\documentclass[reqno]{amsart}
\usepackage{standalone}
\usepackage[utf8]{inputenc}
\usepackage{import}

\usepackage{subcaption} 
\usepackage{caption}

\usepackage[thinlines]{easytable}
\usepackage{amssymb, amsmath, amsthm, enumerate,  mathtools,  }
\usepackage{mathabx} 
\usepackage{esint} 
\usepackage{bm}
\usepackage{tikz}
\usepackage{tikz-3dplot}

\usetikzlibrary{calc,patterns,angles,quotes}

\numberwithin{equation}{section}
\newtheorem{theorem}{Theorem}[section]
\newtheorem{corollary}[theorem]{Corollary}
\newtheorem{lemma}[theorem]{Lemma}
\newtheorem{proposition}[theorem]{Proposition}

\newtheorem{conjecture}[theorem]{Conjecture}

\newtheorem*{claim}{Claim}

\newtheorem*{claim 1}{Claim 1}
\newtheorem*{claim 2}{Claim 2}
\newtheorem*{claim 3}{Claim 3}
\newtheorem*{claim 4}{Claim 4}

\theoremstyle{definition}
\newtheorem{definition}[theorem]{Definition}
\newtheorem{remark}[theorem]{Remark}

\newtheorem{example}[theorem]{Example}

\newtheorem*{acknowledgment}{Acknowledgment}

\newcommand{\C}{\mathbb{C}}
\newcommand{\R}{\mathbb{R}}
\newcommand{\Z}{\mathbb{Z}}
\newcommand{\N}{\mathbb{N}}

\newcommand{\T}{\mathbb{T}}

\newcommand{\ud}{\mathrm{d}}

\renewcommand{\angle}{\measuredangle}

\DeclarePairedDelimiter{\floor}{\lfloor}{\rfloor}
\DeclarePairedDelimiter{\ceil}{\lceil}{\rceil}

\title[Sharp estimates for oscillatory integral operators]{Sharp estimates for oscillatory integral operators via polynomial partitioning}

\author[L. Guth]{ Larry Guth }
\address{Department of Mathematics, Massachusetts Institute of Technology, 182 Memorial Drive, Cambridge, Massachusetts, 02139, USA.}
\email{lguth@math.mit.edu}

\author[J.Hickman]{ Jonathan Hickman }
\address{Eckhart hall Room 414, Department of mathematics, University of Chicago, 5734 S. University Avenue, Chicago, Illinois, 60637, USA.}
\curraddr{Mathematical Institute, North Haugh, St Andrews, Fife, KY16 9SS.}
\email{jeh25@st-andrews.ac.uk}

\author[M. Iliopoulou]{ Marina Iliopoulou }
\address{Department of Mathematics, University of California, Berkeley, CA, 94720-3840, USA.}
\curraddr{Room 262, Sibson Building, Parkwood Road, CT2 7FS.}
\email{m.iliopoulou@kent.ac.uk}

\begin{document}

\begin{abstract} The sharp range of $L^p$-estimates for the class of H\"ormander-type oscillatory integral operators is established in all dimensions under a positive-definite assumption on the phase. This is achieved by generalising a recent approach of the first author for studying the Fourier extension operator, which utilises polynomial partitioning arguments. The main result implies improved bounds for the Bochner--Riesz conjecture in dimensions $n \geq 4$.
\end{abstract}

\maketitle




 \section{Introduction}




\subsection{Statement of results} 

Let $B^{d}$ denote the unit ball in $\R^{d}$ and $\Sigma \colon B^{n-1} \to \R^n$ be a smooth\footnote{In view of the methods of the present article it is convenient to work in the $C^{\infty}$ category, but the forthcoming definitions and questions certainly make sense at lower levels of regularity (in particular, in the $C^2$ class).} parametrisation of a hypersurface. Further let $a \in C_c^{\infty}(\R^{n-1})$ be non-negative and supported in $B^{n-1}$ and suppose $\Sigma$ has non-vanishing Gaussian curvature on the support $\mathrm{supp}\,a$ of $a$. Analytically, this means that $\Sigma$ satisfies the following conditions:
\begin{itemize}
\item[E1)] $\mathrm{rank}\, \partial_{\omega} \Sigma(\omega) = n-1$ for all $\omega \in B^{n-1}$.
\item[E2)] Defining the Gauss map $G \colon B^{n-1} \to S^{n-1}$ by $G(\omega) := \frac{G_0(\omega)}{|G_0(\omega) |}$ where
\begin{equation*}
G_0(\omega) := \bigwedge_{j=1}^{n-1} \partial_{\omega_j} \Sigma(\omega),
\end{equation*}
the curvature condition
\begin{equation*}
\det \partial^2_{\omega \omega} \langle \Sigma(\omega),G(\omega_0)\rangle|_{\omega = \omega_0} \neq 0
\end{equation*}
holds for all $\omega_0 \in \mathrm{supp}\,a$.
\end{itemize}
Here the wedge product of $n-1$ vectors in $\R^n$ is identified with a vector in the usual manner. 

A central problem in harmonic analysis is to understand the Lebesgue space mapping properties of the extension operator $E$ associated to such a parametrised hypersurface. This operator is defined by the formula
\begin{equation}\label{extension operator}
Ef(x) := \int_{B^{n-1}} e^{2 \pi i \langle x, \Sigma(\omega)\rangle} a(\omega) f(\omega)\,\ud \omega
\end{equation}
for all integrable $f \colon B^{n-1} \to \C$. Thus, $E$ is an oscillatory integral operator with associated phase function 
\begin{equation}\label{extension phase}
\phi(x;\omega) := \langle x, \Sigma(\omega)\rangle.
\end{equation} Observe that the parametrisation $\Sigma$ can be recovered from the phase by differentiation; that is,
\begin{equation*}
\partial_{x} \phi(x;\omega) = \Sigma(\omega). 
\end{equation*}
Typically, one is interested in proving local estimates for \eqref{extension operator} of the form\footnote{Given a (possibly empty) list of objects $L$, for real numbers $A_p, B_p \geq 0$ depending on some Lebesgue exponent $p$ the notation $A_p \lesssim_L B_p$ or $B_p \gtrsim_L A_p$ signifies that $A_p \leq CB_p$ for some constant $C = C_{L,n,p} \geq 0$ depending on the objects in the list, $n$ and $p$. In addition, $A_p \sim_L B_p$ is used to signify that $A_p \lesssim_L B_p$ and $A_p \gtrsim_L B_p$.}
\begin{equation}\label{local extension estimate}
\|Ef\|_{L^p(B(0,\lambda))} \lesssim_{\varepsilon} \lambda^{\varepsilon}\|f\|_{L^p(B^{n-1})};
\end{equation}
here the left-hand norm has been localised to a ball of radius $\lambda \geq 1$ and the right-hand constant is allowed some weak dependence on $\lambda$. In particular, the Fourier restriction conjecture asserts that \eqref{local extension estimate} should hold for any $\varepsilon > 0$ in the range $p \geq 2n/(n-1)$. 

In this article the natural variable coefficient generalisations of such extension operators \eqref{extension operator} and estimates \eqref{local extension estimate} are studied. In particular, here more general oscillatory integral operators are considered  whose associated phase function $\phi(x;\omega)$ shares the property of the extension operator that for each $x$ the map $\omega \mapsto \partial_{x} \phi(x;\omega)$ parametrises a hypersurface of non-vanishing Gaussian curvature. Crucially, however, the choice of hypersurface is now allowed to smoothly vary with $x$. 

To formalise this discussion, let $n \geq 2$, $a \in C_c^{\infty}(\R^n \times \R^{n-1})$ be non-negative and supported in $B^n \times B^{n-1}$ and $\phi \colon B^n \times B^{n-1} \to \R$ be a smooth function which satisfies the following conditions:
\begin{itemize}
\item[H1)] $\mathrm{rank}\, \partial_{\omega x}^2 \phi(x;\omega) = n-1$ for all $(x;\omega) \in B^n \times B^{n-1}$.
\item[H2)] Defining the map $G \colon B^n \times B^{n-1} \to S^{n-1}$ by $G(x;\omega) := \frac{G_0(x;\omega)}{|G_0(x;\omega) |}$ where
\begin{equation*}
G_0(x;\omega) := \bigwedge_{j=1}^{n-1} \partial_{\omega_j} \partial_x\phi(x;\omega),
\end{equation*}
the curvature condition
\begin{equation*}
\det \partial^2_{\omega \omega} \langle \partial_x\phi(x;\omega),G(x; \omega_0)\rangle|_{\omega = \omega_0} \neq 0
\end{equation*}
holds for all $(x; \omega_0) \in \mathrm{supp}\,a$.
\end{itemize}

Clearly H1) and H2) agree with E1) and E2) when one restricts to phases of the form \eqref{extension phase}, and this definition therefore leads to a generalisation of the operator $E$ introduced above. 

Suppose  $\phi$ satisfies H1) and H2), for any $\lambda \geq 1$ let $a^{\lambda}(x; \omega) := a(x/\lambda; \omega)$, $\phi^{\lambda}(x;\omega) :=\lambda\phi(x/\lambda;\omega)$ and define the operator $T^{\lambda}$ by
\begin{equation*}
T^{\lambda}f(x) :=  \int_{B^{n-1}} e^{2 \pi i \phi^{\lambda}(x; \omega)}a^{\lambda}(x;\omega) f(\omega)\,\ud \omega
\end{equation*}
for all integrable $f \colon B^{n-1} \to \C$. In this case $T^{\lambda}$ is said to be a \emph{H\"ormander-type operator}. Note that the spatial localisation featured in \eqref{local extension estimate} is now built into the operator. 

\begin{theorem}[Stein \cite{Stein1986}, Bourgain--Guth \cite{Bourgain2011}]\label{general main theorem} Suppose $T^{\lambda}$ is a H\"ormander-type operator. For all $\varepsilon >0$ the estimate 
\begin{equation}\label{general estimate}
\|T^{\lambda}f\|_{L^p(\R^n)} \lesssim_{\varepsilon, \phi, a} \lambda^{\varepsilon}\|f\|_{L^{p}(B^{n-1})}
\end{equation}
holds uniformly for $\lambda \geq 1$ whenever $p$ satisfies
\begin{align}\label{general theorem 1}
p \geq 2 \cdot \frac{n+1}{n-1} & \qquad \textrm{if $n$ is odd;} \\
\nonumber
p \geq 2 \cdot \frac{n+2}{n} & \qquad \textrm{if $n$ is even.}
\end{align}
\end{theorem}

The odd dimensional case is due to Stein \cite{Stein1986}, who in fact showed that the above estimates are valid for $p \geq 2(n+1)/(n-1)$ in all dimensions  without the $\lambda^{\varepsilon}$-loss. The strengthened results in even dimensions were established much later by Bourgain and the first author \cite{Bourgain2011}.\footnote{Strictly speaking, in \cite{Bourgain2011} weaker $L^{\infty}-L^p$ bounds are proven, but the methods can be used to establish the $L^p-L^p$ strengthening: see, for instance, \cite[\S9]{Guth2018} where the $L^p-L^p$ argument appears (although in a slightly disguised form).} A detailed history of this problem is provided later in the introduction. It is remarked that Theorem~\ref{general main theorem} is sharp in the sense that there are examples of H\"ormander-type operators for which \eqref{general estimate} fails whenever $p$ does not satisfy \eqref{general theorem 1}. Such examples originate from work of Bourgain \cite{Bourgain1991} and are discussed in detail in \S\ref{previous results section} and \S\ref{Necessary conditions section}.

The majority of this work concerns the case where the phase satisfies a strengthened version of H2), namely:
\begin{itemize}
\item[H2$^+$)] For all $(x; \omega_0) \in \mathrm{supp}\,a$ the matrix 
\begin{equation*}
\partial^2_{\omega \omega} \langle \partial_x\phi(x;\omega),G(x;\omega_0)\rangle|_{\omega = \omega_0}
\end{equation*}
is positive-definite.
\end{itemize}
If $\phi$ satisfies H1) and H2$^+$), then $T^{\lambda}$ is said to be a \emph{H\"ormander-type operator with positive-definite phase}. Geometrically, this condition implies that the principal curvatures of the hypersurface parametrised by $\omega \mapsto \partial_x \phi(x;\omega)$ are everywhere positive. A hypersurface satisfying this condition is said to be \emph{positively-curved}. 

Lee \cite{Lee2006} observed that for positive-definite phases one may prove estimates beyond the range of Theorem~\ref{general main theorem}\footnote{In particular, Lee \cite{Lee2006} proved that for positive-definite phases \eqref{general estimate} holds for $p\geq 2\cdot\frac{n+2}{n}$ in \textit{all} dimensions, extending the range in Theorem~\ref{general main theorem} when $n$ is odd.}. The main result of this article provides sharp estimates in this setting.

\begin{theorem}\label{main theorem} Suppose $T^{\lambda}$ is a H\"ormander-type operator with positive-definite phase. For all $\varepsilon  > 0$ the estimate 
\begin{equation}\label{Hormander estimate}
\|T^{\lambda}f\|_{L^p(\R^n)} \lesssim_{\varepsilon, \phi, a} \lambda^{\varepsilon}\|f\|_{L^{p}(B^{n-1})}
\end{equation}
holds for all $\lambda \geq 1$ whenever $p$ satisfies
\begin{align}\label{main theorem 1}
p \geq 2 \cdot \frac{3n+1}{3n-3} & \qquad \textrm{if $n$ is odd;} \\
\nonumber
p \geq 2 \cdot \frac{3n+2}{3n-2} & \qquad \textrm{if $n$ is even.}
\end{align}
\end{theorem}

This result improves upon the previous best results of Lee \cite{Lee2006} and Bourgain and the first author \cite{Bourgain2011}. Moreover, it is sharp in the sense that there are H\"ormander-type operators with positive-definite phase for which \eqref{Hormander estimate} fails whenever $p$ does not satisfy \eqref{main theorem 1}. Examples of this kind appear, for instance, in \cite{Bourgain2011, Minicozzi1997} and are discussed in detail in \S\ref{previous results section} and \S\ref{Necessary conditions section}.  It is remarked that range of $p$ obtained in \cite{Lee2006, Bourgain2011} agrees with Theorem \ref{main theorem} for $n=3$ and therefore the sharp result in this case is due to Lee \cite{Lee2006} and Bourgain and the first author \cite{Bourgain2011}. In all higher dimensions \eqref{main theorem 1} is a strictly larger range of $p$ than what was previously known. Finally, away from the endpoint values one may apply $\varepsilon$-removal techniques to establish \eqref{Hormander estimate} without the $\lambda^{\varepsilon}$-loss in the constant (see $\S$\ref{Epsilon removal section}, below). 



\subsection{Applications to the Bochner--Riesz problem}

By the standard Carleson--Sj\"olin reduction \cite{Carleson1972} (see also \cite{Hormander1973}), Theorem~\ref{main theorem} implies new $L^p$-bounds on Bochner--Riesz multipliers in dimensions $n \geq 4$. Theorem~\ref{main theorem} also implies bounds for the oscillatory integral operators of Minicozzi--Sogge \cite{Minicozzi1997} which arise in relation to the Bochner--Riesz problem on compact manifolds (see \cite[Chapter 5]{Sogge2017}). Moreover, for certain choices of manifold, these estimates are sharp. In this section these well-known applications are briefly reviewed.

\subsubsection*{Euclidean Bochner--Riesz} For $\alpha \geq 0$ the \emph{Bochner--Riesz multiplier of order $\alpha$} is the function
\begin{equation*}
    m^{\alpha}(\xi) := (1 - |\xi|)_+^{\alpha}
\end{equation*}
where $(t)_+ := t$ if $t > 0$ and 0 otherwise.

\begin{corollary}\label{Bochner--Riesz corollary} If $p$ satisfies the condition in \eqref{main theorem 1}, then
\begin{equation}\label{Bochner--Riesz bound}
    \|m^{\alpha}(D)f\|_{L^p(\R^n)} \lesssim_{\alpha} \|f\|_{L^p(\R^n)} \qquad \textrm{for $\alpha > \alpha(p) := \max \Big\{n \big| \frac{1}{2} - \frac{1}{p}\big| - \frac{1}{2}, 0 \Big\}$.}
\end{equation}
\end{corollary}

Here $m^{\alpha}(D)$ is the multiplier operator associated to $m^{\alpha}$, defined \emph{a priori} by
\begin{equation*}
    m^{\alpha}(D)f(x) := \int_{\hat{\R}^n} e^{2 \pi i \langle x, \xi \rangle} m^{\alpha}(\xi) \hat{f}(\xi)\,\ud \xi,
\end{equation*}
where $\hat{f}$ denotes the Fourier transform of $f$. 

Recall that the Bochner--Riesz conjecture asserts that \eqref{Bochner--Riesz bound} holds for all $p \geq 2n/(n-1)$.\footnote{Once \eqref{Bochner--Riesz bound} is known in the range $p \geq 2n/(n-1)$, it immediately extends to all $1 \leq p \leq \infty$ by duality and interpolation with the trivial $p = 2$ case.} The conjecture was resolved for $n=2$ by Carleson--Sj\"olin \cite{Carleson1972} but remains open in all higher dimensions. Corollary~\ref{Bochner--Riesz corollary} provides some progress towards this conjecture when $n \geq 4$, improving over earlier results of Fefferman--Stein (see \cite{Fefferman1970}), Lee \cite{Lee2004}, Bourgain--Guth \cite{Bourgain2011} and others. When $n = 3$ the range in Corollary~\ref{Bochner--Riesz corollary} matches that of Lee \cite{Lee2006} and Bourgain--Guth \cite{Bourgain2011}. 

For completeness, here a sketch is provided to show how one may deduce Corollary~\ref{Bochner--Riesz corollary} from Theorem~\ref{main theorem}. This follows a standard argument due to Carleson and Sj\"olin \cite{Carleson1972}. 

A stationary phase computation shows that the kernel $K^{\alpha} := (m^{\alpha})\;\widecheck{}\;$ of $m^{\alpha}(D)$ is given by
\begin{equation*}
    K^{\alpha}(x) = \sum_{\pm} \frac{a_{\pm}(x) e^{\pm 2 \pi i |x|}}{(1 + |x|)^{\frac{n+1}{2} + \alpha}} 
\end{equation*}
where the $a_{\pm}$ are symbols of order 0 in the sense that $|\partial_x^{\beta}a_{\pm}(x)| \lesssim_{\beta} |x|^{-|\beta|}$ for all multi-indices $\beta \in \N_0^n$. After applying a dyadic decomposition, the Bochner--Riesz conjecture is therefore reduced to bounding the \emph{Carleson--Sj\"olin operators} 
\begin{equation}\label{Carleson--Sjolin operator}
    S^{\lambda}f(x) := \int_{\R^n} e^{2 \pi i \lambda|x-y|} a(x,y) f(y)\,\ud y,
\end{equation} 
where $a \in C^{\infty}(\R^n \times \R^n)$ has compact support bounded away from the diagonal $\{(x,x) : x \in \R^n\}$. In particular, Corollary~\ref{Bochner--Riesz corollary} is a consequence of the following bound.

\begin{corollary} If $p$ satisfies the conditions in \eqref{main theorem 1} and $\varepsilon > 0$, then
\begin{equation}\label{Carleson--Sjolin bound}
    \|S^{\lambda}f\|_{L^p(\R^n)} \lesssim_{\varepsilon} \lambda^{(n-1)/2 - n/p + \varepsilon} \|f\|_{L^p(\R^n)} \qquad \textrm{for all $\lambda \geq 1$.}
\end{equation}
\end{corollary}

Theorem~\ref{main theorem} can be used to prove estimates of the form \eqref{Carleson--Sjolin bound}. In particular, one may write $S^{\lambda}$ as a superposition of operators $T^{\lambda}_{y_n}$ for which the $y_n$ variable is frozen:
\begin{equation*}
    S^{\lambda}f(x) = \int_{\R} T^{\lambda}_{y_n}f_{y_n}(x)\,\ud y_n, \qquad f_{y_n}(y') := f(y',y_n).
\end{equation*}
 It is not difficult to check that each $T^{\lambda}_{y_n}$ is a H\"ormander-type operator with positive-definite phase. Theorem~\ref{main theorem} can be applied to the $T^{\lambda}_{y_n}$ with a uniform constant (this uniformity is a consequence of the proof) and Minkowski's inequality can then be used to convert these estimates into bounds for $S^{\lambda}$, yielding Corollary~\ref{Bochner--Riesz corollary}.

\subsubsection*{Bochner--Riesz over compact manifolds} The classical Bochner--Riesz multipliers have natural analogues defined over compact Riemannian manifolds $(M,g)$ without boundary. In this setting, one defines the multiplier operator $m^{\alpha}(D)$ in terms of spectral projectors associated to an eigenbasis for the Laplace--Beltrami operator $-\Delta_g$; see \cite{Sogge1987} or \cite[Chapter 5]{Sogge2017} for further details. Unlike in the euclidean case, Theorem~\ref{main theorem} does not, in general, directly lead to new $L^p$ bounds for the spectral multipliers $m^{\alpha}(D)$. Nevertheless, as described presently, Theorem~\ref{main theorem} does provide bounds for certain variants of the operator \eqref{Carleson--Sjolin operator} which, in some sense, control the ``local behaviour'' of the Bochner--Riesz multipliers on $(M,g)$.

When studying Bochner--Riesz multipliers in the manifold setting, one is led to consider certain variants of the Carleson--Sj\"olin operator \eqref{Carleson--Sjolin operator}, defined by
\begin{equation}\label{manifold Carleson--Sjolin}
    S^{\lambda}_gf(x) := \int_{M} e^{2 \pi i \lambda \mathrm{dist}_g(x,y)}a(x,y) f(y)\,\ud y;
\end{equation}
 here $a \in C^{\infty}(M \times M)$ is supported away from the diagonal whilst $\mathrm{dist}_g$ and $\ud y$ are, respectively, the distance function and measure on $M$ induced by the Riemannian metric. Operators of this form were studied previously by Minicozzi--Sogge \cite{Minicozzi1997} (see also \cite[Chapter 5]{Sogge2017}).

Working in local co-ordinates and then arguing as in the euclidean case, one may use Theorem~\ref{main theorem} to prove the following bound.

\begin{corollary}\label{Bochner--Riesz manifold corollary} Suppose $(M,g)$ is a compact Riemannian manifold of dimension $n \geq 2$ without boundary. If $p$ satisfies the conditions in \eqref{main theorem 1} and $\varepsilon > 0$, then
\begin{equation*}
    \|S_g^{\lambda}f\|_{L^p(M)} \lesssim_{\varepsilon, g} \lambda^{(n-1)/2 - n/p + \varepsilon} \|f\|_{L^p(M)} \qquad \textrm{for all $\lambda \geq 1$.}
\end{equation*}
\end{corollary}
An interesting feature of this result is that there are examples of manifolds $(M,g)$ for which the range of exponents \eqref{main theorem} is sharp: see \cite{Minicozzi1997}. 

There is an analogue of the Carleson--Sj\"olin reduction in the manifold setting, which relies on the Hadamard parametrix for the wave equation on $(M,g)$. This can be used to prove $L^p$ bounds for Bochner--Riesz multipliers over $(M,g)$ in the restricted range $p \geq 2(n+1)/(n-1)$: see \cite[Chapter 5]{Sogge2017}. Unfortunately, the parametrix is only effective for short time intervals and, as a consequence, Corollary~\ref{Bochner--Riesz manifold corollary} does not appear to directly imply any new bounds for Bochner--Riesz multipliers over general $(M,g)$. In particular, it seems that it is necessary to combine Corollary~\ref{Bochner--Riesz manifold corollary} with global geometric information to fully understand the Bochner--Riesz problem. For $p \geq 2(n+1)/(n-1)$ such difficulties can be overcome using $L^p$ eigenfunction estimates of Sogge \cite{Sogge1988}  (see \cite{Sogge1987} or \cite[Chapter 5]{Sogge2017}), but the method appears to be tied down to this restricted range of exponents. 

Finally, it is remarked that new Bochner--Riesz estimates \textit{can} be obtained for certain specific choices of manifold $(M,g)$ which enjoy additional symmetries. The simplest example is the flat torus $\mathbb{T}^n$; indeed, Corollary~\ref{Bochner--Riesz corollary} implies similar $L^p$ bounds in the toral setting via the classical multiplier transference principle. A more involved example is the $n$-dimensional Euclidean sphere $S^n$.\footnote{The authors are grateful to Christopher D. Sogge for drawing their attention to this example.} In this case, using the periodicity of the geodesic flow, one may entirely reduce the Bochner--Riesz problem to bounding operators essentially of the form \eqref{manifold Carleson--Sjolin}, as observed in \cite{Sogge1986} (see also \cite{MSS}). It may be possible to extend these methods to treat the class of Zoll manifolds, following a line of investigation initiated in \cite{MSS}.



\subsection{Historical remarks}\label{previous results section}

The problem of determining $L^p$ estimates for H\"ormander-type operators has an interesting history. H\"ormander  \cite{Hormander1973} asked whether \eqref{Hormander estimate} holds for $p > 2n/(n-1)$ (without $\varepsilon$-loss) under the hypotheses H1) and H2) only and proved that this is indeed the case when $n=2$. This numerology agrees with the Fourier restriction conjecture and also the Bochner--Riesz multiplier problem, both of which would follow from a positive answer to H\"ormander's question.\footnote{The connection with Bochner--Riesz multipliers is made via the classical reduction of Carleson--Sj\"olin \cite{Carleson1972, Hormander1973}, as discussed in the previous subsection.} Stein \cite{Stein1986} provided further evidence for this numerology by proving the estimate
\begin{equation*}
\|T^{\lambda}f\|_{L^p(\R^n)} \lesssim_{\phi, a} \|f\|_{L^2(B^{n-1})} \qquad \textrm{for all $p \geq 2\cdot \frac{n+1}{n-1}$,}
\end{equation*} 
matching what was known about the high-dimensional Fourier restriction and Bochner--Riesz problems at that time. It was therefore somewhat surprising when Bourgain \cite{Bourgain1991} showed that, in general, Stein's theorem is sharp. In particular, he demonstrated that for every odd dimension $n \geq 3$ there exists a H\"ormander-type operator for which
\begin{equation}\label{L infinity bound}
\|T^{\lambda}f\|_{L^p(\R^n)} \lesssim_{\phi, a} \|f\|_{L^{\infty}(B^{n-1})}
\end{equation} 
fails to hold uniformly in $\lambda \geq 1$ whenever $p < 2(n+1)/(n-1)$. Aside from answering H\"ormander's original question in the negative, Bourgain's work hinted at an interesting divergence between the odd and even-dimensional theory. Moreover, it was noted in \cite[p.87]{Bourgain1995} that in even dimensions the $L^{\infty}-L^p$ estimates always hold in a wider range than that of Stein's theorem. Thus, in general, the even-dimensional case is better behaved than the odd-dimensional case. This was further highlighted by Bourgain and the first author \cite{Bourgain2011} who showed that in even dimensions \eqref{L infinity bound} holds for $p > 2(n+2)/n$. Furthermore, in \cite{Bourgain2011} and also implicitly in the work of Wisewell \cite{Wisewell2005}, examples were found in even dimensions which show that \eqref{L infinity bound} can fail for $p < 2(n+2)/n$. Thus, the range of exponents in H\"ormander's original question is valid only when $n=2$.

At this point it is useful to describe the nature of the counter-examples of \cite{Bourgain1991}, \cite{Wisewell2005} and \cite{Bourgain2011} and provide some explanation for the difference between the odd and even-dimensional cases. Roughly speaking, in the odd-dimensional case $T^{\lambda}$ and $f$ can be chosen so that $|T^{\lambda}f|$ is concentrated in the 1-neighbourhood of a low degree algebraic variety $Z$ of dimension $(n+1)/2$. This is the smallest possible dimension for which such concentration is possible. In the even-dimensional case $(n+1)/2$ is no longer an integer, and it transpires that $|T^{\lambda}f|$ can only be concentrated into the $1$-neighbourhood of a variety of relatively large dimension $(n+2)/2$. These observations are related to Kakeya compression phenomena for sets of space curves (see \cite{Wisewell2005} for a thorough introduction to this topic and \cite{Minicozzi1997} for the related problem of Kakeya sets of geodesics in Riemannian manifolds). They also hint at some underlying algebraic structure in the problem.

So far the discussion has focused on operators satisfying the original H1) and H2) hypotheses of H\"ormander. Lee \cite{Lee2006} observed that under the positive-definite hypothesis H2$^+$) one can establish improvements over the range given by Stein's theorem in all dimensions. In particular, he showed that for any H\"ormander-type operator with positive-definite phase \eqref{Hormander estimate} holds for $p \geq 2(n+2)/n$. This coincides with Theorem~\ref{main theorem} when $n=3$, but is weaker in higher dimensions. Wisewell \cite{Wisewell2005} and Minicozzi--Sogge \cite{Minicozzi1997} produced examples (again relying on Kakeya compression phenomena) to show that this result is sharp when $n=3$ (see also \cite[\S6]{Bourgain2011}). 

Comparing the sharp examples under the H2) and H2$^+$) hypotheses highlights another important consideration in addition to Kakeya compression phenomena. This feature relates to how the mass of $|T^{\lambda}f|$ can be distributed in neighbourhoods of low degree varieties. It accounts for the improved behaviour demonstrated by operators with positive-definite phase and is described in detail in $\S$\ref{Necessary conditions section}. 

Given the results of this article, the sharp range of estimates for this problem are now understood under either the H2) or H2$^+$) hypothesis. The corresponding endpoint values for $p$ are concisely tabulated in Figure~\ref{Hormander endpoints}.

\begin{figure}
\begin{center}
\begin{TAB}(c,1cm,2cm)[5pt]{|c|c|c|}{|c|c|c|}
 & $n$ odd & $n$ even    \\
H2) & $\displaystyle 2\cdot\frac{n+1}{n-1}$ & $\displaystyle 2\cdot\frac{n+2}{n}$  \\
H2$^+$) & $\displaystyle 2\cdot\frac{3n+1}{3n-3}$ & $\displaystyle 2\cdot\frac{3n+2}{3n-2}$ \\
\end{TAB}
\caption{Endpoint values for $p$ for H\"ormander-type operators under various hypotheses.}
\label{Hormander endpoints}
\end{center}
\end{figure}
It is remarked that it is possible to prove estimates beyond the range of Theorem~\ref{main theorem} under additional assumptions on the phase function. For example, the first author \cite{Guth2016} has shown that for $n=3$ and all $\varepsilon > 0$ the extension operator $E_{\mathrm{par}}$ associated to the paraboloid satisfies
\begin{equation*}
\|E_{\mathrm{par}}f\|_{L^p(B(0,\lambda))} \lesssim_{\varepsilon} \lambda^{\varepsilon} \|f\|_{L^{p}(B^{2})}
\end{equation*}
for all $p \geq 3 + 1/4$ and this was further improved to $p \geq 3 + 3/13$ by Wang \cite{Wang}. Furthermore, the aforementioned restriction conjecture asserts that the above inequality should be valid in the wider range $p \geq 3$. 

\subsection{Multilinear estimates} The proof of Theorem~\ref{main theorem} follows the strategy introduced by the first author in \cite{Guth2018}. The argument relies on establishing (weakened versions of) multilinear estimates for H\"ormander-type operators. The multilinear approach was introduced in the late 1990s to study oscillatory integral operators (although it was arguably already implicit in many earlier foundational works in the subject \cite{Fefferman1970, Carleson1972}) and has proven an invaluable tool. To describe the $k$-linear setup one first requires the notion of transversality. 

\begin{definition}\label{transverse definition}
Let $1 \leq k \leq n$ and $\mathbf{T}^{\lambda} = (T_1^{\lambda}, \dots, T_k^{\lambda})$ be a $k$-tuple of H\"ormander-type operators, where $T^{\lambda}_j$ has associated phase $\phi_j^{\lambda}$, amplitude $a_j^{\lambda}$ and generalised Gauss map $G_j$ for $1 \leq j \leq  k$. Then $\mathbf{T}^{\lambda}$ is said to be $\nu$-transverse for some $0 < \nu \leq 1$ if
\begin{equation*}
\big|\bigwedge_{j=1}^k G_j(x;\omega_j) \big| \geq \nu \quad \textrm{for all  $(x; \omega_j) \in \mathrm{supp}\,a_j$  for $1 \leq j \leq k$. }
\end{equation*}
\end{definition}

The following conjecture is a natural generalisation of an existent conjecture of Bennett \cite{Bennett2014} for Fourier extension operators.

\begin{conjecture}[$k$-linear H\"ormander conjecture]\label{k-linear conjecture} Let $1 \leq k \leq n$ and suppose $ (T_1^{\lambda}, \dots, T_k^{\lambda})$ is a $\nu$-transverse $k$-tuple of H\"ormander-type operators with positive-definite phase functions. For all  $p \geq \bar{p}(k, n) := 2(n+k)/(n+k -2)$ and $\varepsilon > 0$ the estimate
\begin{equation}\label{k-linear inequality}
\big\| \prod_{j=1}^k |T^{\lambda}_jf_j|^{1/k}\big\|_{L^p(\R^n)} \lesssim_{\varepsilon, \nu, \phi} \lambda^{\varepsilon} \prod_{j=1}^k \|f_j\|_{L^2(B^{n-1})}^{1/k}
\end{equation}
holds for all $\lambda \geq 1$. 
\end{conjecture}

Techniques have been developed by Tao--Vargas--Vega \cite{Tao1998} and Bourgain and the first author \cite{Bourgain2011} to convert $k$-linear into linear inequalities. There are a number of features of the multilinear theory which suggest that it is more approachable than directly tackling the linear estimates. For instance, here the desired inequalities are $L^2$-based, giving greater scope for orthogonality methods.

Some instances of the conjecture are known. 
\begin{itemize} 
\item The $k=1$ case corresponds to Stein's theorem \cite{Stein1986} (which holds without the positive-definite hypothesis).  
\item The $k=2$ case was established by Lee \cite{Lee2006}, who then used the method of Tao--Vargas--Vega \cite{Tao1998} to derive estimates for the linear problem. This approach yields Theorem~\ref{main theorem} in the $n=3$ case, but produces strictly weaker results in higher dimensions (see the discussion in $\S$\ref{previous results section}).
\item The $k=n$ case was established by Bennett--Carbery--Tao \cite{Bennett2006} who also gave partial results at all levels of multilinearity (see also \cite{Bennett2014} for further discussion of this work). Bourgain and the first author \cite{Bourgain2011} later developed a method to deduce improved linear estimates from these multilinear inequalities. 
\end{itemize}

The precise statement of the Bennett--Carbery--Tao theorem \cite{Bennett2006} is as follows.

\begin{theorem}[Bennett--Carbery--Tao \cite{Bennett2006}]\label{Bennett Carbery Tao theorem} Let $2 \leq k \leq n$ and suppose that $ (T_1^{\lambda}, \dots, T_k^{\lambda})$ is a $\nu$-transverse $k$-tuple of H\"ormander-type operators. For all $p \geq 2k/(k-1)$ and $\varepsilon > 0$ the estimate
\begin{equation*}
\big\| \prod_{j=1}^k |T^{\lambda}_jf_j|^{1/k}\big\|_{L^p(\R^n)} \lesssim_{\varepsilon ,(\phi_j)_{j=1}^k} \nu^{-C_{\varepsilon}}\lambda^{\varepsilon} \prod_{j=1}^k \|f_j\|_{L^2(B^{n-1})}^{1/k}
\end{equation*}
 holds for all $\lambda \geq 1$. 
\end{theorem}
The positive-definite assumption does not appear in the hypotheses of Theorem~\ref{Bennett Carbery Tao theorem}.\footnote{In particular, of the results mentioned above only Lee's bilinear estimate \cite{Lee2006} exploits the positive-definite hypothesis.} Combining Theorem~\ref{Bennett Carbery Tao theorem} with the method of \cite{Bourgain2011} leads to the sharp estimates for H\"ormander-type operators stated in Theorem~\ref{general main theorem}. For completeness, the details of this argument are given in $\S$\ref{k-broad to linear section}. 

\subsection{$k$-broad estimates}\label{k-broad subsection}

In \cite{Guth2018} it was observed that, in the context of Fourier extension operators, the method of \cite{Bourgain2011} does not require the full power of the $k$-linear theory, but rather can take as its input inequalities of a weaker form than \eqref{k-linear inequality} known as \emph{$k$-broad estimates}. By applying polynomial partitioning techniques, the first author \cite{Guth2018} was further able to prove the sharp range of $L^2$-based\footnote{These estimates have an $L^2$-norm appearing on the right-hand side. Relaxing $L^2$ to $L^{\infty}$ has led to further improvements on the Fourier restriction problem for the paraboloid \cite{Wang, HickmanRogers2018}.} $k$-broad estimates for the Fourier extension operator associated to the paraboloid. This led to an improvement on the known range of estimates for parabolic restriction in dimensions $n \geq 4$. The main goal of this paper is to extend the theory of $k$-broad estimates to the more general context of H\"ormander-type operators with positive-definite phase. 

The $k$-broad setup involves the notion of a $k$-broad norm, which was introduced in \cite{Guth2018}. Decompose $B^{n-1}$ into finitely-overlapping balls $\tau$ of radius $K^{-1}$, where $K$ is a large constant. These balls will be frequently referred to as \emph{$K^{-1}$-caps}. Given a function $f \colon B^{n-1} \to \C$ write $f = \sum_{\tau} f_{\tau}$ where $f_{\tau}$ is supported in $\tau$. In view of the rescaling $\phi^{\lambda}$ of the phase function, define the rescaled generalised Gauss map 
\begin{equation*}
G^{\lambda}(x; \omega) := G(x/\lambda;\omega) \qquad \textrm{for $(x;\omega) \in \mathrm{supp}\, a^{\lambda}$.}
\end{equation*}
For each $x \in B(0, \lambda)$ there is a range of normal directions associated to the cap $\tau$ given by
\begin{equation*}
G^{\lambda}(x; \tau) := \{G^{\lambda}(x; \omega) : \omega \in \tau,\, (x; \omega) \in \mathrm{supp}\, a^{\lambda} \}.
\end{equation*}
If $V \subseteq \R^n$ is a linear subspace, then let $\angle(G^{\lambda}(x; \tau), V)$ denote the smallest angle between any non-zero vector $v \in V$ and $v' \in G^{\lambda}(x;\tau)$. 

The spatial ball $B(0, \lambda)$ is also decomposed into relatively small balls $B_{K^2}$ of radius $K^2$. In particular, fix $\mathcal{B}_{K^2}$ a collection of finitely-overlapping $K^2$-balls which are centred in and cover $B(0, \lambda)$. For $B_{K^2} \in \mathcal{B}_{K^2}$ centred at some point $\bar{x} \in B(0, \lambda)$ define
\begin{equation}\label{mu definition}
\mu_{T^{\lambda}f}(B_{K^2}) := \min_{V_1,\dots,V_A \in \mathrm{Gr}(k-1, n)} \Big(\max_{\tau : \angle(G^{\lambda}(\bar{x}; \tau), V_a) > K^{-1}  \textrm{ for } 1 \leq a \leq A} \|T^{\lambda}f_{\tau}\|_{L^p(B_{K^2})}^p \Big);
\end{equation}
here $\mathrm{Gr}(k-1, n)$ is the Grassmannian manifold of all $(k-1)$-dimensional subspaces in $\R^n$. It will often be notationally convenient to write $\tau \notin V_a$ to mean that $\angle(G^{\lambda}(\bar{x}; \tau), V_a) > K^{-1}$ (the choice of centre $\bar{x}$ should always be clear from the context); with this notation the above expression becomes
\begin{equation*}
\mu_{T^{\lambda}f}(B_{K^2}) := \min_{V_1,\dots,V_A \in \mathrm{Gr}(k-1, n)} \Big(\max_{\tau : \tau \notin V_a  \textrm{ for } 1 \leq a \leq A} \|T^{\lambda}f_{\tau}\|_{L^p(B_{K^2})}^p \Big).
\end{equation*}
 For $U\subseteq \R^n$ the $k$-broad norm over $U$ is then defined to be
\begin{equation}\label{k-broad norm definition}
\|T^{\lambda}f\|_{\mathrm{BL}^p_{k,A}(U)} := \Big(\sum_{\substack{B_{K^2} \in \mathcal{B}_{K^2} \\ B_{K^2} \cap U \neq \emptyset}} \mu_{T^{\lambda}f}(B_{K^2}) \Big)^{1/p}.
\end{equation}
It is remarked that $\|T^{\lambda}f\|_{\mathrm{BL}^p_{k,A}(U)}$ is not a norm in the traditional sense, but it does satisfy weak variants of certain key properties of $L^p$-norms, as discussed below in $\S$\ref{Broad norm section}. 

Theorem~\ref{main theorem} will be a consequence of certain estimates for $k$-broad norms. These estimates are proved under a further technical assumption that the phase is of \emph{reduced form}. The details of this condition are postponed until $\S$\ref{Reductions section}.

\begin{theorem}\label{k-broad theorem} For $2 \leq k \leq n$ and all $\varepsilon  > 0$ there exists a constant $C_{\varepsilon} > 1$ and an integer $A$ such that whenever $T^{\lambda}$ is a H\"ormander-type operator with positive-definite reduced phase the estimate
\begin{equation}\label{k-broad inequality}
\|T^{\lambda}f\|_{\mathrm{BL}^p_{k,A}(\R^n)} \lesssim_{\varepsilon} K^{C_{\varepsilon}} \lambda^{\varepsilon} \|f\|_{L^2(B^{n-1})}
\end{equation}
holds for all $\lambda \geq 1$  and $K \geq 1$ whenever $p \geq \bar{p}(k, n) := 2(n+k)/(n+k -2)$. 
\end{theorem}

The range of $p$ is sharp for this theorem, as can be seen by considering the extension operator associated to the (elliptic) paraboloid (see \cite{Guth2018}). As explained in $\S$\ref{Broad norm section} below, the $k$-broad estimate \eqref{k-broad inequality} is weaker than the corresponding $k$-linear estimate \eqref{k-linear inequality}, and so Theorem~\ref{k-broad theorem} can be viewed as a weak substitute for Conjecture~\ref{k-linear conjecture}.

To derive $L^p$ estimates from Theorem~\ref{k-broad theorem}, roughly speaking, one argues as follows. The global $L^p$ norm is broken up into contributions over balls $B_{K^2}$; the problem is to estimate each $\|T^{\lambda}f\|_{L^p(B_{K^2})}^p$ and then to sum these estimates in $B_{K^2}$. Fixing one such ball, there exists a collection of $(k-1)$-dimensional subspaces $V_1, \dots, V_A$ such that
\begin{equation*}
\mu_{T^{\lambda}f}(B_{K^2}) = \max_{\tau \notin V_a  \textrm{ for } 1 \leq a \leq A} \|T^{\lambda}f_{\tau}\|_{L^p(B_{K^2})}^p.
\end{equation*}
Thus, by the triangle and H\"older's inequalities,
\begin{equation*}
\|T^{\lambda}f\|_{L^p(B_{K^2})}^p \lesssim_A K^{O(1)}\mu_{T^{\lambda}f}(B_{K^2}) + \sum_{a = 1}^A\big\|\sum_{\tau \in V_a} T^{\lambda}f_{\tau}\big\|_{L^p(B_{K^2})}^p.
\end{equation*}
The $k$-broad estimate \eqref{k-broad norm definition} effectively controls the first term on the right-hand side of the above display, after summing over all $B_{K^2}$. The problem of estimating $\|T^{\lambda}f\|_{L^p(B_{R})}^p$ is therefore reduced to studying expressions of the form
\begin{equation*}
\big\|\sum_{\tau \in V} T^{\lambda}f_{\tau}\big\|_{L^p(B_{K^2})}^p
\end{equation*}
for each $B_{K^2}$, where the sum is over caps $\tau$ which make a small angle with some $(k-1)$-dimensional subspace $V$. This term can then be controlled using a combination of $\ell^p$-decoupling and an induction on scales argument, leading to the proof of Theorem~\ref{main theorem}. The full details of this argument are given in $\S$\ref{k-broad to linear section}.




\subsection*{Structure of the article} The structure of this article is as follows:
\begin{itemize}
\item In $\S$\ref{Necessary conditions section} sharp examples for Theorem~\ref{general main theorem} and Theorem~\ref{main theorem} are discussed in detail.
\item In $\S$\ref{Proof sketch section} the key features of the problem are identified in order to motivate the forthcoming analysis. 
\item In $\S$\ref{Reductions section} some basic reductions are described which allow one to assume the phase is of a certain \emph{reduced form} in the proof of Theorem~\ref{main theorem}.  
\item In $\S$\ref{Wave packet section} and $\S$\ref{Broad norm section} some basic analytic tools are introduced. In particular, the wave packet decomposition for H\"ormander-type operators is defined and studied, some elementary aspects of the $L^2$-theory for H\"ormander-type operators are reviewed, and there  is also a discussion of the basic properties of the $k$-broad norms and their relation to $k$-linear estimates.
\item In $\S$\ref{Algebraic preliminaries section} certain algebraic tools from combinatorial geometry are introduced.  In particular, polynomial partitioning techniques are reviewed and some important geometric lemmas are proved; these techniques will play a fundamental r\^ole in the proof of Theorem~\ref{k-broad theorem}. 
\item In $\S$\ref{Transverse equidistribution section} and $\S$\ref{Adjusting wave packets section} \emph{transverse equidistribution estimates} for $T^{\lambda}$ are introduced and studied. These estimates rely heavily on the positive-definite hypothesis and partially account for the improved behaviour exhibited by operators satisfying the H2$^+$) hypothesis. 
\item In $\S$\ref{Proof section} the proof of $k$-broad estimates of Theorem~\ref{k-broad theorem} is given.
\item In $\S$\ref{k-broad to linear section} the linear estimates of Theorem~\ref{main theorem} are deduced as a consequence of the $k$-broad estimates of Theorem~\ref{k-broad theorem}. For completeness, the same methods are also applied to deduce Theorem~\ref{general main theorem} as a consequence of Corollary~\ref{Bennett Carbery Tao corollary}. 
\item In $\S$\ref{Epsilon removal section} standard $\varepsilon$-removal lemmas are generalised to the variable coefficient setting. This allows one to strengthen Theorem~\ref{main theorem} away from the endpoint by removing the $\lambda^{\varepsilon}$-dependence in the constant.
\item Appended are some remarks concerning (non)-stationary phase arguments used throughout the paper.
\end{itemize}

\begin{acknowledgment} The authors would like to thank the anonymous referees for very thorough and tremendously helpful reports. The authors would also like to thank David Beltran and Christopher D. Sogge for engaging discussions on topics related to this article. The first author is supported by a Simons Investigator Award. This material is based upon work supported by the National Science Foundation under Grant No. DMS-1440140 while the second and third authors were in residence at the Mathematical Sciences Research Institute in Berkeley, California, during the Spring 2017 semester.  
\end{acknowledgment}




\section{Necessary conditions}\label{Necessary conditions section}




\subsection{An overview of the sharp examples} In this section examples of H\"ormander-type operators are studied in view of establishing the necessity of the conditions on the $p$ exponent in the linear estimates of Theorem~\ref{general main theorem} and Theorem~\ref{main theorem}. This analysis will also identify some key features of operators with positive-definite phase which will later be exploited in the proof of Theorem~\ref{main theorem}. As discussed in \S\ref{previous results section}, such examples first arose in the work of Bourgain \cite{Bourgain1991, Bourgain1995} and were later developed by Wisewell \cite{Wisewell2005}, Minicozzi--Sogge \cite{Minicozzi1997} and Bourgain--Guth \cite{Bourgain2011}, amongst others. The presentation in this section follows the lines of \cite{Bourgain1991, Bourgain2011}.

All the examples considered here are of the following general form: for a fixed operator $T^{\lambda}$, a function $f$ is chosen so that $|f|$ is constant whilst $|T^{\lambda}f|$ is concentrated in $N_{\lambda^{\sigma}}(Z) \cap B(0,\lambda)$ for some low-degree algebraic variety $Z$ with $\dim Z = m$; here $N_{\lambda^{\sigma}}(Z)$ is the $\lambda^{\sigma}$-neighbourhood of $Z$. In particular, one has 
\begin{equation}\label{universal example 1}
\|T^{\lambda}f\|_{L^2(\R^n)} \sim \|T^{\lambda}f\|_{L^2(N_{\lambda^{\sigma}}(Z) \cap B(0,\lambda))}.
\end{equation}
The examples will further be chosen so that 
\begin{equation}\label{universal example 2}
\|T^{\lambda}f\|_{L^2(\R^n)} \sim \lambda^{1/2}\|f\|_{L^2(B^{n-1})};
\end{equation}
note that, by H\"ormander's generalisation of the Hausdorff--Young inequality \cite{Hormander1973}, the inequality $\|T^{\lambda}f\|_{L^2(\R^n)} \lesssim \lambda^{1/2}\|f\|_{L^2(B^{n-1})}$ always holds (see also $\S$\ref{Wave packet section}, below). 

Playing \eqref{universal example 1} and \eqref{universal example 2} off against one another yields the necessary conditions on $p$. Indeed, for $f$ as above
\begin{equation*}
 \|f\|_{L^p(B^{n-1})} \sim  \|f\|_{L^2(B^{n-1})} \sim \lambda^{-1/2}\|T^{\lambda}f\|_{L^2(\R^n)} \sim \lambda^{-1/2}\|T^{\lambda}f\|_{L^2(N_{\lambda^{\sigma}}(Z) \cap B(0,\lambda))}. 
\end{equation*}
Now, assuming the estimate $\|T^{\lambda}g\|_{L^p(\R^n)} \lesssim_{\varepsilon} \lambda^{\varepsilon} \|g\|_{L^p(B^{n-1})}$ holds for all $\varepsilon > 0$ and applying H\"older's inequality, it follows that  
\begin{equation*}
 \|f\|_{L^p(B^{n-1})} \lesssim_{\varepsilon}  |N_{\lambda^{\sigma}}(Z) \cap B(0,\lambda)|^{1/2 - 1/p}\lambda^{-1/2}\lambda^{\varepsilon}  \|f\|_{L^p(B^{n-1})}.
\end{equation*}
 By a theorem of Wongkew \cite{Wongkew1993} (see Theorem~\ref{Wongkew theorem} below),
\begin{equation}\label{quasi Wongkew}
    |N_{\lambda^{\sigma}}(Z) \cap B(0,\lambda)| \lesssim \lambda^{m + (n-m)\sigma},
\end{equation}
where the implied constant depends only on $n$. In fact, for the simple varieties used in the arguments below, \eqref{quasi Wongkew} can be shown by direct inspection. Thus, combining the previous displays and recalling that $\lambda$ can be taken arbitrarily large, one concludes that
\begin{equation}\label{universal necessary condition}
p \geq 2 \cdot \frac{\sigma(n-m) + m}{\sigma(n-m) + m -1}.
\end{equation}
This condition depends on the two parameters $m$ and $\sigma$, and becomes more restrictive the more $m$ and $\sigma$ decrease. Therefore, in order to obtain the strongest possible restriction on $p$ for a given phase function, one wishes to find the lowest possible $m$ and $\sigma$, over all $f$, for which the mass of $T^{\lambda}f$ can concentrate in the $\lambda^\sigma$-neighbourhood of an $m$-dimensional low-degree algebraic variety.

\begin{itemize}
\item The optimal choice of $m$ is $n - \lfloor \frac{n-1}{2} \rfloor$. This value arises directly from the theory of Kakeya sets of curves, and will be discussed in more detail in the next subsection. 

\item The optimal choice of $\sigma$ depends on the signature of the phase. For general H\"ormander-type operators, one may find examples for which $\sigma = 0$. If one assumes the positive-definite condition H2$^+$), then $\sigma = 1/2$ is the lowest possible value. This difference in behaviour is governed by \emph{transverse equidistribution estimates} for $T^{\lambda}$, which were introduced in the context of Fourier extension operators in \cite{Guth2018}. This will be discussed in detail in $\S$\ref{mass concentration section}.
\end{itemize}

The optimal pairs $(m, \sigma)$ under the various hypotheses are tabulated in Figure~\ref{Hormander conditions}. Plugging these values into \eqref{universal necessary condition} gives the corresponding sharp range of estimates for $T^{\lambda}$. 

\begin{figure}
\begin{center}
\begin{TAB}(c,1cm,2cm)[5pt]{|c|c|c|}{|c|c|c|}
 & $n$ odd & $n$ even    \\
H2) & $\displaystyle \Big(\frac{n+1}{2}, 0\Big)$ & $\displaystyle \Big(\frac{n}{2}+1, 0\Big)$  \\
H2$^+$) & $\displaystyle \Big(\frac{n+1}{2}, \frac{1}{2}\Big)$ & $\displaystyle \Big(\frac{n}{2}+1, \frac{1}{2}\Big)$ \\
\end{TAB}
\caption{Optimal values of $(m, \sigma)$ for the sharp examples.}
\label{Hormander conditions}
\end{center}
\end{figure}




\subsection{Model operators}
The examples described above will arise from operators with phase of the relatively simple form
\begin{equation}\label{model phase}
\phi(x;\omega) := \langle x', \omega \rangle + \frac{1}{2}\langle \bm{A}(x_n)\omega, \omega\rangle
\end{equation}
where $\bm{A} \colon \R \to \mathrm{Sym}(n-1, \R)$ is a polynomial function taking values in the class of real symmetric matrices which satisfies $\bm{A}(0) = 0$. For such a phase the condition H1) always holds whilst H2) (respectively, H2$^+$)) holds if and only if the component-wise derivative $\bm{A}'(x_n) \in \mathrm{GL}(n-1, \R)$ (respectively, $\bm{A}'(x_n)$ is positive-definite) for all relevant $x_n \in [-1,1]$. Observe that if $\bm{A}(x_n) := x_nA$ for some fixed $A \in  \mathrm{Sym}(n-1, \R) \cap \mathrm{GL}(n-1, \R)$, then the resulting operator is the extension operator associated to the graph of the non-degenerate quadratic form $\omega \mapsto \frac{1}{2}\langle A\omega, \omega \rangle$. For the present purpose, one is interested in examples with higher-order dependence on $x_n$. 

Let $T^{\lambda}$ be an operator associated to the phase function \eqref{model phase} for some $\bm{A}$ and a choice of non-negative amplitude function $a$. The $\omega$-support of $a$ is assumed to lie in $B(0,c)$ where $c > 0$ is a small constant. Cover $B^{n-1}$ by finitely-overlapping balls $\theta$ of radius $\lambda^{-1/2}$; these balls will frequently be referred to as \emph{$\lambda^{-1/2}$-caps}. Let $\psi_{\theta}$ be a smooth partition of unity adapted to this cover. Consider a \emph{wave packet} of the form
\begin{equation*}
f_{\theta,v_{\theta}}(\omega) := e^{-2\pi i \lambda \langle v_{\theta}, \omega - \omega_{\theta}  \rangle} \psi_{\theta}(\omega)
\end{equation*}
 for some choice of $v_{\theta} \in \R^{n-1}$ and $\omega_{\theta}$ the centre of the cap $\theta$. To obtain the necessary conditions for $L^p$-boundedness of $T^{\lambda}$, the operator will be tested against functions given by superpositions of these basic wave packets.

Each localised piece $T^{\lambda}f_{\theta, v_{\theta}}$ is concentrated in a tubular region in $\R^n$. In particular, define the curve
\begin{equation}\label{Kakeya sets of curves 1}
\gamma_{\theta, v_{\theta}}(t) := v_{\theta} - \bm{A}(t)\omega_{\theta}\qquad \textrm{for $t \in (-1,1)$} 
\end{equation}
and let $T_{\theta, v_{\theta}}$ be the curved tube 
\begin{equation*}
T_{\theta, v_{\theta}} := \big\{ x \in B(0, \lambda) : |x' - \lambda\gamma_{\theta, v_{\theta}}(x_n/\lambda)| < c\lambda^{1/2} \big\}
\end{equation*}
for $c > 0$ as above. It is not difficult to show that 
\begin{equation}\label{Kakeya sets of curves 2}
|T^{\lambda}f_{\theta, v_{\theta}}(x)| \gtrsim \lambda^{-(n-1)/2}\chi_{T_{\theta, v_{\theta}}}(x) \qquad \textrm{for all $x \in B(0, \lambda)$,}
\end{equation}
provided that $c$ is chosen suitably small. Indeed, let $a_{\theta}^{\lambda}$ be the rescaled amplitude
\begin{equation*}
  a_{\theta}^{\lambda}(x;\omega) :=  a^{\lambda}(x; \omega_{\theta} + \lambda^{-1/2} \omega) \psi_{\theta}(\omega_{\theta} + \lambda^{-1/2} \omega)
\end{equation*}
and $\phi_{\theta}^{\lambda}$ be the phase function
\begin{equation*}
    \phi_{\theta}^{\lambda}(x; \omega) := \lambda^{-1/2}\langle x' - \lambda\gamma_{\theta, v_{\theta}}(x_n/\lambda), \omega \rangle + \frac{1}{2}\langle \mathbf{A}(x_n/\lambda)\omega,\omega \rangle
\end{equation*}
so that, by a linear change of variables,
\begin{equation}\label{Kakeya sets of curves 3}
    T^{\lambda}f_{\theta, v_{\theta}}(x) = \lambda^{-(n-1)/2} e^{2 \pi i \phi^{\lambda}(x;\omega_{\theta})}\int_{\R^{n-1}} e^{2 \pi i \phi_{\theta}^{\lambda}(x; \omega)}  a_{\theta}^{\lambda}(x; \omega) \,\ud \omega .
\end{equation}
Taking absolute values in \eqref{Kakeya sets of curves 3} and writing $e^{2 \pi i \phi_{\theta}^{\lambda}(x; \omega)}$ in terms of its real and imaginary parts, one deduces that
\begin{equation*}
    |T^{\lambda}f_{\theta, v_{\theta}}(x)| \gtrsim \lambda^{-(n-1)/2}\left|\int_{\R^{n-1}} \cos \big(2 \pi \phi_{\theta}^{\lambda}(x; \omega)\big)  a_{\theta}^{\lambda}(x; \omega) \,\ud \omega \right|.
\end{equation*}
Provided $c$ is sufficiently small,
\begin{equation*}
    |\phi_{\theta}^{\lambda}(x; \omega)| \leq \frac{1}{100} \qquad \textrm{whenever $x \in T_{\theta, v_{\theta}}$ and $(x;\omega) \in \mathrm{supp}\,a_{\theta}^{\lambda}$.}
\end{equation*}
Thus, if $x \in T_{\theta, v_{\theta}}$, then
$\cos \big(2 \pi \phi_{\theta}^{\lambda}(x; \omega)\big) \gtrsim 1$ and the desired bound \eqref{Kakeya sets of curves 2} follows.




\subsection{Kakeya sets of curves}\label{Kakeya sets of curves section} By studying the geometry of the family of tubes $T_{\theta, v_{\theta}}$ one may construct sharp examples for Theorem~\ref{main theorem}. These examples arise owing to Kakeya compression phenomena, whereby the tubes are arranged to lie in a neighbourhood of a low-dimensional set. For $n=3$, the following example appears in Bourgain--Guth \cite{Bourgain2011} (see also \cite{Minicozzi1997, Wisewell2005} for related constructions). Let $\phi$ be of the form \eqref{model phase} where $\bm{A}(t)$ is taken to be the $(n-1)\times(n-1)$ block-diagonal matrix
\begin{equation*}
\bm{A}(t) := 
\underbrace{\begin{pmatrix}
t & t^2 \\
t^2 & t + t^3
\end{pmatrix}
\oplus \dots \oplus 
\begin{pmatrix}
t & t^2 \\
t^2 & t + t^3
\end{pmatrix}}_{\text{$\lfloor \frac{n-1}{2} \rfloor$-fold}}
\oplus \,(t).
\end{equation*}
Here it is understood that the final $1 \times 1$ block appears only when $n$ is even. Observe that the resulting phase \eqref{model phase} satisfies H1) and H2$^+$) on $B^n \times B^{n-1}$.

Suppose that $T^{\lambda}$ is the associated oscillatory integral operator. The estimate
\begin{equation}\label{hypothesised inequality}
\|T^{\lambda}f\|_{L^p(\R^n)} \lesssim_{\varepsilon}\lambda^{\varepsilon} \|f\|_{L^p(B^{n-1})}
\end{equation}
is tested against a superposition of wave packets
\begin{equation*}
f := \sum_{\theta : \lambda^{-1/2}-\mathrm{cap}} \epsilon_{\theta} \cdot f_{\theta, v_{\theta}}
\end{equation*}
where the $\epsilon_{\theta} \in  \{1, -1\}$ are uniformly distributed independent random signs. By Khintchine's theorem (see, for instance, \cite[Appendix D]{Stein1970}), the expected value of $|T^{\lambda}f(x)|$ is given by
\begin{equation*}
\mathbf{E}[\,|T^{\lambda}f(x)|\,]\sim \big(\sum_{\theta : \lambda^{-1/2}-\mathrm{cap}} |T^{\lambda}f_{\theta, v_{\theta}}(x)|^2\big)^{1/2}  \gtrsim \lambda^{-(n-1)/2} \big(\sum_{\theta : \lambda^{-1/2}-\mathrm{cap}} \chi_{T_{\theta, v_{\theta}}}(x) \big)^{1/2}
\end{equation*}
for all $x \in B(0,\lambda)$. Thus, by H\"older's and Minkowski's inequalities,
\begin{equation*}
\lambda^{-(n-1)/2}\big(\int \sum_{\theta : \lambda^{-1/2}-\mathrm{cap}} \chi_{T_{\theta, v_{\theta}}} \big)^{1/2} \lesssim \big| \bigcup_{\theta : \lambda^{-1/2}-\mathrm{cap}} T_{\theta, v_{\theta}}\big|^{1/2 - 1/p} \mathbf{E}[\,\|T^{\lambda}f\|_{L^p(\R^n)}\,].
\end{equation*}
The hypothesis \eqref{hypothesised inequality} together with a direct computation now gives
\begin{equation}\label{Kakeya sets of curves 4}
\|f\|_{L^p(B^{n-1})} \lesssim_{\varepsilon} \big| \bigcup_{\theta : \lambda^{-1/2}-\mathrm{cap}} T_{\theta, v_{\theta}}\big|^{1/2 - 1/p}  \lambda^{-1/2 + \varepsilon} \|f\|_{L^p(B^{n-1})},
\end{equation}
since $\|f\|_{L^p(B^{n-1})} \sim 1$ is independent of the outcomes of the $\epsilon_{\theta}$. 

Varying $v_{\theta}$ corresponds to translating the tube $T_{\theta, v_{\theta}}$ in space in the $x'$ direction. In view of \eqref{Kakeya sets of curves 4}, one wishes to choose the $v_{\theta}$ in order to arrange the tubes so that their union has small measure. For the above choice of phase it is in fact possible to select the $v_{\theta}$ so that the tubes all lie in the $\lambda^{1/2}$-neighbourhood of a low-dimensional, low degree algebraic variety. In particular, let $m := n- \lfloor \frac{n-1}{2} \rfloor$ and $Z := Z(P_1, \dots, P_{n-m})$ be the common zero set of the polynomials 
\begin{equation*}
P_j(x) := \lambda x_{2j} - x_{2j -1}x_n \qquad \textrm{for $1 \leq j \leq \left\lfloor \tfrac{n-1}{2} \right\rfloor$.}
\end{equation*}
Thus, $Z$ is an algebraic variety of dimension $m$ and degree $O_n(1)$. If one defines
\begin{equation*}
v_{\theta,  2j-1} := -\omega_{\theta, 2j} \quad \textrm{and} \quad v_{\theta, 2j} = v_{\theta, n-1} = 0 \qquad \textrm{for $1 \leq j \leq \left\lfloor \tfrac{n-1}{2} \right\rfloor$}
\end{equation*}
for each cap $\theta$, then a simple computation shows that the curve $t \mapsto (\lambda\gamma_{\theta, v_{\theta}}(t/\lambda), t)$ lies in $Z$. Thus, 
\begin{equation*}
 \bigcup_{\theta : \lambda^{-1/2}-\mathrm{cap}} T_{\theta, v_{\theta}} \subseteq N_{\lambda^{1/2}}(Z) \cap B(0,\lambda)
\end{equation*}
and the desired necessary conditions on $p$ follow from \eqref{Kakeya sets of curves 4}. 

In conclusion, here the necessary conditions arise due the fact that it is possible to compress an $(n-1)$-dimensional family of curves into a set of small dimension $m$. The value $m = n- \lfloor \frac{n-1}{2} \rfloor$ is optimal for this kind of behaviour, in view of known estimates for associated Kakeya maximal functions: see \cite{Wisewell2005} and \cite{Bourgain2011}.




\subsection{Mass concentration}\label{mass concentration section} It will be useful to contrast the behaviour in the positive-definite and indefinite cases by considering sharp examples for Theorem~\ref{general main theorem} (that is, for the class of operators satisfying  H1) and the weaker hypothesis H2)). As before, Kakeya compression plays a significant r\^ole in the argument, but one can introduce additional interference between the wave packets which leads to stronger necessary conditions. In particular, this interference causes the mass of $|T^{\lambda}f|$ to concentrate in a tiny $O(1)$-neighbourhood of a variety $Z$; such tight concentration is not possible under the H2$^+$) hypothesis (as demonstrated by Theorem~\ref{main theorem}).

The following example was introduced by Bourgain \cite{Bourgain1991} (see also \cite{Bourgain1995}). Once again, the phase is taken to be of the form \eqref{model phase}. This time $\bm{A}(t)$ is defined to be the $(n-1)\times(n-1)$ block-diagonal matrix
\begin{equation}\label{second example}
\bm{A}(t) := 
\underbrace{\begin{pmatrix}
0 & t \\
t & t^2
\end{pmatrix}
\oplus \dots \oplus 
\begin{pmatrix}
0 & t \\
t & t^2
\end{pmatrix}}_{\text{$\lfloor \frac{n-1}{2} \rfloor$-fold}}
\oplus \,(t).
\end{equation}
Clearly, the corresponding phase satisfies H1) and H2), but H2$^+$) fails. Define the curves $\gamma_{\theta, v_{\theta}}$ as in \eqref{Kakeya sets of curves 1} so that
\begin{equation*}
|T^{\lambda}f_{\theta, v_{\theta}}(x)| \gtrsim \lambda^{-(n-1)/2}\chi_{T_{\theta, v_{\theta}}}(x)  \qquad \textrm{for all $x \in B(0,\lambda)$.}
\end{equation*}
If one takes
\begin{equation*}
v_{\theta,  2j-1} := -\omega_{\theta, 2j-1} \quad \textrm{and} \quad v_{\theta, 2j} = v_{\theta, n-1} = 0 \qquad \textrm{for $1 \leq j \leq \left\lfloor \tfrac{n-1}{2} \right\rfloor$,}
\end{equation*}
then it follows that the curve $t \mapsto (\lambda\gamma_{\theta, v_{\theta}}(t/\lambda), t)$ lies in $Z$ for all $\lambda^{-1/2}$-caps $\theta$, where $Z$ is the same variety as that appearing in the previous subsection (see Figure~\ref{Kakeya figure}).



\begin{figure}
\begin{center}
\tdplotsetmaincoords{75}{15}
\begin{tikzpicture}[tdplot_main_coords, scale=1.9]

	\begin{scope}[rotate around x=5]
   \foreach \a in {-1,-0.8,...,1}
  {
    \foreach \b in {-1,-0.8,...,1}
    {
		\draw[black, very thin] 
		({-\a + \b*1.0},{-1.0},{\a*1.0-\b*(1.0*1.0)})
    \foreach \t in {-1.0,-0.9,...,1.0}
    {
        --({-\a - \b*\t},{\t},{-\a*\t-\b*(\t*\t)})
    }
    ;
    }
    }
		\end{scope}
			\end{tikzpicture}
			\caption{The Kakeya compression phenomenon for the curves arising from the matrix \eqref{second example}.}
\label{Kakeya figure}
\end{center}
\end{figure}
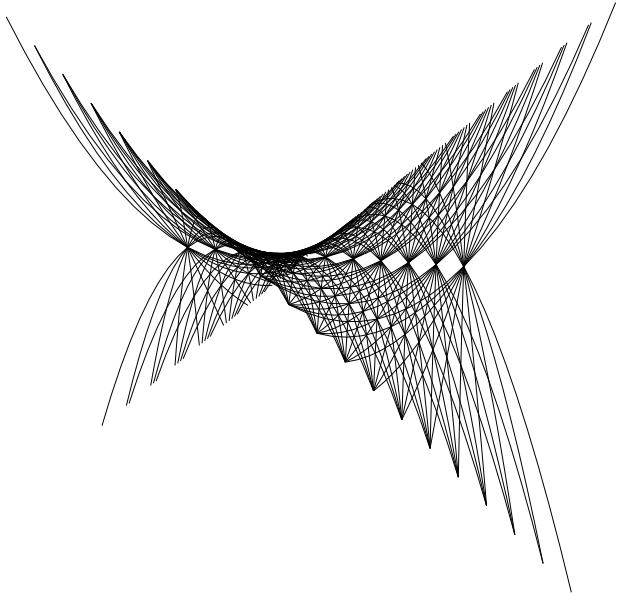


One may repeat the analysis of $\S$\ref{Kakeya sets of curves section} by taking $f$ to be a linear combination of wave packets $f_{\theta, v_{\theta}}$ with random signs. This leads to the same necessary conditions as in the positive-definite case. However, certain deterministic choices of $f$ lead to stronger conditions on $p$. In particular, consider the function 
\begin{equation}\label{mass concentration 2}
\tilde{f} := \sum_{\theta : \lambda^{-1/2}-\mathrm{cap}} e^{2 \pi i \lambda Q(\omega_{\theta})} f_{\theta, v_{\theta}},
\end{equation}
where the $v_{\theta}$ are as defined above and $Q$ is the quadratic polynomial 
\begin{equation*}
Q(\omega) := \frac{1}{2}\sum_{j=1}^{\lfloor \frac{n-1}{2} \rfloor} \omega_{2j-1}^2.
\end{equation*}
Each modulated wave packet appearing in \eqref{mass concentration 2} has a phase given by
\begin{equation*}
\lambda\big(Q(\omega_{\theta}) - \langle v_{\theta}, \omega - \omega_{\theta}  \rangle\big) = \lambda Q(\omega) -  \frac{\lambda}{2}\sum_{j=1}^{\lfloor \frac{n-1}{2} \rfloor} (\omega_{2j-1} - \omega_{\theta, 2j-1})^2. 
\end{equation*}
Since the $\lambda(\omega_{2j-1} - \omega_{\theta, 2j-1})^2$ terms are bounded functions on the support of $\psi_{\theta}$, they do not contribute any significant oscillation. One may therefore heuristically identify $\tilde{f}$ with the function 
\begin{equation*}
f(\omega) := e^{2 \pi i \lambda Q(\omega)}\psi(\omega)
\end{equation*}
where $\psi$ is a bump function supported in $B^{n-1}$. Using a simple stationary phase argument, it was shown in \cite{Bourgain1991} (see also \cite{Bourgain1995, Bourgain2011}) that
\begin{equation}\label{mass concentration 3}
|T^{\lambda}f(x)| \gtrsim \lambda^{-\lfloor n/2 \rfloor/2}\chi_{N_c(Z)}(x) \qquad \textrm{for all $x \in B(0,\lambda)$.}
\end{equation}
Here $0 < c < 1$ is some small, fixed constant (which is independent of $\lambda$) and $Z$ is as defined in $\S$\ref{Kakeya sets of curves section}. With this estimate, one readily deduces the desired necessary conditions on $p$. Indeed, testing the inequality 
        \begin{equation*}
 \|T^{\lambda}f\|_{L^p(\R^n)} \lesssim_{\varepsilon} \lambda^{\varepsilon} \|f\|_{L^p(B^{n-1})}
        \end{equation*}
against the function $f$ as defined above, it follows from \eqref{mass concentration 3} that
\begin{equation*}
  \lambda^{-\floor{n/2}/2}|N_c(Z) \cap B(0,\lambda)|^{1/p} \lesssim_{\varepsilon} \lambda^{\varepsilon}.
\end{equation*}
    Since $|N_c(Z) \cap B(0,\lambda)| \sim \lambda^{n - \floor{(n-1)/2}}$ and $n = \floor{\frac{n}{2}} + \floor{\frac{n-1}{2}} + 1$, one deduces that
    \begin{equation*}
    \lambda^{-\floor{n/2}/2 + (\floor{n/2} +1)/p} \lesssim_{\varepsilon} \lambda^{\varepsilon}.
    \end{equation*}
     Taking $\lambda \geq 1$ large and $0 < \varepsilon < 1$ small, this forces
\begin{equation*}
-\frac{n-1}{2} + \frac{n+1}{p} \leq 0  \quad \textrm{if $n$ is odd} \quad \textrm{and} \quad
-\frac{n}{2} + \frac{n+2}{p} \leq 0  \quad \textrm{if $n$ is even,}
\end{equation*}
corresponding to the constraints on $p$ from Theorem~\ref{general main theorem}.

For completeness, the details of the argument used in \cite{Bourgain1991, Bourgain1995} to prove \eqref{mass concentration 3} are reviewed. Observe that $T^{\lambda}f(x)$ is an oscillatory integral with smooth amplitude $\psi$ and phase 
\begin{equation}\label{mass concentration 4}
\sum_{j=1}^{\lfloor \frac{n-1}{2} \rfloor} x_{2j-1}\omega_{2j-1} + x_{2j}\omega_{2j} + \frac{\lambda}{2}(\omega_{2j-1} + \lambda^{-1}x_n\omega_{2j})^2 + \delta_{\mathrm{e}}\big(x_{n-1} \omega_{n-1} + \frac{x_n}{2}\omega_{n-1}^2\big)
\end{equation}
where $\delta_{\mathrm{e}} = 0$ if $n$ is odd and $\delta_{\mathrm{e}} = 1$ if $n$ is even. Introduce new variables $z_j := \omega_{2j-1} + \lambda^{-1}x_n\omega_{2j}$ for $1 \leq j \leq \left\lfloor \tfrac{n-1}{2} \right\rfloor$. If $x \in Z$, then the phase function \eqref{mass concentration 4} can be expressed as
\begin{equation*}
\sum_{j=1}^{\lfloor \frac{n-1}{2} \rfloor} x_{2j-1}z_j + \frac{\lambda}{2}z_j^2 + \delta_{\mathrm{e}}\big(x_{n-1} \omega_{n-1} + \frac{x_n}{2}\omega_{n-1}^2\big).
\end{equation*}
The integral $T^{\lambda}f$ can now be reduced to a product of $\left\lfloor \tfrac{n-1}{2} \right\rfloor + \delta_{\mathrm{e}} = \left\lfloor \tfrac{n}{2} \right\rfloor$ integrals, each in a single variable. For $x \in Z$ the inequality \eqref{mass concentration 3} follows as a consequence of standard stationary phase estimates applied to each of these integrals (see, for instance, \cite[Chapter VIII, Proposition 3]{Stein1993}). This lower bound can then be extended to some $c$-neighbourhood of $Z$ via a simple estimate on the gradient of $T^{\lambda}f(x)$.





 \section{Key features of the analysis}\label{Proof sketch section}

The examples of the previous section highlight several key features of H\"ormander-type operators. All these features are exploited in the proofs of the linear and $k$-broad estimates.
\begin{enumerate}[1)]
\item \textbf{Algebraic structure.} The sharp examples were given by arranging collections of wave packets to lie in a relatively small neighbourhood of a low degree, low dimensional algebraic variety $Z$. It turns out that this is an essential feature of both the linear and $k$-broad problems. To exploit this underlying algebraic structure, the proof of Theorem~\ref{k-broad theorem} will rely on a variant of the polynomial partitioning method introduced by Katz and the first author \cite{Guth2015}. Roughly speaking, this method allows one to reduce to the case where $|T^{\lambda}f|$ concentrates around some low degree, low dimensional variety, as in the sharp examples. This can be thought of as a dimensional reduction and, indeed, the proof of Theorem~\ref{k-broad theorem} will proceed by an induction on dimension. Polynomial partitioning has played an increasingly important r\^ole in the theory of oscillatory integral operators, beginning with the work on the restriction problem in \cite{Guth2016, Guth2018} and, more recently, in \cite{Wang, HickmanRogers2018}.  
\item \textbf{Non-concentration/transverse equidistribution.} Suppose one does not assume the phase is positive-definite. The example of $\S$\ref{mass concentration section} then shows that interference between the wave packets can cause $|T^{\lambda}f|$ to be concentrated in a tiny neighbourhood of $Z$. In order to prove the sharp range of estimates in the positive-definite case one must rule out the possibility of such concentration. This is achieved by extending the theory of so-called \emph{transverse equidistribution estimates} introduced in \cite{Guth2018} to the variable coefficient setting. These estimates can be interpreted as showing that $|T^{\lambda}f(x)|$ is morally constant along transverse directions to $Z$ in a $\lambda^{1/2}$-neighbourhood of the variety. Consequently, $|T^{\lambda}f(x)|$ cannot concentrate in a smaller neighbourhood.
\item \textbf{Parity of the dimension.} Another key feature of the examples discussed in the previous section is their dependence on the parity of the ambient dimension $n$. Recall that this is directly related to the minimal dimension 
\begin{equation*}
d(n) := \left\{\begin{array}{ll}
\frac{n+1}{2} & \textrm{if $n$ is odd} \\[6pt]
\frac{n+2}{2} & \textrm{if $n$ is even}
\end{array}\right. 
\end{equation*}
of Kakeya sets of curves in $\R^n$.  The parity of the dimension does not play an overt r\^ole in the proof of the $k$-broad estimates, but it becomes a noticeable feature when one wishes to pass from $k$-broad to linear estimates in the proof of Theorem~\ref{main theorem}. In particular, for each fixed value of $2 \leq k \leq n$, the method of $\S$\ref{k-broad to linear section} shows that the $k$-broad estimates imply a (possibly empty/trivial) range of linear estimates. It transpires that to optimise the range of linear estimates obtained in this manner one should choose $k$ to correspond to the dimension $d(n)$ from the Kakeya problem. 
\end{enumerate}

The proof of the $k$-broad estimates follows the same general scheme as that used to study Fourier extension operators in \cite{Guth2018}, and heavily exploits the the features 1) and 2) of the problem highlighted above. A detailed sketch of the argument in the extension context is provided in \cite{Guth2018}; this sketch is likely to be beneficial to readers new to these ideas. 
 



 \section{Reductions}\label{Reductions section}




\subsection{Basic reductions}

The prototypical example of a positive-definite phase function is given by 
\begin{equation}\label{prototypical phase}
\phi_{\mathrm{par}}(x;\omega) := \langle x', \omega \rangle + x_n\frac{|\omega|^2}{2}.
\end{equation}
This is the phase associated to the extension operator for the (elliptic) paraboloid. In general, to prove Theorem~\ref{main theorem} it suffices to only consider phases which are given by small perturbations of $\phi_{\mathrm{par}}$. Observations of this kind have been used previously in the theory of oscillatory integral operators and the arguments of this section are inspired by \cite{Hormander1973, Lee2006}. 

To understand why such a reduction is possible, first recall that the class of operators under consideration are those of the form
\begin{equation*}
T^{\lambda}f(x) := \int_{\R^{n-1}} e^{2 \pi i \phi^{\lambda}(x;\omega)} a^{\lambda}(x; \omega) f(\omega)\, \ud \omega
\end{equation*}
where $\phi$ satisfies H1) and H2$^+$). In addition, one may assume a number of fairly stringent conditions on the form of $\phi$ on the support of $a$.  

\begin{lemma}\label{reduction lemma} To prove Theorem~\ref{main theorem} for some fixed $\varepsilon > 0$ it suffices to consider the case where $a$ is supported on $X \times \Omega$ where $X := X' \times X_n$ and $X' \subset B^{n-1}$, $X_n \subset B^1$ and $\Omega \subset B^{n-1}$ are small balls centred at 0 upon which the phase $\phi$ has the form
\begin{equation}\label{reduced phase}
\phi(x; \omega) = \langle x', \omega \rangle + x_nh(\omega) + \mathcal{E}(x;\omega).
\end{equation}
Here $h$ and $\mathcal{E}$ are smooth functions, $h$ is quadratic in $\omega$ and $\mathcal{E}$ is quadratic in $x$ and $\omega$.\footnote{Explicitly, if $(\alpha, \beta) \in \N_0 \times \N_0^{n-1}$ is a pair of multi-indices, then:
\begin{enumerate}[i)]
\item $\partial^{\beta}_{\omega}h(0) = \partial^{\beta}_{\omega}\partial^{\alpha}_x\mathcal{E}(x;0)=0$ whenever $x \in X$ and $|\beta| \leq 1$;
\item  $\partial^{\beta}_{\omega}\partial^{\alpha}_x\mathcal{E}(0;\omega)=0$ whenever $\omega \in \Omega$ and $|\alpha| \leq 1$.
\end{enumerate}} Furthermore, letting $c_{\mathrm{par}} > 0$ be a small constant, which may depend on the admissible parameters $n$, $p$ and $\varepsilon$, one may assume that
\begin{gather}\label{close to identity}
\|\partial_{\omega x'}^2 \phi(x; \omega) - \mathrm{I}_{n-1}\|_{\mathrm{op}} < c_{\mathrm{par}}, \quad |\partial_{\omega}\partial_{x_n} \phi(x; \omega)| < c_{\mathrm{par}}, \\
\label{close to identity again} \|\partial_{\omega \omega}^2\partial_{x_k} \phi(x; \omega) - \delta_{kn}\mathrm{I}_{n-1}\|_{\mathrm{op}} < c_{\mathrm{par}}
\end{gather}
 for all $(x;\omega) \in X \times \Omega$ and $1\leq k \leq n$. 
\end{lemma}
Here $\mathrm{I}_{n-1}$ denotes the $(n-1)\times(n-1)$ identity matrix, $\delta_{ij}$ the Kronecker $\delta$-function and $\|\,\cdot\,\|_{\mathrm{op}}$ the operator norm.

It is perhaps useful to provide a brief interpretation of the formula \eqref{reduced phase}. Since $h$ is quadratic, $h(\omega) = \frac{1}{2}\langle \partial_{\omega\omega}^2h(0) \omega, \omega \rangle + O(|\omega|^3)$. Although unnecessary for the forthcoming analysis, by rotating the $\omega$ co-ordinates one may further assume that $h(\omega) = \frac{|\omega|^2}{2} + O(|\omega|^3)$. Thus, the phase in \eqref{reduced phase} is given by
\begin{equation*}
    \phi(x;\omega) = \phi_{\mathrm{par}}(x;\omega) + \textrm{higher order terms}
\end{equation*}
and is therefore a perturbation of the prototypical example $\phi_{\mathrm{par}}$.

The proof of the lemma is based upon three elementary principles:

\paragraph{\emph{Localisation}} If a property $P$ of a phase holds locally on $\mathrm{supp}\,a$, then typically one may assume $P$ holds on the whole of $\mathrm{supp}\,a$ by applying a partition of unity, the triangle inequality and shifting co-ordinates.

\paragraph{\emph{Parametrisation invariance}} By the change of variables formula, one may compose $\phi$ with a smooth change of either the $x$ or $\omega$ variables. 

\paragraph{\emph{Modulation invariance}} One is free to add smooth functions to the phase which depend only on either the $x$ or on the $\omega$ variables. In particular, $\phi$ can be replaced by
\begin{equation*}
\phi(x;\omega) + \phi(0;0) - \phi(0;\omega) - \phi(x;0)
\end{equation*}
and therefore one may assume that
\begin{equation}\label{modulation invariance}
\partial_{x}^{\alpha}\phi(x;0) = 0 \quad \textrm{and} \quad \partial_{\omega}^{\beta}\phi(0;\omega) = 0
\end{equation}
for all multi-indices $(\alpha, \beta) \in \N_0^n \times \N_0^{n-1}$. 

The following argument provides an example of these three principles working together. Rotating the $x$-co-ordinates, one may assume that $\partial_{\omega} \partial_{x_n} \phi(0;0) = 0$. By H2) it follows that 
\begin{equation*}
\det \partial_{x'\omega}^2 \phi(0;0) \neq 0.
\end{equation*}
The inverse function theorem now implies the existence of local inverses to the functions $\omega \mapsto \partial_{x'}\phi(x; \omega)$ and $x' \mapsto \partial_{\omega}\phi(x; \omega)$ in a neighbourhood of 0. Thus, by localisation, one may assume $\mathrm{supp}\,a$ is contained in $X\times \Omega$ where $X = X' \times X_n$ for $X'\subset B^{n-1}$, $X_n \subset B^1$ and $\Omega \subset B^{n-1}$ small balls centred at 0 and that there exist smooth functions $\Phi$ and $\Psi$ taking values in $X$ and $\Omega$, respectively, such that
\begin{equation}\label{change of variables}
\partial_{x'}\phi(x; \Psi(x; u)) = u \quad \textrm{and} \quad \partial_{\omega}\phi(\Phi(z', x_n; \omega), x_n; \omega ) = z'.
\end{equation}
The former identity can be thought of as a generalisation of the fact that any hypersurface can be locally parametrised as a graph. The latter identity features in the proof of Lemma~\ref{reduction lemma} and it is useful to highlight some further properties of $\Phi$. For each $(x_n, \omega) \in X_n \times \Omega$ the map $z' \mapsto \Phi(z',x_n;\omega)$ is a diffeomorphism from its domain onto $X'$; this provides a useful change of variables on $X'$. Furthermore, it is easy to see that $0$ lies in the domain of this map when $x_n = 0$, $\omega = 0$ and that
\begin{equation}\label{Phi properties}
\Phi(0;0) = 0, \quad  \partial_{x_n}\Phi(0;0) = 0 \quad \textrm{and} \quad \partial_{x'}\Phi(0;0)= \partial_{x'\omega}^2\phi(0;0)^{-1}.
\end{equation} 
Indeed, the first identity follows directly from \eqref{modulation invariance} whilst the remaining identities are obtained by differentiating the defining expression for $\Phi$ from \eqref{change of variables}.

\begin{proof}[Proof (of Lemma~\ref{reduction lemma})] By \eqref{modulation invariance} one may assume that
\begin{equation*}
\phi(x; \omega) = \langle \partial_\omega \phi(x;0), \omega \rangle  + \rho(x;\omega)
\end{equation*}
where $\rho$ is quadratic in $\omega$ and satisfies $\rho(0;\omega) = 0$ for all $\omega \in \Omega$. Let $\Phi$ be the function defined in \eqref{change of variables} and, using localisation and parametrisation invariance, perform the change of variables $x' \mapsto \Phi(x',x_n;0)$ on $X'$ so that the phase becomes
\begin{equation*}
\phi(x; \omega) =  \langle x', \omega \rangle + \rho(\Phi(x;0), x_n;\omega).
\end{equation*}
By \eqref{Phi properties} one has $\partial_{x_n}\Phi(x', 0;0) = O(|x|)$ and, taking a Taylor expansion of $\rho$ in $x_n$,
\begin{equation*}
\rho(\Phi(x;0), x_n; \omega) = \rho(\Phi(x', 0;0), 0;\omega) + (\partial_{x_n}\rho)(\Phi(x', 0;0), 0;\omega) x_n + O(|x|^2).
\end{equation*}
Note that the first expression on the right-hand side satisfies 
\begin{equation*}
\rho(\Phi(x', 0;0),0;\omega) = \langle \partial_{x'\omega}\phi(0;0)^{-1}x', (\partial_{x'}\rho)(0;\omega) \rangle + O(|x|^2)
\end{equation*}
whilst, Taylor-expanding now in $x'$, it follows that
\begin{equation*}
(\partial_{x_n} \rho)(\Phi(x', 0;0),0;\omega) = (\partial_{x_n}\rho)(0;\omega) + O(|x|).
\end{equation*}
Combining these observations, and noting, for instance, that \eqref{modulation invariance} implies that $\partial_x^{\alpha}\rho(x;0) = 0$ for all $\alpha \in \N_0^n$ and $x \in B^n$, one deduces that
\begin{equation*}
\phi(x; \omega) =  \big\langle x', \omega + \partial_{x'\omega}\phi(0;0)^{-\top} (\partial_{x'}\rho)(0;\omega)\big\rangle + x_n(\partial_{x_n}\rho)(0;\omega) + O(|x|^2|\omega|^2);
\end{equation*}
Here the symbol $\top$ is used to denote the matrix transpose and $-\top$ the inverse matrix transpose. 

Since $\rho$ is quadratic in $\omega$, it follows that $\omega \mapsto \omega + \partial_{x'\omega}\phi(0;0)^{-\top} (\partial_{x'}\rho)(0;\omega)$ is a well-defined change of variables in a neighbourhood of the origin and so, once again by localisation and parametrisation invariance, the problem is reduced to considering phase functions of the from \eqref{reduced phase}. By the construction $h$ and $\mathcal{E}$ are quadratic. Finally, the condition H2$^+$) implies that the matrix $\partial_{\omega \omega}^2\partial_{x_n} \phi(0; 0)$ is positive definite. Applying a linear co-ordinate change, one may therefore suppose that $\partial_{\omega \omega}^2\partial_{x_n} \phi(0; 0) = \mathrm{I}_{n-1}$. On the other hand, clearly $\partial_{\omega}\partial_{x_n}\phi(0;0) = 0$, $\partial_{\omega x'}^2\phi(0; 0) = \mathrm{I}_{n-1}$ and $\partial_{\omega \omega}^2\partial_{x_k} \phi(0; 0) = 0_{n-1}$ for $1 \leq k \leq n-1$. By continuity, if the support of $a$ is sufficiently small, then the conditions of \eqref{close to identity} and \eqref{close to identity again} are valid on the support of $a$. 
\end{proof}

\subsection{Parabolic rescaling}\label{parabolic rescaling subsection} In addition to the reductions of Lemma~\ref{reduction lemma}, it will be useful to have control over higher order derivatives of the phase, and also the amplitude function. Such control is made possible using a simple scaling argument.

Consider the constant coefficient case $\phi(x;\omega) := \langle x',\omega \rangle + x_n h(\omega)$ where $h(\omega) = \frac{|\omega|^2}{2} + O(|\omega|^3)$, so that $\phi$ is a perturbation of the prototypical phase $\phi_{\mathrm{par}}$ defined in \eqref{prototypical phase}. The corresponding operator
\begin{equation*}
    Ef(x) := \int_{B^{n-1}} e^{2 \pi i \phi(x;\omega)} f(\omega)\,\ud \omega
\end{equation*}
(that is, the extension operator associated to the graph of $h$) has a special scaling structure. Let $\bar{\omega} \in B^{n-1}$ and $\rho \geq 1$ and note that
\begin{equation*}
    \phi(x;\bar{\omega} + \rho^{-1}\omega) - \phi(x;\bar{\omega}) = \tilde{\phi}(\tilde{x}; \omega)
\end{equation*}
where $\tilde{\phi}$ is defined the same way as $\phi$ but with $h$ replaced with 
\begin{equation*}
    \tilde{h}(\omega) := \rho^2 \big( h(\bar{\omega} + \rho^{-1}\omega) - h(\bar{\omega}) - \rho^{-1}\langle \partial_{\omega}h(\bar{\omega}), \omega \rangle \big). 
\end{equation*}
and $\tilde{x} := (\rho^{-1}\partial_{\omega}\phi(x;\bar{\omega}),\rho^{-2}x_n)$. The linear map $x \mapsto \tilde{x}$ is referred to as a \emph{parabolic rescaling}, owing to the $\rho^{-1}$ and $\rho^{-2}$ scaling factors.

Now consider the special case where $h(\omega) = h_{\mathrm{par}}(\omega) := \frac{|\omega|^2}{2}$ (so that $\phi = \phi_{\mathrm{par}}$), noting that
\begin{equation}\label{something trivial}
    |\partial_{\omega}^{\beta}h_{\mathrm{par}}(\omega)| = 0  \qquad \textrm{for all $\beta \in \N_0^n$ with $|\beta| \geq 3$.}
\end{equation}
Write $E_{\mathrm{par}}$ for the operator $E$ and observe that $\tilde{h}_{\mathrm{par}} = h_{\mathrm{par}}$ and, consequently, $\tilde{\phi}_{\mathrm{par}} = \phi_{\mathrm{par}}$. Thus, if $\mathrm{supp}\,f \subseteq B(\bar{\omega},\rho^{-1})$, then 
\begin{equation*}
    |E_{\mathrm{par}}f(x)| = |E_{\mathrm{par}}\tilde{f}(\tilde{x})|
\end{equation*}
where $\tilde{f}(\omega) := \rho^{-(n-1)}f(\bar{\omega}+\rho^{-1}\omega)$ is supported in $B^{n-1}$. 

For general $h$ such a clean scaling identity does not hold. In particular, parabolic rescaling does not preserve $E$ but transforms it into the operator $\tilde{E}$ associated to the phase $\tilde{\phi}$: that is,
\begin{equation*}
  |Ef(x)| = |\tilde{E}\tilde{f}(\tilde{x})|.  
\end{equation*}
A useful feature, however, is that the new phase $\tilde{\phi}$ more closely resembles the prototypical example $\phi_{\mathrm{par}}$ (relative to $\phi$). In particular,
\begin{equation*}
    |\partial_{\omega}^{\beta} \tilde{h}(\omega)| \lesssim_{\beta} \rho^{-(|\beta| - 2)} \qquad \textrm{for all $\beta \in \N_0^n$ with $|\beta| \geq 3$}
\end{equation*}
so that, as $\rho$ is large, the function $\tilde{h}$ is `closer' to satisfying \eqref{something trivial} (relative to $h$). 

These observations can be extended to the variable coefficient setting to prove the following reduction. 

\begin{lemma}\label{second reduction lemma}  To prove Theorem~\ref{main theorem} for some fixed $\varepsilon > 0$ it suffices to consider the case where, in addition to the properties described in Lemma~\ref{reduction lemma}, the phase satisfies
\begin{equation*}
\|\partial_{\omega}^{\beta}\partial_x^{\alpha} \phi \|_{L^{\infty}(X \times \Omega)} < c_{\mathrm{par}} \qquad \textrm{for $1 \leq |\alpha| \leq N_{\mathrm{par}}$, $3 \leq |\beta| \leq N_{\mathrm{par}}$}
\end{equation*}
for some small constant $c_{\mathrm{par}}$ and large integer $N_{\mathrm{par}} \in \N$, which can be chosen to depend on $n$, $p$ and $\varepsilon$. Furthermore, one may assume that the amplitude satisfies
\begin{equation*}
\|\partial_{\omega}^{\beta}a\|_{L^{\infty}(X \times \Omega)} \lesssim_{\beta} 1 \qquad \textrm{for all $0 \leq |\beta| \leq N_{\mathrm{par}}$.}
\end{equation*}
\end{lemma}

\begin{proof} One may assume that the phase of $T^{\lambda}$ is given by $\phi^{\lambda}(x;\omega) := \lambda\phi(x/\lambda;\omega)$ where
\begin{equation*}
 \phi(x;\omega) = \langle x',\omega\rangle + x_nh(\omega) + \mathcal{E}(x;\omega) \quad \textrm{for $(x;\omega) \in X \times \Omega$.}
\end{equation*}
Let $\rho \geq 1$, $f \in L^1(B^{n-1})$ and cover $B^{n-1}$ by finitely-overlapping $\rho^{-1}$-balls. Provided $\rho$ is chosen to depend only on $\phi$ and $\varepsilon > 0$, by the triangle inequality one may assume that $f$ is supported on one such ball, say $B(\bar{\omega}, \rho^{-1})$ where $\bar{\omega} \in B^{n-1}$. Thus, by a linear change of variables,
\begin{equation*}
|T^{\lambda}f(x)| = \big|\int_{B^{n-1}} e^{2\pi i (\phi^{\lambda}(x;\bar{\omega} + \rho^{-1} \omega) - \phi^{\lambda}(x;\bar{\omega}))} a^{\lambda}(x;\bar{\omega} + \rho^{-1} \omega) \tilde{f}(\omega) \,\ud \omega\big|
\end{equation*}
where $\tilde{f}(\omega) := \rho^{-(n-1)}f(\bar{\omega}+\rho^{-1}\omega)$. The phase function appearing in the above oscillatory integral may be expressed as
\begin{equation*}
\rho^{-1}\langle (\partial_{\omega}\phi^{\lambda})(x;\bar{\omega}), \omega \rangle + \rho^{-2}\big(x_n\tilde{h}(\omega) + \lambda\tilde{\mathcal{E}}_1(x/\lambda; \omega)\big)
\end{equation*}
where
\begin{equation*}
\tilde{h}(\omega) := \rho^{2}\big(h(\bar{\omega}+\rho^{-1}\omega) - h(\bar{\omega}) - \rho^{-1}\langle \partial_{\omega}h(\bar{\omega}), \omega\rangle\big)
\end{equation*}
and
\begin{equation*}
\tilde{\mathcal{E}}_1(x;\omega) := \rho^{2}\big(\mathcal{E}(x;\bar{\omega}+\rho^{-1}\omega) - \mathcal{E}(x;\bar{\omega}) - \rho^{-1}\langle \partial_{\omega}\mathcal{E}(x; \bar{\omega}), \omega\rangle \big).
\end{equation*}
These definitions are motivated by the analysis of the constant coefficient case, as in the discussion prior to the statement of the lemma. Defining 
\begin{equation*}
  \tilde{\mathcal{E}}(x;\omega) := \tilde{\mathcal{E}}_1(\Phi(\rho^{-1}x', x_n;\bar{\omega}), x_n;\omega)  
\end{equation*}
where $\Phi$ is the function introduced in \eqref{change of variables}, under the change of variables
\begin{equation*}
    x' \mapsto \lambda\Phi(\rho x'/\lambda, \rho^2 x_n/\lambda; \bar{\omega}), \quad x_n \mapsto \rho^2 x_n
\end{equation*}
the phase and amplitude are transformed into $\tilde{\phi}^{\lambda/\rho^2}(x; \omega)$ and $\tilde{a}^{\lambda/\rho^2}(x;\omega)$, respectively, where
\begin{equation*}
\tilde{\phi}(x; \omega) := \langle x', \omega \rangle + x_n\tilde{h}(\omega) + \tilde{\mathcal{E}}(x; \omega)
\end{equation*}
and
\begin{equation*}
\tilde{a}(x;\omega) := a(\Phi(\rho^{-1} x', x_n;\bar{\omega}),x_n;\bar{\omega} + \rho^{-1}\omega).
\end{equation*}
In particular, defining 
\begin{equation}\label{rescaled operator}
 \tilde{T}^{\lambda/\rho^2}g(x) := \int_{\R^{n-1}} e^{2\pi i \tilde{\phi}^{\lambda/\rho^2}(x;\omega)} \tilde{a}^{\lambda/\rho^2}(x;\omega) g(\omega) \,\ud \omega
\end{equation}
it follows that 
\begin{equation*}
 \|T^{\lambda}f\|_{L^p(\R^n)} \lesssim \rho^{(n+1)/p} \|\tilde{T}^{\lambda/\rho^2}\tilde{f}\|_{L^p(\R^n)}.
\end{equation*}
It is easy to verify that the phase $\tilde{\phi}$ satisfies the conditions of Lemma~\ref{reduction lemma} and, provided $\rho$ is chosen appropriately (depending on $\phi$ and $a$), it also satisfies the additional conditions described in Lemma~\ref{second reduction lemma}. The same is true for the amplitude $\tilde{a}$, except that the $\omega$ support has now been enlarged. However, by applying a partition of unity and translation argument, it is possible to assume without loss of generality that $\tilde{a}$ satisfies the desired support condition. This facilitates the reduction.
\end{proof}

\subsection{Controlling higher order $x$ derivatives} A final, elementary scaling argument allows one to control higher order derivatives in $x$.

\begin{lemma}\label{third reduction lemma}
To prove Theorem~\ref{main theorem} for some fixed $\varepsilon > 0$ it suffices to consider the case where, in addition to the properties described in Lemma~\ref{reduction lemma} and  Lemma~\ref{second reduction lemma}, the phase satisfies
\begin{equation}\label{small x derivatives}
\|\partial_{\omega}^{\beta}\partial_x^{\alpha} \phi \|_{L^{\infty}(X \times \Omega)} < c_{\mathrm{par}} \qquad \textrm{for $2 \leq |\alpha| \leq N_{\mathrm{par}}$, $0 \leq |\beta| \leq N_{\mathrm{par}}$.}
\end{equation}
Furthermore, one may assume that the amplitude satisfies
\begin{equation*}
\|\partial_{\omega}^{\beta}\partial_x^{\alpha} a\|_{L^{\infty}(X \times \Omega)} \lesssim_{\alpha, \beta} 1 \qquad \textrm{for all $0 \leq |\alpha|, |\beta| \leq N_{\mathrm{par}}$.}
\end{equation*}
\end{lemma}

\begin{proof} Let $T^{\lambda}$ be an operator associated to a phase $\phi$ and amplitude $a$ satisfying the conditions of Lemma~\ref{second reduction lemma}. Let $\rho \geq 1$ be a large constant, which will be chosen depending on $a$ and $\phi$, $n$ and $\varepsilon$ only, and define
\begin{equation*}
    \tilde{\phi}(x;\omega) := \rho\phi(x/\rho;\omega), \qquad \tilde{a}(x;\omega) := a(x/\rho;\omega).
\end{equation*}
One may easily verify that $\tilde{\phi}$ and $\tilde{a}$ satisfy the conditions in Lemma~\ref{second reduction lemma}, except for an enlargement of the $x$-support which may be dealt with via a partition of unity. Furthermore,
\begin{equation*}
    \|T^{\lambda}f\|_{L^p(\R^n)} = \|\tilde{T}^{\lambda/\rho}f\|_{L^p(\R^n)}
\end{equation*}
and so to prove $L^p$ estimates for $T^{\lambda}$ it suffices to prove corresponding estimates for $\tilde{T}^{\lambda}$. Finally, provided $\rho$ is suitably chosen, it follows that $\tilde{\phi}$ and $\tilde{a}$ satisfy the additional conditions in the statement of Lemma~\ref{third reduction lemma}.

\end{proof}

\begin{definition}\label{reduced definition} Henceforth $c_{\mathrm{par}} > 0$ and $N_{\mathrm{par}} \in \N$ are assumed to be fixed constants, chosen to satisfy the requirements of the forthcoming arguments. A positive-definite phase satisfying the properties of Lemma~\ref{third reduction lemma} for this choice of $c_{\mathrm{par}}$ and $N_{\mathrm{par}}$ is said to be \emph{reduced}.
\end{definition} 
This notion of reduced positive-definite phase is precisely that which appears, then undefined, in the statement of the $k$-broad estimates of Theorem~\ref{k-broad theorem}.


\subsection{Geometric consequences}

Henceforth, unless otherwise stated, all positive-definite phase functions $\phi$ are assumed to be reduced in the sense described above. The strategy is to obtain uniform estimates over this class of phases. 

By the definition of H\"ormander-type operators, for each $x \in X$ the map $\omega \mapsto \partial_x \phi(x;\omega)$ parametrises a smooth hypersurface $\Sigma_x$. In many respects, these hypersurfaces are geometrically very similar to the paraboloid $\omega \mapsto (\omega, \frac{|\omega|^2}{2})$. To see this, recall that $\Psi \colon U \to \Omega$ is a smooth function which satisfies
\begin{equation}\label{R change of variables}
\partial_{x'}\phi(x; \Psi(x; u)) = u 
\end{equation}
for all $(x; u) \in U \subset X \times \R^{n-1}$. On each of the fibres $U_x := \{u \in \R^{n-1} : (x;u) \in U\}$ of the domain $U$, the map $u \mapsto \Psi(x;u)$ is a diffeomorphism. Thus, \eqref{R change of variables} implies that $\Sigma_x$ is the graph of the function 
\begin{equation*}
h_x(u) := \partial_{x_n}\phi(x; \Psi(x;u))
\end{equation*}
over the fibre $U_x$. Each $h_x$ is a perturbation of $\frac{|u|^2}{2}$ in the following sense.  

\begin{lemma}\label{hx nearly paraboloid} The function $h_x$ satisfies $h_x(0) = 0$, $\partial_{u}h_x(0) = 0$ and
\begin{equation*}
\|\partial_{u u}^2 h_x(u) - \mathrm{I}_{n-1} \|_{\mathrm{op}} = O(c_{\mathrm{par}}) 
\end{equation*}
for all $u \in U_x$.
\end{lemma}

Before proving the lemma, some simple properties of $\Psi$ are recorded. By \eqref{modulation invariance} it follows that $\Psi(x; 0) = 0$. The implicit function theorem implies that $\partial_{u}\Psi(x;u) = \partial_{x'\omega}^2 \phi(x;\Psi(x;u))^{-1}$ so that, by \eqref{close to identity} and the local Lipschitz continuity of taking matrix inverses,
\begin{equation}\label{Psi bound 1}
\|\partial_{u}\Psi(x;u) - \mathrm{I}_{n-1}\|_{\mathrm{op}} = O(c_{\mathrm{par}}).
\end{equation}
As a consequence of this identity (and choosing $c_{\mathrm{par}}$ to be sufficiently small),
\begin{equation*}
|\Psi(x;u) - \Psi(x;u')| \sim |u - u'| \qquad \textrm{for all $u, u' \in U_x$,}
\end{equation*}
where the implied constant depends only on $n$. In addition, if $1 \leq k \leq n-1$, then by twice differentiating $\partial_{x_k}\phi(x; \Psi(x; u)) = u_k$ in the $u$ variables one may deduce that the $k$-th co-ordinate $\Psi_k$ of $\Psi$ satisfies
\begin{equation}\label{Psi bound 3}
\|\partial_{uu}^2\Psi_k(x;u)\|_{\mathrm{op}} = O(c_{\mathrm{par}}).
\end{equation}
The stated properties of $h_x$ now easily follow. 

\begin{proof}[Proof (of Lemma~\ref{hx nearly paraboloid})] By \eqref{modulation invariance} one has
\begin{equation*}
h_x(0) = \partial_{x_n}\phi(x; \Psi(x;0)) = \partial_{x_n}\phi(x; 0) = 0.
\end{equation*}
Similarly, $\partial_{u}h_x(0) = 0$ since $\partial_{\omega}\partial_{x_n}\phi(x; 0) = 0$. Finally, 
\begin{equation*}
\partial_{uu}^2 h_x(u) = (\partial_{u}\Psi)^{\top}(x;u) (\partial_{\omega \omega}^2 \partial_{x_n}\phi)(x; \Psi(x;u)) \partial_{u}\Psi(x;u) + E(x;u)
\end{equation*}
where $E(x; \omega)$ is the $n-1 \times n-1$ matrix whose $(i,j)$th entry is given by
\begin{equation*}
E_{ij}(x; u) = \langle (\partial_{\omega}\partial_{x_n}\phi)(x; \Psi(x;u)), \partial_{u_iu_j}\Psi(x;u) \rangle. 
\end{equation*}
By \eqref{close to identity} and \eqref{Psi bound 3} it follows that $\|E(x;u)\|_{\mathrm{op}} = O(c_{\mathrm{par}})$, whilst \eqref{close to identity again} and multiple applications of  \eqref{Psi bound 1} imply that
\begin{equation*}
\|(\partial_{u}\Psi)^{\top}(x;u) (\partial_{\omega \omega}^2 \partial_{x_n}\phi)(x; \Psi(x;u)) \partial_{u}\Psi(x;u) - \mathrm{I}_{n-1}\|_{\mathrm{op}} = O(c_{\mathrm{par}}). 
\end{equation*}
This concludes the proof. 
\end{proof}

 
Similar reasoning can be used to provide useful uniform estimates for the generalised Gauss map associated to $T^{\lambda}$. To state the result, let 
\begin{equation*}
X^{\lambda} := \{x \in \R^n : x/\lambda \in X\}    
\end{equation*}
 denote the $\lambda$-dilate of $X$, so that $a^{\lambda}$ is supported in $X^{\lambda} \times \Omega$.

\begin{lemma}\label{Gauss Lipschitz} For all $x, \bar{x} \in X^{\lambda}$ and $\omega, \bar{\omega} \in \Omega$ the estimates
\begin{equation*}
\angle(G^{\lambda}(x; \omega), G^{\lambda}(x; \bar{\omega})) \sim |\omega - \bar{\omega}| \quad \textrm{and} \quad
\angle(G^{\lambda}(x;\omega), G^{\lambda}(\bar{x};\omega))\lesssim \lambda^{-1}|x-\bar{x}|
\end{equation*}
hold with implied constants which depend only on the dimension. 
\end{lemma}

The proof, which is an elementary calculus exercise in the style of the proof of Lemma~\ref{hx nearly paraboloid}, is omitted. 

If the $x$ parameter is restricted to a relatively small ball, then the second inequality in Lemma~\ref{Gauss Lipschitz} often allows the Gauss map to be treated as if it were constant in $x$. This is consistent with the idea that $T^{\lambda}$ is a small perturbation of a constant coefficient operator and can therefore be effectively approximated by constant coefficient operators at certain spatial scales. 




 \section{Basic analytic preliminaries}\label{Wave packet section}




\subsection{Wave packet decomposition} 

Throughout the following sections $\varepsilon > 0$ is a fixed small parameter and $\delta > 0$ is a tiny number satisfying\footnote{For $A, B \geq 0$ the notation $A \ll B$ or $B \gg A$ is used to denote that $A$ is `much smaller' than $B$; a more precise interpretation of this is that $A \leq C_{\varepsilon}^{-1}B$ for some constant $C_{\varepsilon} \geq 1$ which can be chosen to be large depending on $n$ and $\varepsilon$.} $\delta \ll \varepsilon$ and $\delta \sim_{\varepsilon} 1$.

A wave packet decomposition is carried out with respect to some spatial parameter $1\ll R \ll \lambda$. Cover $B^{n-1}$ by finitely-overlapping balls $\theta$ of radius $R^{-1/2}$ and let $\psi_{\theta}$ be a smooth partition of unity adapted to this cover. These $\theta$ will frequently be referred to as $R^{-1/2}$-\emph{caps}. Cover $\R^{n-1}$ by finitely-overlapping balls of radius $CR^{(1+\delta)/2}$ centred on points belonging to the lattice $R^{(1+\delta)/2}\Z^{n-1}$. By Poisson summation one may find a bump function adapted to $B(0, R^{(1+\delta)/2})$ so that the functions $\eta_v(z) := \eta(z- v)$ for $v \in R^{(1+\delta)/2}\Z^{n-1}$ form a partition of unity for this cover. Let $\mathbb{T}$ denote the collection of all pairs $(\theta, v)$. Thus, for $f \colon \R^{n-1} \to \C$ with support in $B^{n-1}$ and belonging to some suitable \emph{a priori} class one has
\begin{equation*}
f = \sum_{(\theta, v) \in \mathbb{T}} (\eta_v(\psi_{\theta}f)\;\widecheck{}\;)\;\widehat{}\; = \sum_{(\theta, v) \in \mathbb{T}} \widehat{\eta_v} \ast(\psi_{\theta}f).
\end{equation*} 
For each $R^{-1/2}$-cap $\theta$ let $\omega_{\theta} \in B^{n-1}$ denote its centre. Choose a real-valued smooth function $\tilde{\psi}$ so that the function $\tilde{\psi}_{\theta}(\omega) := \tilde{\psi}(R^{1/2}(\omega - \omega_{\theta}))$ is supported in $\theta$ and $\tilde{\psi}_{\theta}(\omega) = 1$ whenever $\omega$ belongs to a $cR^{-1/2}$ neighbourhood of the support of $\psi_{\theta}$ for some small constant $c > 0$. Finally, define
\begin{equation*}
f_{\theta, v} := \tilde{\psi}_{\theta} \cdot [\widehat{\eta_v} \ast (\psi_{\theta}f)]. 
\end{equation*}

\begin{definition}  The notation $ \mathrm{RapDec}(R)$ is used to denote any quantity $C_R$ which is rapidly decaying in $R$. More precisely, $C_R = \mathrm{RapDec}(R)$ if 
\begin{equation*}
    |C_R| \lesssim_{\varepsilon} R^{-N} \qquad \textrm{for all $N \leq \sqrt{N_{\mathrm{par}}}$.}
\end{equation*}
Here $N_{\mathrm{par}}$ is the large integer appearing in the definition of reduced phase from Definition~\ref{reduced definition}. By choosing $N_{\mathrm{par}}$ large one may assume, say,
\begin{equation*}
    \varepsilon^{-10^n} \ll \sqrt{N_{\mathrm{par}}} \ll \varepsilon^{10^n} N_{\mathrm{par}}.
\end{equation*} 
\end{definition}

The function $\widehat{\eta_v}(\omega)$ is rapidly decaying for $|\omega| \gtrsim R^{-(1+\delta)/2}$ and, consequently, 
\begin{equation*}
\|f_{\theta, v} - \widehat{\eta_v} \ast (\psi_{\theta}f)\|_{L^{\infty}(\R^{n-1})} \leq \mathrm{RapDec}(R)\|f\|_{L^2(B^{n-1})}.
\end{equation*}
 It follows that 
\begin{equation*}
\|f - \sum_{\theta, v}f_{\theta, v}\|_{L^{\infty}(\R^{n-1})} \leq \mathrm{RapDec}(R)\|f\|_{L^2(B^{n-1})}.
\end{equation*}

The functions $f_{\theta, v}$ are almost orthogonal: if $\mathbb{S} \subseteq \T$, then
\begin{equation*}
\big\| \sum_{(\theta, v) \in \mathbb{S}} f_{\theta, v} \big\|_{L^2(\R^{n-1})}^2 \sim  \sum_{(\theta, v) \in \mathbb{S}} \|f_{\theta, v} \|_{L^2(\R^{n-1})}^2.
\end{equation*}

Let $T^{\lambda}$ be an operator with reduced phase $\phi$ and amplitude $a$ supported in $X\times \Omega$ as in Lemma~\ref{reduction lemma}. For $(\theta, v)\in \T$ define the curve $\gamma_{\theta, v}^1 \colon I_{\theta,v}^{1} \to \R^{n-1}$ by setting $\gamma_{\theta, v}^1(t) := \Phi(v, t; \omega_{\theta})$, where $\Phi$ is the function introduced in $\S$\ref{Reductions section} and 
\begin{equation*}
I_{\theta,v}^{1} := \big\{ t \in X_n : \partial_{\omega}\phi(x', t; \omega_{\theta}) = v \textrm{ for some } x' \in X' \big\}.
\end{equation*}
 Thus, $\partial_{\omega} \phi(\gamma_{\theta,v}^1(t), t;\omega_{\theta}) = v$ for all $t \in I_{\theta,v}^{1}$. Moreover, the rescaled curve $\gamma^\lambda_{\theta,v}(t) :=\lambda \gamma^1_{\theta,v/\lambda}\left(t/\lambda\right)$ satisfies 
\begin{equation*}
\partial_{\omega}\phi^\lambda( \gamma^\lambda_{\theta, v}(t), t; \omega_{\theta}) = v \quad \textrm{for all $t \in I_{\theta,v}^{\lambda} := \{t \in \R : t/\lambda \in I_{\theta,v/\lambda}^{1}\}$.}
\end{equation*}
Let $\Gamma_{\theta, v}^{\lambda} \colon I_{\theta,v}^{\lambda} \to \R^n$ denote the graphing map $\Gamma_{\theta, v}^{\lambda}(t) := (\gamma_{\theta,v}^{\lambda}(t),t)$; by an abuse of notation $\Gamma_{\theta,v}^{\lambda}$ will also be used to denote the image of this mapping. The geometry of the curves $\Gamma_{\theta,v}^{\lambda}$ is related to the generalised Gauss map $G^{\lambda}$ by the following elementary lemma.

\begin{lemma} The tangent space $T_{\Gamma^{\lambda}_{\theta,v}(t)}\Gamma^{\lambda}_{\theta,v}$ lies in the direction of the unit vector $G^{\lambda}(\Gamma^{\lambda}_{\theta,v}(t); \omega_{\theta})$ for all $t \in I^{\lambda}_{\theta,v}$.
\end{lemma}

\begin{proof} Differentiating the equation $\partial_{\omega} \phi(\Gamma_{\theta,v}^{\lambda}(t);\omega_{\theta}) = v$, it follows that 
\begin{equation*}
    \partial_{\omega x}^2 \phi^{\lambda}(\Gamma_{\theta,v}^{\lambda}(t);\omega_{\theta})(\Gamma_{\theta,v}^{\lambda})'(t) = 0.
\end{equation*}
 Thus, $(\Gamma_{\theta,v}^{\lambda})'(t)$ must be parallel to $G^{\lambda}(\Gamma^{\lambda}_{\theta,v}(t); \omega_{\theta})$ since, by definition, the latter vector spans the kernel of $\partial_{\omega x}^2 \phi^{\lambda}(\Gamma_{\theta,v}^{\lambda}(t);\omega_{\theta})$. 
\end{proof}

Define the curved $R^{1/2+\delta}$-tube
\begin{equation*}
T_{\theta,v} := \big\{ (x',x_n) \in B(0,R) : x_n \in I_{\theta, v}^{\lambda} \textrm{ and } |x' - \gamma^\lambda_{\theta, v}(x_n)| \leq R^{1/2 + \delta} \big\}.
\end{equation*}
The curve $\Gamma_{\theta,v}^{\lambda}$ is referred to as the \emph{core} of $T_{\theta,v}$. Observe that, since $\phi$ is of the reduced form defined in $\S$\ref{Reductions section}, one has
\begin{equation}\label{wave packet concentration 0}
|x' - \gamma^\lambda_{\theta, v}(x_n)|\sim |\partial_\omega\phi^\lambda\left(x;\omega_\theta\right)-v|,
\end{equation}
for all $x =(x',x_n)\in X^{\lambda}$ with $x_n \in I_{\theta,v}^{\lambda}$ (uniformly in $\lambda$).

\begin{example} Let $\phi^{\lambda}(x; \omega) := \langle x', \omega \rangle + x_nh(\omega)$ and observe that $\gamma^\lambda_{\theta, v}(t) = v - t\partial_{\omega}h(\omega_{\theta})$ parametrises a straight line through $v$ in the direction of $\partial_{\omega}h(\omega_{\theta})$. The tube is given by
\begin{equation*}
T_{\theta,v} = \big\{ x \in B(0,R) : |x' + x_n \partial_{\omega}h(\omega_{\theta}) - v| \leq R^{1/2 + \delta} \big\},
\end{equation*}
which agrees with those studied in the case of the extension operator. 
\end{example}

\begin{lemma}\label{wave packet concentration lemma} If $1 \ll R \ll \lambda$ and $x \in B(0,R) \setminus T_{\theta, v}$, then 
\begin{equation}\label{wave packet concentration equation}
|T^\lambda f_{\theta, v}(x)| \leq (1 + R^{-1/2}|\partial_{\omega}\phi^{\lambda}(x;\omega_{\theta}) - v|)^{-(n+1)}\mathrm{RapDec}(R)\|f\|_{L^2(B^{n-1})}. 
\end{equation}
\end{lemma}

\begin{proof} Observe that 
\begin{equation*}
T^\lambda f_{\theta,v}(x) = \int_{\R^{n-1}} \widehat{\eta_v} \ast (\psi_{\theta}f) \cdot \overline{G_x}
\end{equation*}
where 
\begin{equation*}
G_x(\omega) := e^{-2\pi i \phi^\lambda(x;\omega)} a^\lambda (x; \omega) \tilde{\psi}_{\theta}(\omega)
\end{equation*}
and so, by Plancherel,
\begin{equation*}
T^\lambda f_{\theta,v}(x) = \int_{\R^{n-1}} \eta_v \cdot (\psi_{\theta}f)\;\widecheck{}\; \cdot \overline{\check{G}_x}.
\end{equation*}
By a simple change of variables, if one defines the phase and amplitude functions
\begin{align*}
    \phi_{x,z}^{\lambda,R}(\omega) &:= R^{1/2}\phi^\lambda (x;\omega_{\theta} + R^{-1/2}\omega) - \langle z, \omega \rangle, \\
    a_{x}^{\lambda,R}(\omega) &:= a^\lambda(x; \omega_{\theta} + R^{-1/2}\omega) \tilde{\psi}(\omega), 
\end{align*} 
then 
\begin{equation*}
|\check{G}_x(z)| = R^{-(n-1)/2} \big|\int_{\R^{n-1}} e^{-2 \pi i R^{-1/2}\phi_{x,z}^{\lambda,R}(\omega)} a_{x}^{\lambda,R}(\omega)\,\ud \omega \big|.
\end{equation*}
This integral is analysed using (non-)stationary phase. 

\begin{claim} Fixing $x \in B(0,R) \setminus T_{\theta,v}$, $z \in \mathrm{supp}\,\eta_v$ and $R \gg 1$, the estimates
\begin{enumerate}[i)]
    \item $R^{-1/2}|\partial_{\omega} \phi_{x,z}^{\lambda,R}(\omega)| \sim R^{-1/2}|\partial_{\omega}\phi^{\lambda}(x;\omega_{\theta}) - v| \gtrsim R^{\delta}$,
    \item $|\partial_{\omega}^{\alpha} \phi_{x,z}^{\lambda,R}(\omega)| \lesssim |\partial_{\omega} \phi_{x,z}^{\lambda,R}(\omega)|$ for all $2 \leq |\alpha| \leq N_{\mathrm{par}}$,
    \item $|\partial_{\omega}^{\alpha} a_{x}^{\lambda,R}(\omega)| \lesssim_{\varepsilon} 1$ for all $|\alpha| \leq N_{\mathrm{par}}$
\end{enumerate}
hold on $\mathrm{supp}\,a_x^{\lambda,R}$.
\end{claim}

Once the claim is established, repeated integration-by-parts yields
\begin{equation}\label{wave packet concentration 1}
|\check{G}_x(z)| = (1 + R^{-1/2}|\partial_{\omega}\phi^{\lambda}(x;\omega_{\theta}) - v|)^{-(n+1)}\mathrm{RapDec}(R).
\end{equation}
Such integration-by-parts or non-stationary phase arguments are standard but, for the reader's convenience (and since arguments of this kind will feature repeatedly in the article), the details are appended. The precise result used here is Lemma~\ref{integration-by-parts lemma}. The desired inequality \eqref{wave packet concentration equation} is an immediate consequence of \eqref{wave packet concentration 1} together with Young's inequality and Plancherel's theorem. 

It remains to establish the claim. Here the stated uniformity in the estimates is a consequence of the reductions made in \S\ref{Reductions section}. In view of Lemma~\ref{second reduction lemma}, the bound iii) for the amplitude is immediate and so it remains to show the bounds for the phase.

\subsubsection*{Proof of i)} Computing the derivative of the phase function,
\begin{align*}
\partial_{\omega} \phi_{x,z}^{\lambda,R}(\omega) &= \partial_{\omega}\phi^\lambda(x;\omega_{\theta} + R^{-1/2}\omega) -  z \\
&= \big[\partial_{\omega}\phi^\lambda(x;\omega_{\theta}) - v\big] + [v - z] + \big[\partial_{\omega}\phi^\lambda(x;\omega_{\theta} + R^{-1/2}\omega) - \partial_{\omega}\phi^\lambda(x;\omega_{\theta})\big].
\end{align*}
Observe that $|v - z| \lesssim R^{(1+\delta)/2}$ whenever $z \in \mathrm{supp}\,\eta_v$. Moreover,
\begin{equation*}
|\partial_{\omega_j}\phi^\lambda(x;\omega_{\theta} + R^{-1/2}\omega) - \partial_{\omega_j}\phi^\lambda(x;\omega_{\theta})| \leq R^{-1/2} \int_0^1 |\langle \partial_{\omega}\partial_{\omega_j} \phi^\lambda(x; \omega_{\theta} + tR^{-1/2}\omega), \omega \rangle| \,\ud t.
\end{equation*}
Since $\partial_{\omega_i\omega_j}^2 \phi^\lambda(0; \omega) = 0$ for all $\omega \in \Omega$ and $1 \leq i,j \leq n-1$, it follows that
\begin{align}
\nonumber
|\partial_{\omega_i\omega_j}^2 \phi^\lambda(x; \omega_{\theta} + tR^{-1/2}\omega)| &\leq \int_0^1 | \langle \partial_x \partial_{\omega_i\omega_j}^2 \phi^\lambda(sx; \omega_{\theta} + tR^{-1/2}\omega), x \rangle|\,\ud s \\
\label{wave packet concentration 2}
&\leq \|\partial_x \partial_{\omega_i\omega_j}^2 \phi^\lambda\|_{L^{\infty}(X^{\lambda}\times \Omega)}|x| \lesssim R
\end{align}
for $x \in B(0,R) \cap X^{\lambda}$, where the uniformity in the last inequality is due to Lemma~\ref{reduction lemma}. Combining these estimates,
\begin{equation*}
|\partial_{\omega}\phi^\lambda(x;\omega_{\theta} + R^{-1/2}\omega) - \partial_{\omega}\phi^\lambda(x;\omega_{\theta})|\lesssim R^{1/2}. 
\end{equation*}
On the other hand, for $x \in B(0,R)\setminus T_{\theta, v}$ it is claimed that
\begin{equation}\label{wave packet concentration 3}
|\partial_{\omega}\phi^{\lambda}(x;\omega_{\theta}) - v| \gtrsim R^{1/2 + \delta}.
\end{equation}
If $x_n \in I_{\theta,v}^{\lambda}$, then \eqref{wave packet concentration 3} follows directly from the definition of $T_{\theta,v}$ and \eqref{wave packet concentration 0}. Temporarily assuming \eqref{wave packet concentration 3} holds in general, it follows that the $\partial_{\omega}\phi^{\lambda}(x;\omega_{\theta}) - v$ term dominates in the above expansion of $\partial_{\omega} \phi_{x,z}^{\lambda,R}(\omega)$. In particular, for all $z \in \mathrm{supp}\,\eta_v$, one concludes that
\begin{equation*}
R^{-1/2}|\partial_{\omega} \phi_{x,z}^{\lambda,R}(\omega)| \sim R^{-1/2}|\partial_{\omega}\phi^{\lambda}(x;\omega_{\theta}) - v| \gtrsim R^{\delta}
\end{equation*}
whenever $R \gg 1$. 

It remains to establish \eqref{wave packet concentration 3} for $x \in B(0,R)\setminus T_{\theta, v}$ with $x_n \notin I_{\theta,v}^{\lambda}$. The condition on $x_n$ implies that $v \notin \tilde{X}'$ where $\tilde{X}'$ is defined to be the image of $X'$ under the diffeomorphism $z' \mapsto \lambda\partial_{\omega}\phi(z', x_n/\lambda;\omega_{\theta})$. The set $\tilde{X}'$ will contain a ball centred at $\lambda\partial_{\omega} \phi(0, x_n/\lambda; \omega_{\theta})$ of radius $2c\lambda$ for some dimensional constant $c >0$. Since $|x_n| < R$ and $\partial_{\omega} \phi(0; \omega_{\theta}) = 0$, one observes that $\lambda|\partial_{\omega} \phi(0, x_n/\lambda; \omega_{\theta})| \lesssim R$ and so $B(0, c\lambda) \subseteq \tilde{X}'$, provided $R \ll \lambda$. Consequently, $|v| \gtrsim \lambda$ whilst, on the other hand, $|\partial_{\omega}\phi^{\lambda}(x; \omega_{\theta})| \lesssim R$ and so, again provided $R \ll \lambda$, the estimate \eqref{wave packet concentration 3} immediately follows.\footnote{This argument can also be used to show that for $|v| \lesssim \lambda$, the domain $I_{\theta, v}^{\lambda}$ contains an interval about 0 of length $\sim \lambda$.} 

\subsubsection*{Proof of ii)} Fix $\alpha \in \N_0^n$ with $2 \leq |\alpha| \leq N_{\mathrm{par}}$. By arguing as in \eqref{wave packet concentration 2}, one obtains
\begin{align*}
    |\partial_{\omega}^{\alpha} \phi^{\lambda,R}_{x,z}(\omega)| &= R^{-(|\alpha| - 1)/2}|(\partial_{\omega}^{\alpha} \phi^{\lambda})(x; \omega_{\theta} + R^{-1/2}\omega)| \\
    &\leq R^{-(|\alpha| - 1)/2}\|\partial_x \partial_{\omega}^{\alpha} \phi^\lambda\|_{L^{\infty}(X^{\lambda}\times \Omega)}|x| \lesssim R^{1/2},
\end{align*}
where the uniformity in the last inequality is due to Lemma~\ref{reduction lemma} and Lemma~\ref{second reduction lemma}. Since $|\partial_{\omega} \phi^{\lambda,R}_{x,z}(\omega)| \geq R^{1/2}$ by i), this concludes the proof. 

\end{proof}




\subsection{An $L^2$ estimate} The following standard $L^2$-bound, which has been mentioned previously in $\S$\ref{Necessary conditions section}, will play a significant r\^ole in the forthcoming analysis.  

\begin{lemma}[H\"ormander \cite{Hormander1973}]\label{Hormander L2} If $1 \leq R\leq \lambda$ and $B_R$ is any ball of radius $R$, then
\begin{equation*}
\|T^\lambda f\|_{L^2(B_{R})} \lesssim R^{1/2} \|f\|_{L^2(B^{n-1})}.
\end{equation*}
\end{lemma}

This lemma is a direct corollary of the following lemma which, in turn, is a consequence of H\"ormander's generalisation of the Hausdorff--Young inequality \cite{Hormander1973}. 

\begin{lemma}\label{Hormander L2 again} For any fixed $x_n \in \R$, the estimate
\begin{equation*}
\|T^\lambda f\|_{L^2(\R^{n-1}\times \{x_n\})} \lesssim \|f\|_{L^2(B^{n-1})} 
\end{equation*}
holds. 
\end{lemma}

\begin{proof} Defining $Sf(x') := T^\lambda f(\lambda x',x_n)$, the problem is to show that
\begin{equation}\label{Hormander Hausdorff-Young}
\|Sf\|_{L^2(\R^{n-1})} \lesssim \lambda^{-(n-1)/2}\|f\|_{L^2(\R^{n-1})}.
\end{equation}
Observe that
\begin{equation*}
Sf(x') = \int_{\R^{n-1}} e^{2\pi i \lambda \phi(x', x_n/\lambda;\omega)} a(x', x_n/\lambda;\omega) f(\omega)\,\ud \omega.
\end{equation*}
The original hypotheses on the phase $\phi$ imply that 
\begin{equation*}
|\det\partial_{x'\omega}^2\phi(x', x_n/\lambda;\omega)| \gtrsim 1
\end{equation*}
whilst $(x'; \omega) \mapsto a(x', x_n/\lambda;\omega)$ has support in some bounded subset of $\R^{n-1} \times \R^{n-1}$. Both these conditions hold uniformly in $x_n$ and $\lambda$. Thus, the operator $S$ satisfies the conditions of H\"ormander's generalisation of the Hausdorff--Young inequality \cite{Hormander1973} (see also, for instance, \cite[p. 377]{Stein1993}) uniformly in  $x_n$ and $\lambda$. Applying H\"ormander's theorem immediately yields \eqref{Hormander Hausdorff-Young}. 
\end{proof}




\subsection{The locally constant property}\label{locally constant section} As a final analytic preliminary, some simple consequences of the uncertainty principle are discussed. It is remarked that the result of this subsection (that is, Lemma~\ref{locally constant lemma}) only plays a r\^ole in the proof of Theorem~\ref{main theorem} much later in the argument (namely, in the parabolic rescaling argument in $\S$\ref{k-broad to linear section}). It does, however, feature in an independent discussion in the following section. 

\begin{definition} A function $\zeta \colon \R^n \to [0,\infty)$ is said to be \emph{locally constant at scale $\rho$} for some $\rho > 0$ if $\zeta(x) \sim \zeta(y)$ for all $x, y \in \R^n$ with $|x - y| \lesssim \rho$.
\end{definition}

Owing to the uncertainty principle, heuristically one expects the following: if $f$ is supported on a $\rho^{-1}$-cap, then $|T^{\lambda}f|$ is essentially constant at scale $\rho$. For extension operators this is due to the fact that, under the support hypothesis on the input function, $Ef$ has (distributional) Fourier support inside a $\rho^{-1}$-ball. For general H\"ormander-type operators $T^{\lambda}$ the Fourier transform of $T^{\lambda}f$ does not necessarily have compact support. It will, however, be concentrated in some $\rho^{-1}$-ball and this is sufficient to ensure the locally constant property holds. This discussion is formalised by the following lemma.

\begin{lemma}\label{locally constant lemma} Let $T^{\lambda}$ be a H\"ormander-type operator and $1 \leq k \leq n$. There exists a smooth, rapidly decreasing function $\zeta \colon \R^n \to [0,\infty)$ with the following properties:
\begin{enumerate}[1)]
\item $\zeta$ is locally constant at scale 1.
\item If $\delta > 0$ and $1 \leq \rho \leq \lambda^{1- \delta}$, then the pointwise inequality
\begin{equation*}
|T^{\lambda}f|^{1/k} \lesssim |T^{\lambda}f|^{1/k}\ast \zeta_{\rho} +  \mathrm{RapDec}(\lambda)\|f\|_{L^2(B^{n-1})}^{1/k}
\end{equation*}
holds whenever $f$ is supported in some $\rho^{-1}$-ball. Here $\zeta_{\rho}(x) := \rho^{-n}\zeta(x/\rho)$. 
\end{enumerate}
\end{lemma}

It is useful to work with the parameter $k$ here in order to apply the locally constant property effectively in $k$-linear settings. 

The locally constant property of $\zeta$ implies that
\begin{equation*}
|T^{\lambda}f|^{1/k}\ast \zeta_{\rho}(x) \sim |T^{\lambda}f|^{1/k}\ast \zeta_{\rho}(y)  \quad \textrm{for all $x, y \in \R^n$ with $|x - y| \lesssim \rho$;}
\end{equation*}
namely, $|T^{\lambda}f|^{1/k}\ast \zeta_{\rho}$ is locally constant at scale $\rho$. This is a rigorous formulation of the locally constant heuristic discussed above. 

\begin{proof}[Proof (of Lemma~\ref{locally constant lemma})] Suppose that $\mathrm{supp}\,f \subset B(\bar{\omega},\rho^{-1})$ where $\bar{\omega} \in \Omega$ and observe that
\begin{equation*}
[e^{-2\pi i \phi^{\lambda}(\,\cdot\,;\bar{\omega})}T^{\lambda}f]\;\widehat{}\;(\xi) = \int_{\R^{n-1}} K^{\lambda}(\xi;\omega) f(\omega)\,\ud \omega
\end{equation*}
where the function $K^{\lambda}$ is given by
\begin{equation*}
K^{\lambda}(\xi;\omega) = \lambda^n \int_{\R^n} e^{-2\pi i\lambda(\langle x, \xi \rangle - \phi(x ;\omega) + \phi(x;\bar{\omega}))} a(x;\omega)\,\ud x.
\end{equation*}
This oscillatory integral is estimated via (non-)stationary phase, using the simple estimate
\begin{equation*}
|\partial_x\phi(x ;\omega) - \partial_x \phi(x;\bar{\omega}) | \lesssim \rho^{-1} \qquad \textrm{for $(x;\omega) \in X \times \Omega$ with $\omega \in \mathrm{supp}\,f$.}
\end{equation*}
In particular, if $|\xi| \geq C\rho^{-1}$ for a suitably large constant $C \geq 1$, then repeated integration-by-parts, combined with the control on the derivatives of $a$ ensured by Lemma~\ref{third reduction lemma}, shows that
\begin{equation*}
|K^{\lambda}(\xi;\omega)| \leq  \mathrm{RapDec}(\lambda) (1 + |\xi|)^{-(n+1)}.
\end{equation*}
Let  $\eta$ be a Schwartz function on $\R^n$ with $\hat{\eta}(\xi) = 1$ for all $|\xi| < C$ for a suitable constant $C \geq 1$ and support in $B(0,2C)$. Such a function can further be chosen so that $|\eta|^{1/k}$ admits a smooth, rapidly decreasing majorant $\zeta$ which is locally constant at scale 1. From the above observations,
\begin{equation*}
[e^{-2\pi i \phi^{\lambda}(\,\cdot\,;\bar{\omega})}T^{\lambda}f]\;\widehat{}\;(\xi)  = [e^{-2\pi i \phi^{\lambda}(\,\cdot\,;\bar{\omega})}T^{\lambda}f]\;\widehat{}\;(\xi)\widehat{\eta_{\rho}}(\xi) +  E(f,\lambda)(\xi)
\end{equation*}
where $|E(f,\lambda)(\xi)| \leq \mathrm{RapDec}(\lambda) (1 + |\xi|)^{-(n+1)}\|f\|_{L^2(B^{n-1})}$. Applying Fourier inversion and using the triangle inequality to estimate the error, 
\begin{equation*}
e^{-2\pi i \phi^{\lambda}(x,;\bar{\omega})}T^{\lambda}f(x) = [e^{-2\pi i \phi^{\lambda}(\,\cdot\,;\bar{\omega})}T^{\lambda}f] \ast \eta_{\rho}(x) + \mathrm{RapDec}(\lambda)\|f\|_{L^2(B^{n-1})}
\end{equation*}
and, in particular, 
\begin{equation*}
|T^{\lambda}f(x)| \leq \int_{\R^n}|T^{\lambda}f(x-y)\eta_{\rho}(y)|\,\ud y +  \mathrm{RapDec}(\lambda)\|f\|_{L^2(B^{n-1})}.
\end{equation*}
Observe that, for fixed $x$, the function appearing in absolute values in the above integrand has Fourier support in a ball of radius $O(\rho^{-1})$.  Bernstein's inequality\footnote{More precisely, here the proof uses a general form of Bernstein's inequality, valid for exponents less than 1. In particular, if $0< p\leq q \leq \infty$ and $g$ is an integrable function on $\R^n$ satisfying $\mathrm{supp}\,\hat{g} \subseteq B_r$, then 
\begin{equation*}
\|g\|_{L^q(\R^n)} \lesssim r^{n (1/p - 1/q)} \|g\|_{L^p(\R^n)}.
\end{equation*}
This extension follows from the classical Bernstein inequality (that is, the above estimate in the restricted range $1\leq p \leq q \leq \infty$) in a rather straight-forward manner. The classical Bernstein inequality is itself a direct consequence of Young's convolution inequality: see, for instance, \cite[\S5]{Wolff2003}.} may therefore be applied to dominate the right-hand side by 
\begin{equation*}
 \Big(\int_{\R^n}|T^{\lambda}f(x-y)|^{1/k} \zeta_{\rho}(y)\,\ud y\Big)^k +  \mathrm{RapDec}(\lambda)\|f\|_{L^2(B^{n-1})},
\end{equation*}
which concludes the proof.
\end{proof}




\section{Properties of the $k$-broad norms}\label{Broad norm section}

\subsection{$k$-broad triangle inequality and logarithmic convexity} The functional $f \mapsto \|T^{\lambda}f\|_{\mathrm{BL}^p_{k,A}(U)}$ is not a norm in a literal sense, but it does exhibit some properties similar to those of $L^p$-norms. For instance, the map $U \mapsto \|T^{\lambda}f\|_{\mathrm{BL}^p_{k,A}(U)}^p$ behaves similarly to a measure.
 
\begin{lemma}[Finite (sub)-additivity] Let $U_1, U_2 \subseteq \R^n$ and $U := U_1 \cup U_2$. If $1 \leq p < \infty$ and $A$ is a non-negative integer, then
\begin{equation*}
\|T^{\lambda}f\|_{\mathrm{BL}^p_{k,A}(U)}^p \leq \|T^{\lambda}f\|_{\mathrm{BL}^{p}_{k,A}(U_1)}^p + \|T^{\lambda} f\|_{\mathrm{BL}^{p}_{k,A}(U_2)}^p
\end{equation*} 
holds for all integrable $f \colon B^{n-1} \to \C$.
\end{lemma}

This is an immediate consequence of the definition of the $k$-broad norms. A slightly less trivial observation is that $\|T^{\lambda}f\|_{\mathrm{BL}^p_{k,A}(U)}$ also satisfies weak versions of the triangle and logarithmic convexity inequalities. 

\begin{lemma}[Triangle inequality \cite{Guth2018}]\label{triangle inequality lemma} If $U \subseteq \R^n$, $1 \leq p < \infty$ and $A := A_1 + A_2$ for $A_1, A_2$ non-negative integers, then
\begin{equation*}
\|T^{\lambda}(f_1 + f_2)\|_{\mathrm{BL}^p_{k,A}(U)} \lesssim \|T^{\lambda}f_1\|_{\mathrm{BL}^p_{k,A_1}(U)} + \|T^{\lambda} f_2\|_{\mathrm{BL}^p_{k,A_2}(U)}
\end{equation*} 
holds for all integrable $f_1, f_2 \colon B^{n-1} \to \C$.
\end{lemma}

\begin{lemma}[Logarithmic convexity \cite{Guth2018}]\label{logarithmic convexity inequality lemma} Suppose that $U \subseteq \R^n$, $1 \leq p, p_1, p_2 < \infty$ and $0 \leq \alpha_1, \alpha_2 \leq 1$ satisfy $\alpha_1 + \alpha_2 = 1$ and
\begin{equation*}
\frac{1}{p} = \frac{\alpha_1}{p_1} + \frac{\alpha_2}{p_2}.
\end{equation*}
If $A := A_1 + A_2$ for $A_1, A_2$ non-negative integers, then
\begin{equation*}
\|T^{\lambda}f\|_{\mathrm{BL}^p_{k,A}(U)} \lesssim \|T^{\lambda}f\|_{\mathrm{BL}^{p_1}_{k,A_1}(U)}^{\alpha_1} \|T^{\lambda} f\|_{\mathrm{BL}^{p_2}_{k,A_2}(U)}^{\alpha_2}
\end{equation*} 
holds for all integrable $f \colon B^{n-1} \to \C$.
\end{lemma}

These estimates are proven in the context of Fourier extension operators in \cite{Guth2018}. The arguments are entirely elementary and readily generalise to the variable coefficient case. It is remarked that the parameter $A$ appears in the definition of the $k$-broad norm to allow for these weak triangle and logarithmic convexity inequalities.

\subsection{$k$-broad versus $k$-linear} A relationship between $k$-broad and $k$-linear estimates is given by the following proposition.

\begin{proposition}\label{k-broad vrs k-linear proposition} Let $\mathcal{T}$ be a class of H\"ormander-type operators which is closed under translation,\footnote{That is, if $T^{\lambda} \in \mathcal{T}$ and $a \in \R^n$, then the translated operator $T^{\lambda}_a$ defined by $T^{\lambda}_af(x) := T^{\lambda}f(x+a)$ also belongs to $\mathcal{T}$.} $2 \leq p < \infty$, $2 \leq k \leq n$ and $\varepsilon >0$. Suppose that for all $1 \ll R \leq \lambda$ and $R$-balls $B_R$ the $k$-linear inequality
\begin{equation*}
\big\|\prod_{j=1}^k|T^{\lambda}_jf_j|^{1/k}\|_{L^p(B_R)} \lesssim_{\varepsilon, (\phi_j)_{j=1}^k} \nu^{-C_{\varepsilon}} R^{\varepsilon}\|f\|_{L^2(B^{n-1})}
\end{equation*}
holds whenever $(T_1^{\lambda}, \dots, T_k^{\lambda}) \in \mathcal{T}^k$ is a $\nu$-transverse $k$-tuple of H\"ormander-type operators. Then for all $1 \ll R \leq \lambda$ and $R$-balls $B_R$ the $k$-broad inequality
\begin{equation*}
\|T^{\lambda}f\|_{\mathrm{BL}^p_{k,1}(B_R)} \lesssim_{\varepsilon, \phi} K^{C_{\varepsilon}} R^{\varepsilon}\|f\|_{L^2(B^{n-1})}
\end{equation*}
holds for any $T^{\lambda} \in \mathcal{T}$. 
\end{proposition}

Recall, the notion of $\nu$-transversality was introduced in Definition~\ref{transverse definition}. The parameter $K$ in the above theorem is the same as that which appears in the definition of the $k$-broad norms; the $C_{\varepsilon}$ denote constants, which may vary from line to line, which depend only on $n$ and $\varepsilon$. 

The (local version of the) Bennett--Carbery--Tao theorem \cite{Bennett2006} therefore implies a version of Theorem~\ref{k-broad theorem} which holds for all H\"ormander-type operators (that is, without the positive-definite hypothesis) with a restricted range of $p$.\footnote{The version of the Bennett--Carbery--Tao theorem used here is not explicitly stated in \cite{Bennett2006} (there the variable coefficient estimates are only presented at the $n$-linear level). Nevertheless, $k$-linear inequalities for H\"ormander-type operators are readily obtained by combining the analysis of \S5 and \S6 of \cite{Bennett2006}: see \cite[\S5]{Bourgain2011}.}

\begin{corollary}\label{Bennett Carbery Tao corollary} Let $T^{\lambda}$ be a H\"ormander-type operator. For all $2 \leq k \leq n$,  $p \geq 2k/(k-1)$ and $\varepsilon > 0$ the estimate
\begin{equation*}
\|T^{\lambda}f\|_{\mathrm{BL}^p_{k,1}(B_R)} \lesssim_{\varepsilon, \phi} K^{C_{\varepsilon}} R^{\varepsilon} \|f\|_{L^2(B^{n-1})}
\end{equation*}
holds for all $\lambda \geq 1$. 
\end{corollary}

For completeness the proof of Proposition~\ref{k-broad vrs k-linear proposition} is given; the result itself will not be used in the proof of Theorem~\ref{main theorem} and is included mainly for expository purposes. Thus, readers interested only in the proof of Theorem~\ref{main theorem} may safely skip to the next section.  

\begin{proof}[Proof (of Proposition~\ref{k-broad vrs k-linear proposition})] Let $\mathcal{Z} \subset B_R$ be a maximal set of points with the property that the balls $B(z, R/2\bar{C}K)$ for $z \in \mathcal{Z}$ are pairwise disjoint. Here $\bar{C} \geq 1$ is a suitable constant, chosen to meet the forthcoming requirements of the proof. Letting $B_z := B(z, R/\bar{C}K)$ for $z\in \mathcal{Z}$, it follows that $\# \mathcal{Z} \lesssim K^n$ and
\begin{equation*}
\|T^{\lambda}f\|_{\mathrm{BL}^p_{k,1}(B_R)}^p \leq \sum_{z \in \mathcal{Z}} \|T^{\lambda}f\|_{\mathrm{BL}^p_{k,1}(B_z)}^p.
\end{equation*}
Fixing $z \in \mathcal{Z}$  it therefore suffices to show that
\begin{equation*}
\|T^{\lambda}f\|_{\mathrm{BL}^p_{k,1}(B_z)} \lesssim_{\varepsilon} K^{C_{\varepsilon}}R^{\varepsilon} \|f\|_{L^2(B^{n-1})},
\end{equation*}
since summing the contributions from each choice of $z \in \mathcal{Z}$ only introduces an acceptable $K^n$ factor into the estimate. By introducing a bump function into the definition of the operator, one may further assume that the amplitude $a^{\lambda}$ has $x$-support in the ball concentric to $B_z$ with twice the radius.

 Fix a ball $B_{K^2} = B(\bar{x}, K^2) \in \mathcal{B}_{K^2}$ with $B_{K^2} \cap B_z \neq \emptyset$ and suppose that $V$ is a $(k-1)$-dimensional subspace which realises the minimum in $\mu_{T^{\lambda}f}(B_{K^2})$. Thus, by definition, if $\tau$ is a $K^{-1}$-cap for which 
\begin{equation*}
 \|T^{\lambda}f_{\tau}\|_{L^p(B_{K^2})}^p > \mu_{T^{\lambda}f}(B_{K^2}), 
\end{equation*}
then $\tau \in V$, where the inclusion symbol is used in the non-standard sense described in the introduction. Amongst all such subspaces $V$ choose one which maximises the cardinality of the set 
\begin{equation*}
\mathcal{T}(V) :=  \big\{ \tau \in V : \|T^{\lambda} f_{\tau}\|_{L^p(B_{K^2})}^p \geq \mu_{T^{\lambda}f}(B_{K^2})\big\}.
\end{equation*}
By definition there exists some cap $\tau^* \notin V$ such that  $\|T^{\lambda}f_{\tau^*}\|_{L^p(B_{K^2})}^p = \mu_{T^{\lambda}f}(B_{K^2})$. 

Suppose there exists a $(k-2)$-dimensional subspace $W \subset \R^n$ such that $\tau \in W$ for all $\tau \in \mathcal{T}(V)$. Then defining $V' := \mathrm{span}\,\big( W \cup \{G^{\lambda}(\bar{x}, \omega_{\tau^*})\}\big)$ where $\omega_{\tau^*}$ is the centre of $\tau^*$, it follows that $\tau^* \in V'$ and $\tau \in V'$ for all $\tau \in \mathcal{T}(V)$. On the other hand, $V'$ also realises the minimum in the definition in $\mu_{T^{\lambda}f}(B_{K^2})$, since
\begin{equation*}
\|T^{\lambda}f_\tau\|_{L^p(B_{K^2})}^p \leq \mu_{T^{\lambda}f}(B_{K^2}) \qquad \textrm{for all $\tau \notin V'$ with $\tau \in V$;}
\end{equation*}
this is immediate by the fact that, if $\tau \notin V'$, then $\tau \notin W$, so $\tau$ does not belong to $\mathcal{T}(V)$. These observations contradict the maximality of $V$ and, consequently, no such subspace $W$ can exist.

By the preceding discussion, one may find a family of caps $\tau_1^*, \dots, \tau_{k-1}^* \in \mathcal{T}(V)$ satisfying
\begin{equation}\label{transversal 1}
\big|\bigwedge_{j=1}^kG^{\lambda}(\bar{x}, \omega_j) \big| \gtrsim K^{-(k-1)} \quad \textrm{for all $\omega_j \in \tau_j^*$, $1 \leq j \leq k$.}
\end{equation}
Thus,
\begin{equation}\label{k broad vrs k linear 1}
 \mu_{T^{\lambda}f}(B_{K^2}) \leq \prod_{j=1}^k \|T^{\lambda} f_{\tau_j^*}\|_{L^p(B_{K^2})}^{p/k}
 \end{equation}
for $\tau_k^* := \tau^*$. To apply the hypothesised multilinear estimate one wishes to exchange the order of taking the norm and product on the right-hand side of the above expression; that is, one wishes to prove an estimate of the form
\begin{equation*}
 \prod_{j=1}^k \|T^{\lambda} f_{\tau_j^*}\|_{L^p(B_{K^2})}^{p/k} \lesssim K^{O(1)} \big\|\prod_{j=1}^k |T^{\lambda}f_{\tau_j^*}|^{1/k}\big\|_{L^p(B_{K^2})}^{p}.
\end{equation*}
This is achieved by exploiting the locally constant property of the $T^{\lambda}f_{\tau}$, as discussed in $\S$\ref{locally constant section}. In particular, by Lemma~\ref{locally constant lemma} and H\"older's inequality there exists a non-negative, rapidly decreasing, locally constant function $\zeta$ such that
\begin{equation}\label{k broad vrs k linear 2}
|T^{\lambda}f_{\tau}|^{p/k} \lesssim |T^{\lambda}f_{\tau}|^{p/k}\ast \zeta_{K} + \mathrm{RapDec}(\lambda)\|f\|_{L^2(B^{n-1})}^{p/k}
\end{equation}
holds for all $K^{-1}$-caps $\tau$. Since rapidly decaying error terms are entirely harmless, henceforth they will be suppressed in the notation. Observe that for all $z \in B(\bar{x}, K^2)$ and $y \in \R^n$ one has $\zeta_K(z-y) \lesssim K^{O(1)} w_{K}(\bar{x}-y)$ where $w_{K}(y) := (1+ |y|/K)^{-N}$ for some choice suitable of large exponent $N$ satisfying $N = O(1)$. Combining these observations with a second application of \eqref{k broad vrs k linear 2} yields
\begin{equation}\label{k broad vrs k linear 3}
\|T^{\lambda}f_{\tau}\|_{L^p(B_{K^2})}^{p/k} \lesssim K^{O(1)} \int_{\R^n} |T^{\lambda}f_{\tau}|^{p/k}\ast \zeta_{K}(y) w_{K}(\bar{x}-y)\ud y.
\end{equation}
 Temporarily fix $x \in B(0, K)$ and note that the locally constant property of $\zeta$ implies that
\begin{equation}\label{k broad vrs k linear 4}
|T^{\lambda} f_{\tau}|^{p/k}\ast \zeta_{K}(y) \lesssim |T^{\lambda}f_{\tau}|^{p/k}\ast \zeta_{K}(x+y) \qquad \textrm{for all $y \in \R^n$.}
\end{equation}
Thus, by \eqref{k broad vrs k linear 1}, \eqref{k broad vrs k linear 3} and \eqref{k broad vrs k linear 4}, one deduces that 
\begin{equation*}
\mu_{T^{\lambda}f}(B_{K^2}) \lesssim  K^{O(1)} \int_{(\R^n)^k}\prod_{j=1}^k |T^{\lambda}f_{\tau_j^*}|^{p/k}\ast \zeta_{K}(x+y_j)w_{K}(\bar{x} - y_j)\ud \mathbf{y}.
\end{equation*}
Taking the average of both sides of this estimate over all $x \in B(0,K)$ and shifting the $y_j$ variables, 
\begin{equation*}
\mu_{T^{\lambda}f}(B_{K^2}) \lesssim  K^{O(1)} \int\displaylimits_{(\R^n)^{2k}} \int\displaylimits_{B(\bar{x},K)}  \prod_{j=1}^k |T^{\lambda}f_{\tau_j^*}\circ\sigma_{y_j,z_j}(x)|^{p/k}\,\ud x Z_K(\mathbf{y},\mathbf{z})\ud \mathbf{y}\ud \mathbf{z}
\end{equation*}
where $Z_K(\mathbf{y}, \mathbf{z}) := \prod_{j=1}^k w_{K}(y_j)\zeta_K(z_j)$ and $\sigma_{y_j,z_j}(x) := x+ y_j -z_j$. Since both $w_K$ and $\zeta_K$ decay (at least) as rapidly as $|y|^{-N}$ away from $B(0,K)$, one may restrict the integral in $\mathbf{y}$ and $\mathbf{z}$ from the whole space $(\R^n)^{2k}$ to the bounded region $B(0,\lambda/\bar{C}K)^{2k}$ at the expense of an additional harmless error term. 

It is possible to localise to a finer scale than $\lambda/\bar{C}K$, but this scale suffices for the purposes of the proof. In particular, if $|y|, |z| < \lambda/\bar{C}K$, then it follows from Lemma~\ref{Gauss Lipschitz} that 
\begin{equation}\label{k broad vrs k linear 5}
|G^{\lambda}(x, \omega)  - G^{\lambda}(\sigma_{y,z}(x), \omega)| \lesssim  \frac{|x - \sigma_{y,z}(x)|}{\lambda} \lesssim \bar{C}^{-1}K^{-1} \quad \textrm{for all $(x,\omega) \in \mathrm{supp}\,a$.}
\end{equation}
If $\bar{C}$ is chosen sufficiently large, then this bound is enough to ensure that pre-composing by $\sigma_{y,z}$ preserves certain transversality properties, as discussed below. 

Given a $K^{-1}$-cap $\tau$, let $T_{\tau}^{\lambda}$ be a H\"ormander-type operator given by replacing the amplitude $a^{\lambda}$ in the definition of $T^{\lambda}$ with some amplitude $a_{\tau}^{\lambda}$ which has $\omega$-support in a $2K^{-1}$-cap concentric to $\tau$ and which satisfies $T_{\tau}^{\lambda}f_{\tau} = T^{\lambda}f_{\tau}$. One now wishes to bound
\begin{equation}\label{k broad vrs k linear 6}
 \int\displaylimits_{B(0,\lambda/\bar{C}K)^{2k}} \int\displaylimits_{B(\bar{x},K)}  \prod_{j=1}^k |T_{\tau_j^*}^{\lambda}f_{\tau_j^*}\circ\sigma_{y_j,z_j}(x)|^{p/k}\,\ud x Z_K(\mathbf{y},\mathbf{z})\,\ud \mathbf{y}\ud \mathbf{z}
\end{equation}
For the purposes of this proof a $k$-tuple $(\tau_1, \dots, \tau_k)$  of $K^{-1}$-caps is said to be \emph{transverse} if $(T_{\tau_1}^{\lambda}, \dots, T_{\tau_k}^{\lambda})$ is a $cK^{-(k-1)}$-transverse $k$-tuple of H\"ormander-type operators, for a suitable choice of small constant $c > 0$. Now suppose $(x,\omega) \in \mathrm{supp}\,a$ so that, by the original decomposition, both $x$ and $\bar{x}$ lie in a ball of radius $R/\bar{C}K$. Since $R \leq \lambda$, if follows from Lemma~\ref{Gauss Lipschitz} that
\begin{equation*}
|G^{\lambda}(\bar{x}, \omega)  - G^{\lambda}(x, \omega)| \lesssim  \frac{|\bar{x} - x|}{\lambda} \lesssim \bar{C}^{-1}K^{-1} \quad \textrm{for all $(x,\omega) \in \mathrm{supp}\,a$.}
\end{equation*}
Thus, choosing $\bar{C}$ sufficiently large, in addition to \eqref{transversal 1}, one may assume that $(\tau_1^*, \dots, \tau_k^*)$ is transverse. The expression \eqref{k broad vrs k linear 6} is therefore dominated by
 \begin{equation*}
\sum_{(\tau_j)_{j=1}^k\, \mathrm{trans.}} \,\,\int\displaylimits_{B(0,\lambda/\bar{C}K)^{2k}} \int\displaylimits_{B(\bar{x},K)}  \prod_{j=1}^k |T_{\tau_j}^{\lambda}f_{\tau_j}\circ\sigma_{y_j,z_j}(x) |^{p/k}\,\ud x Z_K(\mathbf{y}, \mathbf{z})\,\ud \mathbf{y}\ud \mathbf{z},
\end{equation*}
where the sum is over all choices of transverse $k$-tuples of caps. Summing both sides of this inequality over all $B(\bar{x},K^2) \in \mathcal{B}_{K^2}$ with $B(\bar{x},K^2) \cap B_z \neq \emptyset$, it suffices to show that
\begin{equation*}
 \int_{\R^n}  \prod_{j=1}^k |T_{\tau_j}^{\lambda}f_{\tau_j}\circ\sigma_{y_j,z_j}(x)|^{p/k}\,\ud x \lesssim R^{\varepsilon} \|f\|_{L^2(B^{n-1})}^p
\end{equation*}
for any choice of $K^{-(k-1)}$-transverse tuple $(\tau_1, \dots, \tau_k)$ and any $\mathbf{y}, \mathbf{z} \in B(0,\lambda/\bar{C}K)^k$. However, defining $T_j^{\lambda}f(x) := T_{\tau_j}^{\lambda}f\circ\sigma_{y_j,z_j}(x)$ and again choosing $\bar{C}$ to be sufficiently large, it follows from \eqref{k broad vrs k linear 5} that these operators are $\sim K^{-(k-1)}$-transverse in the sense of Definition~\ref{transverse definition}. Thus, the desired estimate is an immediate consequence of the hypothesised multilinear inequality.

\end{proof}




 \section{Algebraic preliminaries}\label{Algebraic preliminaries section}




\subsection{Basic definitions and results} Let $0 \leq m \leq n$ and consider a collection of real polynomials $P_j \in \R[X_1, \dots, X_n]$, $1 \leq j \leq n-m$. Let $Z(P_1, \dots, P_{n-m})$ denote their zero locus; that is,
\begin{equation*}
Z(P_1, \dots, P_{n-m}) := \big\{ x \in \R^n : P_j(x) = 0 \textrm{ for $1 \leq j \leq n-m$} \big\}.
\end{equation*}
A set of the above form is referred to as a \emph{variety} and the \emph{maximum degree} of $Z(P_1, \dots, P_{n-m})$ is defined to be the number
\begin{equation*}
\overline{\deg}\, Z(P_1, \dots, P_{n-m}) := \max_{1 \leq j \leq n-m} \deg P_j. 
\end{equation*}

\begin{remark}
 The notion of maximum degree is unnatural from a geometric perspective: it is not an intrinsic quantity associated to the variety but depends on the choice of defining polynomials $P_1, \dots, P_{n-m}$. Nevertheless, it is a convenient quantity to work with for the purposes of this article.  
\end{remark}

Throughout this article it will be convenient to work with varieties which satisfy the additional property that the $n \times (n-m)$ matrix $(\nabla P_1(z) \dots \nabla P_{n-m}(z))$ has full rank whenever $z \in Z(P_1, \dots, P_{n-m})$. In this case $Z(P_1, \dots, P_{n-m})$ is said to be a \emph{transverse complete intersection}. Clearly any transverse complete intersection is a smooth $m$-dimensional submanifold of $\R^n$. 

For 0-dimensional transverse complete intersections the following well-known variant of the classical B\'ezout theorem holds (see, for instance, \cite{Chen2010}). 

\begin{theorem}[B\'ezout's theorem]\label{Bezout's theorem} Suppose $Z = Z(P_1, \dots, P_n)$ is a transverse complete intersection. Then $Z$ is finite and $\#Z \leq \prod_{j=1}^n \deg P_j$.
\end{theorem}

A key tool in the present analysis of H\"ormander-type operators is the following polynomial partitioning result, which is a variant of the polynomial partitioning theorem introduced in \cite{Guth2015} and is based on the classical polynomial ham sandwich theorem of Stone and Tukey \cite{Stone1942}. 

\begin{theorem}[Polynomial partitioning \cite{Guth2018}]\label{polynomial partitioning theorem} Suppose $W \in L^1(\R^n)$ is non-negative. For any degree $D \in \N$ there is a polynomial $P$ of degree $\deg P \lesssim D$ such that the following hold. 
\begin{enumerate}[i)]
\item The set $Z(P)$ is a finite union of $\sim \log D$ transverse complete intersections.
\item If $\{O_i\}_{i \in \mathcal{I}}$  denotes the set of connected components of $\R^n\setminus Z(P)$, then $\#\mathcal{I} \lesssim D^n$ and 
\begin{equation*}
\int_{O_i} W \sim D^{-n} \int_{\R^n} W \qquad \textrm{for all $i \in \mathcal{I}$.}
\end{equation*}
\end{enumerate}
\end{theorem}

The connected components $O_i$ of the set $\R^n \setminus Z(P)$ are referred to as \emph{cells}. It is remarked that in \cite{Guth2018} a stronger version of the above theorem is stated and proved, which provides further structural information about the polynomial $P$ (in particular, the full result is stable under certain small perturbations of $P$). Whilst the methods of this article will require this strengthened version of Theorem~\ref{polynomial partitioning theorem}, the full statement of the result is not reproduced here (it is only needed to address certain technical aspects of the analysis). 

It was observed in recent work of the first author \cite{Guth2016, Guth2018} that polynomial partitioning is a useful tool for studying oscillatory integral operators. Roughly speaking, Theorem~\ref{polynomial partitioning theorem} can be used to effectively reduce the problem to situations where the mass of $T^{\lambda}f$ is concentrated in the neighbourhood of some low-degree algebraic variety; note that this is precisely the setup in the sharp examples discussed in $\S$\ref{Necessary conditions section}. 




\subsection{Polynomial approximation}\label{polynomial approximation subsection} Recall that the operators $T^{\lambda}$ are defined with respect to data belonging to the $C^{\infty}$ category. In order to apply algebraic methods to the problem, one must approximate certain $C^{\infty}$ functions by polynomials. This applies, in particular, to the core curves $\Gamma_{\theta, v}^{\lambda}$ which appear in the definition of the wave packets in $\S$\ref{Wave packet section}. Similar issues were addressed in \cite{Bourgain2011, Zahl2012} via a Jackson-type approximation theorem (see, for instance, \cite{Bagby2002}); for the present purpose an entirely elementary Taylor approximation argument is all that is required. 

Let $\varepsilon > 0$ be a small parameter and define $N = N_{\varepsilon} := \lceil 1/2\varepsilon \rceil \in \N$. Suppose that $\Gamma \colon (-1, 1) \to \R^n$ is a smooth curve satisfying
\begin{equation*}
\|\Gamma\|_{C^{N+1}(-1,1)} := \max_{0 \leq k \leq N+1} \sup_{|t| < 1} |\Gamma^{(k)}(t)| \lesssim 1.
\end{equation*}
The following lemma implies that
\begin{equation*}
\|\Gamma^1_{\theta,v}\|_{C^{N+1}(-1,1)}\lesssim 1,
\end{equation*}
revealing further properties of the core curves of the tubes defined in \S\ref{Wave packet section}.

\begin{lemma} The curves $\Gamma_{\theta,v}^1$ satisfy
\begin{equation*}
\begin{array}{rll}
    |(\Gamma_{\theta,v}^1)'(t)| &\sim 1  &\textrm{for all $t \in I_{\theta,v}^1$,}  \\[5pt]
     \displaystyle\sup_{t \in I_{\theta,v}^1} |(\Gamma_{\theta,v}^1)^{(k)}(t)| &\lesssim c_{\mathrm{par}} & \textrm{for $2 \leq k \leq N$.}
\end{array}
\end{equation*}
 \end{lemma}
 
 \begin{proof} This follows from the reductions made in $\S$\ref{Reductions section}. Indeed, recall that $\Gamma_{\theta,v}^1(t) = (\gamma^1_{\theta,v}(t), t)$ satisfies $\partial_{\omega}\phi(\gamma_{\theta,v}^1(t),t;\omega_{\theta}) = v$. Differentiating this identity yields
 \begin{equation*}
     (\gamma^1_{\theta,v})'(t) = -\partial_{\omega x'}^2\phi(\gamma_{\theta,v}^1(t),t;\omega_{\theta})^{-1}\partial_{\omega}\partial_{x_n}\phi(\gamma_{\theta,v}^1(t),t;\omega_{\theta}).
 \end{equation*}
 The bounds now follow from \eqref{close to identity} and \eqref{small x derivatives}, provided $N_{\mathrm{par}}$ is chosen to be sufficiently large.
 \end{proof}
  
Let $[\Gamma]_{\varepsilon} \colon \R \to \R^n$ denote the polynomial curve given by the degree $N$ Taylor approximation of $\Gamma$ around 0. Observe that
\begin{equation*}
\|[\Gamma]_{\varepsilon}\|_{C^{\infty}(-2, 2)} \leq e^2 \|\Gamma\|_{C^{N}(-1,1)} \lesssim 1. 
\end{equation*}
Given $\lambda \gg 1$, noting that $\lambda^{-\varepsilon N} \leq \lambda^{-1/2}$, Taylor's theorem yields
\begin{equation*}
|\Gamma^{(i)}(t) - [\Gamma]_{\varepsilon}^{(i)}(t)| \lesssim_{\varepsilon} \lambda^{-1/2}|t|^{1 - i} \quad \textrm{for all $|t| \lesssim_{\varepsilon} \lambda^{-\varepsilon}$ and $i = 0,1$.} 
\end{equation*}
Letting $\Gamma^{\lambda} \colon (-\lambda, \lambda) \to \R^n$ denote the rescaled curve $\Gamma^{\lambda}(t) := \lambda \Gamma(t/\lambda)$, the above inequalities trivially imply that
\begin{equation}\label{Taylor approximation 1}
\|[\Gamma^{\lambda}]_{\varepsilon}'\|_{C^{\infty}(-2\lambda, 2\lambda)} \lesssim 1 \quad \textrm{and} \quad \|[\Gamma^{\lambda}]_{\varepsilon}''\|_{C^{\infty}(-2\lambda, 2\lambda)} \lesssim \lambda^{-1} 
\end{equation}
and
\begin{equation}\label{Taylor approximation 2}
|(\Gamma^{\lambda})^{(i)}(t) - ([\Gamma^{\lambda}]_{\varepsilon})^{(i)}(t)| \lesssim_{\varepsilon} \lambda^{-1/2}|t|^{1 - i} \quad \textrm{for all $|t| \lesssim_{\varepsilon} \lambda^{1-\varepsilon}$ and $i = 0,1$.} 
\end{equation}
Combining the $i = 1$ case of the above estimate with the elementary inequality\footnote{This follows by estimating the area of a triangle with sides of length $|x|,|y|,|x-y|$, taking the length of the `base' to be $\min \{|x|, |y|\}$ and bounding the `perpendicular height' by $|x-y|$.}
\begin{equation*}
|x \wedge y| \leq \min \{|x|, |y|\} |x-y| \qquad \textrm{for all $x,y \in \R^n \setminus \{0\}$,}
\end{equation*}c 
together with the fact that $|(\Gamma^{\lambda})'(t)| \sim |[\Gamma^{\lambda}]_{\varepsilon}'(t)|  \sim 1$, one observes that the tangent spaces to the curves $\Gamma^{\lambda}$ and $[\Gamma^{\lambda}]_{\varepsilon}$ have a small angular separation; more precisely,
\begin{equation}\label{Taylor approximation 3}
\angle (T_{\Gamma^{\lambda}(t)}\Gamma^{\lambda}, T_{[\Gamma^{\lambda}]_{\varepsilon}(t)}[\Gamma^{\lambda}]_{\varepsilon}) \lesssim_{\varepsilon} \lambda^{-1/2} \qquad \textrm{for all $|t| \lesssim_{\varepsilon} \lambda^{1-\varepsilon}$.}
\end{equation}




\subsection{Transverse interactions between curved tubes and varieties} 

Let $Z = Z(P_1, \dots, P_{n-m})$ be a transverse complete intersection and fix a polynomial curve $\Gamma \colon \R \to \R^n$. The purpose of this subsection is to study transverse interactions between $\Gamma$ and an $r$-neighbourhood of $Z$; that is, roughly speaking, one wishes to understand the set of points at which the curve $\Gamma$ enters $\mathcal{N}_r Z$ at a large angle. More precisely, given $\alpha, r > 0$ the problem is to estimate the size of the set
\begin{equation*}
Z_{> \alpha, r, \Gamma} := \big\{z \in Z : \exists\, x \in \Gamma \textrm{ with } |x - z| < r \textrm{ and } \angle(T_zZ, T_x\Gamma) > \alpha \big\}.
\end{equation*}
It will be convenient to assume that $\Gamma$ is a \emph{polynomial graph}, by which it is meant that the curve can be rotated so that it is given by $\Gamma(t) = (\gamma(t), t)$ for some polynomial mapping $\gamma \colon \R \to \R^{n-1}$. 

\begin{lemma}\label{transverse interaction lemma} Let $n \geq 2$, $1 \leq m \leq n$ and $Z = Z(P_1, \dots, P_{n-m}) \subseteq \R^n$ be a transverse complete intersection. Suppose $\Gamma \colon \R \to \R^n$ is a polynomial graph satisfying
\begin{equation}\label{transverse interaction 1}
\|\Gamma'\|_{L^{\infty}(-2\lambda, 2\lambda)} \lesssim 1 \quad \textrm{and} \quad \|\Gamma''\|_{L^{\infty}(-2\lambda, 2\lambda)} \leq \delta
\end{equation}
for some $\lambda, \delta > 0$. There exists a dimensional constant $\bar{C} > 0$ such that for all $\alpha > 0$ and $0 <r < \lambda$ satisfying $\alpha \geq \bar{C}\delta r$ the set $Z_{> \alpha, r, \Gamma} \cap B(0, \lambda)$ is contained in a union of $O((\overline{\deg}\, Z \cdot \deg \Gamma)^n)$ balls of radius $r/\alpha$. 
\end{lemma}

The case of interest is given by taking $\Gamma := [\Gamma_{\theta,v}^{\lambda}]_{\varepsilon}$ to be the polynomial approximant of the curve $\Gamma_{\theta,v}^{\lambda}$ introduced in the previous subsection. Here $\deg \Gamma \lesssim_{\varepsilon} 1$ and, by \eqref{Taylor approximation 1}, the condition \eqref{transverse interaction 1} holds with $\delta \sim_{\varepsilon} 1/\lambda$; thus, Lemma~\ref{transverse interaction lemma} implies that for $\alpha > 0$ and $0 <r < \lambda$  satisfying $\alpha \gtrsim r/\lambda$, the set $Z_{> \alpha, r, \Gamma} \cap B(0, \lambda)$ is contained in a union of $O_{\varepsilon}((\overline{\deg}\, Z)^n)$ balls of radius $r/\alpha$.

Using B\'ezout's theorem (that is, Theorem~\ref{Bezout's theorem}), Lemma~\ref{transverse interaction lemma} was established in the case where $\Gamma$ is a line by the first author in  \cite[Lemma 5.7]{Guth2018}. If $\Gamma = \ell$ is a line, then the condition \eqref{transverse interaction 1} holds for any $\lambda > 0$ and any $\delta > 0$ and therefore the lemma implies that for any $\alpha, r >0$ the set $Z_{> \alpha, r, \ell}$ is contained in a union of $O((\overline{\deg}\, Z)^n)$ balls of radius $r/\alpha$. The result for general curves $\Gamma$ is, in fact, a rather straight-forward consequence of the special case of lines. 

\begin{proof}[Proof (of Lemma~\ref{transverse interaction lemma})] Since the problem is rotationally invariant, one may assume that $\Gamma(t) = (\gamma(t), t)$ where $\gamma \colon \R \to \R^{n-1}$ is a polynomial mapping. 

The function $\Upsilon \colon \R^n \to \R^n$ given by $\Upsilon(x',  x_n) := (x' - \gamma(x_n), x_n)$ is clearly a diffeomorphism which maps bijectively between $\Gamma$ and the vertical line $\ell = \mathrm{span}\{e_n\}$. Furthermore, it easily follows that the image set
\begin{equation*}
\tilde{Z} := \Upsilon \big( Z(P_1, \dots, P_{n-m})\big) = Z(P_1 \circ \Upsilon^{-1}, \dots, P_{n-m} \circ \Upsilon^{-1})
\end{equation*}
is a transverse complete intersection of maximum degree $\overline{\deg}\,\tilde{Z} \leq \overline{\deg}\, Z \cdot \deg \Gamma$. 

Let $\lambda, \alpha, r$ satisfy the hypotheses of the lemma for some suitably large dimensional constant $\bar{C} \geq 1$. The key observation is as follows.

\begin{claim} There exist dimensional constants $0 < c \leq 1$ and $C \geq 1$ such that 
\begin{equation*}
Z_{> \alpha, r, \Gamma} \cap \big(\R^{n-1} \times (-\lambda, \lambda)\big) \subseteq \Upsilon^{-1} \big( \tilde{Z}_{>c\alpha, Cr, \ell} \big).
\end{equation*}
\end{claim}

Once this claim is verified, Lemma~\ref{transverse interaction lemma} easily follows. Indeed, one may apply the special case of Lemma~\ref{transverse interaction lemma} for lines (which, as previously remarked, is proved in \cite[Lemma 5.7]{Guth2018}) to conclude that $\tilde{Z}_{>c\alpha, Cr, \ell}$ is contained in a union of $O((\overline{\deg}\, Z \cdot \deg \Gamma)^n)$ balls of radius $r/\alpha$. On the other hand, as a consequence of the first hypothesis in \eqref{transverse interaction 1},
\begin{equation}\label{transverse interaction 2}
|\Upsilon(x) - \Upsilon(x')| \sim |x - x'| \qquad \textrm{for all $x, x' \in \R^{n-1} \times (-\lambda, \lambda)$.}
\end{equation}
Combining these observations, it follows that the set $Z_{ > \alpha, r, \Gamma} \cap B(0, \lambda)$ can be covered by $O((\overline{\deg}\, Z \cdot \deg \Gamma)^n)$ balls of radius $r/\alpha$, as required.

Turning to the proof of the claim, let $z \in Z_{> \alpha, r, \Gamma} \cap B(0, \lambda)$ and note that there exists some $x = \Gamma(x_n) \in \Gamma$ with $|x - z| < r$ and $\angle(T_zZ, T_x\Gamma) > \alpha$. Defining $\tilde{z} := \Upsilon(z) \in \tilde{Z}$ and $\tilde{x} := \Upsilon(x) \in \ell$, it follows from \eqref{transverse interaction 2} that $|\tilde{x} - \tilde{z}| \lesssim r$. Thus, the problem is reduced to showing that $\angle(T_{\tilde{z}}\tilde{Z}, e_n) = \angle(T_{\tilde{z}}\tilde{Z}, T_{\tilde{x}} \ell) \gtrsim \alpha$. 

Observe that, provided $\bar{C}$ is sufficiently large depending only on $n$,
\begin{equation}\label{transverse interaction 2.5}
\angle(T_zZ, T_{\Gamma(z_n)}\Gamma) > \alpha/2.
\end{equation}
Indeed, $|z_n| < \lambda$ and $|x_n| < \lambda + r < 2\lambda$ and so, by the second condition in \eqref{transverse interaction 1}, 
\begin{equation*}
|\Gamma'(x_n) - \Gamma'(z_n)| \leq \delta|x_n - z_n| <\delta r < \bar{C}^{-1} \alpha.
\end{equation*}
Thus, if $\bar{C}$ is appropriately chosen, then $\angle (T_{\Gamma(x_n)}\Gamma, T_{\Gamma(z_n)}\Gamma)<\alpha/2$, which immediately yields \eqref{transverse interaction 2.5}. 

Combining the observations of the previous paragraphs, the claim follows provided one can show that $\angle(T_{\tilde{z}}\tilde{Z}, e_n) \sim \angle(T_zZ, T_{\Gamma(z_n)}\Gamma)$. Let $\eta$ be a smooth curve in $Z$ containing $z$ and define $\tilde{\eta} := \Upsilon(\eta)$; thus, $\tilde{\eta}$ is a smooth curve in $\tilde{Z}$ containing $\tilde{z}$. The problem is now reduced to proving that
\begin{equation}\label{transverse interaction 3}
\angle(T_{\tilde{z}}\tilde{\eta}, e_n) \sim \angle(T_z\eta, T_{\Gamma(z_n)}\Gamma). 
\end{equation}  
If $T_z\eta$ lies in the hyperplane $e_n^{\perp}$ orthogonal to $e_n$, then the above estimate easily follows. Indeed, the tangent space $T_{\Gamma(z_n)}\Gamma$ is spanned by $\Gamma'(z_n) = (\gamma'(z_n), 1)$ and therefore, by \eqref{transverse interaction 1}, one has $\angle(T_z\eta, T_{\Gamma(z_n)}\Gamma) \sim 1$. On the other hand, it is clear from the definition of $\Upsilon$ that $T_{\tilde{z}}\tilde{\eta}$ also lies in $e_n^{\perp}$ and so $\angle(T_{\tilde{z}}\tilde{\eta}, e_n) = \pi/2$. Thus, \eqref{transverse interaction 3} holds in this case.

If $T_z\eta$ does not lie in the hyperplane $e_n^{\perp}$, then $\eta$ can be locally parametrised as a graph over the $x_n$-variable. By an abuse of notation, let $\eta$ denote this graph parametrisation and $\tilde{\eta} := \Upsilon\circ \eta$ so that $\eta(z_n) = z$ and $\tilde{\eta}(z_n) = \tilde{z}$. One may easily verify that $|\tilde{\eta}'(z_n) \wedge e_n| = |\eta'(z_n) \wedge \Gamma'(z_n)|$ and so
\begin{equation*}
\sin \angle (T_{\tilde{z}}\tilde{\eta}, e_n)|\tilde{\eta}'(z_n)| = \sin \angle(T_z\eta, T_{\Gamma(z_n)}\Gamma) |\eta'(z_n)| |\Gamma'(z_n)|.
\end{equation*} 
By the first hypothesis in \eqref{transverse interaction 1}, one has $|\tilde{\eta}'(z_n)| \sim |\eta'(z_n)|$ and $|\Gamma'(z_n)| \sim 1$, and \eqref{transverse interaction 3} follows.
\end{proof}




\section{Transverse equidistribution estimates}\label{Transverse equidistribution section}




\subsection{Tangential wave packets and transverse equidistribution}

In this section the theory of transverse equidistribution estimates, as introduced in \cite{Guth2018}, is extended to the variable coefficient setting. This is a key step in the proof of Theorem~\ref{k-broad theorem} and here the positive-definite hypothesis H2$^+$) plays a crucial r\^ole in the argument.

The first step is to give a precise definition of what it means for a wave packet to be `tangential' to a transverse complete intersection $Z$. Throughout this section let $T^{\lambda}$ be a Hormander-type operator with reduced positive-definite phase $\phi$ and for some $R \ll \lambda$ define the (curved) tubes $T_{\theta,v}$ as in $\S$\ref{Wave packet section}. Furthermore, let $\delta_m$ denote a small parameter satisfying $0 < \delta \ll \delta_m \ll 1$ (here $\delta$ is the same parameter as that which appears in the definition of the wave packets). 

\begin{definition}\label{tangent definition} Suppose $Z = Z(P_1, \dots, P_{n-m})$ is a transverse complete intersection. A tube $T_{\theta, v}$ is $R^{-1/2 + \delta_m}$-tangent to $Z$ in $B(0,R)$ if
\begin{equation*}
T_{\theta, v} \subseteq N_{R^{1/2 + \delta_m}}(Z)
\end{equation*}
and
\begin{equation*}
\angle(G^{\lambda}(x; \omega_{\theta}), T_zZ) \leq \bar{c}_{\mathrm{tang}} R^{-1/2 + \delta_m}
\end{equation*}
for any $x \in T_{\theta, v}$ and $z \in Z \cap B(0,2R)$ with $|x - z| \leq\bar{C}_{\mathrm{tang}} R^{1/2 + \delta_m}$.
\end{definition}

Here $\bar{c}_{\mathrm{tang}} > 0$ (respectively, $\bar{C}_{\mathrm{tang}} \geq 1$) is a dimensional constant, chosen to be sufficiently small (respectively, large) for the purposes of the following arguments. 

\begin{definition} If $\mathbb{S} \subseteq \T$, then $f$ is said to be concentrated on wave packets from $\mathbb{S}$ if
\begin{equation*}
f = \sum_{(\theta, v) \in \mathbb{S}} f_{\theta, v} + \mathrm{RapDec}(R)\|f\|_{L^2(B^{n-1})}.
\end{equation*}
\end{definition}

One wishes to study functions concentrated on wave packets from the collection
\begin{equation*}
\T_Z := \big\{(\theta, v) \in \T : T_{\theta, v} \textrm{ is $R^{-1/2 + \delta_m}$-tangent to $Z$ in $B(0,R)$}\big\}.
\end{equation*}

Let $B \subseteq \R^n$ be a fixed ball of radius $R^{1/2 + \delta_m}$ with centre $\bar{x} \in B(0,R)$. Throughout this section the analysis will be essentially confined to a spatially localised operator $\eta_B\cdot T^{\lambda}g $ where $\eta_B$ is a suitable choice of Schwartz function concentrated on $B$. For any $(\theta, v)$, a stationary phase argument shows that the Fourier transform of $\eta_B\cdot T^{\lambda}g_{\theta, v}$ is concentrated near the surface
\begin{equation}\label{Sigma definition}
\Sigma := \{ \Sigma(\omega) : \omega \in \Omega \} \quad \textrm{where} \quad \Sigma(\omega) := \partial_{x}\phi^{\lambda}(\bar{x};\omega).
\end{equation}
Now consider the refined set of wave packets 
\begin{equation*}
\T_{Z,B} := \big\{(\theta, v) \in \T_Z : T_{\theta, v} \cap B \neq \emptyset \big\}.
\end{equation*}
If $(\theta, v) \in \T_{Z,B}$, then the direction $G^{\lambda}(\bar{x}; \omega_{\theta})$ of the curved tube $T_{\theta, v}$ on the ball $B$ must make a small angle with each of the tangent spaces $T_zZ$ for all $z \in Z \cap B$. It transpires that this essentially constrains the frequency $\Sigma(\omega_{\theta})$ to lie in a small neighbourhood of some fixed (depending on the choice of ball $B$) $m$-dimensional manifold $S_{\xi}$ (here $m=\dim Z$).\footnote{The subscript $\xi$ is used here to indicate that $S_{\xi}$ lies in the $\xi$ parameter space (that is, $\hat{\R}^n$). In particular, it does \textit{not} denote a dependence on some choice of $\xi$. Variants of this notation (such as $A_{\xi}$, $S_{\omega}$, $A_{\omega}$, \&c) feature throughout this section with the obvious corresponding intended meaning.} In the case of the parabolic extension operator $E_{\mathrm{par}}$, which is studied in \cite{Guth2018}, the relationship between the normal direction $G^{\lambda}(\bar{x}; \omega_{\theta})$ and the frequency  $\Sigma(\omega_{\theta})$ is particularly simple. Here $\Sigma(\omega_{\theta}) = \big(\omega_{\theta}, \tfrac{|\omega_{\theta}|^2}{2}\big)$ is constrained to lie in roughly the $R^{-1/2}$-neighbourhood some \emph{affine} subspace $A_{\xi}$.
  \begin{example}
    Suppose, for simplicity, that $Z$ is an $m$-dimensional affine plane so that $T_zZ = V$ for all $z \in Z$ where $V$ is the $m$-dimensional linear subspace parallel to $Z$. To avoid degenerate situations, also assume $V$ makes a small angle with the $e_n$ direction. For the prototypical case of the parabolic extension operator $E_{\mathrm{par}}$ the (unnormalised) Gauss map is an affine map: $G_0(\omega) = (-\omega, 1)$. Consequently, 
    \begin{equation*}
        A_{\omega} := \{\omega \in \R^{n-1} : G_0(\omega) \in V\}
    \end{equation*}    
    is an affine subspace of $\R^{n-1}$ of dimension $m-1$. Thus, if  $G_0(\omega) \in V$, then $\Sigma(\omega) \in A_{\xi} := A_{\omega} \times \R$. 
    
       Note that for general H\"ormander-type operators the condition $G^{\lambda}(\bar{x};\omega) \in V$ defines a (possibly curved) submanifold rather than an affine subspace. 
    \end{example}
    In view of this frequency concentration in the case of $E_{par}$, the uncertainty principle then suggests that if $g$ concentrated on wave packets from $\T_{Z,B}$, then the function $|E_{\mathrm{par}}g(x)|$ is morally constant as one varies $x$ by $R^{1/2}$ in directions perpendicular to $A_{\xi}$. Furthermore, it can be shown that the affine subspace $A_{\xi}$ makes a small angle with the tangent planes $T_zZ$ for $z \in Z \cap B$ and so $|E_{\mathrm{par}}g(x)|$ is morally constant as one varies $x$ by $R^{1/2}$ in directions transverse to $Z \cap B$.

One wishes to extend the above observations for $E_{\mathrm{par}}$ to the variable coefficient setting; that is, for $g$ concentrated on wave packets from $\T_{Z,B}$,\footnote{In fact, in the general case a more stringent hypothesis on $g$ is required, as discussed below.} the problem is to show that $|T^{\lambda}g|$ is morally constant in directions transverse to $Z\cap B$. More precisely, one wishes to establish an inequality roughly of the form
\begin{equation}\label{ideal transverse equidistribution}
 \fint_{N_{\rho^{1/2 + \delta_m}}(Z)\cap B} |T^{\lambda}g|^2 \lesssim \fint_B |T^{\lambda}g|^2 
\end{equation}
for $0 < \rho < R$; this would show that the $L^2$ mass of $T^{\lambda}g$ is unable to concentrate in a small neighbourhood of $Z\cap B$. For the parabolic extension operator the observations of the previous paragraph can be used to prove \eqref{ideal transverse equidistribution} (up to a rapidly decaying error term). The general case is more complicated, however. First of all, the surface $S_{\xi}$ described above is no longer necessarily an affine subspace and may possess curvature. One way to circumvent this issue is to introduce a further constraint on the family of wave packets. Let $R^{1/2} < \rho \ll R$ and throughout this section let $\tau \subset \R^{n-1}$ be a fixed cap of radius $O(\rho^{-1/2 + \delta_m})$ centred at a point in $B^{n-1}$. Now define
\begin{equation*}
\T_{Z, B,\tau} := \big\{(\theta, v) \in \T_Z : \theta \cap \tau \neq \emptyset \textrm{ and }  T_{\theta, v} \cap B \neq \emptyset  \big\}.
\end{equation*}
The frequencies $\Sigma(\omega_{\theta})$ for $(\theta, v) \in \T_{Z, B,\tau}$ are further constrained to lie in a small region of $\Sigma$ upon which the curved space $S_{\xi}$ can be effectively approximated by an affine space $A_{\xi}$. Consequently, one can carry out a similar analysis as in the parabolic extension case.  

The second issue is to ensure that the resulting affine space $A_{\xi}$ makes a small angle with the tangent spaces $T_zZ$ for $z \in Z \cap B$. This is crucial to ensure that the morally constant property holds in directions transverse to $Z$. For general H\"ormander-type operators this property can fail (a simple example is given by the extension operator associated to the hyperbolic paraboloid, as discussed below). In order to ensure the angle condition one needs to exploit the additional positive-definite hypothesis H2$^+$).

In practice, the rigorous formulation of these heuristics is somewhat messier than \eqref{ideal transverse equidistribution}, and it is convenient to state the key estimate in the following manner. 

\begin{lemma}\label{Transverse equidistribution lemma 1} With the above setup, if $\overline{\deg}\,Z \lesssim_{\varepsilon} 1$ and $g$ is concentrated on wave packets from $\T_{Z, B,\tau}$, then
\begin{equation}\label{transverse equidistribution lemma estimate}
\int_{N_{\rho^{1/2 + \delta_m}}(Z) \cap B} |T^{\lambda}g|^2 \lesssim R^{1/2+O(\delta_m)}(\rho/R)^{(n-m)/2}\|g\|^2_{L^2(B^{n-1})}.
\end{equation}
\end{lemma}

The inequality \eqref{transverse equidistribution lemma estimate} is related to the heuristic inequality \eqref{ideal transverse equidistribution} via H\"ormander's $L^2$ bound
\begin{equation*}
 \|T^{\lambda}g\|_{L^2(B)}^2 \lesssim R^{1/2 + \delta_m} \|g\|_{L^2(B^{n-1})}^2.
\end{equation*}
The estimate is presented in this way (rather than in a form more closely resembling \eqref{ideal transverse equidistribution}) as it provides a relatively clean statement and, moreover, \eqref{transverse equidistribution lemma estimate} happens to be the precise bound required later in the proof.




\subsection{Uncertainty principle preliminaries} If $G \colon \R^n \to \C$ is frequency supported on a ball of radius $r > 0$, then the uncertainty principle dictates that $G$ should be essentially constant at spatial scale $r^{-1}$. In particular, the $L^2$-mass of $G$ cannot be highly concentrated in any ball of radius $\rho < r^{-1}$ and so one has
\begin{equation*}
 \fint_{B(x_0,\rho)} |G|^2 \lesssim \fint_{B(x_0,r^{-1})} |G|^2.
\end{equation*}
Strictly speaking, for this inequality to hold the right-hand integral should be taken with respect to a rapidly decaying weight function rather than over the compact region $B(x_0,r^{-1})$ (see, for instance, \cite[Section 6]{Guth2018}). There is a variant of this estimate which is effective in cases where $G$ has the property that $\hat{G}$ is merely concentrated in (rather than supported in) an $r$-ball. 

\begin{lemma}\label{concentration uncertainty} If $r^{-1/2} \leq \rho \leq r^{-1}$, then for any ball $B(x_0,\rho)$, $\xi_0 \in \hat{\R}^n$ and $\delta > 0$ one has
\begin{equation*}
 \fint_{B(x_0,\rho)} |G|^2 \lesssim_{\delta} \|\hat{G}w_{B(\xi_0, r)}^{-1}\|_{\infty}^{2\delta/(1+\delta)}  \frac{1}{|B(0, r^{-1})|}\big(\int_{\R^n} |G|^2\big)^{1/(1+\delta)}.
\end{equation*}
Here $w_{B(\xi_0, r)}$ is a weight concentrated on $B(\xi_0, r)$ given by
\begin{equation}\label{weight function}
w_{B(\xi_0, r)}(\xi) :=  (1 + r^{-1}|\xi - \xi_0|)^{-N}
\end{equation}
for some large $N = N_{\delta} \in \N$. 
\end{lemma}

Hence, if $\hat{G}$ is concentrated in $B(\xi_0, r)$ in the sense that $|\hat{G}(\xi)| \lesssim M w_{B(\xi_0, r )}(\xi)$ for some controllable constant $M \geq 0$, then the lemma produces a favourable estimate. 

\begin{remark} Lemma~\ref{concentration uncertainty} is \textit{not} sharp in terms of the dependence on the $r$ and $\rho$ parameters. It differs, however, from the sharp inequality only by $O(\delta)$ powers and such losses are negligible for the purposes of this article. 
 \end{remark}

\begin{proof}[Proof (of Lemma~\ref{concentration uncertainty})] Define $\psi_{\rho}$ by $(\psi_{\rho})\,\widecheck{}\,(x) := \check{\psi}(\rho^{-1}(x-x_0))$ where $\psi$ is a Schwartz function which satisfies $|\check{\psi}(x)| \gtrsim 1$ on $B(0,1)$. Thus, by Plancherel,
\begin{equation*}
 \int_{B(x_0,\rho)} |G|^2 \lesssim  \int |\check{\psi}_{\rho} G|^2 = \int |\psi_{\rho} \ast \hat{G}|^2.
\end{equation*}
Using the rapid decay of $\psi$, one deduces that
\begin{equation*}
 |\psi_{\rho} \ast \hat{G}(\xi)| \lesssim_{\delta} \rho^n \int_{\hat{\R}^n} w_{B(\xi,\rho^{-1})}(\eta)^{\delta/(1+\delta)} |\hat{G}(\eta)|\,\ud \eta
\end{equation*}
for all $\xi \in \hat{\R}^n$. By expressing the right-hand integral as
\begin{equation*}
\int_{\hat{\R}^n} \Big(w_{B(\xi,\rho^{-1})}(\eta)w_{B(\xi_0,r)}(\eta)^{1/2}\Big)^{\delta/(1+\delta)}\Big(|\hat{G}(\eta)|^{1+\delta}w_{B(\xi_0,r)}(\eta)^{-\delta/2}\Big)^{1/(1+\delta)} \,\ud \eta 
\end{equation*}
and applying H\"older's inequality, it follows that
\begin{equation*}
|\hat{\psi} \ast \hat{G}(\xi)|  \lesssim |B(x_0,\rho)| \cdot \mathrm{I}(\xi)^{\delta/(1+ \delta)}\cdot \mathrm{II}(\xi)^{1/(1+ \delta)}
\end{equation*}
where
\begin{align*}
\mathrm{I}(\xi) &:= \int_{\hat{\R}^n} w_{B(\xi , \rho^{-1})}(\eta)w_{B(\xi_0,r)}(\eta)^{1/2}\,\ud \eta, \\
\mathrm{II}(\xi) &:= \int_{\hat{\R}^n} |\hat{G}(\eta)|^{1+\delta} w_{B(\xi_0,r)}(\eta)^{-\delta/2}\,\ud \eta .
\end{align*}

To estimate $\mathrm{I}(\xi)$ first perform the variable shift $\eta \mapsto \eta + \xi_0$ and then decompose the range of integration into the regions $|\eta|< |\xi-\xi_0|/2$ and $|\eta| \geq |\xi-\xi_0|/2$. Since $w_{B(\xi - \xi_0, \rho^{-1})}(\eta) \lesssim_{\delta} w_{B(\xi_0, \rho^{-1})}(\xi)$ for $|\eta| < |\xi-\xi_0|/2$ and $w_{B(0, \rho^{-1})}(\eta) \lesssim_{\delta} w_{B(\xi_0, \rho^{-1})}(\xi)$ for $|\eta| \geq |\xi-\xi_0|/2$, it follows that 
\begin{equation*}
\mathrm{I}(\xi) \lesssim_{\delta} |B(x_0,\rho)|^{-1} w_{B(\xi_0,\rho^{-1})}(\xi)^{1/2}.
\end{equation*}

To estimate $\mathrm{II}(\xi)$ note that, provided $N_{\delta}$ is chosen sufficiently large, by Cauchy--Schwarz and Plancherel's theorem one has
\begin{align*}
\mathrm{II}(\xi) &\leq \|\hat{G}w_{B(\xi_0,r)}^{-1}\|_{\infty}^{\delta}\int |\hat{G}(\eta)|w_{B(\xi_0,r)}(\eta)^{\delta/2}\,\ud \eta \\
&\lesssim_{\delta} |B(0,r^{-1})|^{-1/2}\|\hat{G}w_{B(\xi_0,r)}^{-1}\|_{\infty}^{\delta}\|G\|_{L^2(\R^n)}.
\end{align*}

 Combining these observations, one obtains the desired estimate but with an additional factor of $(\rho r^{1/2})^{-2n\delta / (1+\delta)}$ on the right-hand side. Since $1 \leq \rho r^{1/2}$, the result immediately follows. 
\end{proof}




\subsection{Wave packets tangential to linear subspaces} Here, as a step towards  Lemma~\ref{Transverse equidistribution lemma 1}, transverse equidistribution estimates are proven for functions concentrated on wave packets tangential to some fixed linear subspace $V \subseteq \R^n$. As before, let $B$ be a ball of radius $R^{1/2 + \delta_m}$ with centre $\bar{x} \in \R^n$ and define
\begin{equation*}
\T_{V,B} := \big\{(\theta, v) : \angle(G^{\lambda}(\bar{x}, \omega_{\theta}), V) \lesssim R^{-1/2 + \delta_m}  \textrm{ and }  T_{\theta, v} \cap B \neq \emptyset   \big\}.
\end{equation*}
Let $R^{1/2} < \rho < R$ and for $\tau \subset \R^{n-1}$ a ball of radius $O(\rho^{-1/2+ \delta_m})$ centred at a point in $B^{n-1}$ define
\begin{equation*}
\T_{V, B,\tau} := \big\{(\theta, v) \in \T_{V,B} : \theta \cap (\tfrac{1}{10}\cdot \tau) \neq \emptyset\big\}
\end{equation*}
where $(\tfrac{1}{10}\cdot \tau)$ is the cap concentric to $\tau$ but with $1/10$th of the radius. 

The key estimate is the following.

\begin{lemma}\label{Transverse equidistribution lemma 2} If $V \subseteq \R^n$ is a linear subspace, then there exists a linear subspace $V'$ with the following properties:
\begin{enumerate}[1)]
\item $\dim V + \dim V' = n$.
\item $V, V'$ are quantitatively transverse in the sense that there exists a uniform constant $c_{\mathrm{trans}} > 0$ such that
\begin{equation*}
 \angle(v, v') \geq 2c_{\mathrm{trans}} \qquad \textrm{for all non-zero vectors $v \in V$ and $v' \in V'$.}
\end{equation*}
\item If $g$ is concentrated on wave packets from $\T_{V, B,\tau} $, $\Pi$ is any plane parallel to $V'$ and $x_0 \in \Pi \cap B$, then the inequality
\begin{align*}
\int\displaylimits_{\Pi \cap B(x_0, \rho^{1/2 + \delta_m})}\!\!\!\!\!\!\!\!\!\!\!\! |T^{\lambda}g|^2 \lesssim_{\delta} R^{O(\delta_m)} (\rho/R)^{\dim V'/2} \|g\|_{L^2(B^{n-1})}^{2\delta/(1+\delta)} \big(\int\displaylimits_{\Pi \cap 2B} |T^{\lambda}g|^2 \big)^{1/(1+\delta)}
\end{align*}
holds up to the inclusion of a $\mathrm{RapDec}(R)\|g\|_{L^2(B^{n-1})}$ term on the right-hand side. 
\end{enumerate}
\end{lemma}

\begin{proof} The argument is presented in a number of stages.
 


 
\subsection*{Constructing the subspace $V'$} Recall that $\omega \mapsto G^{\lambda}(\bar{x};\omega) := \frac{G_0^{\lambda}(\bar{x};\omega)}{|G_0^{\lambda}(\bar{x};\omega)|}$ is the Gauss map associated to the hypersurface $\Sigma$, defined in \eqref{Sigma definition}. Since $G(x; 0) = e_n$ for all $x \in X$, Lemma~\ref{Gauss Lipschitz} implies that 
\begin{equation*}
\angle(G^{\lambda}(\bar{x}; \omega), e_n) \sim |\omega| \qquad \textrm{for all $\omega \in \Omega$.}
\end{equation*}
Consequently, by choosing $\mathrm{diam}\,\Omega$ to be sufficiently small in the initial reductions, one may assume that the Gauss map $\omega \mapsto G^{\lambda}(\bar{x}; \omega)$ always makes a wide angle with the hyperplane $e_n^{\perp}  = \R^{n-1} \times \{0\}$. In particular,
\begin{equation*}
\angle(G^{\lambda}(\bar{x}; \omega), e_n^{\perp}) \gtrsim 1 \quad \textrm{for all $\omega \in \Omega$.}
\end{equation*}
Since the situation is trivial if $\mathbb{T}_{V, B} = \emptyset$, one may assume that  
\begin{equation}\label{TE2 1}
\overline{\angle}(V, e_n^{\perp}) := \max_{v \in V\cap S^{n-1}} \angle(v, e_n^{\perp}) \gtrsim 1. 
\end{equation}
Define $S_{\omega} \subset \R^{n-1}$ by 
\begin{equation*}
S_{\omega} := \big\{ \omega \in \Omega : G^{\lambda}(\bar{x}; \omega) \in V \big\}.
\end{equation*}
Fixing an orthonormal basis $\{N_1, \dots, N_{n - \dim V}\}$ for $V^{\perp}$, one has
\begin{equation*}
S_{\omega} = \big\{ \omega \in \Omega : \langle G_0^{\lambda}(\bar{x}; \omega), N_k \rangle = 0 \textrm{ for $1 \leq k \leq n - \dim V$} \big\}. 
\end{equation*}

\begin{claim 1} If $S_{\omega} \neq \emptyset$, then $S_{\omega}$ is a smooth surface in $\R^{n-1}$ of dimension $\dim V - 1$.
\end{claim 1}

\begin{proof}[Proof of Claim 1] Let $\omega \in S_{\omega}$ and note that each $N_k$ is tangential to $\Sigma$ at $\Sigma(\omega)$. Hence, one may write
\begin{equation*}
N_k = \sum_{j=1}^{n-1} N_k^{(j)}(\omega) \partial_{\omega_j}\partial_x\phi^{\lambda}(\bar{x}; \omega)
\end{equation*}
for some choice of coefficients $N_k^{(j)}(\omega) \in \R$. A computation now shows that
\begin{equation*}
\partial_{\omega_i} \langle G_0^{\lambda}(\bar{x}; \omega), N_k \rangle = -\sum_{j=1}^{n-1} \langle \partial^2_{\omega_i\omega_j}\partial_x\phi^{\lambda}(\bar{x}; \omega) , G_0^{\lambda}(\bar{x}; \omega) \rangle N_k^{(j)}(\omega).
\end{equation*}
The condition H2) implies the invertibility of the $(n-1)\times(n-1)$ matrix whose $(i,j)$th entry is given by  
\begin{equation*}
\langle \partial^2_{\omega_i\omega_j}\partial_x\phi^{\lambda}(\bar{x}; \omega) , G_0^{\lambda}(\bar{x}; \omega) \rangle.
\end{equation*}
Thus, the Jacobian of $\omega \mapsto (\langle G_0^{\lambda}(\bar{x}; \omega), N_k \rangle)_{k=1}^{n - \dim V}$ has maximal rank, and the claim follows by the implicit function theorem. 
\end{proof}

At this point it is convenient to switch to a graph parametrisation of $\Sigma$ via the change of variables $u \mapsto \Psi^{\lambda}(\bar{x}; u)$ where $\Psi^{\lambda}$ is the (appropriate $\lambda$-rescaling of the) function introduced in $\S$\ref{Reductions section}. For convenience, let $\Psi \colon U \to \Omega$ denote this mapping; that is, $\Psi(u) := \Psi^{\lambda}(\bar{x}; u)$. Recall that the hypersurface $\Sigma$ coincides with the graph of the function 
\begin{equation}\label{h bar definition}
\bar{h} \colon U \to \R, \quad \bar{h}(u) := \partial_{x_n}\phi^{\lambda}(\bar{x}; \Psi(u)). 
\end{equation}
If $S_{\omega} \cap \tau = \emptyset$, then it follows by Lemma \ref{Gauss Lipschitz} that
\begin{equation*}
\angle(G^{\lambda}(\bar{x}; \theta), V) \gtrsim \rho^{-1/2 + \delta_m} > R^{-1/2 + \delta_m}
\end{equation*}
whenever $\theta \cap (\tfrac{1}{10}\cdot\tau) \neq \emptyset$. Consequently, $\T_{V, B, \tau} = \emptyset$ and the situation is trivial. Thus, one may assume without loss of generality that $S_{\omega} \cap \tau \neq \emptyset$ and so, letting
\begin{equation*}
 S_u := \Psi^{-1}(S_{\omega}) = \{u \in U : G_0^{\lambda}(\bar{x};\Psi (u)) \in V\},
\end{equation*}
it follows that $S_u \cap \Psi^{-1}(\tau) \neq \emptyset$. The properties of the mapping $\Psi$ discussed in $\S$\ref{Reductions section} imply that $\Psi^{-1}(\tau)$ is roughly a ball of radius $O(\rho^{-1/2 + \delta_m})$. 

Fix some $u_0 \in S_u \cap \Psi^{-1}(\tau)$ and let $A_u$ denote the tangent plane to $S_u$ at $u_0$. Here, the tangent plane is interpreted as a $(\dim V - 1)$-dimensional affine subspace of $\R^{n-1}$ through $u_0$. Now define $A_{\xi} := A_u \times \R \subseteq \R^n$, so that $\dim A_{\xi} = \dim V$, and let $V_u$ and $V_{\xi}$ be the linear subspaces parallel to $A_u$ and $A_{\xi}$, respectively. Finally, let $V' := V_{\xi}^{\perp}$ so that $\dim V + \dim V' = n$.




\subsection*{Verifying the transverse equidistribution estimate in 3)} Suppose $\Pi \subseteq \R^n$ is an affine subspace parallel to $V'$ which intersects $B$ and $x_0 \in \Pi \cap B$. Let $\eta_B(x) := \eta((x-\bar{x})/R^{1/2 + \delta_m})$ where $\eta$ is a Schwartz function which satisfies $\eta(x) = 1$ for $x \in B(0,2)$ and, for any $(\theta,v) \in \T$, consider
\begin{equation*}
\big(\eta_B\cdot T^{\lambda}g_{\theta, v} |_{\Pi} \big)\;\widehat{}\;(\xi)  = e^{-2\pi i \langle x_0, \xi \rangle}R^{\dim V'(1/2+ \delta_m)}\int_{B^{n-1}} K^{\lambda,R}(\xi;\omega)g_{\theta, v}(\omega)\,\ud \omega 
\end{equation*}
where the kernel $K^{\lambda,R}$ is given by
\begin{equation*}
    K^{\lambda,R}(\xi;\omega) := \int_{V'} e^{2 \pi i \phi_{\omega}^{\lambda, R}(z)} a^{\lambda, R}_{\omega}(z)\,\ud z 
\end{equation*}
for the phase an amplitude functions
\begin{align*}
\phi_{\omega}^{\lambda, R}(z) &:= \phi^{\lambda}(x_0 + R^{1/2 + \delta_m}z; \omega) -R^{1/2 +\delta_m} \langle z, \xi \rangle\\
a^{\lambda, R}_{\omega}(z) &:= a^{\lambda}(x_0 + R^{1/2+\delta_m}z;\omega)\tilde{\eta}(z)
\end{align*}
and $\tilde{\eta}(z) := \eta(z + (x_0 - \bar{x})/R^{1/2 + \delta_m})$. 

\begin{claim 2} Fixing $\omega \in \Omega$, $\xi \in \hat{\R}^n$ such that $|\xi - \mathrm{proj}_{V'}\Sigma(\omega)| \gtrsim R^{-1/2 +\delta_m}$ and $R \gg 1$, the estimates
\begin{enumerate}[i)]
    \item $|\partial_{z} \phi_{\omega}^{\lambda,R}(z)| \sim R^{1/2 + \delta_m}|\xi - \mathrm{proj}_{V'}\Sigma(\omega)| \gtrsim R^{2\delta_m}$,
    \item $|\partial_{z}^{\alpha} \phi_{\omega}^{\lambda,R}(z)| \lesssim |\partial_{z} \phi_{\omega}^{\lambda,R}(z)|$ for all $2 \leq |\alpha| \leq N_{\mathrm{par}}$,
    \item $|\partial_{z}^{\alpha} a_{\omega}^{\lambda,R}(z)| \lesssim_{\varepsilon} 1$ for all $|\alpha| \leq N_{\mathrm{par}}$
\end{enumerate}
hold on $\mathrm{supp}\,a_{z}^{\lambda,R}$. Here $\Sigma(\omega) := \partial_x\phi^{\lambda}(\bar{x};\omega)$ is as defined in \eqref{Sigma definition}.
\end{claim 2}

Once the claim is established, repeated integration-by-parts (see Lemma~\ref{integration-by-parts lemma}) shows that $K^{\lambda,R}$ is rapidly decaying whenever $|\xi - \mathrm{proj}_{V'}\Sigma(\omega)| \gtrsim R^{-1/2 +\delta_m}$ and, in particular, 
\begin{equation*}
|K^{\lambda,R}(\xi; \omega)| \lesssim_{\varepsilon} (1 + R^{1/2}|\xi - \mathrm{proj}_{V'}\Sigma(\omega)|)^{-N} \qquad \textrm{for all $N \leq N_{\mathrm{par}}$}.
\end{equation*}

\begin{proof}[Proof of Claim 2] The uniformity in the estimates is due to the reductions from \S\ref{Reductions section}. The bound iii) for the amplitude immediately follows from Lemma~\ref{third reduction lemma} and it remains to prove the bounds for the phase. 

\subsubsection*{Proof of i)} The $z$-gradient of the phase $\phi^{\lambda,R}_{\omega}$ is equal to 
\begin{equation*}
R^{1/2+\delta_m}\big(\mathrm{proj}_{V'}\big[(\partial_x\phi^{\lambda})(x_0 + R^{1/2+\delta_m}z; \omega) - (\partial_x\phi^{\lambda})(\bar{x}; \omega)] -[\xi  - \mathrm{proj}_{V'}\Sigma(\omega)\big] \big)
\end{equation*}
where, by Lemma~\ref{third reduction lemma}, the first term satisfies
\begin{equation*}
\big|\mathrm{proj}_{V'}\big[(\partial_x\phi^{\lambda})(x_0 + R^{1/2+\delta_m}z; \omega) - (\partial_x\phi^{\lambda})(\bar{x}; \omega)\big]\big| \lesssim R^{1/2 + \delta_m}/\lambda \ll R^{-1/2 + \delta_m}.
\end{equation*}
Thus, if $|\xi - \mathrm{proj}_{V'}\Sigma(\omega)| \gtrsim R^{-1/2 +\delta_m}$, then the desired bound immediately follows. 
\subsubsection*{Proof of ii)} Fix $\alpha \in \N_0^n$ with $2 \leq |\alpha| \leq N_{\mathrm{par}}$. It follows that 
\begin{equation*}
\partial_z^{\alpha} |\phi^{\lambda,R}_{\omega}(z)| \leq \lambda (R^{1/2 + \delta_m}/\lambda)^{|\alpha|} \|\partial_x^{\alpha}\phi\|_{L^{\infty}(X \times \Omega)} \lesssim R^{2 \delta_m}
\end{equation*}
and the desired bound now follows from i). 
\end{proof}

If $\omega \in \mathrm{supp}\, g_{\theta, v}$, then $|\omega - \omega_{\theta}| < R^{-1/2}$ and so $|\Sigma(\omega) - \xi_{\theta}| \lesssim R^{-1/2}$ where $\xi_{\theta} := \Sigma(\omega_{\theta})$. Consequently, 
\begin{equation}\label{variable coefficient decoupling 100}
  | (\eta_B \cdot T^{\lambda}g_{\theta, v} |_{\Pi})\;\widehat{}\;(\xi) | \lesssim_N  R^{O(1)} w_{B(\mathrm{proj}_{V'}\xi_{\theta},R^{-1/2})}(\xi) \|g_{\theta, v}\|_{L^2(B^{n-1})}
\end{equation}
where the definition of the weight function 
\begin{equation*}
w_{B(\mathrm{proj}_{V'}\xi_{\theta},R^{-1/2})}(\xi) := (1 + R^{1/2}|\xi - \xi_{\theta}|)^{-N}
\end{equation*}
agrees with that of Lemma~\ref{concentration uncertainty} (although here the weights are thought of as functions on $V'$), and so $N = N_{\delta}$ is a large integer, depending on $\delta$. 

The following geometric observation is key to the proof of property 3). 

\begin{claim 3}  If $(\theta, v) \in \mathbb{T}_{B,\tau, V}$, then $\mathrm{dist}(\xi_{\theta}, A_{\xi}) \lesssim R^{-1/2 + \delta_m}$.
\end{claim 3}

Temporarily assume this claim and recall that $V' := V_{\xi}^{\perp}$ where $V_{\xi}$ is the linear subspace parallel to the affine subspace $A_{\xi}$. Thus, if $(\theta, v) \in \mathbb{T}_{B,\tau, V}$, then $\mathrm{proj}_{V'}\xi_{\theta}$ lies in some fixed ball of radius $O(R^{-1/2 + \delta_m})$. Letting $\xi_* \in V'$ denote the centre of this ball, it follows that $w_{B(\mathrm{proj}_{V'}\,\xi_{\theta},R^{-1/2})} \lesssim_{\delta} w_{B(\xi_{*},R^{-1/2+\delta_m})}$ and so
\begin{equation*}
\sum_{(\theta, v) \in \mathbb{T}_{V, B, \tau}} w_{B(\mathrm{proj}_{V'}\,\xi_{\theta},R^{-1/2})} \lesssim_{\delta} R^{O(1)} w_{B(\xi_{*},R^{-1/2+\delta_m})}.
\end{equation*}  
Recalling \eqref{variable coefficient decoupling 100}, 
\begin{equation*}
\|(\eta_B \cdot T^{\lambda}g |_{\Pi})\;\widehat{}\;w_{B(\xi_*,R^{-1/2+\delta_m})}^{-1}\|_{\infty} \lesssim R^{O(1)} \|g\|_{L^2(B^{n-1})}
\end{equation*}
and, applying Lemma~\ref{concentration uncertainty}, one concludes that
\begin{equation*}
 \int\displaylimits_{B(x_0, \rho^{1/2 + \delta_m}) \cap \Pi} \!\!\!\!\!\!\!\!\!\! |T^{\lambda}g|^2 \lesssim_{\delta} R^{O(\delta_m)}\big(\frac{\rho^{1/2}}{R^{1/2}}\big)^{\dim V'} \|g\|_{L^2(B^{n-1})}^{2\delta/(1+\delta)} \big(\int\displaylimits_{\Pi} |T^{\lambda}g|^2 |\eta_B|^2\big)^{1/(1+\delta)}.
\end{equation*}
If $\xi \notin 2B$, then $\xi \notin \bigcup_{(\theta, v) \in \mathbb{T}_B} T_{\theta, v}$ and so $|T^{\lambda}g_{\theta, v}(\xi)| = \mathrm{RapDec}(R)\|g\|_{L^2(B^{n-1})}$ for all $(\theta, v) \in \mathbb{T}_{V, B, \tau}$. Hence
\begin{equation*}
\int_{\Pi} |T^{\lambda}g|^2 |\eta_B|^2 \leq \int_{2B \cap \Pi} |T^{\lambda}g|^2 + \mathrm{RapDec}(R)\|g\|_{L^2(B^{n-1})},
\end{equation*} 
completing the proof of property 3) under the assumption that the above claim holds. 

\begin{proof}[Proof of Claim 3] Fix $(\theta, v) \in \mathbb{T}_{B,\tau, V}$ and let $u_{\theta} := \mathrm{proj}_{x_n^{\perp}}\,\Sigma(\omega_{\theta})$. Recalling that $A_{\xi} = A_u \times \R$  and applying triangle inequality, one deduces that
\begin{equation*}
\mathrm{dist}(\xi_{\theta}, A_{\xi}) = \mathrm{dist}(u_{\theta}, A_u) \leq \mathrm{dist}\big(u_{\theta}, S_u \cap \Psi^{-1}(\tau)\big) + \sup_{u_* \in S_u \cap \Psi^{-1}(\tau)}\mathrm{dist}(u_*, A_u).
\end{equation*}
Furthermore, by Lemma \ref{Gauss Lipschitz},
\begin{equation*}
 \mathrm{dist}\big(u_{\theta}, S_u \cap \Psi^{-1}(\tau)\big) \sim \mathrm{dist}(\omega_{\theta}, S_{\omega}\cap \tau) \lesssim \angle(G^{\lambda}(\bar{x};\omega_{\theta}), V) \lesssim R^{-1/2 + \delta_m},
\end{equation*}
where the last inequality is by the definition of $\mathbb{T}_{B,\tau, V}$. On the other hand, fixing $u_* \in S_u \cap \Psi^{-1}(\tau)$, one wishes to estimate $\mathrm{dist}(u_*, A_u)$. Provided $\rho$ is sufficiently large (so that $\mathrm{diam}\,\tau$ is sufficiently small), the surface $S_u \cap \Psi^{-1}(\tau)$ can be parametrised as the graph of some function $\psi \colon \mathcal{W}  \to V_u^{\perp} \subseteq \R^{n-1}$ where $\mathcal{W} \subset V_u$ is an open set about the origin of diameter $O(\rho^{-1/2 + \delta_m})$. In particular, one may write
\begin{equation*}
S_u \cap \Psi^{-1}(\tau) = \big\{ w + \psi(w) : w \in \mathcal{W} \big\} + u_0
\end{equation*}
where $\psi(0) = 0$ and $\partial_{w} \psi_j(0) = 0$ for $1 \leq j \leq n -  \dim V_u$. Thus, $u_* = w_* + \psi(w_*) + u_0$ for some $w_* \in \mathcal{W}$ and, since $w_* \in V_u+u_0 = A_u$, it follows that $ \mathrm{dist}(u_*, A_u) \leq |\psi(w_*)|$. By Taylor's theorem (here using the hypothesis that $R^{1/2} \leq \rho$),
\begin{equation*}
|\psi_j(w_*)| \leq \int_0^1 (1-t) |\langle \partial_{ww}^2\psi_j(tw_*)w_*, w_* \rangle| \, \ud t \lesssim |w_*|^2 \lesssim \rho^{-1 + 2\delta_m} \leq R^{-1/2 + \delta_m}
\end{equation*}
 for $1 \leq j \leq n -  \dim V_u$ and combining these observations yields the desired estimate.

\end{proof}




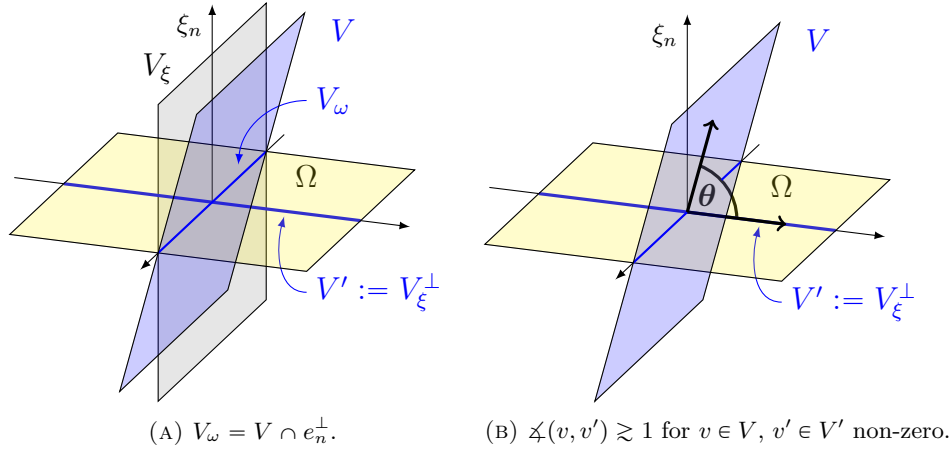
\begin{figure} 
\centering 
  \begin{subfigure}[t]{0.49\linewidth}
  %
  %
\tdplotsetmaincoords{70}{110}
\begin{tikzpicture}[tdplot_main_coords,font=\sffamily,scale=0.7]
\draw[-latex] (-4,0,0) -- (4,0,0);
\draw[-latex] (0,-4,0) -- (0,4,0);
\draw[-latex] (0,0,0) -- (0,0,4) node[anchor=north east]{$\xi_n$};
\node[anchor=south west,align=center] at (-0.8,1.2,0) {\Large$\Omega$};

\draw[fill=black,opacity=0.1] (-3,0,-3) -- (-3,0,3) -- (3,0,3) -- (3,0,-3) -- cycle;
\draw[black, thin] (-3,0,-3) -- (-3,0,3) -- (3,0,3) -- (3,0,-3) -- cycle;
\node[anchor=south west,align=center] (line) at (0,-1.6,2) {\Large $V_{\xi}$};

\draw[very thick, blue](0,-3,0)--(0,3,0);
\node[anchor=south west,align=center] (line) at (3,3,-1) {\Large{\color{blue}$V' := V_{\xi}^{\perp}$}};
\draw[-latex, blue] (line) to[out=180,in=-105] (0.4,1.6,0.05);

\draw[fill=yellow,opacity=0.2] (-3,-3,0) -- (-3,3,0) -- (3,3,0) -- (3,-3,0) -- cycle;
\draw[black, thin] (-3,-3,0) -- (-3,3,0) -- (3,3,0) -- (3,-3,0) -- cycle;
\draw[thick, blue](-3,0,0)--(3,0,0);
\node[anchor=south west,align=center] (line) at (3,3,3) {\Large ${\color{blue}V_{\omega}}$};
\draw[-latex, blue] (line) to[out=180,in=85] (-2,-0.2,0.05);

\node[anchor=south west,align=center] (line) at (3,3.3,4.5) {\Large${\color{blue}V}$};

\tdplotsetrotatedcoords{-90}{75}{0}
\begin{scope}[tdplot_rotated_coords]
\draw[fill=blue,opacity=0.2] (-3,-3,0) -- (-3,3,0) -- (3,3,0) -- (3,-3,0) -- cycle;
\draw[black, thin] (-3,-3,0) -- (-3,3,0) -- (3,3,0) -- (3,-3,0) -- cycle;

\end{scope}
\end{tikzpicture}
    \caption{$V_{\omega} = V \cap e_n^{\perp}$.} 
  \end{subfigure}
  %
  %
\begin{subfigure}[t]{0.49\linewidth}
\tdplotsetmaincoords{70}{110}
\begin{tikzpicture}[tdplot_main_coords,font=\sffamily, scale=0.7]
\draw[-latex] (0,0,0) -- (4,0,0);
\draw (-4,0,0) -- (-3,0,0);
\draw[-latex] (0,-4,0) -- (0,4,0);
\draw[-latex] (0,0,0) -- (0,0,4) node[anchor=north east]{$\xi_n$};

\node[align=center] at (0,0.4,0.4) {\Large{$\bm{\theta}$}};

\node[anchor=south west,align=center] at (-0.8,1.2,0) {\Large$\Omega$};

\begin{scope}[canvas is yz plane at x=0]
\draw[very thick,black] (1,0) arc (0:75:1cm) node[midway] (line) {};

   \end{scope}

\draw[very thick, blue](0,-3,0)--(0,3,0);
\node[anchor=south west,align=center] (line) at (3,3,-1) {\Large{\color{blue}$V' := V_{\xi}^{\perp}$}};
\draw[-latex, blue] (line) to[out=180,in=-105] (0.4,1.6,0.05);

\draw[fill=yellow,opacity=0.2] (-3,-3,0) -- (-3,3,0) -- (3,3,0) -- (3,-3,0) -- cycle;
\draw[black, thin] (-3,-3,0) -- (-3,3,0) -- (3,3,0) -- (3,-3,0) -- cycle;
\draw[blue, thick](0,0,0)--(3,0,0);
\draw[blue, thick](-3,0,0)--(-1.9,0,0);
\draw[->,black, very thick] (0,0,0) -- (0,2,0);

\node[anchor=south west,align=center] (line) at (3,3.3,4.5) {\Large${\color{blue}V}$};

\tdplotsetrotatedcoords{-90}{75}{0}
\begin{scope}[tdplot_rotated_coords]
\draw[fill=blue,opacity=0.2] (-3,-3,0) -- (-3,3,0) -- (3,3,0) -- (3,-3,0) -- cycle;
\draw[black, thin] (-3,-3,0) -- (-3,3,0) -- (3,3,0) -- (3,-3,0) -- cycle;
\draw[->,black, very thick] (0,0,0) -- (-2,0,0);
\end{scope}
\end{tikzpicture}
\caption{$\angle(v, v') \gtrsim 1$ for $v \in V$, $v' \in V'$ non-zero.}
\end{subfigure}
 \captionsetup{singlelinecheck=off}
\caption[.]{The transversality condition in the prototypical case of the extension operator $E_{\mathrm{par}}$. Here $S_{\omega}$ is an affine subspace and $V_{\omega}$ is defined to be the linear subspace parallel to $S_{\omega}$. Moreover, $V_{\omega}$ coincides with $V_u$ and is given by the intersection of $V$ with the horizontal slice $e_n^{\perp} = \R^{n-1} \times \{0\}$. From the right-hand diagram it is clear that 
\begin{equation*}
    \theta := \min_{v \in V\setminus\{0\}, v' \in V' \setminus\{0\}}\angle(v, v') = \overline{\angle}(V, e_n^{\perp}) \gtrsim 1; \qquad\qquad\qquad\qquad
\end{equation*}
see \cite[Sublemma 6.6]{Guth2018} for a formal proof of this fact. For general operators $T^{\lambda}$ with positive-definite phase, Claim 4 guarantees that $V_u$ is a slight perturbation of the horizontal slice $V \cap e_n^{\perp}$ (see Figure~\ref{perturbed figure}) so that the angle condition still holds.}\label{transverse figure}
\end{figure}  

\subsection*{Verifying the transversality condition in 2)}

In the prototypical case of the parabolic extension operator $E_{\mathrm{par}}$, the transversality condition holds by a straightforward computation and the minimum angle can be explicitly computed (see \cite[Sublemma 6.6]{Guth2018} and Figure~\ref{transverse figure}). To establish the result for variable coefficient operators, one uses the localisation to the cap $\tau$ and ball $B$ to show that the situation is only a slight perturbation of the prototypical case. This argument is carried out in detail below.

The transversality of the planes $V$ and $V'$ heavily relies upon the positive-definite hypothesis H2$^+$); the property does not hold in general if one only assumes the weaker condition H2).

\begin{example}
 For $\phi_{\mathrm{hyp}}(x;\omega) := \langle x', \omega \rangle + x_3\omega_1\omega_2$ for $(x; \omega) \in \R^3 \times \R^2$ the oscillatory integral operator 
 \begin{equation*}
  E_{\mathrm{hyp}}f(x) := \int_{B^{2}} e^{2 \pi i \phi_{\mathrm{hyp}}(x;\omega)} f(\omega)\,\ud \omega
 \end{equation*}
is the extension operator associated to the hyperbolic paraboloid. This is the prototypical example of a H\"ormander-type operator for which H2$^+$) fails. Here $G_0(\omega) = ( - \omega_2,  - \omega_1,  1)^{\top}$ and therefore if $V := \{x \in \R^3 : x_1 = 0\}$, then 
\begin{equation*}
 S_{\omega} := \{ \omega \in B^2 : G_0(\omega) \in V \} = \{ \omega \in B^2 : \omega_2 = 0 \}.
\end{equation*}
It follows that $V_{\xi} = \{\xi \in \hat{\R}^3 : \xi_2 = 0\}$ and so $V' := V_{\xi}^{\perp}$ is a vector subspace of $V$. Clearly, in this situation the desired transversality condition completely fails.
\end{example}

The present analysis concerns H\"ormander-type operators with reduced positive-definite phase $\phi^{\lambda}$, so that $\phi^{\lambda}$ is a small perturbation of $\phi_{\mathrm{par}}$. Such phases do not exhibit the phenomenon observed in the above example: the following claim is key to understanding this.

\begin{claim 4} Let $c_{\mathrm{par}}$ be the constant defined in $\S$\ref{Reductions section}. Then
\begin{equation*}
\max_{v^* \in V \cap (S^{n-2} \times \{0\})} \angle(v^*, V_u) = O(c_{\mathrm{par}}).
\end{equation*}
Here $V_u$ is identified with a subspace of $e_n^{\perp} = \R^{n-1} \times \{0\}$ in the natural manner. 
\end{claim 4}

\begin{example} Returning to the example of the hyperbolic paraboloid with $V := \{x \in \R^3 : x_1 = 0\}$, the spaces $V \cap (\R^2 \times \{0\})$ and $V_u := \{x \in \R^3 : x_2 = x_3 = 0\}$ are mutually orthogonal, and so the claim does not hold in this case. 
\end{example}

Provided $c_{\mathrm{par}} > 0$ is chosen sufficiently small, the claim implies the transversality condition. Indeed, let $\{v_1^*, \dots, v_{\dim V - 1}^*\}$ be an orthonormal basis for $V \cap e_n^{\perp} $. Fix a unit vector $v^*_{\dim V} \in V$ which is perpendicular to $V \cap  e_n^{\perp} $ so that $\{v_1^*, \dots, v_{\dim V}^*\}$ forms an orthonormal basis for $V$. By the above claim, there exist $v_k \in V_u \cap S^{n-2}\subset e_n^{\perp}$ such that 
\begin{equation*}
 \angle(v_k^*, v_k) = O(c_{\mathrm{par}}) \qquad \textrm{for $1 \leq k \leq \dim V -1$.}
\end{equation*}
Applying the Gram--Schmidt process, one may further assume that $\{v_1,\dots, v_{\dim V -1}\}$ forms an orthonormal basis of $V_u$; adjoining $e_n$ to this set then gives an orthonormal basis of $V_{\xi}$. Given $v \in V \cap S^{n-1}$ and writing
\begin{equation*}
v = \sum_{k=1}^{\dim V -1} \langle v, v_k^*\rangle v_k + \langle v, v_{\dim V}^* \rangle v_{\dim V}^* + \sum_{k=1}^{\dim V -1} \langle v, v_k^*\rangle (v_k^*-v_k),
\end{equation*}
since $\sin \angle(v, V') = |\mathrm{proj}_{V_{\xi}}v|$, it follows that
\begin{align*}
\sin \angle(v, V') &=  \Big(\sum_{k=1}^{\dim V -1} |\langle v, v_k^*\rangle|^2 + |\langle v, v_{\dim V}^* \rangle|^2|\langle v_{\dim V}^*, e_n \rangle|^2 \Big)^{1/2} - O(c_{\mathrm{par}}) \\
& \geq \Big(\sum_{k=1}^{\dim V} |\langle v, v_k^*\rangle|^2 \Big)^{1/2}\cdot|\langle v_{\dim V}^*, e_n \rangle| - O(c_{\mathrm{par}}).
\end{align*}
Consequently, provided that $c_{\mathrm{par}}$ is chosen to be sufficiently small,
\begin{equation*}
\sin \angle(v, V') \geq |\langle v_{\dim V}^*, e_n \rangle| - O(c_{\mathrm{par}}) \gtrsim 1;
\end{equation*}
indeed, the last inequality holds since \eqref{TE2 1} implies that
\begin{equation*}
|\langle v_{\dim V}^*, e_n \rangle| = \angle(v_{\dim V}^*, e_n^{\perp} ) = \overline{\angle}(V, e_n^{\perp} ) \gtrsim 1;
\end{equation*}
see Figure~\ref{perturbed figure}. This concludes the proof of the transversality condition, conditional on the above claim.

\begin{figure} 
\centering 
\begin{tikzpicture}[scale=0.7]

\coordinate (origin) at (0,0);
\coordinate (mary) at (0,4);
\coordinate (bob) at (4,1);

\node[anchor=south west,align=center] (line) at (1,3.5) {\Large ${\color{blue}V \cap e_n^{\perp}}$};
\draw[-latex, blue] (line) to[out=270,in=0] (0.2,2.3);

\draw[->][black] 
        (-4,0)->(4,0);   
\draw[->][black] 
        (0,-4,0)->(0,4);          
  \draw[-][blue, very thick] 
		(0,-3)->(0,3) node[anchor=north]{};
		
\node[anchor=east] at (2,-1) {\Large $U$};  
		
\draw[fill=yellow,opacity=0.2] (-3,3) -- (3,3) -- (3,-3) -- (-3,-3) -- cycle;
\draw[black, thin] (-3,-3) -- (-3,3) -- (3,3) -- (3,-3) -- cycle;

  \draw[-][blue, very thick] 
		(-3,-3/4)->(3,3/4) node[anchor=west]{\Large $V' := V_{\xi}^{\perp}$};  
  \draw[-][purple, thick] 
        (3/4,-3)->(-3/4,3);
\node[anchor=east] at (-0.6,1.2) {\Large ${\color{purple}V_u}$};  
    
        
\pic [draw, black, thick, <->, "\Large $\:\:\theta$", angle eccentricity=1.2, angle radius=1cm] {angle = bob--origin--mary};

\end{tikzpicture}
 \captionsetup{singlelinecheck=off}
\caption[.]{In the variable coefficient case, Claim 4 guarantees that the subspace $V_u$ makes a small angle with the horizontal slice $V \cap e_n^{\perp}$. This implies that the angle $\theta$ is always large (that is, bounded below by an absolute constant) and thereby ensures that $V$ and $V'$ are transverse.} \label{perturbed figure}
\end{figure}
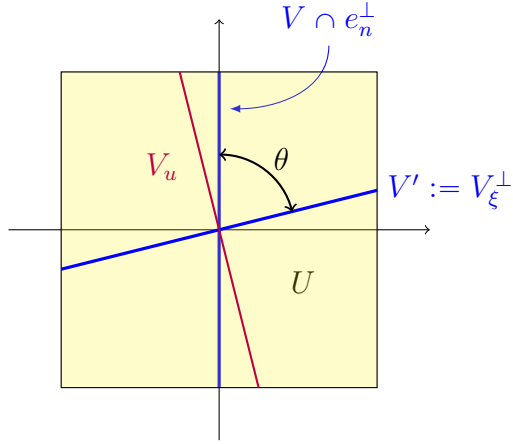

\begin{proof}[Proof of Claim 4] Fix $v^* \in  V \cap (S^{n-2} \times \{0\})$  and let $v \in V_u \cap S^{n-2}$ denote the unit normalisation of the vector $\mathrm{proj}_{V_u} v^*$. It suffices to show that $\angle(v^*, v) = O(c_{\mathrm{par}})$. Since 
\begin{equation*}
v^* - v = (|\mathrm{proj}_{V_u} v^*| - 1)v +  \mathrm{proj}_{V_u^{\perp}}v^*
\end{equation*}
and $||\mathrm{proj}_{V_u} v^*| - 1| \leq |\mathrm{proj}_{V_u^{\perp}}v^*|$, the problem is further reduced to proving that
\begin{equation}\label{V space claim 2}
|\mathrm{proj}_{V_u^{\perp}}v^*| = O(c_{\mathrm{par}}).
\end{equation}

Recall that 
\begin{align*}
  S_u &:= \{u \in U : G_0^{\lambda}(\bar{x};\Psi(u)) \in V \} \\
  &=   \{u \in U : \langle G_0^{\lambda}(\bar{x};\Psi(u)), N_k \rangle = 0 \textrm{ for $1 \leq k \leq n - \dim V$}\}
\end{align*}
 where, as above, $\{N_1, \dots, N_{n - \dim V}\}$ is a choice of orthonormal basis for $V^{\perp}$. If $\bar{h}$ is the function introduced in \eqref{h bar definition}, then $u \mapsto (u, \bar{h}(u))$ is a graph parametrisation  of the surface $\Sigma$ and  $u \mapsto G_0^{\lambda}(\bar{x};\Psi(u))$ is the (unnormalised) Gauss map associated to this parametrisation. Thus, the surface $S_u$ is defined by the equations \begin{equation*}
   -\langle\partial_{u} \bar{h}(u), N_k' \rangle + N_{k,n} = 0 \quad \textrm{for $1 \leq k \leq n-\dim V$,}
 \end{equation*}
 where $N_k = (N'_k,N_{k,n}) \in \R^{n-1} \times \R$. By differentiating these expressions, one deduces that a basis for $V_u^{\perp}$ is given by $\{M_1, \dots, M_{n - \dim V}\}$ where
\begin{equation*}
M_k := \partial_{uu}^2 \bar{h}(u_0)N_k' \qquad \textrm{for $1 \leq k \leq n - \dim V$.}
\end{equation*}
Let $1 \leq k \leq n - \dim V$ and recall from Lemma~\ref{hx nearly paraboloid} that 
\begin{equation}\label{V space claim 3}
\|\partial_{uu}^2 \bar{h}(u_0) - \mathrm{I}_{n-1}\|_{\mathrm{op}} = O(c_{\mathrm{par}}).
\end{equation}
 Consequently, 
\begin{equation}\label{V space claim 4}
|M_k - N_k'| = O(c_{\mathrm{par}})
\end{equation}
and, combining this with the fact that $\langle v^*, N_k' \rangle = 0$ for $1 \leq k \leq n - \dim V$ (where $v^*$ is identified with a vector in $\R^{n-1}$ in the natural manner), it follows that
\begin{equation}\label{V space claim 5}
\langle v^*, M_k \rangle = \langle v^*, M_k - N_k' \rangle   = O(c_{\mathrm{par}}).
\end{equation}
Let $\mathbf{M}$ be the $(n-1) \times (n - \dim V)$ matrix whose $k$th column is given by the vector $M_k$. The orthogonal projection of $v^*$ onto the subspace $V_u^{\perp}$ can be expressed in terms of $\mathbf{M}$ via the formula  
\begin{equation*}
\mathrm{proj}_{V_u^{\perp}} v^* := \mathbf{M}(\mathbf{M}^{\top}\mathbf{M})^{-1}\mathbf{M}^{\top} v^*.
\end{equation*}
By \eqref{V space claim 5}, the components of the vector $\mathbf{M}^{\top} v^*$ are all $O(c_{\mathrm{par}})$. Thus, to prove \eqref{V space claim 2} (and thereby establish the claim) it suffices to show that $\|\mathbf{M}(\mathbf{M}^{\top}\mathbf{M})^{-1}\|_{\mathrm{op}} \lesssim 1$, which would in turn follow from
\begin{equation*}
\|\mathbf{M}\|_{\mathrm{op}} \lesssim 1 \quad \textrm{and} \quad \|(\mathbf{M}^{\top}\mathbf{M})^{-1}\|_{\mathrm{op}} \lesssim 1. 
\end{equation*} 
The bound for $\mathbf{M}$ is an immediate consequence of the definition of the $M_k$ and \eqref{V space claim 3}. The remaining estimate would follow if one could show that, provided  $c_{\mathrm{par}}$ is sufficiently small, $|\lambda| \gtrsim 1$ for every eigenvalue $\lambda$ of the symmetric matrix $\mathbf{M}^{\top}\mathbf{M}$. By \eqref{V space claim 4} and continuity of eigenvalues, it suffices to show that the matrix $\mathbf{N}^{\top}\mathbf{N}$ satisfies the same property, where $\mathbf{N}$ is the $(n-1) \times (n - \dim V)$ matrix whose $k$th column is given by the vector $N_k'$. By \eqref{TE2 1} the vectors $N_1', \dots, N_{n - \dim V}' \in \R^{n-1}$ are linearly independent and, moreover, satisfy 
\begin{equation*}
|\det \mathbf{N}^{\top}\mathbf{N}| = |N_1' \wedge \dots \wedge N_{n - \dim V}'|^2 \gtrsim 1.
\end{equation*} 
Therefore, the desired condition on the eigenvalues holds if the spectral radius of $\mathbf{N}^{\top}\mathbf{N}$ is $O(1)$. But the latter property is an obvious consequence of the Newton--Girard identity
\begin{equation*}
\sum_{i=1}^{m}a_i^2  = \big(\sum_{i=1}^{m} a_i\big)^2 - 2\sum_{1 \leq i_1 < i_2 \leq m} a_{i_1}a_{i_2} \qquad a_i \in \R \textrm{ for }1 \leq i \leq m
\end{equation*}
and the fact that the entries of $\mathbf{N}^{\top}\mathbf{N}$ are all $O(1)$.
\end{proof}
This concludes the proof of Lemma~\ref{Transverse equidistribution lemma 2}.
\end{proof}




\subsection{The proof of the transverse equidistribution estimate} It remains to demonstrate how to pass from Lemma~\ref{Transverse equidistribution lemma 2} to Lemma~\ref{Transverse equidistribution lemma 1}. At this stage, the proof is very similar to the argument found in \cite{Guth2018}, but the details are nevertheless included for completeness. 

There are two additional ingredients needed for the proof of Lemma~\ref{Transverse equidistribution lemma 1}. The first is the following theorem of Wongkew \cite{Wongkew1993} (see also \cite{Guth2016, Zhang2017}), which controls the size of a neighbourhood of a variety.

\begin{theorem}[Wongkew \cite{Wongkew1993}]\label{Wongkew theorem} Suppose $Z = Z(P_1, \dots, P_{n-m})$ is an $m$-dimensional transverse complete intersection in $\R^n$ with $\overline{\deg}\, Z \leq D$. For any $0 < \rho \leq R$ and $R$-ball $B_R$ the neighbourhood $N_{\rho}(Z\cap B_R)$ can be covered by $O_D( (R/\rho)^{m})$ balls of radius $\rho$.
\end{theorem}
 
The second ingredient is a geometric lemma concerning planar slices of neighbourhoods of varieties. The statement of this result requires a general quantitative notion of transversality for pairs of linear subspaces in $\R^n$. Any $m$-dimensional linear subspace $V$ can be expressed as a transverse complete intersection $V = Z(P_{N_1}, \dots, P_{N_{n-m}})$ where $\{N_1,\dots, N_{n-m}\}$ forms an orthonormal basis of $V^{\perp}$ and $P_{N_j}(x) := \langle x, N_j \rangle$. Suppose $V_1, V_2$ are linear subspaces in $\R^n$ satisfying 
\begin{equation}\label{dimension hypothesis}
\dim V_1 + \dim V_2 \geq n.
\end{equation}
It is easy to verify that the subspace $V_1 \cap V_2$ is a transverse complete intersection if and only if  $\dim V_1 \cap V_2 = \dim V_1 + \dim V_2 - n$ (of course, the inequality $\dim V_1 \cap V_2 \geq \dim V_1 + \dim V_2 - n$ always holds so the latter condition says that $V_1 \cap V_2$ is as small as possible). 

\begin{definition} A pair $V_1, V_2$ of linear subspaces in $\R^n$ satisfying \eqref{dimension hypothesis} is said to be quantitatively transverse if the following hold:
\begin{enumerate}[i)]
\item $\dim(V_1 \cap V_2) = \dim V_1 + \dim V_2 - n$;
\item $\angle(v_1, v_2) \geq c_{\mathrm{trans}}$ for all non-zero vectors $v_j \in (V_1 \cap V_2)^{\perp} \cap V_j$, $j=1,2$.
\end{enumerate}
\end{definition}

\begin{remark}
 In the special case where $\dim V_1 + \dim V_2 = n$, it follows that the pair $V_1, V_2$ is quantitatively transverse if and only if $\angle(v_1, v_2) \geq c_{\mathrm{trans}}$ for all non-zero vectors $v_1 \in V_1$, $v_2 \in V_2$. Thus, up to the minor disparity between the choice of constant $c_{\mathrm{trans}}$, this agrees with the transversality condition appearing in the statement of Lemma~\ref{Transverse equidistribution lemma 2}.
\end{remark}

\begin{lemma}\label{quantitatively transverse lemma} There exists some dimensional constant $C > 0$ such that the following holds. Let $B_r \subseteq \R^n$ be an $r$-ball, $V \subseteq \R^n$ a linear subspace, $Z$ a transverse complete intersection and suppose that $\dim Z + \dim V \geq n$ and $T_zZ, V$ is a quantitatively transverse pair for all $z \in Z \cap 2B_r$. Then
\begin{equation*}
 V \cap B_r \cap N_{\rho}(Z) \subseteq N_{C\rho}(V \cap Z)
\end{equation*}
for all $0 < \rho \ll r$.
\end{lemma}

The proof of this simple lemma is postponed until the end of this section.

\begin{proof}[Proof (of Lemma~\ref{Transverse equidistribution lemma 1}).] If $T_{\theta, v} \cap N_{R^{1/2 + \delta_m}}(Z)\cap B = \emptyset$, then it follows that
\begin{equation*}
 |T^{\lambda}g_{\theta, v}(x)| = \mathrm{RapDec}(R)\|g\|_{L^2(B^{n-1})} \qquad \textrm{for all $x \in N_{\rho^{1/2 + \delta_m}}(Z)\cap B$.}
\end{equation*}
Consequently, one may assume that $g$ is concentrated on only those wave packets from $\mathbb{T}_{B, \tau, Z}$ for which $T_{\theta,v}$ intersects $N_{R^{1/2+\delta_m}}(Z) \cap B$ non-trivially. Suppose $(\theta, v) \in \mathbb{T}_{B, \tau, Z}$ has this property and let  $x \in T_{\theta, v} \cap N_{R^{1/2 + \delta_m}}(Z) \cap B$. If $z \in Z \cap 2B$, then $|x - z| \lesssim R^{1/2 + \delta_m}$ and, by the $R^{1/2 + \delta_m}$-tangent condition,
\begin{equation*}
 \angle(G^{\lambda}(x; \theta), T_zZ) \lesssim R^{-1/2 + \delta_m}.
\end{equation*}
Since $|G^{\lambda}(\bar{x}; \theta) - G^{\lambda}(x; \theta)| \lesssim |\bar{x} - x|/\lambda \lesssim R^{-1/2 + \delta_m}$, one concludes that
\begin{equation*}
  \angle(G^{\lambda}(\bar{x}; \theta), T_zZ) \lesssim R^{-1/2 + \delta_m} \qquad \textrm{for all $z \in Z \cap 2B$.}
\end{equation*}
Thus, there exists a subspace $V \subseteq \R^n$ of minimal dimension $\dim V \leq \dim Z$ such that
\begin{equation*}
   \angle(G^{\lambda}(\bar{x}; \theta), V) \lesssim R^{-1/2 + \delta_m}
\end{equation*}
for all $(\theta, v) \in \mathbb{T}_{B, \tau, Z}$ for which $T_{\theta, v} \cap N_{R^{1/2+\delta_m}}(Z) \cap B \neq \emptyset$. In particular, $g$ is concentrated on wave packets from $\mathbb{T}_{B,\tau, V}$. One may apply Lemma~\ref{Transverse equidistribution lemma 2} to find a subspace $V'$ of dimension $n - \dim V$ such that 
\begin{equation}\label{key transversality condition}
 \angle(v, v') \geq 2c_{\mathrm{trans}} \qquad \textrm{for all non-zero vectors $v \in V$ and $v' \in V'$} 
\end{equation}
and 
\begin{equation}\label{transverse equidistribution 1}
\int\displaylimits_{\Pi \cap B(x_0, \rho^{1/2 + \delta_m})} \!\!\!\!\!\!\!\!\!\!\!\!|T^{\lambda}g|^2 \lesssim_{\delta} R^{O(\delta_m)} (\rho/R)^{\dim V'/2} \|g\|_{L^2(B^{n-1})}^{2\delta/(1+\delta)} \big(\int\displaylimits_{\Pi \cap 2B} |T^{\lambda}g|^2 \big)^{1/(1+\delta)}
\end{equation}
for every affine subspace $\Pi$ parallel to $V'$. More precisely, the above estimate holds up to the inclusion of some additional rapidly decreasing term. This small error will propagate through the remainder of the argument but in the end will be harmless and is therefore suppressed in the notation.
 
It is claimed that for each $z \in Z \cap 2B$ the tangent space $T_zZ$ forms a quantitatively transverse pair $T_zZ, V'$ with $V'$. Indeed, if this fails, then it is easy to see that for some $z \in Z \cap 2B$ there exists a subspace $W \subseteq T_zZ$ of dimension
\begin{equation*}
 \dim W > \dim Z - \dim V
\end{equation*}
with the property that $\angle(w, V') < c_{\mathrm{trans}}$ for all $w \in W \setminus \{0\}$. Consequently, the crucial angle condition \eqref{key transversality condition} guarantees that
\begin{equation*}
 \angle(w, V) \geq c_{\mathrm{trans}} \qquad \textrm{for all $w \in W \setminus \{0\}$.}
\end{equation*}
This implies that there exists a linear map $L \colon \R^n \to V$ such that $L$ restricted to $V$ is the identity, $L$ restricted to $W$ is zero and $\|L\|_{\mathrm{op}} \lesssim 1$. Recall that for each $(\theta, v) \in \T_{B,Z}$ one has $\angle(G^{\lambda}(\bar{x}, \theta), V) \lesssim R^{-1/2 + \delta_m}$ and so 
\begin{equation*}
 \sup_{\omega, \omega' \in \theta} |L(G^{\lambda}(\bar{x}, \omega)) -  G^{\lambda}(\bar{x}, \omega')| \lesssim R^{-1/2 + \delta_m}.
\end{equation*}
On the other hand, $G^{\lambda}(\bar{x}, \theta) \subset N_{C_1R^{-1/2 + \delta_m}}(T_zZ) \cap S^{n-1}$ and so $L(G^{\lambda}(\bar{x},\theta))$ lies in 
\begin{equation*}
 L(N_{C_1R^{-1/2 + \delta_m}}(T_zZ) \cap S^{n-1}) \subseteq N_{C_2R^{-1/2 + \delta_m}}(L(T_zZ)).
\end{equation*}
This shows that for all $(\theta, v) \in \T_{B,Z}$ one has
\begin{equation*}
 \angle(G^{\lambda}(\bar{x}, \theta), L(T_zZ)) \lesssim R^{-1/2 + \delta_m}.
\end{equation*}
Since $L$ vanishes on $W$, by rank-nullity $L(T_zZ)$ is a subspace of dimension at most $\dim Z - \dim W < \dim V$. This contradicts the minimality of $V$ and so $(T_zZ, V')$ is a quantitatively transverse pair for all $z \in Z \cap 2B$. 

By Lemma~\ref{quantitatively transverse lemma} one deduces that
\begin{equation*}
 \Pi \cap N_{\rho^{1/2 + \delta_m}}(Z) \cap B \subseteq N_{C\rho^{1/2 + \delta_m}}(\Pi \cap Z) \cap 2B.
\end{equation*}
Since $\Pi \cap Z$ is a transverse complete intersection of dimension $\dim V' + \dim Z - n$, Wongkew's theorem now implies that $\Pi \cap N_{\rho^{1/2 + \delta_m}}(Z) \cap B$ can be covered by 
\begin{equation*}
 O\big(R^{O(\delta_m)} (R/\rho)^{(\dim V' + \dim Z - n)/2} \big)
\end{equation*}
balls of radius $\rho^{1/2 + \delta_m}$. Applying the estimate \eqref{transverse equidistribution 1} to each of these balls and summing, one deduces that 
\begin{equation*}
\int\displaylimits_{\Pi \cap N_{\rho^{1/2 + \delta_m}}(Z) \cap B}\!\!\!\!\!\!\! |T^{\lambda}g|^2 \lesssim_{\delta} R^{O(\delta_m)} (\rho/R)^{(n-m)/2} \|g\|_{L^2(B^{n-1})}^{2\delta/(1+\delta)} \big(\int\displaylimits_{\Pi \cap 2B} |T^{\lambda}g|^2 \big)^{1/(1+\delta)}.
\end{equation*}
 Integrating over planes $\Pi$ parallel to $V'$ and applying H\"older's inequality (recalling that $\delta \ll \delta_m$), it follows that 
\begin{equation*}
 \int\displaylimits_{N_{\rho^{1/2 + \delta_m}}(Z) \cap B} |T^{\lambda}g|^2 \lesssim_{\delta} R^{O(\delta_m)} (\rho/R)^{(n-m)/2} \|g\|_{L^2(B^{n-1})}^{2\delta/(1+\delta)} \big(\int_{2B} |T^{\lambda}g|^2 \big)^{1/(1+\delta)}.
\end{equation*}
Finally, recalling H\"ormander's bound,
\begin{equation*}
 \big(\int_{2B} |T^{\lambda}g|^2 \big)^{1/(1+\delta)} \lesssim R^{1/2 +O(\delta_m)} \big(\int_{B^{n-1}} |g|^2\big)^{1/(1+\delta)}
\end{equation*}
and absorbing the implied rapidly decaying error into the main term, one concludes that
\begin{equation*}
 \int\displaylimits_{N_{\rho^{1/2 + \delta_m}}(Z) \cap B} |T^{\lambda}g|^2 \lesssim_{\delta} R^{1/2+O(\delta_m)}(\rho/R)^{(n-m)/2}\|g\|^2_{L^2(B^{n-1})},
\end{equation*}
which is the desired estimate. 
\end{proof}

It remains to prove Lemma~\ref{quantitatively transverse lemma}. 

\begin{proof}[Proof (of Lemma~\ref{quantitatively transverse lemma})] Applying a rotation, one may assume that $V$ is the span of the co-ordinate vectors $e_1, \dots, e_{\dim V}$. For the purposes of this proof, $\gamma_{V} := (\gamma_1, \dots, \gamma_{\dim V})$ and $\gamma_{V^{\perp}} = (\gamma_{\dim V +1}, \dots, \gamma_n)$ will denote the orthogonal projections of a space curve $\gamma$ onto $V$ and $V^{\perp}$, respectively.

Suppose that $x \in V \cap B_r \cap N_{\rho}(Z) $ and fix some $z_0 \in Z\cap N_{\rho}(B_r)$ with $0 < |x - z_0| < \rho$. Let $\gamma \colon \R \to \R^n$ be the constant speed parametrisation of the line through $\gamma(0) := z_0$ and $\gamma(1) := x$. To prove the lemma it suffices to show that there exists a curve $\tilde{\gamma} \colon [0,1] \to \R^n$ such that for all $t \in [0,1]$ the following hold:
\begin{enumerate}[1)]
\item $\tilde{\gamma}(0) = \gamma(0) = z_0$,
\item $\tilde{\gamma}(t) \in Z$,
\item $\tilde{\gamma}_{V^{\perp}}(t) = \gamma_{V^{\perp}}(t)$,
\item $|\tilde{\gamma}'(t)| \leq \bar{C} |\tilde{\gamma}_{V^{\perp}}'(t)|$ where $\bar{C} := (\sin c_{\mathrm{trans}})^{-1}$.
\end{enumerate}
Indeed, once this is established, observe that $z_1 := \tilde{\gamma}(1) \in Z \cap V$ by properties 2) and 3). Furthermore, 3) and 4) ensure that 
\begin{equation*}
|\tilde{\gamma}'(t)| \leq \bar{C}  |\tilde{\gamma}_{V^{\perp}}'(t)| \leq \bar{C} |\gamma'(t)| < \bar{C}  \rho
\end{equation*}
and so, combining this observation with 1), 
\begin{equation}\label{quantitatively transverse lemma 1}
|x - \tilde{\gamma}(t)| \leq |x - z_0| + |z_0-  \tilde{\gamma}(t)| < (1+  \bar{C}t )\rho
\end{equation}
for all $t \in [0,1]$. In particular, $|x - z_1| \lesssim \rho$, giving the desired conclusion. 

The transversality condition implies that the distribution (in the sense of Frobenius: see, for instance, \cite[Chapter 1]{Warner})
\begin{equation*}
    W_z := (T_zZ \cap V)^{\perp}\cap T_zZ
\end{equation*}
has rank $n-\dim V$ on $Z \cap 2B_r$ and, moreover, $\mathrm{proj}_{W_z^{\perp}} |_V \colon V \to W_z^{\perp}$ is an isomorphism for all $z \in Z \cap 2B_r$. Smoothly extend $W_z$ to a small neighbourhood $U$ of $Z \cap 2B_r$ so that
\begin{equation}\label{quantitatively transverse lemma 2}
\mathrm{proj}_{W_x^{\perp}} |_V \colon V \to W_x^{\perp} \qquad \textrm{is an isomorphism for all $x \in U$.}
\end{equation}
The curve $\tilde{\gamma}$ will be chosen so that its tangent always lies in this distribution. Given that $\tilde{\gamma}_{V^{\perp}}$ is already defined by property 3), to satisfy this condition $\tilde{\gamma}_V$ must be a solution to the ODE
\begin{equation*}
\left\{\begin{array}{rcl}
\mathrm{proj}_{W_{x(t)}^{\perp}}(y'(t), \gamma_{V^{\perp}}'(t)) & =& 0\\
y(0) & = & \mathrm{proj}_V z_0
\end{array} \right. 
\end{equation*}
where $x(t) := (y(t), \gamma_{V^{\perp}}(t))$. By \eqref{quantitatively transverse lemma 2}, solving the above ODE is equivalent to solving a system of the form
\begin{equation}\label{quantitatively transverse lemma 3}
\left\{\begin{array}{rcl}
y' & =& g(t,y)\\
y(0) & = & \mathrm{proj}_V z_0
\end{array} \right. 
\end{equation}
for $g$ a smooth function defined on $\{(t,y) \in \R \times V : (y, \gamma_{V^{\perp}}(t)) \in U\}$. Note that $g$ can be described explicitly in terms of the inverse of $\mathrm{proj}_{W_x^{\perp}}|_V$ and, provided $U$ is appropriately chosen, the derivatives of $g$ are bounded. 

The Picard--Lindel\"of existence theorem implies that the system \eqref{quantitatively transverse lemma 3} has a solution $\tilde{\gamma}_{V}$ defined on an interval $[0,T]$ for some $T > 0$ such that $\tilde{\gamma} := (\tilde{\gamma}_V, \gamma_{V^{\perp}})$ satisfies $\tilde{\gamma}(t) \in 2B_r$ for all $t \in [0,T]$. It can be checked that on this interval the curve $\tilde{\gamma}$ further satisfies 1) and 3) and, by the tangency condition which motivated the definition of the ODE, 2) also holds. If $t \in [0,T]$, then it follows that $\tilde{\gamma}(t) \in Z \cap 2B_r$ and $\tilde{\gamma}'(t) \in W_{\tilde{\gamma}(t)}$ and so the transversality hypothesis implies that 
\begin{equation*}
\bar{C} ^{-1} = \sin c_{\mathrm{trans}} \leq \sin\angle(\tilde{\gamma}'(t), V) = \frac{|\tilde{\gamma}_{V^{\perp}}'(t)|}{|\tilde{\gamma}'(t)|}.
\end{equation*}
Rearranging, one concludes that properties 1) to 4) all hold on $[0,T]$.

It remains to show that $T$ can be chosen to satisfy $T \geq 1$. If $\mathrm{dist}(\tilde{\gamma}(T), U^{c}) \gtrsim 1$, then the regularity of $g$ implies that the interval of existence can be extended by a fixed increment. Thus, one may assume that at least one of the following holds: $T \geq 1$ or $|\tilde{\gamma}(T) - x| \geq r/2$.\footnote{Indeed, suppose both conditions fail for $T$. The failure of the latter condition implies that $\mathrm{dist}(\tilde{\gamma}(T), (2B_r)^{c}) \geq r/2$. Since $\tilde{\gamma}(T) \in Z$ by property 2), one concludes that $\tilde{\gamma}(T)$ is far from $U^{c}$ and thus the interval of existence for $\tilde{\gamma}$ can be extended by a fixed increment. One may redefine $T$ to be some value in the interval of existence incrementally larger than the original value of $T$ and repeat this procedure until at least one of the stated conditions hold.} Supposing the latter holds, by the choice of $T$ and \eqref{quantitatively transverse lemma 1}, one deduces that
\begin{equation*}
r/2 \leq |\tilde{\gamma}(T) - x| \leq (1+ \bar{C}T ) \rho.
\end{equation*}
Provided $r$ is chosen to be sufficiently large compared to $\rho$, the desired bound immediately follows.
\end{proof}




\section{Comparing wave packets at different spatial scales}\label{Adjusting wave packets section}




\subsection{Wave packet decomposition at scale $\rho$}

The proof of Theorem~\ref{k-broad theorem} relies on a multi-scale analysis and for this it is necessary to compare wave packets at different scales.

Let $1 \ll R \ll \lambda$ and recall the decomposition
\begin{equation*}
T^{\lambda} f(x)=\sum_{(\theta,v) \in \T}T^{\lambda} f_{\theta,v}(x)+\textrm{RapDec}(R)\|f\|_{L^2(B^{n-1})}
\end{equation*}
described in $\S$\ref{Wave packet section}. Consider a smaller spatial scale\footnote{Later it will be useful to assume the more stringent condition $R^{1/2} \leq \rho \leq R^{1 - 2\delta}$.} $R^{1/2}\leq \rho \leq R$  and fix $B(y,\rho)\subset B(0,R)$ with centre $y \in X^{\lambda}$. Each of the $T^{\lambda}f_{\theta,v}$ can be further decomposed into wave packets at scale $\rho$ over $B(y,\rho)$. To do this, first apply a transformation to recentre $B(y, \rho)$ at the origin. For $g\colon B^{n-1} \to \C$ integrable define $\tilde{g}:= e^{2\pi i \phi^{\lambda}(y;\;\cdot)}g$ so that
\begin{equation}\label{translating the operator}
T^{\lambda} g(x)=\tilde{T}^{\lambda} \tilde{g}(\tilde{x}) \qquad \textrm{for }\tilde{x} = x - y
\end{equation}
where $\tilde{T}^{\lambda}$ is the H\"ormander-type operator with phase $\tilde{\phi}^{\lambda}$ and amplitude $\tilde{a}^{\lambda}$ given by
\begin{equation}\label{translated phase}
\tilde{\phi}(x;\omega) := \phi\Big(x+\frac{y}{\lambda};\omega\Big)-\phi\Big(\frac{y}{\lambda};\omega\Big) \quad \textrm{and} \quad \tilde{a}(x;\omega) := a\Big(x+\frac{y}{\lambda};\omega\Big).
\end{equation}
Applying this identity to the wave packet decomposition above,
\begin{equation*}
T^{\lambda} f(x) =\sum_{(\theta,v) \in \T}\tilde{T}^{\lambda} ((f_{\theta,v})\;\widetilde{}\;)(\tilde{x})+\textrm{RapDec}(R)\|f\|_{L^2(B^{n-1})}.
\end{equation*}
Each $T^{\lambda} f_{\theta,v}$ is (spatially) concentrated on the curved $R^{1/2+\delta}$-tube $T_{\theta,v}$ and, consequently, each $\tilde{T}^{\lambda} (f_{\theta,v})\;\widetilde{}\;$ is  concentrated on the translate $T_{\theta,v}-y$. Since\footnote{For every fixed $\omega,v$, here $\gamma^\lambda_{\omega,v}$ is used to denote the curve satisfying $\partial_\omega\phi^\lambda((\gamma_{\omega,v}^\lambda(t),t);\omega)=v$ for all (admissible) $t\in (-R,R)$. In the notation of \S\ref{Wave packet section}, $\gamma^\lambda_{\omega,v}=\gamma^\lambda_{\theta,v}$ for a cap $\theta$ with $\omega=\omega_\theta$.}
\begin{align} 
\nonumber
\partial_{\omega}\tilde{\phi}^{\lambda}((\gamma^{\lambda}_{\omega,v}(t),t)-y;\omega) &= \partial_{\omega}\phi^{\lambda}((\gamma^{\lambda}_{\omega,v}(t),t);\omega)-\partial_{\omega}\phi^{\lambda}(y;\omega) \\
\label{pre curve identity}
&=v-\partial_{\omega}\phi^{\lambda}(y;\omega),
\end{align}
the core curve $\Gamma^{\lambda}_{\theta,v}-y$ of $T_{\theta,v}-y$ is equal to $\Gamma_{\theta,v-\bar{v}(y;\omega_{\theta})}$ where
\begin{equation*}
\bar{v}(y;\omega):=\partial_{\omega}\phi^{\lambda}(y;\omega).
\end{equation*}

Now repeat the construction of the wave packets for each $\tilde{T}^{\lambda} (f_{\theta,v})\;\widetilde{}\;$, but at scale $\rho$. In particular, cover $\Omega$ by finitely overlapping caps $\tilde{\theta}$ of radius $\rho^{-1/2}$ and $\R^{n-1}$ by finitely-overlapping balls of radius $\rho^{(1+\delta)/2}$ centered at vectors $\tilde{v}\in \rho^{(1+\delta)/2}\mathbb{Z}^{n-1}$. Let $\tilde{\T}$ denote the set of all pairs $(\tilde{\theta}, \tilde{v})$. For each $(\theta,v) \in \T$ one may decompose 
\begin{equation*}
(f_{\theta, v})\;\widetilde{}\; = \sum_{(\tilde{\theta}, \tilde{v}) \in \tilde{\T}} (f_{\theta, v})\;\widetilde{}_{\!\!\tilde{\theta}, \tilde{v}} + \mathrm{RapDec}(R)\|f\|_{L^2(B^{n-1})}
\end{equation*} 
as in $\S$\ref{Wave packet section}. The significant contributions to this sum arise from pairs $(\tilde{\theta}, \tilde{v})$ belonging to
\begin{equation*}
\tilde{\T}_{\theta,v}:=\{(\tilde{\theta},\tilde{v}) \in \tilde{\T}: \textrm{dist}(\theta,\tilde{\theta})\lesssim \rho^{-1/2}\text{ and }|v-\bar{v}(y;\omega_{\theta})-\tilde{v}|\lesssim R^{(1+\delta)/2}\},
\end{equation*}
as demonstrated by the following lemma.

\begin{lemma}\label{concentration on smaller scale wave packets lemma} If $R^{1/2} \leq \rho \leq R$, then, with the above definitions, the function $(f_{\theta,v})\;\widetilde{}\;$ is concentrated on wave packets from $\tilde{\T}_{\theta,v}$; that is,
\begin{equation*}
(f_{\theta,v})\;\widetilde{}\; = \sum_{(\tilde{\theta},\tilde{v}) \in \tilde{\T}_{\theta,v}} (f_{\theta,v})\;\widetilde{}_{\!\!\tilde{\theta}, \tilde{v}} + \mathrm{RapDec}(R)\|f\|_{L^2(B^{n-1})}.
\end{equation*}
\end{lemma}

\begin{proof} Since $(f_{\theta,v})\;\widetilde{}\;$ is supported in $\theta$, clearly the wave packets of $(f_{\theta,v})\;\widetilde{}\;$ at scale $\rho$ are all supported in
\begin{equation*}
\bigcup_{\tilde{\theta}:\; \textrm{dist}(\tilde{\theta},\theta)\lesssim \rho^{-1/2}} \tilde{\theta}.
\end{equation*}
Thus, it suffices to show that for each $(\theta, v) \in \T$ and $\rho^{-1/2}$-cap $\tilde{\theta}$ one has
\begin{equation*}
\sum_{\tilde{v} : |v - \bar{v}(y;\omega_{\theta}) - \tilde{v}| \gtrsim R^{(1+\delta)/2}} (f_{\theta,v})\;\widetilde{}_{\!\!\tilde{\theta}, \tilde{v}} = \mathrm{RapDec}(R)\|f\|_{L^2(B^{n-1})}.
\end{equation*}
Fixing $(\theta, v) \in \T$ and $(\tilde{\theta},\tilde{v}) \in \tilde{\T}$ with $|v - \bar{v}(y;\omega_{\theta}) - \tilde{v}| \gtrsim R^{(1+\delta)/2}$, the above estimate would follow if one could show that
\begin{equation*}
(f_{\theta,v})\;\widetilde{}_{\!\!\tilde{\theta}, \tilde{v}} = \big(1 + R^{-1/2}|v - \bar{v}(y;\omega_{\theta}) - \tilde{v}|\big)^{-(n+1)}\mathrm{RapDec}(R)\|f\|_{L^2(B^{n-1})}.
\end{equation*}
By definition, $(f_{\theta,v})\;\widetilde{}_{\!\!\tilde{\theta}, \tilde{v}}=\tilde{\psi}_{\tilde{\theta}}\cdot [(\eta_{\tilde v})\;\hat{}\ast (\psi_{\tilde{\theta}}f_{\theta,v})]$ for the bump functions $\tilde{\psi}_{\tilde{\theta}}, \eta_{\tilde{v}}, \psi_{\tilde{\theta}}$ as defined in \S\ref{Wave packet section}. Thus, Fourier inversion yields the pointwise bound
\begin{align*}
|(f_{\theta,v})\;\widetilde{}_{\!\!\tilde{\theta}, \tilde{v}}(\omega)| &\leq \|(\tilde{\psi}_{\tilde{\theta}})\,\check{}\,\|_{L^1(\R^n)}\big\|\eta_{\tilde{v}} \big(\psi_{\tilde{\theta}} (f_{\theta,v})\;\widetilde{}\;\big)\;\widecheck{}\;\big\|_{L^{\infty}(\R^n)}  \\
&\lesssim \big\| (\psi_{\tilde{\theta}})\,\check{}\,\ast \big((f_{\theta,v})\;\widetilde{}\;\big)\;\widecheck{}\;\big\|_{L^{\infty}(B(\tilde{v}, C\rho^{(1+\delta)/2}))}
\end{align*}
for all $\omega \in \R^{n-1}$. Since $(\psi_{\tilde{\theta}})\,\check{}\,$ is concentrated in $B(0, \rho^{1/2})$, the problem is further reduced to showing that
\begin{equation*}
\big((f_{\theta,v})\;\widetilde{}\;\big)\;\widecheck{}\;(z) = \big(1 + R^{-1/2}|z - v + \bar{v}(y;\omega_{\theta}) |\big)^{-(n+1)}\mathrm{RapDec}(R)\|f\|_{L^2(B^{n-1})}
\end{equation*}
whenever $|z - v +\bar{v}(y;\omega_{\theta})| \gtrsim R^{(1+\delta)/2}$. 

Let $\tilde{\tilde{\psi}}$ be a Schwartz function on $\R^{n-1}$ satisfying $\tilde{\tilde{\psi}}(\omega) = 1$ for $\omega \in B^{n-1}$ and define $\tilde{\tilde{\psi}}_{\theta}(\omega) := \tilde{\tilde{\psi}}(R^{1/2}(\omega - \omega_{\theta}))$ so that
\begin{equation*}
\big((f_{\theta,v})\;\widetilde{}\;\big)\;\widecheck{}\; =\big( \tilde{\tilde{\psi}}_{\theta} e^{2\pi i \phi^{\lambda}(y;\;\cdot)}\big)\;\widecheck{}\;\ast (f_{\theta,v})\;\widecheck{}\;.
\end{equation*}
On the one hand, since $\mathrm{supp}\,\eta_v \subset B(v, CR^{(1+\delta)/2})$, it is not difficult to show that 
\begin{equation}\label{concentration lemma 1}
|(f_{\theta,v})\;\widecheck{}\;(z)| = (1 + R^{-1/2}|z - v|)^{-(n+1)} \mathrm{RapDec}(R)\|f\|_{L^2(B^{n-1})}
\end{equation}
whenever $|z - v| \gtrsim R^{(1+\delta)/2}$. On the other hand, it is claimed that
\begin{equation}\label{concentration lemma 2}
\big|\big( \tilde{\tilde{\psi}}_{\theta} e^{2\pi i \phi^{\lambda}(y;\;\cdot)}\big)\;\widecheck{}\;(z)\big|=(1 + R^{-1/2}|z + \bar{v}(y;\omega_{\theta})|)^{-(n+1)}\textrm{RapDec}(R)  
\end{equation}
whenever $|z + \bar{v}(y;\omega_{\theta})| \gtrsim R^{(1+\delta)/2}$. A routine argument then shows that \eqref{concentration lemma 1} and \eqref{concentration lemma 2} imply the desired estimate for $\big((f_{\theta,v})\;\widetilde{}\;\big)\;\widecheck{}\;$ and so it only remains to prove the claim.

The inverse Fourier transform in \eqref{concentration lemma 2} can be expressed as 
\begin{equation*}
R^{-(n-1)/2} \int_{\R^{n-1}}e^{2\pi i (\langle z,\omega_{\theta}+R^{-1/2}\omega\rangle+\phi^{\lambda}(y;\omega_{\theta}+R^{-1/2}\omega))}\tilde{\tilde{\psi}}(\omega)\ud\omega,
\end{equation*}
where the $\omega$-gradient of the phase is given by
\begin{equation}\label{concentration lemma 3}
R^{-1/2}\big([\partial_{\omega}\phi^{\lambda}(y;\omega_{\theta}+R^{-1/2}\omega)-\partial_{\omega}\phi^{\lambda}(y;\omega_{\theta})]+[\bar{v}(y;\omega_{\theta})+z]\big).
\end{equation}
Using the fact that  $\partial^2_{\omega\omega}\phi^{\lambda}(0;\omega)=0$, the first term in \eqref{concentration lemma 3} can be estimated thus:
\begin{equation*}
|\partial_{\omega}\phi^{\lambda}(y;\omega_{\theta}+R^{-1/2}\omega)-\partial_{\omega}\phi^{\lambda}(y;\omega_{\theta})|\lesssim R^{-1/2}|y| \leq R^{1/2}.
\end{equation*}
Consequently, if $z\not\in B(-\bar{v}(y;\omega_{\theta}), R^{(1+\delta)/2})$, then the second term in \eqref{concentration lemma 3} dominates the $\omega$-gradient of the phase and  \eqref{concentration lemma 3} is $\gtrsim R^{\delta/2}$ in norm. Repeated integration-by-parts now implies \eqref{concentration lemma 1}, concluding the proof.
\end{proof}




\subsection{Tangency properties} 
Let $Z$ be a transverse complete intersection of dimension $m$ and suppose $h$ is a function which is concentrated on wave packets from 
\begin{equation*}
\T_{Z,B(y,\rho)}:=\big\{(\theta,v)\in \T_Z:\; T_{\theta,v}\cap B(y,\rho)\neq \emptyset \big\}.
\end{equation*}
What can be said about the scale $\rho$ wave packets of $\tilde{h}$? In particular, do the lower scale wave packets inherit the tangency property; namely, is $\tilde{h}$ concentrated on scale $\rho$ wave packets which are $\rho^{-1/2 + \delta_m}$-tangent to some variety? It transpires that this is \textbf{not} true in general. It is true, however, that $\tilde{h}$ can be broken up into pieces which are each made up of scale $\rho$ wave packets tangential to some translate of $Z$. 

In particular, while all the scale $\rho$ wave packets in question form very small angle with $Z-y$, they can be traced all the way up to distance $\sim R^{1/2+\delta_m}$ (rather than $\lesssim \rho^{1/2+\delta_m}$) from $Z-y$, which means that they generally live too far from $Z-y$ to be tangential to it at scale $\rho$. Translations of $Z-y$ however, up to distance $\sim R^{1/2+\delta_m}$, are tangential to such remote wave packets.

To make the above discussion precise, let $\tilde{\gamma}^{\lambda}_{\omega,v}$ be the curve defined by
\begin{equation}\label{core curve at scale rho}
\partial_\omega\tilde{\phi}^{\lambda}(\tilde{\gamma}^{\lambda}_{\omega, v}(t),t;\omega)=v \qquad \textrm{for $t \in (-\rho, \rho)$.}
\end{equation}
It is remarked that \eqref{pre curve identity} implies the relation
\begin{equation}\label{curve identity} 
\gamma^{\lambda}_{\omega, v} (t) =\tilde{\gamma}^{\lambda}_{\omega, v - \bar{v}(y;\omega)}(t-y_n)  + y'. 
\end{equation}
Let $\tilde{T}_{\omega,v}$ be the $\rho^{1/2+\delta}$-tube with core curve $\tilde{\Gamma}^{\lambda}_{\omega,v} = (\tilde{\gamma}^{\lambda}_{\omega,v}(t), t)$ (defined analogously to the $R^{1/2 + \delta}$-tube $T_{\theta,v}$) and for $b \in \R^n$ define
\begin{equation*}
\tilde{\T}_{Z+b}:=\big\{(\tilde{\theta},\tilde v) \in \tilde{\T} : \tilde{T}_{\tilde{\theta},\tilde v}\text{ is }\rho^{-1/2+\delta_m}\text{-tangent to }Z+b\text{ in }B(0,\rho)\big\}.
\end{equation*}
The key observation is as follows.

\begin{proposition}\label{tangential thin packets} Let $R^{1/2} \leq \rho \leq R^{1-\delta}$ and $Z\subseteq \R^n$ be a transverse complete intersection. 
\begin{enumerate}[1)]
\item Let $(\theta,v) \in \T_Z$ and $b\in B(0,2R^{1/2+\delta_m})$. If $(\tilde{\theta}, \tilde{v})\in \tilde{\T}_{\theta,v}$ satisfies 
\begin{equation*}
\tilde{T}_{\tilde{\theta},\tilde{v}} \cap N_{\frac{1}{2}\rho^{1/2+\delta_m}}( Z-y+b) \neq \emptyset,
\end{equation*}
 then $(\tilde{\theta}, \tilde{v}) \in \tilde{\T}_{Z-y+b}$. 
\item If $h$ is concentrated on wave packets in $\T_{Z,B(y,\rho)}$, then $\tilde{h}$ is concentrated on wave packets in $\bigcup_{|b| \lesssim R^{1/2+\delta_m}}\tilde{\T}_{Z-y+b}$.
\end{enumerate}
\end{proposition}
In view of the forthcoming analysis, before proving these statements a simple application is discussed. Under the hypotheses of the proposition, if one defines
\begin{equation*}
\tilde{\T}_b := \Big\{ (\tilde{\theta}, \tilde{v}) \in \!\!\!\! \bigcup_{(\theta, v) \in \T_{Z, B(y,\rho)}} \tilde{\T}_{\theta, v} : \tilde{T}_{\tilde{\theta},\tilde{v}} \cap N_{\frac{1}{2}\rho^{1/2+\delta_m}}( Z-y+b) \neq \emptyset \Big\},
\end{equation*}
then it follows that $\tilde{\T}_b \subseteq \tilde{\T}_{Z-y+b}$.  Given a function $h$ concentrated on wave packets in $\T_{Z,B(y,\rho)}$, consider a function of the form
\begin{equation}\label{b function definition}
\tilde{h}_b:=\sum_{(\tilde{\theta},\tilde{v}) \in \tilde{\T}_b}\tilde{h}_{\tilde{\theta},\tilde v}.
\end{equation}
Expressions of the form \eqref{b function definition} will play an important r\^ole in later arguments. Proposition~\ref{tangential thin packets} implies that
\begin{equation} \label{eq:relative} 
\tilde{T}^{\lambda} \tilde{h}_b\; (\tilde{x}) = T^{\lambda} h_b(x)\chi_{N_{\rho^{1/2+\delta_m}}(Z+b)}(x) + \mathrm{RapDec}(R)\|h\|_{L^2(B^{n-1})}
\end{equation}
for all $x=\tilde{x}+y \in B(y,\rho)$.



The proof of Proposition~\ref{tangential thin packets} relies on the following lemma.

\begin{lemma}\label{close curves lemma}  If $(\theta, v) \in \T$ and $(\tilde{\theta}, \tilde{v}) \in \tilde{\T}_{\theta, v}$, then 
\begin{equation*}
|\tilde{\Gamma}_{\tilde{\theta}, \tilde{v}}^{\lambda}(t)  - (\Gamma_{\theta, v}^{\lambda}(t+y_n) - y)| \lesssim R^{(1+\delta)/2} \qquad \textrm{for all $t \in (-\rho, \rho)$.}
\end{equation*}
\end{lemma}

\begin{proof} By the identity \eqref{curve identity} and the definition of $\tilde{\T}_{\theta, v}$, it suffices to show that if $(\omega_1, v_1), (\omega_2, v_2) \in \Omega \times \R^{n-1}$ satisfy $|\omega_1 - \omega_2| \lesssim \rho^{-1/2}$ and $|v_1 - v_2| \lesssim R^{(1+ \delta)/2}$, then  
\begin{equation*}
|\tilde{\gamma}_{\omega_1, v_1}^{\lambda}(t)  - \tilde{\gamma}_{\omega_2, v_2}^{\lambda}(t)| \lesssim R^{(1+\delta)/2} \qquad \textrm{for all $t \in (-\rho, \rho)$.}
\end{equation*}
 Fixing $t \in (-\rho, \rho)$, let $x_t:=(\tilde{\gamma}^{\lambda}_{\omega_1,v_1}(t),t)$ and  $v_t :=\partial_\omega\tilde{\phi}^{\lambda}(x_t;\omega_2)$ and note that, since the value of $\tilde{\gamma}_{\omega,v}(t)$ is uniquely determined by \eqref{core curve at scale rho}, $x_t= (\tilde{\gamma}_{\omega_2, v_t}^{\lambda}(t),t)$. Observe that
\begin{equation*}
|v_1-v_t|=|\partial_\omega\tilde{\phi}^{\lambda}(x_t;\omega_1)-\partial_\omega\tilde{\phi}^{\lambda}(x_t;\omega_2)|\lesssim |\omega_1 - \omega_2||x_t|\lesssim \rho^{1/2}.
\end{equation*}
Since $|v_1 - v_2| \lesssim R^{(1+\delta)/2}$, it follows that $|v_t-v_2|\lesssim R^{(1+\delta)/2}$. Therefore,
\begin{equation*}
|\tilde{\gamma}^{\lambda}_{\omega_1,v_1}(t)-\tilde{\gamma}^{\lambda}_{\omega_2,v_2}(t)|=|\tilde{\gamma}^{\lambda}_{\omega_2,v_t}(t)-\tilde{\gamma}^{\lambda}_{\omega_2,v_2}(t)|\sim |v_t-v_2|\lesssim R^{(1+\delta)/2},
\end{equation*}
which establishes the lemma. 
\end{proof}

One may now turn to the proof of Proposition~\ref{tangential thin packets}.

\begin{proof}[Proof (of Proposition~\ref{tangential thin packets})] The proof is broken into stages.

\subsubsection*{The angle condition.} Fix $(\theta, v) \in \T_Z$ and $(\tilde{\theta}, \tilde{v}) \in \tilde{\T}_{\theta, v}$. Motivated by the definition of tangency, let $x \in \tilde{T}_{\tilde{\theta}, \tilde{v}}$ and suppose $z \in Z$ and $b \in B(0,2R^{1/2+\delta_m})$ are such that $z - y + b \in B(0,4\rho)$ and $|x - (z - y + b)| \leq \bar{C}_{\mathrm{tang}} \rho^{1/2 + \delta_m}$. It is claimed that
\begin{equation}\label{tangential thin packets 1}
\angle(\tilde{G}^{\lambda}(x;\omega_{\tilde{\theta}}), T_{z - y+b}(Z - y +b)) \leq \bar{c}_{\mathrm{tang}}\rho^{-1/2+\delta_m},
\end{equation}
where $\tilde{G}^{\lambda}$ is the generalised Gauss map associated to the phase $\tilde{\phi}^{\lambda}$. It is easy to see that $\tilde{G}^{\lambda}(x; \omega) = G^{\lambda}(x+y;\omega)$ and $T_{z - y+b}(Z - y +b) = T_zZ$ so the above estimate can be written as
\begin{equation}\label{tangential thin packets 1.5}
\angle(G^{\lambda}(x+y;\omega_{\tilde{\theta}}), T_{z}Z) \leq \bar{c}_{\mathrm{tang}} \rho^{-1/2+\delta_m}.
\end{equation}
By Lemma~\ref{close curves lemma}, the definition of $\tilde{T}_{\tilde{\theta}, \tilde{v}}$ and the hypothesis $\rho \leq R^{1-\delta}$, it follows that 
\begin{equation}\label{tangential thin packets 2}
|x + y - \Gamma^{\lambda}_{\theta, v}(x_n +y_n)| \lesssim R^{(1+\delta)/2},
\end{equation}
which, by Lemma~\ref{Gauss Lipschitz}, implies that
\begin{equation*}
\angle(G^{\lambda}(x+y;\omega_{\tilde{\theta}}), T_{z}Z) \lesssim \angle(G^{\lambda}(\Gamma^{\lambda}_{\theta, v}(x_n+y_n); \omega_{\theta}), T_zZ) + \rho^{-1/2}.
\end{equation*}
Finally, $\Gamma^{\lambda}_{\theta, v}(x_n+y_n) \in T_{\theta, v}$ and this tube is $R^{-1/2 + \delta_m}$-tangent to $Z$. Note that $z \in Z \cap B(0,2R)$ whilst, recalling \eqref{tangential thin packets 2}, one has
\begin{equation*}
|\Gamma^{\lambda}_{\theta, v}(x_n+y_n) - z| \leq |x+ y - \Gamma^{\lambda}_{\theta, v}(x_n+y_n)| + |x - (z - y + b)| + |b| \lesssim R^{1/2 + \delta_m}.
\end{equation*}
Thus, if the constant $\bar{C}_{\mathrm{tang}}$ in Definition~\ref{tangent definition} is appropriately chosen, then the tangency of $T_{\theta, v}$ implies that
\begin{equation*}
\angle(G^{\lambda}(\Gamma^{\lambda}_{\theta, v}(x_n+y_n); \omega_{\theta}), T_zZ) \leq \bar{c}_{\mathrm{tang}} R^{-1/2 + \delta_m}
\end{equation*}
 and, provided that $R$ is sufficiently large, \eqref{tangential thin packets 1.5} (and therefore \eqref{tangential thin packets 1}) immediately follows. 

\subsubsection*{Containment properties.} The angle condition \eqref{tangential thin packets 1} implies that any tube $\tilde{T}_{\tilde{\theta}, \tilde{v}}$ which intersects $N_{\frac{1}{2}\rho^{1/2 + \delta_m}}(Z - y +b) \cap B(0, \rho)$ is actually contained in $N_{\rho^{1/2 + \delta_m}}(Z - y +b)$. To demonstrate this containment property, continue with the setup of the previous stage, but now assume the slightly stronger conditions that $z - y + b \in B(0,\rho)$ and $|x - (z-y+b)| \leq (1/2)\rho^{1/2+\delta_m}$. Define a time-dependent vector field $X_{\tilde{\theta}, \tilde{v}} \colon (-\rho,\rho) \times Z \cap B(0,2R) \to \R^m$ on $Z \cap B(0,2R)$ by 
\begin{equation*}
X_{\tilde{\theta}, \tilde{v}}(t,z) := \mathrm{proj}_{T_zZ}(\tilde{\Gamma}_{\tilde{\theta}, \tilde{v}}^{\lambda})'(t) \qquad \textrm{for all $(t, z) \in [-\rho, \rho] \times Z \cap B(0,2R)$.}
\end{equation*}
This can be smoothly extended to a map on $[-\rho, \rho] \times U$ where $U \subseteq \R^n$ is a small open neighbourhood of $Z \cap B(0,2R)$. By the Picard--Lindel\"of existence theorem for ODE there exists some smooth mapping $Z_{\tilde{\theta}, \tilde{v}} \colon (-\rho, \rho) \to Z$ such that $Z_{\tilde{\theta}, \tilde{v}}(x_n) = z$ and $Z_{\tilde{\theta}, \tilde{v}}'(t) = X_{\tilde{\theta}, \tilde{v}}(t, Z_{\tilde{\theta}, \tilde{v}}(t))$ for all $t \in (-\rho,\rho)$. Here $x = (x',x_n) \in \tilde{T}_{\tilde{\theta},\tilde{v}}$ are the points fixed above. 

Observe that
\begin{equation*}
|\tilde{\Gamma}_{\tilde{\theta}, \tilde{v}}^{\lambda}(x_n) - (z - y + b)| \leq |\tilde{\Gamma}_{\tilde{\theta}, \tilde{v}}^{\lambda}(x_n) - x| + |x -  (z - y + b)| < (2/3) \rho^{1/2+\delta_m}.
\end{equation*}
Let $I$ denote the set of all $t  \in (-\rho, \rho)$ such that $t \geq x_n$ and
\begin{equation*}
 |\tilde{\Gamma}_{\tilde{\theta}, \tilde{v}}^{\lambda}(s) - (Z_{\tilde{\theta},\tilde{v}}(s) - y + b)| \leq (9/10) \rho^{1/2 + \delta_m} \textrm{ for all $x_n \leq s \leq t$}.
\end{equation*} 
It is claimed that $t_* := \sup I = \rho$. To see this, first note that if $t_* < \rho$, then 
\begin{equation*}
(9/10)\rho^{1/2+\delta_m} = |\tilde{\Gamma}_{\tilde{\theta}, \tilde{v}}^{\lambda}(t_*) - (Z_{\tilde{\theta},\tilde{v}}(t_*)-y+b)|.
\end{equation*}
The angle condition \eqref{tangential thin packets 1} implies that 
\begin{equation*}
\angle((\tilde{\Gamma}_{\tilde{\theta}, \tilde{v}}^{\lambda})'(t), T_{Z_{\tilde{\theta}, \tilde{v}}(t)}Z) \leq \bar{c}_{\mathrm{tang}} \rho^{-1/2 + \delta_m} \quad \textrm{for all $x_n \leq t \leq t_*$.}
\end{equation*}
Combining the previous two displays with the identity
\begin{equation*}
\tilde{\Gamma}_{\tilde{\theta}, \tilde{v}}^{\lambda}(t_*) - Z_{\tilde{\theta},\tilde{v}}(t_*) = \int_{x_n}^{t_*} \mathrm{proj}_{(T_{Z_{\tilde{\theta}, \tilde{v}}(t)}Z)^{\perp}} (\tilde{\Gamma}_{\tilde{\theta}, \tilde{v}}^{\lambda})'(t)\,\ud t + (\tilde{\Gamma}_{\tilde{\theta}, \tilde{v}}^{\lambda}(x_n) -z),
\end{equation*} 
one concludes that
\begin{align*}
(9/10)\rho^{1/2+\delta_m}  &< \int_{x_n}^{t_*} \sin \angle\big((\tilde{\Gamma}_{\tilde{\theta}, \tilde{v}}^{\lambda})'(t), T_{Z_{\tilde{\theta}, \tilde{v}}(t)}Z\big)|(\tilde{\Gamma}_{\tilde{\theta}, \tilde{v}}^{\lambda})'(t)|\,\ud t +  (2/3)\rho^{1/2+\delta_m} \\
&\leq 2\bar{c}_{\mathrm{tang}}\rho^{-1/2 + \delta_m}|t_* - x_n| + (2/3) \rho^{1/2+\delta_m}. 
\end{align*}
Since $|t_* - x_n| \leq 2\rho$, if $\bar{c}_{\mathrm{tang}}$ is appropriately chosen, then this yields a contradiction and, consequently, $[x_n,  \rho) \subseteq I$. A similar argument shows that $(-\rho, x_n]\subseteq I$ and so $\tilde{\Gamma}_{\tilde{\theta}, \tilde{v}}^{\lambda}((-\rho, \rho)) \subseteq N_{\frac{9}{10}\rho^{1/2+\delta_m}}(Z-y +b)$. One therefore concludes that $\tilde{T}_{\tilde{\theta}, \tilde{v}} \subseteq N_{\rho^{1/2+\delta_m}}(Z-y +b)$.

\subsubsection*{Proof of Proposition~\ref{tangential thin packets}, 1).} Let $b \in B(0, 2R^{1/2+\delta_m})$ and suppose that $\tilde{T}_{\tilde{\theta}, \tilde{v}} \cap N_{\frac{1}{2}\rho^{1/2 + \delta}}(Z-y+b) \cap B(0, \rho) \neq \emptyset$; the problem is to show that $\tilde{T}_{\tilde{\theta}, \tilde{v}}$ is $\rho^{-1/2 + \delta_m}$-tangential to $Z-y+b$. By hypothesis, there exists some $x \in \tilde{T}_{\tilde{\theta}, \tilde{v}}$ and $z \in Z$ such that $z - y + b \in B(0,\rho)$ and $|x - (z - y + b)| \leq (1/2)\rho^{1/2 + \delta_m}$. The preceding analysis therefore implies that  $\tilde{T}_{\tilde{\theta}, \tilde{v}} \subseteq N_{\rho^{1/2+\delta_m}}(Z-y +b)$, which is the desired containment condition for tangency. On the other hand, the angle condition for tangency always holds by \eqref{tangential thin packets 1}, and so the proof of part 1) is complete. 

\subsubsection*{Proof of Proposition~\ref{tangential thin packets}, 2).}
By Lemma~\ref{concentration on smaller scale wave packets lemma} it suffices to prove that 
\begin{equation*}
\bigcup_{(\theta, v) \in \T_{Z, B(y, \rho)}} \tilde{\T}_{\theta, v} \subseteq \bigcup_{|b| \lesssim R^{1/2+\delta_m}} \tilde{\T}_{Z-y+b}.
\end{equation*}
Fixing $(\theta, v) \in \T_{Z, B(y, \rho)}$ and $(\tilde{\theta}, \tilde{v}) \in \tilde{\T}_{\theta, v}$, by \eqref{tangential thin packets 1} the problem is further reduced to showing that there exists some $|b| \lesssim R^{1/2+\delta_m}$ such that $\tilde{T}_{\tilde{\theta}, \tilde{v}} \subseteq N_{\rho^{1/2 + \delta_m}}(Z-y+b)$. Lemma~\ref{close curves lemma} implies that $\tilde{\Gamma}_{\tilde{\theta}, \tilde{v}}^{\lambda}(t) \in N_{CR^{1/2 + \delta_m}}(Z -y)$ for $t \in [-\rho, \rho]$. Consequently, fixing $t_0 \in [-\rho, \rho]$ there exists some $|b| \lesssim R^{1/2 + \delta_m}$ such that $\tilde{\Gamma}_{\tilde{\theta}, \tilde{v}}^{\lambda}(t_0) \in Z - y +b$. The desired inclusion now follows from the containment property discussed earlier in the proof.

\end{proof}




\subsection{Sorting the wave packets}

If $(\theta, v) \in \T$ and $(\tilde{\theta}, \tilde{v}) \in \tilde{\T}_{\theta, v} $, then Lemma~\ref{close curves lemma} implies that\footnote{Here $\mathrm{dist}_H$ denotes the Hausdorff distance.}
\begin{equation}\label{children close to mothers}
\mathrm{dist}_{H}\big(T_{\theta, v} \cap B(y,\rho), \tilde{T}_{\tilde{\theta}, \tilde{v}} + y) \lesssim R^{1/2+\delta}.
\end{equation}
In particular, if a pair of wave packets $(\theta_1, v_1), (\theta_2, v_2) \in \T$ are such that $\tilde{\T}_{\theta_1, v_1} \cap \tilde{\T}_{\theta_2, v_2} \neq \emptyset$, then the tubes $T_{\theta_1, v_1}$, $T_{\theta_2, v_2}$ are approximately equal on $B(y,\rho)$.\footnote{More precisely, enlarging the radius of either one of the $T_{\theta_j, v_j}$ by a constant factor produces a tube which contains $(T_{\theta_1, v_1} \cup T_{\theta_2, v_2}) \cap B(y, \rho)$.} This suggests sorting the scale $R$ wave packets $(\theta, v) \in \T$ into disjoint sets for which the associated tubes essentially agree on $B(y, \rho)$.

Let $\mathcal{T}$ denote the collection of all pairs $(\tilde{\theta}, w)$ formed by a $\rho^{-1/2}$-cap $\tilde{\theta}$ and $w \in R^{(1+\delta)/2}\Z^{n-1}$. For each $(\tilde{\theta}, w) \in \mathcal{T}$ choose some
\begin{equation*}
\mathcal{T}_{\tilde{\theta}, w} \subseteq \big\{(\theta,v) \in \T : \mathrm{dist}(\theta, \tilde{\theta}) \lesssim \rho^{-1/2}\text{ and }|v-\bar{v}(y; \omega_{\theta})-w|\lesssim R^{(1+\delta)/2}\big\}
\end{equation*}
so that the family $\{\mathcal{T}_{\tilde{\theta}, w} : (\tilde{\theta}, w) \in \mathcal{T}\}$ forms a covering of $\T$ by disjoint sets. Defining 
\begin{equation*}
T_{\tilde{\theta},w}:=\bigcup_{(\theta,v)\in \mathcal{T}_{\tilde{\theta},w}}T_{\theta,v}\cap B(y,\rho),
\end{equation*}
one obtains the following corollary to Lemma~\ref{close curves lemma}. 

\begin{corollary}\label{same fat intersection} If $(\tilde{\theta}, w) \in \mathcal{T}$ and $(\theta, v) \in \mathcal{T}_{\tilde{\theta}, w}$, then
\begin{equation*}
\mathrm{dist}_{H}\big(T_{\theta,v} \cap B(y,\rho), T_{\tilde{\theta}, w}) \lesssim R^{1/2+\delta}.
\end{equation*}
\end{corollary}



\begin{figure}
\begin{center}

\resizebox {0.7\textwidth} {!} {
\begin{tikzpicture}

\filldraw[blue!10,cm={cos(0-4) ,sin(0-4) ,-sin(0-4) ,cos(0-4) ,(-0.5 cm, 0.3 cm)}](0,1) circle (0.7cm);

\foreach \a in {-1, -0.5}
    {
 \filldraw[yellow!20,cm={cos(0+2*\a) ,-sin(0+2*\a) ,sin(0+2*\a) ,cos(0+2*\a) ,(0.5*\a cm, -0.3*\a cm)}](0.075,5) -- (0.075,-5) -- (-0.075,-5) -- (-0.075,5) -- (0.075,5);

}
\foreach \a in {-1, -0.5}
    {
 \filldraw[yellow!20,cm={cos(0+4*\a) ,sin(0+4*\a) ,-sin(0+4*\a) ,cos(0+4*\a) ,(0.5*\a cm, -0.3*\a cm)}](0.075,5) -- (0.075,-5) -- (-0.075,-5) -- (-0.075,5) -- (0.075,5);

}

\fill[fill=white] (4cm,0) arc [white,radius=4cm, start angle=0, delta angle=180]
                  -- (-5.4cm,0) arc [white, radius=5.4cm, start angle=180, delta angle=-180]
                  -- cycle;
\fill[fill=white] (4cm,0) arc [white,radius=4cm, start angle=0, delta angle=-180]
                  -- (-5cm,0) arc [white, radius=5cm, start angle=180, delta angle=180]
                  -- cycle;

\foreach \a in {-1, -0.5}
    {
 \draw[black,cm={cos(0+2*\a) ,-sin(0+2*\a) ,sin(0+2*\a) ,cos(0+2*\a) ,(0.5*\a cm, -0.3*\a cm)}] (0.075,5) -- (0.075,-5);
 \draw[black,cm={cos(0+2*\a) ,-sin(0+2*\a) ,sin(0+2*\a) ,cos(0+2*\a) ,(0.5*\a cm, -0.3*\a cm)}] (-0.075,-5) -- (-0.075,5);
}
\foreach \a in {-1, -0.5}
    {
 \draw[black,cm={cos(0+4*\a) ,sin(0+4*\a) ,-sin(0+4*\a) ,cos(0+4*\a) ,(0.5*\a cm, -0.3*\a cm)}] (0.075,5) -- (0.075,-5);
 \draw[black,cm={cos(0+4*\a) ,sin(0+4*\a) ,-sin(0+4*\a) ,cos(0+4*\a) ,(0.5*\a cm, -0.3*\a cm)}] (-0.075,-5) -- (-0.075,5);
}

\draw[cyan,thick,dashed] (0,0) circle (4cm);
\draw[blue, thick,cm={cos(0-4) ,sin(0-4) ,-sin(0-4) ,cos(0-4) ,(-0.5 cm, 0.3 cm)}](0,1) circle (0.7cm);
\filldraw[blue,cm={cos(0-4) ,sin(0-4) ,-sin(0-4) ,cos(0-4) ,(-0.5 cm, 0.3 cm)}](0,1) circle (1pt);

 \node[below=0.2cm, left, scale=1.5] at  (-0.7cm,-0.5) {$ T_{\tilde{\theta},w}$};
 \node[right, scale=1.2] at  (0.3cm,1.6cm) {$B(x_0, CR^{1/2+\delta})$};
\node[right, scale=1.5] at  (3.2cm,-3.2cm) {$B(y, \rho)$};
		\end{tikzpicture}}
		
\caption{The set $ T_{\tilde{\theta},w} := \bigcup_{(\theta,v)\in \mathcal{T}_{\tilde{\theta},w}}T_{\theta,v}\cap B(y,\rho)$ is highlighted in yellow. Fixing $x_0 \in T_{\tilde{\theta},w}$, for every $(\theta, v) \in \mathcal{T}_{\tilde{\theta},w}$ the tube $T_{\theta,v}$ intersects the ball $B(x_0, CR^{1/2+\delta})$.  }
\label{tube diagram}
\end{center}
\end{figure}
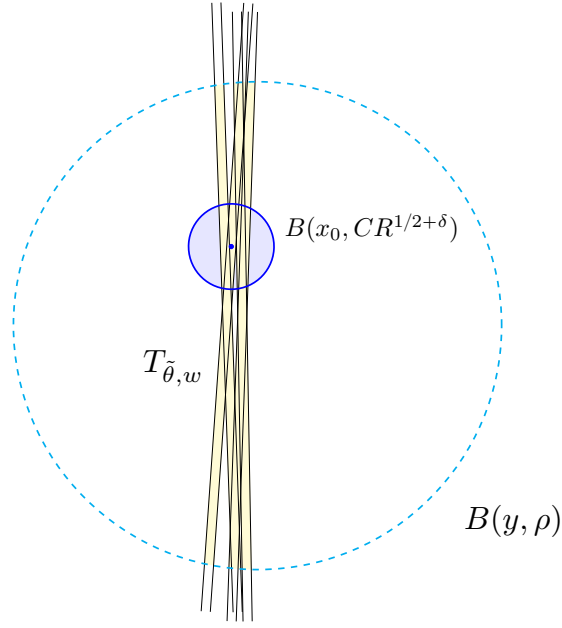


Let $g \colon B^{n-1} \to \C$ be integrable and define
\begin{equation*}
g_{\tilde{\theta},w}:=\sum_{(\theta,v)\in \mathcal{T}_{\tilde{\theta}, w}}g_{\theta,v} \qquad \textrm{for all $(\tilde{\theta},w) \in \mathcal{T}$.}
\end{equation*}
Since the $\mathcal{T}_{\tilde{\theta}, w}$ cover $\T$ and are disjoint, it follows that
\begin{equation}\label{sorting wave packets 1}
g=\sum_{(\tilde{\theta},w)\in \mathcal{T}}g_{\tilde{\theta},w} + \textrm{RapDec}(R)\|g\|_{L^2(B^{n-1})};
\end{equation}
furthermore, the functions $g_{\tilde{\theta},w}$ are almost orthogonal and, consequently,
\begin{equation}\label{sorting wave packets 2}
\|g\|_{L^2(B^{n-1})}^2\sim\sum_{(\tilde{\theta},w)\in \mathcal{T}}\|g_{\tilde{\theta},w}\|_{L^2(B^{n-1})}^2.
\end{equation}

By Lemma~\ref{concentration on smaller scale wave packets lemma}, the function $(g_{\tilde{\theta},w})\;\widetilde{}\;$ is concentrated on scale $\rho$ wave packets belonging to $\bigcup_{(\theta,v)\in \mathcal{T}_{\tilde{\theta}, w}} \tilde{\T}_{\theta, v}$. This union is contained in 
\begin{equation*}
\tilde{\mathcal{T}}_{\tilde{\theta},w} := \big\{(\tilde{\theta}',\tilde{v}) \in \tilde{\T}: \mathrm{dist}(\tilde{\theta}',\tilde{\theta})\lesssim \rho^{-1/2}\text{ and }|\tilde{v}-w|\lesssim R^{(1+\delta)/2}  \big\}
\end{equation*}
and therefore each $(g_{\tilde{\theta},w})\;\widetilde{}\;$ is concentrated on wave packets from $\tilde{\mathcal{T}}_{\tilde{\theta}, w}$. The family $\{\tilde{\mathcal{T}}_{\tilde{\theta}, w} : (\tilde{\theta}, w) \in \mathcal{T}\}$ forms a covering of $\tilde{\T}$ by almost disjoint sets. This implies almost orthogonality between the scale $\rho$ wave packets of the different functions $(g_{\tilde{\theta},w})\;\widetilde{}\;$. A particular consequence of this observation is that
\begin{equation}\label{sorting wave packets 3}
\big\|\sum_{(\tilde{\theta},w)\in \mathcal{T}}(g_{\tilde{\theta},w})\;\widetilde{}_{ \!\!b}\;\big\|_{L^2(B^{n-1})}^2\sim\sum_{(\tilde{\theta},w)\in \mathcal{T}}\|(g_{\tilde{\theta},w})\;\widetilde{}_{ \!\!b}\;\|_{L^2(B^{n-1})}^2,
\end{equation}
where $\tilde{h}_b$ is defined for a given function $h$ as in \eqref{b function definition}.




\subsection{Transverse equidistribution revisited}

Let $Z$ be an $m$-dimensional transverse complete intersection, $(\tilde{\theta},w) \in \mathcal{T}$ and $h$ be a function concentrated on wave packets in $\T_{Z\cap B(y,\rho)}\cap \mathcal{T}_{\tilde{\theta},w}$. Here the key example to have in mind is $h=g_{\tilde{\theta},w}$ for some function $g$ concentrated on wave packets in $\T_{Z, B(y,\rho)}$. 

Every scale $R$ wave packet of $h$ intersects $B(y,\rho)$ on the set $T_{\tilde{\theta},w}$ which, by Corollary~\ref{same fat intersection}, is comparable to $T_{\theta, v} \cap B(y,\rho)$ for any $(\theta,v) \in \mathcal{T}_{\tilde{\theta},w}$.\footnote{Here `the scale $R$ wave packets of $h$' refers to the scale $R$ wave packets upon which $h$ is concentrated.} Consequently, if $x_0 \in T_{\tilde{\theta},w}$, then all the scale $R$ wave packets of $h$ intersect $B(x_0, CR^{1/2+\delta_m})$ (see Figure~\ref{tube diagram}). Moreover, \eqref{children close to mothers} implies that the scale $\rho$ wave packets of $\tilde{h}$ will intersect $B(x_0-y, CR^{1/2+\delta_m})$. Under these conditions a useful reverse-type version of H\"ormander's $L^2$ bound holds.

\begin{lemma}\label{reverse Plancherel} Let $T^{\lambda}$ be a H\"ormander-type operator with phase $\phi^{\lambda}$ given by a translate of a reduced phase in the sense of \eqref{translated phase} and $1 \leq R^{1/2+\delta} \leq r \lesssim \lambda^{1/2}$. There exists a family of H\"ormander-type operators $\mathbf{T}^{\lambda}$ all with phase $\phi^{\lambda}$ such that the following hold:
\begin{enumerate}[i)]
    \item Each $T^{\lambda} \in \mathbf{T}^{\lambda}$ is again an operator with phase given by a translate of a reduced phase in the sense of \eqref{translated phase} (in particular, all the relevant bounds from \S\ref{Reductions section} hold on the support of the amplitude).
    \item $\# \mathbf{T}^{\lambda} = O(1)$.
    \item If $f$ is concentrated on wave packets (with respect to $T^{\lambda}$) which intersect some $B(\bar{x}, r) \subseteq B(0,R)$, then 
\begin{equation*}
 \|f\|_{L^2(B^{n-1})}^2\lesssim r^{-1} \|T_*^{\lambda} f\|^2_{L^2(B(\bar{x},Cr))} 
\end{equation*}
holds for some $T_*^{\lambda} \in \mathbf{T}^{\lambda}$.
\end{enumerate}
\end{lemma}

Lemma~\ref{reverse Plancherel} can be proven for extension operators fairly directly via Plancherel's theorem (see \cite[$\S$3]{Guth2018}). Establishing the general (variable coefficient) version of Lemma~\ref{reverse Plancherel} involves a number of additional technicalities and the proof is therefore postponed until the end of the section. 

For $h$ as above,  $x_0\in T_{\tilde{\theta},w}$ and $|b| \lesssim R^{1/2 + \delta_m}$ the preceding discussion implies that the function $\tilde{h}_b$, as defined in \eqref{b function definition}, is a sum of wave packets which intersect $B(x_0 - y, CR^{1/2 + \delta_m})$. Lemma~\ref{reverse Plancherel} can therefore be applied at scale $\rho$ with $r \sim R^{1/2 + \delta_m}$ to yield
\begin{equation*}
\|\tilde{h}_b\|_{L^2(B^{n-1})}^2 \lesssim R^{-1/2-\delta_m} \|\tilde{T}_*^{\lambda}\tilde{h}_b \|^2_{L^2(B(x_0-y,CR^{1/2+\delta_m}) )}.
\end{equation*}
The wave packets defined by $T^{\lambda}$ and $T_*^{\lambda}$ will have identical geometric properties (since these properties are essentially independent of the precise choice of amplitude). By \eqref{eq:relative}, one concludes that
\begin{equation}\label{eq:cakeslice}
\|\tilde{h}_b\|_{L^2(B^{n-1})}^2\lesssim R^{-1/2-\delta_m} \|T_*^{\lambda} h_b \|^2_{L^2(N_{\rho^{1/2+\delta_m}}(Z+b)\cap B(x_0,CR^{1/2+\delta_m}))}.
\end{equation}
This observation has several useful consequences. First of all, by applying H\"ormander's $L^2$ bound one obtains the following. 

\begin{lemma} \label{induceb} Let $h$ be concentrated on wave packets from $\T_{Z\cap B(y,\rho)} \cap \mathcal{T}_{\tilde{\theta},w}$, for some $(\tilde{\theta}, w) \in \mathcal{T}$. Let $\mathfrak{B}\subset B(0,CR^{1/2+\delta_m})$ be such that the sets $N_{\rho^{1/2+\delta_m}}(Z+b)\cap B(x_0,CR^{1/2+\delta_m})$ are pairwise disjoint over all $b\in \mathfrak{B}$. Then,
\begin{equation*}
\sum_{b\in\mathfrak{B}}\|\tilde{h}_b\|_{L^2(B^{n-1})}^2\lesssim \|h\|_{L^2(B^{n-1})}^2.
\end{equation*}
\end{lemma}

A further consequence of \eqref{eq:cakeslice} is the following transverse equidistribution result.

\begin{lemma}\label{equidistributedinputs1}  Let $(\tilde{\theta}, w) \in \mathcal{T}$, $|b| \lesssim R^{1/2+\delta_m}$ and $Z$ be an $m$-dimensional transverse complete intersection with $\overline{\deg}\,Z \lesssim_{\varepsilon} 1$. If $h$ is concentrated on wave packets from $\T_{Z\cap B(y,\rho)} \cap \mathcal{T}_{\tilde{\theta},w}$, then
\begin{equation*}
\|\tilde{h}_b\|_{L^2(B^{n-1})}^2\lesssim R^{O(\delta_m)}(\rho/R)^{(n-m)/2}\|h\|_{L^2(B^{n-1})}^2.
\end{equation*}
\end{lemma}

\begin{proof}  The transverse equidistribution estimate in Lemma~\ref{Transverse equidistribution lemma 1} implies that
\begin{equation*}
\|T_*^{\lambda} \tilde{h}_b \|^2_{L^2(N_{\rho^{1/2+\delta_m}}(Z+b)\cap B(x_0,CR^{1/2+\delta_m}))}\lesssim R^{1/2+O(\delta_m)}(\rho/R)^{(n-m)/2}\|h\|_{L^2(B^{n-1})}^2.
\end{equation*}
Combining this with \eqref{eq:cakeslice} completes the proof.
\end{proof}

Let $g$ be concentrated on wave packets of $\T_{Z, B(y, \rho)}$. For each $(\tilde{\theta},w) \in \mathcal{T}$ the function $g_{\tilde{\theta},w}$ is concentrated on wave packets in $\T_{Z\cap B(y,\rho)} \cap \mathcal{T}_{\tilde{\theta},w}$. It follows that Lemma~\ref{induceb} and Lemma~\ref{equidistributedinputs1} hold for $h=g_{\tilde{\theta},w}$. Combining the contributions from distinct $\mathcal{T}_{\tilde{\theta},w}$, one obtains the following.

\begin{lemma} \label{equidistributedinputs2} Let $|b| \lesssim R^{1/2+\delta_m}$ and $Z$ be an $m$-dimensional transverse complete intersection with $\overline{\deg}\,Z \lesssim_{\varepsilon} 1$. If $g$ is concentrated on wave packets from $\T_{Z,B(y,\rho)}$, then
\begin{equation*}
\|\tilde{g}_b\|_{L^2(B^{n-1})}\lesssim R^{O(\delta_m)}(\rho/R)^{(n-m)/4}\|g\|_{L^2(B^{n-1})}.
\end{equation*}
\end{lemma}

\begin{proof} By \eqref{sorting wave packets 1} and the linearity of the map $h \mapsto \tilde{h}_b$ it follows that
\begin{equation*}
\tilde{g}_b=\sum_{(\tilde{\theta},w)\in \mathcal{T} }(g_{\tilde{\theta},w})\;\widetilde{}_{\!\! b}\;+\textrm{RapDec}(R)\|g\|_{L^2(B^{n-1})}.
\end{equation*}
The almost orthogonality relation \eqref{sorting wave packets 3} between the $(g_{\tilde{\theta},w})\;\widetilde{}_{\!\! b}\;$ implies that
\begin{equation*}
\|\tilde{g}_b\|_{L^2(B^{n-1})}^2 \lesssim \sum_{(\tilde{\theta},w)\in \mathcal{T}}\|(g_{\tilde{\theta},w})\;\widetilde{}_{\!\! b}\;\|_{L^2(B^{n-1})}^2+ \textrm{RapDec}(R)\|g\|_{L^2(B^{n-1})}^2.
\end{equation*}
By Lemma~\ref{equidistributedinputs1}, the right-hand side of the above display is in turn dominated by 
\begin{equation*}
R^{O(\delta_m)}(\rho/R)^{(n-m)/2}\sum_{(\tilde{\theta},w)\in \mathcal{T}}\|g_{\tilde{\theta},w}\|_{L^2(B^{n-1})}^2+ \textrm{RapDec}(R)\|g\|_{L^2(B^{n-1})}^2.
\end{equation*}
An application of \eqref{sorting wave packets 2} yields the desired estimate.

\end{proof}




\subsection{The proof of the reverse $L^2$ bound}

\begin{proof}[Proof (of Lemma~\ref{reverse Plancherel})] One may assume without loss of generality that $\bar{x} = 0$ and $r = R^{1/2 + \delta}$. Indeed, the first reduction follows from the formula \eqref{translating the operator}, which can be used to replace $T^{\lambda}$ and $f$ with $\tilde{T}^{\lambda}$ and $\tilde{f}$, here defined with $y := \bar{x}$. Lemma~\ref{concentration on smaller scale wave packets lemma} and the identity \eqref{curve identity} imply that $\tilde{f}$ is concentrated on scale $R$ wave packets associated to $\tilde{T}^{\lambda}$ which intersect $B(0,r)$. For the second reduction, suppose the result is known for $r = R^{1/2+\delta}$ and let $R^{1/2+\delta} \leq r \leq R$ and $f$ be as in the statement of the theorem. For $a \in \R$ consider the slab 
\begin{equation*}
S_a:=\R^{n-1}\times [a-R^{1/2+\delta},a+R^{1/2+\delta}]\;\cap B(0,Cr)
\end{equation*}
where $C \geq 2$ is a constant, chosen large enough for the purposes of the argument. Cover $S_a$ with a collection $\{B_j\}_{j \in \mathcal{J}}$ of finitely-overlapping $R^{1/2+\delta}$-balls satisfying $B_j \cap S_a \neq \emptyset$ for all $j \in \mathcal{J}$. By the initial reductions, any tube $T_{\theta,v}$ makes a small angle with the $e_n$ direction and thus intersects at most $O(1)$ of these balls. Orthogonality of the wave packets together with the hypothesised estimate therefore imply that
\begin{equation}\label{reverse Plancherel -1}
\|g\|_{L^2(B^{n-1})}^2\lesssim R^{-1/2-\delta} \|T^{\lambda} g\|^2_{L^2(N_{CR^{1/2 + \delta}}(S_a))}
\end{equation}
for any $g$ concentrated on wave packets at scale $R$ which intersect $S_a$. If $(\theta,v) \in \T_{B(0,r)}$, then the aforementioned angle condition implies that $T_{\theta,v}$ intersects every slab $S_a$ for which $S_a \cap B(0,r) \neq \emptyset$. Hence, if $f$ is concentrated on wave packets from $\T_{B(0,r)}$, then one may sum \eqref{reverse Plancherel -1} over a collection of $\sim r/R^{1/2+\delta}$ slabs which cover $B(0,r)$ to obtain the desired result.

Fix a function $f$ satisfying the hypotheses of the lemma with $\bar{x} = 0$ and $r = R^{1/2+\delta}$ and note that
\begin{equation*}
\|f\|_{L^2(B^{n-1})} \sim \|\sum_{(\theta, v) : T_{\theta,v} \cap B(0, R^{1/2+\delta}) \neq \emptyset} f_{\theta, v}\|_{L^2(B^{n-1})}. 
\end{equation*} 

 Recall that the change of variables $u \mapsto \Psi(\bar{x}/\lambda; u)$, where $\Psi$ is the function introduced in $\S$4, reparametrises the surface $\{\partial_x\phi^\lambda(\bar{x};\omega):\omega\in\Omega\}$ as the graph of the function $\partial_{x_n}\phi^\lambda(\bar{x};\Psi(\bar{x}/\lambda;u))$, for any $\bar{x}\in X^\lambda$. With an abuse of notation for the sake of simplicity, let $\Psi$ denote the above change of variables for $\bar{x}=0$; that is, $\Psi(u) := \Psi(0;u)$. Thus, $\Psi \colon U \to \Omega$ is a diffeomorphism that reparametrises the surface $\{\partial_{x}\phi^{\lambda}(0;\omega) : \omega \in \Omega\}$ as the graph of the function $\bar{h}(u) := \partial_{x_n}\phi^{\lambda}(0;\Psi(u))$; in particular, $\partial_{x'}\phi^{\lambda}\left(0;\Psi(u)\right)=u$ for all $u \in U$. 
 
Applying the change of variables $u \mapsto \Psi(u)$, denoting by $J_{\Psi}$ the absolute value of the determinant of the corresponding Jacobian matrix and recalling that $J_{\Psi}\sim 1$ by \eqref{Psi bound 1}, one obtains 
 \begin{equation*}
  \|g\|_{L^2(B^{n-1})} \sim \|g\circ \Psi\cdot J_\Psi^{1/2}\|_{L^2(B^{n-1})}\sim\|g\circ \Psi\cdot J_\Psi\|_{L^2(B^{n-1})}     
 \end{equation*}
for all $g\in L^2(B^{n-1})$. It follows that 
\begin{equation*}
    \|f\|_{L^2(B^{n-1})} \sim \|f_{\Psi}\|_{L^2(B^{n-1})}
\end{equation*}
where
\begin{equation*}
f_{\Psi} := \big(\sum_{(\theta, v) : T_{\theta,v} \cap B(0, R^{1/2+\delta}) \neq \emptyset} f_{\theta, v}\circ \Psi \big) \cdot J_{\Psi}.
\end{equation*}

Let $E$ denote the extension operator
\begin{equation*}
Eg(x) := \int_{U} e^{2 \pi i (\langle x', u \rangle + x_n\bar{h}(u))} g(u)\,\ud u 
\end{equation*}
associated to the graph $u \mapsto (u, \bar{h}(u))$. For any $x_n \in \R$ and $g$ a square integrable function supported on $U$, Plancherel's theorem implies that
\begin{equation*}
\|g\|_{L^2(B^{n-1})}^2 =\|e^{2\pi i x_n \bar{h}}g\|_{L^2(B^{n-1})}^2 = \|Eg\|_{L^2(\R^{n-1}\times \{x_n\})}^2.
\end{equation*}
Hence, averaging this estimate over $|x_n| < R^{1/2+\delta}$ one obtains
\begin{equation}\label{reverse Plancherel 1}
\|g\|_{L^2(B^{n-1})}^2 \sim R^{-1/2 - \delta} \|Eg\|_{L^2(\R^{n-1}\times (-R^{1/2+\delta}, R^{1/2+\delta}))}^2.
\end{equation}
The key observation is that the hypothesis on $f$ implies that the right-hand $L^2$-norm can be localised.

\begin{claim} If $|x_n| < R^{1/2 + \delta}$ and $x' \notin B(0,CR^{1/2 + \delta})$, then
\begin{equation*}
Ef_{\Psi}(x) = (1+R^{-1/2}|x'|)^{-(n+1)}\mathrm{RapDec}(R)\|f\|_{L^2(B^{n-1})}.
\end{equation*}
\end{claim}

This concentration property is very similar to that detailed in Lemma~\ref{wave packet concentration lemma}, the main difference being that the condition $(\theta, v) \in \T_{B(0, R^{1/2+\delta})}$ is defined with respect to the operator $T^{\lambda}$ whilst the above identity concerns the linearised version $E$. The proof is a minor  adaptation of the stationary phase analysis used to establish Lemma ~\ref{wave packet concentration lemma} and is therefore omitted. 

For the specific choice of function $g = f_{\Psi}$, the claim implies that \eqref{reverse Plancherel 1} may be strengthened to
\begin{equation}\label{reverse Plancherel 2}
 \|f\|_{L^2(B^{n-1})}^2 \sim  \|f_{\Psi}\|_{L^2(B^{n-1})}^2 \sim R^{-1/2 - \delta} \|Ef_{\Psi}\|_{L^2(B(0, CR^{1/2 + \delta}))}^2.
\end{equation}
This is easily seen to imply the lemma. Indeed, reversing the earlier change of variables,
\begin{equation*}
Ef_{\Psi}(z) = \int_{\Omega} e^{2\pi i \phi^{\lambda}(z;\omega) }e^{-2\pi i \lambda\mathcal{E}(z/\lambda;\omega)}f(\omega)\ud\omega + \mathrm{RapDec}(R)\|f\|_{L^2(B^{n-1})}
\end{equation*}
for all $z\in B(0, CR^{1/2+\delta})$, where $\mathcal{E}$ is the error term in the Taylor expansion
\begin{equation*}
\phi (z;\omega)=\phi(0;\omega)+\langle \partial_x \phi(0;\omega),z\rangle +\mathcal{E}(z;\omega).
\end{equation*}

Were it not for the factor $e^{-2\pi i \lambda\mathcal{E}(z/\lambda;\omega)}$ the functions $Ef_{\Psi}(z)$ and $T^{\lambda}f(z)$ would be equal up to a negligible error term and \eqref{reverse Plancherel 2} would directly imply the desired estimate. This unwanted additional factor can be removed via a Fourier series decomposition. 

More precisely, it holds by the integral form of the remainder that
\begin{equation}\label{integral form remainder}
    \partial_{\omega}^{\beta}\mathcal{E}(z;\omega) = \sum_{|\gamma| = 2} \frac{2}{\alpha!} \int_0^1 (1-t) \partial_x^{\gamma}\partial_{\omega}^{\beta}\phi(tz; \omega)\,\ud t z^{\gamma} \qquad \textrm{for all $\beta \in \N_0^{n-1}$}.
\end{equation}
Applying the uniform bounds on the derivatives of the reduced phase function $\phi$ as described in \S\ref{Reductions section} and recalling the hypothesis $R^{1/2 + \delta} \lesssim \lambda^{1/2}$, \eqref{integral form remainder} implies that
\begin{equation}\label{reverse Plancherel 3}
|\partial_{\omega}^{\beta} \partial_{z}^{\alpha} \lambda \mathcal{E}(R^{1/2 + \delta}z/\lambda;\omega)| \lesssim c_{\mathrm{par}} \quad \textrm{for $0\leq |\alpha|, |\beta|\leq N_{par}$ and $|z| \lesssim 1 $.}
\end{equation}
 Let $\psi \in C^{\infty}_c(\R^n \times \R^{n-1})$ be supported on $X \times \Omega$ and equal to 1 on $\mathrm{supp}\,a$. By forming the Fourier series expansion in both the $x$ and $\omega$ variables one obtains 
\begin{equation*}
e^{-2\pi i \lambda\mathcal{E}(z/\lambda;\omega)}\psi(z/R^{1/2 +\delta};\omega) =\sum_{\substack{k \in \mathbb{Z}^n\times \mathbb{Z}^{n-1} \\ k = (k_1,k_2)}}(1 + |k|)^{-2n} c_k e^{2\pi i (\langle z/R^{1/2+\delta} , k_1 \rangle + \langle \omega , k_2 \rangle)},
  \end{equation*}
where the $c_k$ are weighted Fourier coefficients. Observe that \eqref{reverse Plancherel 3} implies that $|c_k| \lesssim 1$ for all $k \in \Z^n \times \Z^{n-1}$. Thus, \eqref{reverse Plancherel 2} now yields
\begin{equation*}
\|f\|_{L^2(B^{n-1})} \lesssim R^{-(1/4+\delta/2)}\sum_{\substack{k \in \mathbb{Z}^n\times \mathbb{Z}^{n-1} \\ k = (k_1,k_2)}}(1+|k|)^{-2n}\big\|T^{\lambda}(e^{2\pi i \langle \,\cdot\,,k_2 \rangle}f )\big\|_{L^2(B(0,CR^{1/2+\delta}))}.
\end{equation*}
The above sum is split into a sum over $k = (k_1, k_2)$ satisfying $|k| > C_n$ and a sum over the remaining $k$ where $C_n$ is a dimensional constant, chosen large enough for the present purpose. The sum over large $k$ is bounded above by
\begin{equation}\label{reverse Plancherel 4}
    A\sum_{|k|> C_n}(1+|k|)^{-2n}\|f\|_{L^2(B^{n-1})}
\end{equation}
for some dimensional constant in $A$; indeed, this follows by applying H\"ormander's $L^2$-estimate from Lemma~\ref{Hormander L2} to each of the summands (the constant in H\"ormander's theorem can be made uniform over the class of reduced phases). By choosing $C_n$ to be sufficiently large so that $A \sum_{|k|> C_n}(1+|k|)^{-2n} \leq \frac{1}{2}$, the term \eqref{reverse Plancherel 4} can be absorbed into the left-hand side of the inequality. Thus, one obtains
\begin{equation*}
\|f\|_{L^2(B^{n-1})} \lesssim R^{-(1/4 + \delta/2)} \sum_{\substack{k \in \Z^{n-1} \\|k| \leq C_n}} \|T^{\lambda} (e^{2\pi i \langle \,\cdot\,,k \rangle} f ) \|_{L^2(B(0,CR^{1/2+\delta}))}.
\end{equation*}
Finally, define the class of operators $\widebar{\mathbf{T}}^{\lambda} := \{T^{\lambda}_k : k \in \Z^{n-1}, |k| \leq C_n \}$ where $T^{\lambda}_k$ has phase $\phi^\lambda$ and amplitude $a_k^\lambda$, for
\begin{equation*}
a_k(z;\omega) := \psi(z;\omega)e^{2 \pi i \langle \omega, k \rangle}.  
\end{equation*}
It is easy to see that each such amplitude $a_k$ can be written as a linear combination of $O(1)$ amplitudes satisfying the conditions of Lemma \ref{third reduction lemma}, with complex coefficients of order of magnitude $O(1)$. Defining $\mathbf{T}^\lambda$ to be the union, over all $|k|\leq C_n$, of the operators with phase $\phi^\lambda$ and the corresponding rescaled amplitudes, it follows by pigeonholing that there exists at least one operator $T_*^{\lambda} \in \mathbf{T}^{\lambda}$ for which the desired inequality holds.

\end{proof}
 



\section{Proof of the $k$-broad estimate}\label{Proof section}




\subsection{A more general result} In this section the proof of the $k$-broad estimate in Theorem~\ref{k-broad theorem} is given. In order to facilitate an inductive argument, a more general result will be established, which is described presently. 

Throughout this section $T^{\lambda}$ denotes an arbitrary choice of a translate of a H\"ormander-type operator with reduced positive-definite phase. That is, $T^{\lambda}$ is of the form of the operator $\tilde{T}^{\lambda}$ discussed in the previous section, with phase and amplitude of the type described in \eqref{translated phase}. Many of the estimates involving $T^{\lambda}$, such as \eqref{equation: k-linear for m general} below, are understood to hold uniformly for the entire class of such operators; it is important to work with the whole class rather than a single choice of $T^{\lambda}$ in order to run certain induction arguments. 

In order to state the result, first define the exponent 
\begin{equation*}
e_{k,n}(p) := \frac{1}{2}\Big(\frac{1}{2}-\frac{1}{p}\Big)(n+k)
\end{equation*}
and note that $-e_{k,n}(p)+1/2\leq 0$ if and only if $p \geq \bar{p}(k,n)$ where $\bar{p}(k,n)=2(n+k)/(n+k-2)$ is the exponent appearing in Theorem~\ref{k-broad theorem}.

\begin{proposition}\label{k-linear for m proposition} Given $\varepsilon>0$ sufficiently small and $1\leq m\leq n$ there exist 
\begin{equation*}
0<\delta\ll\delta_{n-1}\ll \delta_{n-2} \ll \ldots \ll \delta_1\ll \varepsilon
\end{equation*}
and constants $\bar{C}_\varepsilon$, ${\bar A}_\varepsilon$ dyadic, $D_{m,\varepsilon} \lesssim_{\varepsilon} 1$ and $\vartheta_m < \varepsilon$ such that the following holds.

Suppose $Z=Z(P_1,\ldots,P_{n-m})$ is a transverse complete intersection with $\overline{\deg}\, Z\leq D_{m,\varepsilon}$. For all $2 \leq k \leq n$, $1\leq A\leq {\bar A}_\varepsilon$ dyadic and $1\leq K\leq R \leq \lambda$ the inequality 
\begin{equation}\label{equation: k-linear for m general} 
\|T^{\lambda} f\|_{\mathrm{BL}^p_{k,A}(B(0,R))}\lesssim_{\varepsilon} K^{\bar{C}_{\varepsilon}} R^{\vartheta_m + \delta\left(\log {\bar{A}}_\varepsilon - \log A\right)-e_{k,n}(p)+1/2}\|f\|_{L^2(B^{n-1})}
\end{equation}
holds whenever $f$ is concentrated on wave packets from $\T_Z$ and
\begin{equation*}
2\leq p \leq \bar{p}_0(k,m) := \left\{
\begin{array}{ll}
\bar{p}(k,m)          & \textrm{if $k < m$} \\
\bar{p}(m,m)+\delta   & \textrm{if $k = m$}
\end{array} \right. .
\end{equation*}
\end{proposition}

Here $\T_Z$ is defined as in $\S$\ref{Transverse equidistribution section}; that is,
\begin{equation*}
\mathbb{T}_Z := \{(\theta,v) \in \T :\; T_{\theta,v}\text{ is }R^{-1/2+\delta_m}\text{-tangent to } Z\text{ in }B(0,R)\}.
\end{equation*}

Proposition~\ref{k-linear for m proposition} immediately yields the desired $k$-broad estimate. 

\begin{proof}[Proof (of Theorem~\ref{k-broad theorem})] Theorem~\ref{k-broad theorem} is a special case of Proposition~\ref{k-linear for m proposition}. Indeed, $Z=\R^n$ is a transverse complete intersection of dimension $n$ and $\mathbb{T}_Z$ contains all wave packets in $B(0,R)$. Thus, taking $A={\bar A}_\varepsilon$ and $p=\bar{p}(k,n)$ yields the endpoint case of Theorem~\ref{k-broad theorem}. The general result follows by interpolating with the trivial $p = \infty$ estimate via the logarithmic convexity of the $k$-broad norms (that is, Lemma~\ref{logarithmic convexity inequality lemma}).
\end{proof}




\subsection{Reducing to $R\lesssim \lambda^{1-\varepsilon}$.} Turning to the proof of Proposition~\ref{k-linear for m proposition}, the first step is a technical reduction on the radii $R$ which is needed to facilitate certain polynomial partitioning arguments. In particular, it will be necessary to approximate the curves $\Gamma^\lambda_{\theta,v}$ by polynomial curves of degree independent of $R$; by the observations of $\S$\ref{polynomial approximation subsection}, this is possible when $1\leq R\lesssim_\varepsilon \lambda^{1-\varepsilon}$ and thus \eqref{equation: k-linear for m general} will first be proved for this restricted range of $R$. The result can then be extended to the full $R \leq \lambda$ range by a triangle inequality argument (incurring a permissible $R^{O(\varepsilon)}$ loss). The concentration hypothesis on $f$ in Proposition~\ref{k-linear for m proposition} creates some difficulties here, which are addressed by the following lemma. 

\begin{lemma}\label{for rough induction} Let $1 \leq \rho \leq R \leq \lambda$ and assume that for any transverse complete intersection  $Z=Z(P_1,\ldots,P_{n-m})$ with $\overline{\deg}\, Z \leq D$ the inequality
\begin{equation}\label{for rough induction equation}
\| T^{\lambda} f\|_{\mathrm{BL}^{p}_{k,A}(B(0,\rho))}\leq E\|f\|_{L^2(B^{n-1})}
\end{equation}
holds with some constant $E > 0$ whenever $f$ is concentrated on wave packets from $\tilde{\mathbb{T}}_{Z}$. Then for all $Z$ as above, the inequality
\begin{equation*}
\| T^{\lambda} f\|_{\mathrm{BL}^{p}_{k,A}(B(0,R))}\lesssim (R/\rho)^{O(1)}E \|f\|_{L^2(B^{n-1})}
\end{equation*}
holds for all functions $f$ concentrated on wave packets in $\mathbb{T}_{Z}$.
\end{lemma}

Here $\tilde{\mathbb{T}}_{Z}$ denotes the collection of wave packets at scale $\rho$ that are $\rho^{1/2 + \delta_m}$-transverse to $Z$ in $B(0,\rho)$; this notation is consistent with that used in $\S$\ref{Adjusting wave packets section}.

\begin{proof}[Proof (of Lemma~\ref{for rough induction})] Let $f$ be a function concentrated on wave packets in $\mathbb{T}_{Z}$ for some transverse complete intersection $Z$ as in the statement of the lemma. Fix a cover $\mathcal{B}_{\rho}$ of $B(0,R)$ by finitely overlapping $\rho$-balls. By the sub-additivity of the $k$-broad norms and Lemma~\ref{wave packet concentration lemma}, there exists some $B = B(y,\rho) \in \mathcal{B}_{\rho}$ such that
\begin{equation*}
\| T^{\lambda} f\|_{\mathrm{BL}^{p}_{k,A}(B(0,R))}^p\lesssim (R/\rho)^{O(1)}\|T^{\lambda} h\|_{\mathrm{BL}^{p}_{k,A}(B)}^p + \mathrm{RapDec}(R)\|f\|_{L^2(B^{n-1})}^p,
\end{equation*}
where $h:=\sum_{(\theta,v)\in \mathbb{T}_{Z,B}}f_{\theta,v}$; here and below the notation $\T_{Z,B}$ is consistent with that used in $\S$\ref{Transverse equidistribution section}. As in $\S$\ref{Adjusting wave packets section}, write $T^{\lambda}h(x+y) = \tilde{T}^{\lambda}\tilde{h}(x)$ so that, suppressing the harmless rapidly decaying term in the notation, one has\footnote{\label{translation footnote} Strictly speaking, in order for \eqref{for rough induction 1} to hold the $k$-broad norm on the right-hand side should be defined with respect to a translate of the family of balls $\mathcal{B}_{K^2}$. Since the estimates will be uniform over all choices of families $\mathcal{B}_{K^2}$ of bounded multiplicity, this slight technicality does not affect the argument.}
\begin{equation}\label{for rough induction 1}
\| T^{\lambda} f\|_{\mathrm{BL}^{p}_{k,A}(B(0,R))}^p\lesssim (R/\rho)^{O(1)}\|\tilde{T}^{\lambda} \tilde{h}\|_{\mathrm{BL}^{p}_{k,A}(B(0,\rho))}.
\end{equation}

In general, $\tilde{h}$ is not concentrated on wave packets which are $\rho^{-1/2+\delta_m}$-tangential to a suitable variety inside $B(0,\rho)$; thus, hypothesis \eqref{for rough induction equation} cannot be applied directly to estimate the right-hand side of \eqref{for rough induction 1}. Rather, one approaches the problem via the methods of $\S$\ref{Adjusting wave packets section}. By Proposition~\ref{tangential thin packets}, $\tilde{h}$ is concentrated on wave packets in $\bigcup_{b \in \mathfrak{B}}\tilde{\mathbb{T}}_{Z-y+b}$ where $\mathfrak{B}$ is a discrete set of cardinality $\lesssim (R/\rho)^{O(1)}$ such that $|b| \lesssim R^{1/2 + \delta_m}$ for all $b \in \mathfrak{B}$. Consequently, by the sub-additivity of the $k$-broad norms and Lemma~\ref{wave packet concentration lemma} and Proposition~\ref{tangential thin packets}, there exists some $b \in \mathfrak{B}$ such that
\begin{equation}
\label{for rough induction 2}
\|\tilde{T}^{\lambda}\tilde{h}\|_{\mathrm{BL}^{p}_{k,A}(B(0,\rho))}^p \lesssim  (R/\rho)^{O(1)}\|\tilde{T}^{\lambda}\tilde{h}_b\|_{\mathrm{BL}^{p}_{k,A}(B(0,\rho))}^p,
\end{equation}
for $\tilde{h}_b$ as defined in $\S$\ref{Adjusting wave packets section}. Recall from Proposition~\ref{tangential thin packets} that $\tilde{h}_b$ is concentrated on wave packets in $\tilde{\mathbb{T}}_{Z-y+b}$ and satisfies 
\begin{equation}\label{for rough induction 3}
\|\tilde{h}_b\|_{L^2(B^{n-1})} \lesssim \|f\|_{L^2(B^{n-1})}.
\end{equation}
Combining \eqref{for rough induction 1} and \eqref{for rough induction 2}, the desired estimate now follows by applying hypothesis \eqref{for rough induction equation} to the function $\tilde{h}_b$ and then using \eqref{for rough induction 3} to bound the resulting right-hand expression.
\end{proof}




\subsection{Setting up the induction argument}

Henceforth it is assumed that $1\leq R\lesssim_\varepsilon \lambda^{1-\varepsilon}$. Under this hypothesis, given $\varepsilon > 0$ sufficiently small, Proposition~\ref{k-linear for m proposition} will be established for the following choice of parameters:
\begin{gather}\label{parameter values}
D_{m,\varepsilon}:=\varepsilon^{-\delta^{-(2n-m)}}, \qquad \vartheta(\varepsilon) := \varepsilon - c_n\delta_m, \qquad \bar{A}_{\varepsilon} := \ceil{e^{10n/\delta}}, \\
\nonumber
\delta_i = \delta_i(\varepsilon):=\varepsilon^{2i+1} \quad \textrm{for all $i=1,\ldots, n-1$.}
\end{gather}
Here $0<\delta = \delta(\varepsilon) \ll \delta_{n-1}(\varepsilon)$ and $c_n >0$ is a fixed dimensional constant.

The proof is by induction on the radius $R$ and the dimension $m$; presently the base cases for this induction are established. 




\subsubsection*{Base case for the radius: $R\lesssim_\varepsilon K^n$} Provided that the implied constant in \eqref{equation: k-linear for m general} and $C_\varepsilon$ are chosen to be sufficiently large, in this case Proposition~\ref{k-linear for m proposition} follows immediately from the trivial inequality
\begin{equation}\label{base case 1}
\|T^{\lambda} f \|_{\mathrm{BL}^p_{k,A}(B(0,R))}\lesssim R^{n/p} \|f\|_{L^2(B^{n-1})},
\end{equation}
valid for all $A \in \N$ and $1 \leq p \leq \infty$.




\subsubsection*{Base case for the dimension: $m\leq k-1$} Assuming (without loss of generality) that $K\lesssim_\varepsilon R^{1/n}$, one can show in this case that
\begin{equation}\label{base case 2}
\|T^{\lambda} f\|_{\mathrm{BL}^p_{k,A}(B(0,R))}= \mathrm{RapDec}(R) \|f\|_{L^2(B^{n-1})}.
\end{equation}
Indeed, fix a ball $B_{K^2}\in \mathcal{B}_{K^2}$ with $B_{K^2}\cap B(0,R)\neq \emptyset$; here $\mathcal{B}_{K^2}$ denotes the collection of $K^2$-balls featured in the definition of the $k$-broad norm \eqref{k-broad norm definition}. Let $\mathfrak{T}_{B_{K^2}}$ denote the collection of all $K^{-1}$-caps $\tau$ for which there exists some $(\theta,v) \in \T_{Z,B_{K^2}}$ with $\theta \cap \tau \neq \emptyset$. Observe that if $\tau \notin \mathfrak{T}_{B_{K^2}}$, then  
\begin{equation}\label{base case 3}
\int_{B_{K^2}}|T^{\lambda} f_\tau|^p =\mathrm{RapDec}(R)\|f\|_{L^2(B^{n-1})}^p,
\end{equation}
since $f$ is concentrated on wave packets in $\mathbb{T}_Z$.

It is claimed that there exists some $V\in {\rm Gr}(k-1,n)$, such that 
\begin{equation}\label{base case 4}
\angle(G^{\lambda}(\bar{x};\tau), V)\leq K^{-1} \qquad \textrm{for all $\tau \in \mathfrak{T}_{B_{K^2}}$,}
\end{equation}
where $\bar{x} \in \R^n$ denotes the centre of $B_{K^2}$. Indeed, by \eqref{base case 3} one may assume  without loss of generality that $\mathfrak{T}_{B_{K^2}}\neq \emptyset$ and hence $\T_{Z,B_{K^2}} \neq \emptyset$. Thus, there exists $z\in Z\cap B(0,R)$ with $|z-\bar{x}|\lesssim R^{1/2+\delta_m}$ and, taking $V\in {\rm Gr}(k-1,n)$ to be any subspace that contains $T_zZ$, the claim is easily deduced from the definition of $R^{1/2 + \delta_m}$-tangency (see Definition~\ref{tangent definition}), together with the hypothesis $K\lesssim R^{1/n}$. 

Recalling the definition of $\mu_{T^{\lambda}f}(B_{K^2})$ from \eqref{mu definition}, it follows from \eqref{base case 4} that 
\begin{equation*}
\mu_{T^{\lambda} f}(B_{K^2}) \leq \max_{\tau \notin V }\int_{B_{K^2}}|T^{\lambda} f_\tau |^p \leq \max_{\tau\notin \mathfrak{T}_{B_{K^2}}}\int_{B_{K^2}}|T^{\lambda} f_\tau |^p
\end{equation*}
and the desired estimate \eqref{base case 2} is now a consequence of \eqref{base case 3}.




\subsubsection*{Reduction to $A \geq 2$:} Recall that $\bar{A}_{\varepsilon} \geq e^{10n/\delta}$ so that $\delta\log \bar{A}_{\varepsilon} \geq 10n$. Although the argument does not require one to induct on $A$, it is useful to note that \eqref{base case 1} implies that Proposition~\ref{k-linear for m proposition} holds for $A=1$. This allows one to assume $A \geq 2$ throughout the following argument, and therefore permits the use of the $k$-broad triangle and logarithmic-convexity inequalities.




\subsection{An overview of the inductive step} Let $2\leq k\leq n-1$,  $k\leq m\leq n$ and $R\gtrsim_\varepsilon K^n$. Assume, by way of induction hypothesis, that \eqref{equation: k-linear for m general} holds whenever the dimension of the transverse complete intersection $Z$ is at most $m-1$ or the radial parameter is at most $R/2$. 

Fix $\varepsilon >0$, $1 < A \leq \bar{A}_\varepsilon$ and a transverse complete intersection  $Z=Z(P_1,\ldots,P_{n-m})$ with $\overline{\deg}\, Z \leq D_{m,\varepsilon}$, where the parameters $\bar{A}_{\varepsilon}$ and $D_{m, \varepsilon}$ are as defined in \eqref{parameter values}. Let $f$ be concentrated on wave packets from $\T_Z$. 

It suffices to show that \eqref{equation: k-linear for m general} holds at the endpoint $p=\bar{p}_0(k,m)$. Indeed, observe that Lemma~\ref{Hormander L2} implies the $L^2$-bound
\begin{equation*}
\|T^{\lambda} f\|_{\mathrm{BL}^2_{k,A}(B(0,R))}^2 \lesssim \sum_{\tau : K^{-1}-\mathrm{cap}} \|T^{\lambda} f_{\tau}\|_{L^2(B(0,R))}^2 \lesssim R\|f\|_{L^2(B^{n-1})}^2.
\end{equation*}
Once \eqref{equation: k-linear for m general} is established for $p = \bar{p}_0(k,m)$, one may use the logarithmic convexity of the $k$-broad norms to interpolate the $p = \bar{p}_0(k,m)$ estimate against the above inequality and thereby obtain \eqref{equation: k-linear for m general} in the desired range.
 
The analysis proceeds by considering two different cases.

\subsubsection*{The algebraic case} There exists a transverse complete intersection $Y^l \subseteq Z$ of dimension $1 \leq l \leq m-1$ of maximum degree at most $(D_{m,\varepsilon})^n$ such that
\begin{equation}\label{equation: algebraic case}
\|T^{\lambda} f\|_{\mathrm{BL}^p_{k,A}(N_{\frac{1}{4}R^{1/2+\delta_m}}(Y^l)\cap B(0,R))}^p\geq c_{\mathrm{alg}} \| T^{\lambda} f\|_{\mathrm{BL}^p_{k,A}(B(0,R))}^p.
\end{equation}
Here $c_{\mathrm{alg}} > 0$ is a constant depending only on $n$ and $\varepsilon$ which is chosen to be sufficiently small to suit the needs of the forthcoming argument. 

\subsubsection*{The cellular case} The negation of the algebraic case: for every transverse complete intersection $Y^l \subseteq Z$ of dimension $1 \leq l \leq m-1$ and maximum degree at most $(D_{m,\varepsilon})^n$ the inequality
\begin{equation}\label{cellular case equation}
\|T^{\lambda} f\|_{\mathrm{BL}^p_{k,A}(N_{\frac{1}{4}R^{1/2+\delta_m}}(Y^l)\cap B(0,R))}^p < c_{\mathrm{alg}} \| T^{\lambda} f\|_{\mathrm{BL}^p_{k,A}(B(0,R))}^p
\end{equation}
holds.\\

The cellular case is the simplest situation and will be treated first. Here a polynomial partitioning argument is employed which splits the mass of the $k$-broad norm into small pieces; these pieces can then be treated individually using the radial induction hypothesis. The algebraic case is the most involved situation; it encapsulates the kind of behaviour exhibited by the sharp examples in $\S$\ref{Necessary conditions section}. In this case $T^{\lambda} f$ can be thought of as concentrated near a low-dimensional and low-degree variety $Y^l$ (in a $k$-broad sense). If the wave packets of $f$ are also tangential to this variety, then one may use induction on the dimension to conclude the argument. This might not be the case, however, and if many of the wave packets of $f$ are transverse to $Y^l$, then an alternative argument is required.  Thus, the algebraic case naturally splits into two sub-cases, a tangential and a transverse sub-case, which are discussed in detail below. Lemma~\ref{transverse interaction lemma} can be applied to show that a given tube $T_{\theta, v}$ can only interact transversely with the variety $Y^l$ inside a small number of finitely overlapping balls of some fixed radius $\rho \ll R$ (more precisely, the radius is chosen to satisfy $\rho^{1/2 + \delta_l} = R^{1/2 + \delta_m}$); this eventually allows one to also close the induction in the transverse situation.




\subsection{The cellular case}

The cellular case can be treated using polynomial partitioning. Roughly speaking, since by hypothesis $T^{\lambda}f$ is concentrated in a neighbourhood of an $m$-dimensional transverse complete intersection, for any $D \geq 1$ Theorem~\ref{polynomial partitioning theorem} can be applied in $m$ dimensions to show that there exists a non-zero polynomial $P$ of degree at most $D$ such that, letting $\{O_i\}_{i \in \mathcal{I}}$ denote the connected components of $\R^n\setminus Z(P)$ (which, recall, are referred to as \emph{cells}), one has $\#\{ O_i : i \in \mathcal{I}\} \sim D^m$ and
\begin{equation}\label{cellular 1}
\|T^{\lambda}f\|_{\mathrm{BL}_{k,A}^p(O_i)}^p \sim D^{-m}\|T^{\lambda}f\|_{\mathrm{BL}_{k,A}^p(B(0,R))}^p \qquad \textrm{for all $i \in \mathcal{I}$.}
\end{equation}
Thus, the mass of the $k$-broad norm is essentially equally distributed amongst the cells. Moreover, using the hypothesis of the cellular case, one can construct $P$ so that very little of the mass lies near the cell wall 
\begin{equation*}
W := N_{2R^{1/2+\delta}}(Z(P)) \cap B(0,R).
\end{equation*}
In particular, the estimate \eqref{cellular 1} essentially still holds if the $O_i$ are replaced with the shrunken cells $O_i' := O_i\setminus W$. The $O_i'$ can be thought of as well-separated\footnote{In particular, the distance between a pair of distinct $O_i'$ is wider than the width $R^{1/2+\delta}$ of any tube $T_{\theta,v}$.} and this facilitates a divide-and-conquer-style argument. More precisely, the fact that a non-zero univariate polynomial of degree at most $D$ has at most $D$ roots quickly leads to the following observation.

\begin{lemma}\label{each tube enters few cells} If $P$ is a polynomial of degree $\deg P \leq D$ and $\{O_i'\}_{i \in \mathcal{I}}$ are defined as above, then each tube $T_{\theta,v}$ enters at most $D/\varepsilon$ of the cells $O_i'$.
\end{lemma}

It is important to note that, in general, Lemma~\ref{each tube enters few cells} does not hold if the $O_i'$ are replaced with the cells $O_i$.

\begin{proof}[Proof (of Lemma~\ref{each tube enters few cells})] Let $[\Gamma_{\theta,v}^{\lambda}]_{\varepsilon} : \R \rightarrow \R^{n-1}$ denote the polynomial approximant of $\Gamma_{\theta,v}^{\lambda}$, as defined in $\S$\ref{polynomial approximation subsection}. Thus, $\deg [\Gamma_{\theta,v}^{\lambda}]_{\varepsilon}  \leq \lceil 1/2\varepsilon\rceil$ and \eqref{Taylor approximation 2} implies that
\begin{equation*}
|[\Gamma_{\theta,v}^{\lambda}]_{\varepsilon}(t)-\Gamma_{\theta,v}^{\lambda}(t)|\leq R^{1/2} \qquad \textrm{for all $t\in (-R,R)$.}
\end{equation*}

Suppose that $x\in O_i'\cap T_{\theta,v}$. By the definition of $O_i'$, the ball $B\left(x, 2R^{1/2+\delta}\right)$ contains no points of $Z(P)$, and is therefore contained in $O_i$. On the other hand, $\mathrm{dist}(x, \Gamma_{\theta,v}^{\lambda}) <  R^{1/2 + \delta}$ and therefore  $\mathrm{dist}(x, [\Gamma_{\theta,v}^{\lambda}]_{\varepsilon}) < 2R^{1/2+\delta}$. Consequently, $[\Gamma_{\theta,v}^{\lambda}]_{\varepsilon}$ enters $B(x, 2R^{1/2+\delta}) \subseteq O_i$. Thus, if $T_{\theta,v}$ enters a cell $O_i'$, then the polynomial curve $[\Gamma_{\theta,v}^{\lambda}]_{\varepsilon}$ enters $O_i$ whilst, by the simple property of univariate polynomials quoted above, $[\Gamma_{\theta,v}^{\lambda}]_{\varepsilon}$ can enter at most $\deg P\cdot \deg [\Gamma_{\theta,v}^{\lambda}]_{\varepsilon}+1 \leq  D/\varepsilon$ cells $O_i$. 
\end{proof}

Some aspects of the discussion prior to Lemma~\ref{each tube enters few cells} are not entirely precise; for instance, to apply the polynomial partitioning theorem one must work with an $L^1$ function rather than a $k$-broad norm. In view of this, let $\mu$ denote the measure on $\R^n$ with Radon--Nikodym derivative 
\begin{equation*}
\sum_{B_{K^2} \in \mathcal{B}_{K^2}} \mu_{T^{\lambda}f}(B_{K^2}) \frac{1}{|B_{K^2}|} \chi_{B_{K^2}}
\end{equation*}
 with respect to the Lebesgue measure. One may easily verify that 
\begin{equation}\label{cellular 2}
\mu(U) \leq \|T^{\lambda}f\|_{\mathrm{BL}_{k,A}^p(U)}^p \quad \textrm{and} \quad \|T^{\lambda}f\|_{\mathrm{BL}_{k,A}^p(B(0,R))}^p \leq \mu(B(0,2R))
\end{equation}
for all Lebesgue measurable sets $U \subseteq \R^n$. These inequalities ensure that the measure $\mu$ acts as an effective surrogate for the $k$-broad norm in the polynomial partitioning argument. 

By combining the cellular hypothesis with Theorem~\ref{polynomial partitioning theorem}, one obtains the following partitioning result. 

\begin{lemma}[Polynomial partitioning \cite{Guth2018}]\label{partitioning lemma} There exists a polynomial $P$ of degree $\deg P \leq D_{m,\varepsilon}$ such that, if $\{O_i\}_{i \in \mathcal{I}}$ and $W$ are defined as above and $O_i' := O_i \setminus W$ for all $i \in \mathcal{I}$, then $\# \mathcal{I} \lesssim (D_{m,\varepsilon})^{m}$ and 
\begin{equation}\label{partitioning lemma equation}
\mu(O_i') \sim (D_{m,\varepsilon})^{-m} \mu(B(0,2R))
\end{equation}
for at least 99\%  of the cells $O_i'$.
\end{lemma} 

This lemma is contained in the work of the first author \cite[$\S$8.1]{Guth2018} and the details of the proof are not reproduced here. The basic idea is as follows: by hypothesis, the mass of $\mu$ is concentrated in $N_{R^{1/2 +\delta_m}}(Z)$ where $Z$ is an $m$-dimensional algebraic variety; this allows one to apply Theorem~\ref{polynomial partitioning theorem} in $m$-dimensions to construct a polynomial $P$ which satisfies the desired properties with $O_i$ in place of $O_i'$. The hypothesis of the cellular case implies that the mass of $\mu$ cannot concentrate in a neighbourhood of a certain type of algebraic variety and this can be used to show, in particular, that the mass cannot concentrate around the cell wall $W$. Provided the constant $c_{\mathrm{alg}}$ is chosen to be sufficiently small, this allows one to pass to the shrunken cells $O_i'$ in \eqref{partitioning lemma equation} (at least for 99\% of the cells). 

There are a number of technicalities involved in rigorously carrying out this argument. In particular, one must justify the application of Theorem~\ref{polynomial partitioning theorem} in dimension $m$; this requires locally approximating $Z$ by some tangent plane $T_zZ$ and applying the theorem to the push-forward of $\mu$ onto $T_zZ$ under orthogonal projection. The partitioning variety in $T_zZ$ is lifted to a variety $\tilde{Z}$ in $\R^n$ by taking the pre-image under the orthogonal projection; it is possible to define $\tilde{Z}$ in this way so that it is transverse to $Z$. The cells $O_i$ are then defined with respect to $\tilde{Z}$.\footnote{To carry out this argument rigorously, one must further ensure that all the relevant varieties are transverse complete intersections of dimension at most $m-1$ and controlled degree in order to invoke \eqref{cellular case equation}. Such technicalities account for the choice of maximum degree $(D_{m,\varepsilon})^n$ in the definition of the algebraic and cellular cases.}  

Presently, it is shown how together Lemma~\ref{each tube enters few cells} and Lemma~\ref{partitioning lemma} easily yield the proof of Proposition~\ref{k-linear for m proposition} in the cellular case. Applying Lemma~\ref{partitioning lemma} one obtains a partition of $\R^n\setminus W$ into disjoint cells $\{O_i\}_{i \in \mathcal{I}}$. For each $i \in \mathcal{I}$ let 
\begin{equation*}
\mathbb{T}_i:=\{(\theta,v)\in \mathbb{T}_Z:\;T_{\theta,v} \cap O_i'\neq\emptyset\}
\quad \textrm{and} \quad f_i:=\sum_{(\theta,v)\in\mathbb{T}_i}f_{\theta,v}.
\end{equation*}
By Lemma~\ref{wave packet concentration lemma} one has
\begin{equation*}
\|T^{\lambda} f\|_{\mathrm{BL}^p_{k,A}(O_i')}^p  \leq \|T^{\lambda} f_i\|_{\mathrm{BL}^p_{k,A}(O_i')}^p + \mathrm{RapDec}(R)\|f\|_{L^2(B^{n-1})}^p.
\end{equation*}
Combining this inequality with \eqref{cellular 2} and Lemma~\ref{partitioning lemma}, one deduces that at least 99\% of the cells $O_i'$ have the property that
\begin{equation}\label{cellular 3}
\|T^{\lambda} f\|_{\mathrm{BL}^p_{k,A}(B(0,R))}^p \lesssim (D_{m,\varepsilon})^m\|T^\lambda f_i\|^p_{\mathrm{BL}^p_{k,A}(O_i')}  + \mathrm{RapDec}(R)\|f\|_{L^2(B^{n-1})}^p.
\end{equation}
On the other hand, by Lemma~\ref{each tube enters few cells} and orthogonality between the $f_{\theta,v}$, one has
\begin{align*}
\sum_{i \in \mathcal{I}} \|f_i\|_{L^2(B^{n-1})}^2 &\sim \sum_{(\theta,v)\in \mathbb{T}_Z} \#\{ i \in \mathcal{I} : (\theta, v) \in \mathbb{T}_i\}\|f_{\theta,v}\|_{L^2(B^{n-1})}^2 \\
&\lesssim_{\varepsilon} D_{m,\varepsilon} \|f\|_{L^2(B^{n-1})}^2.
\end{align*}
Since there are roughly $(D_{m,\varepsilon})^m$ cells in total, Markov's inequality shows that at least 99\% of the cells $O_i'$ have the property that
\begin{equation}\label{cellular 4}
\|f_i\|_{L^2(B^{n-1})}^2 \lesssim_{\varepsilon} D_{m,\varepsilon}^{-(m-1)}\|f\|_{L^2(B^{n-1})}^2.
\end{equation}
Therefore, there exists some cell $O_i'$ for which \eqref{cellular 3} and \eqref{cellular 4} simultaneously hold; henceforth, attention is focused on a single such cell $O_i'$. 

Let $E_{m,A}(R)$ denote the constant appearing on the right-hand side of \eqref{equation: k-linear for m general}; namely,
\begin{equation*}
E_{m,A}(R) := C_{m,\varepsilon} K^{\bar{C}_{\varepsilon}} R^{\varepsilon - c_n\delta_m + \delta\left(\log {\bar{A}}_\varepsilon - \log A\right)-e_{k,n}(p)+1/2}.
\end{equation*}
By the radial induction hypothesis, Proposition~\ref{k-linear for m proposition} holds for the radius $R/2$. For the fixed choice of $i$ as above, cover $O_i'$ with $O(1)$ balls of radius $\rho = R/2$. Applying Lemma~\ref{for rough induction} to $f_i$ on each of these balls, one obtains
\begin{equation*}
\|T^{\lambda} f_i\|_{\mathrm{BL}^p_{k,A}(B(0,R))}\lesssim E_{m, A}(R/2)\|f_i\|_{L^2(B^{n-1})} \lesssim E_{m,A}(R)\|f_i\|_{L^2(B^{n-1})}.
\end{equation*}

Combining the above estimate with \eqref{cellular 3} and \eqref{cellular 4}, one deduces that
\begin{equation*}
\|T^{\lambda} f\|_{\mathrm{BL}^p_{k,A}(B(0,R))}\leq C_{\varepsilon} (D_{m,\varepsilon})^{-(m-1)/2+ m/p} E_{m,A}(R) \|f\|_{L^2(B^{n-1})}
\end{equation*}
for some constant $C_{\varepsilon} > 0$. The $D_{m,\varepsilon}$ exponent is negative if and only if $p>2m/(m-1)$; note this is the case for the choice of exponent $p = \bar{p}_0(k,m)$.\footnote{It is for this reason that one works with the modified exponent $\bar{p}_0(k,m)$ rather than $\bar{p}(k,m)$.} Thus, recalling the definition of $D = D_{m,\varepsilon}$ and assuming $\varepsilon$ is sufficiently small depending on $n$, it follows that $C_{\varepsilon} (D_{m,\varepsilon})^{-(m-1)/2+ m/p} \leq 1$. This establishes the desired estimate \eqref{equation: k-linear for m general} and closes the induction in the cellular case.




\subsection{The algebraic case} Fix a transverse complete intersection $Y^l$ of dimension $1 \leq l \leq m-1$ and maximum degree $\overline{\deg}\, Y^l \leq (D_{m,\varepsilon})^n$ which satisfies \eqref{equation: algebraic case}. Let $R^{1/2} \ll  \rho \ll R$ be such that $\rho^{1/2+\delta_l}=R^{1/2+\delta_m}$ and note that
\begin{equation}\label{rho bounds}
R \leq R^{2\delta_l}\cdot \rho \quad \textrm{and} \quad \rho \leq R^{-\delta_l/2}\cdot R. 
\end{equation}

Let $\mathcal{B}_{\rho}$ be a finitely-overlapping cover of $B(0,R)$ by $\rho$-balls and for each $B\in \mathcal{B}_{\rho}$ define
\begin{equation*}
\T_B:=\{(\theta,v) \in \T:  T_{\theta,v}\cap N_{\frac{1}{4}R^{1/2+\delta_m}}(Y^l)\cap B\neq \emptyset\}
\end{equation*}
and $f_B:= \sum_{(\theta,v)\in \T_B}f_{\theta,v}$. Thus, by \eqref{equation: algebraic case} and Lemma~\ref{wave packet concentration lemma},
\begin{equation*}
\| T^{\lambda} f\|_{\mathrm{BL}^p_{k,A}(B(0,R))}^p\lesssim \sum_{B\in\mathcal{B}_{\rho}}\|T^{\lambda} f_B\|_{\mathrm{BL}^p_{k,A}(N_{\frac{1}{4}R^{1/2+\delta_m}}(Y^l)\cap B)}^p
\end{equation*}
holds up to the inclusion of a rapidly decreasing error term. In what follows, such error terms, which are harmless, are suppressed in the notation. 

For $B = B(y, \rho) \in \mathcal{B}_{\rho}$ let $\T_{B,\mathrm{\mathrm{tang}}}$ denote the set of all $(\theta,v)\in \T_B$ with the property that whenever $x\in T_{\theta,v}$ and $z\in Y^l\cap B(y, 2\rho)$ satisfy $|x-z|\leq 2\bar{C}_{\mathrm{tang}} \rho^{1/2+\delta_l}$, it follows that
\begin{equation*}
\angle\left(G^{\lambda}(x;\omega_\theta), T_z Y^l \right)\leq \frac{\bar{c}_{\mathrm{tang}}}{2}\rho^{-1/2+\delta_l},
\end{equation*}
where $\bar{C}_{\mathrm{tang}}$ and $\bar{c}_{\mathrm{tang}}$ are the constants appearing in the definition of tangency. Furthermore, let $\T_{B,\mathrm{\mathrm{trans}}}:=\T_B\setminus \T_{B,\mathrm{\mathrm{tang}}}$ and define
\begin{equation*}
f_{B,\mathrm{tang}}:=\sum_{(\theta,v)\in \T_{B,\mathrm{\mathrm{tang}}}} f_{\theta,v} \quad \textrm{and} \quad f_{B,\mathrm{\mathrm{trans}}}:=\sum_{(\theta,v)\in \T_{B,\mathrm{\mathrm{trans}}}}f_{\theta,v}.
\end{equation*}
It follows that $f_B=f_{B,\mathrm{tang}}+f_{B,\mathrm{\mathrm{trans}}}$ and, by the triangle inequality for broad norms (that is, Lemma~\ref{triangle inequality lemma}), one concludes that
\begin{equation*}
\| T^{\lambda} f\|_{\mathrm{BL}^p_{k,A}(B(0,R))}^p\lesssim \!\! \sum_{B\in\mathcal{B}_{\rho}} \!\! \| T^{\lambda} f_{B,\mathrm{tang}}\|_{\mathrm{BL}^p_{k,A/2}(B)}^p+ \!\! \sum_{B\in\mathcal{B}_{\rho}} \!\! \| T^{\lambda} f_{B,\mathrm{\mathrm{trans}}}\|_{\mathrm{BL}^p_{k,A/2}(B)}^p.
\end{equation*}
Either the tangential or transverse contribution to the above sum dominates, and each case is treated separately. 




\subsection*{The tangential sub-case}

Suppose that the tangential term dominates; that is,
\begin{equation}\label{tangent case 1}
\| T^{\lambda} f\|_{\mathrm{BL}^p_{k,A}(B(0,R))}^p \lesssim \sum_{B\in\mathcal{B}_{\rho}} \| T^{\lambda} f_{B,\mathrm{tang}}\|_{\mathrm{BL}^p_{k,A/2}(B)}^p.
\end{equation}
Each term in the right-hand sum is bounded using the dimensional induction hypothesis. Fix $B = B(y,\rho) \in \mathcal{B}_{\rho}$ and, as in $\S$\ref{Adjusting wave packets section}, let $\tilde{T}^{\lambda}(f_{B,\mathrm{tang}})\;\widetilde{}\;(x) = T^{\lambda}f_{B,\mathrm{tang}}(x+y)$ so that\footnote{See footnote~\ref{translation footnote} on page \pageref{translation footnote}.}
\begin{equation}\label{tangent case 2}
\| T^{\lambda} f_{B,\mathrm{tang}}\|_{\mathrm{BL}^p_{k,A/2}(B(y,\rho))} = \| \tilde{T}^{\lambda} (f_{B,\mathrm{tang}})\;\widetilde{}\;\|_{\mathrm{BL}^p_{k,A/2}(B(0,\rho))}.
\end{equation}
Since $\overline{deg}\,Y^l \leq D_{l,\epsilon}$, in order to apply the induction hypothesis, one must verify that $(f_{B,\mathrm{tang}})^{\;\widetilde{}\;}$ is concentrated on scale $\rho$ wave packets that are $\rho^{-1/2+\delta_l}$-tangent to $Y^l-y$ in $B(0,\rho)$. By Lemma~\ref{concentration on smaller scale wave packets lemma}, $(f_{B,\mathrm{tang}})^{\;\widetilde{}\;}$ is concentrated on scale $\rho$ wave packets from
\begin{equation*}
\tilde{\T}_{B,\mathrm{tang}} := \bigcup_{(\theta, v) \in \T_{B, \mathrm{tang}}} \tilde{\T}_{\theta, v},
\end{equation*}
where the $\tilde{\T}_{\theta, v}$ are as defined in $\S$\ref{Adjusting wave packets section}. Fix $(\tilde{\theta}, \tilde{v}) \in \tilde{\T}_{B,\mathrm{tang}}$ and recall from \eqref{children close to mothers} that
\begin{equation}\label{children close to mothers recalled} 
\mathrm{dist}_H(\tilde{T}_{\tilde{\theta}, \tilde{v}}, (T_{\theta, v} - y) \cap B(0, \rho) ) \lesssim R^{1/2 + \delta} \ll \rho^{1/2 + \delta_l}.
\end{equation}
Combining this with the definition of $\T_{B, \mathrm{tang}}$, it is easy to see that $\tilde{T}_{\tilde{\theta}, \tilde{v}}$ satisfies the angle condition for tangency and it remains to verify the containment property $\tilde{T}_{\tilde{\theta}, \tilde{v}} \subseteq N_{\rho^{1/2 + \delta_l}}(Y^l - y)$. The definition of $\T_B$ and \eqref{children close to mothers recalled} imply that $\tilde{T}_{\tilde{\theta}, \tilde{v}} \cap N_{\frac{1}{2}\rho^{1/2 + \delta_l}}(Y^l - y) \cap B(0,\rho) \neq \emptyset$ and so the containment property follows from the angle condition, as in the proof of Proposition~\ref{tangential thin packets}. 

Thus, the dimensional induction hypothesis may be applied to $(f_{B,\mathrm{tang}})\;\widetilde{}\;$, and one therefore deduces that
\begin{equation*}
\| \tilde{T}^{\lambda} (f_{B,\mathrm{tang}})^{\;\widetilde{}\;}\|_{\mathrm{BL}^p_{k,A/2}(B(0,\rho))} \leq E_{l, A/2}(\rho) \|f_{B,\mathrm{tang}}\|_{L^2(B^{n-1})}.
\end{equation*}
Combining this estimate with \eqref{tangent case 1} and \eqref{tangent case 2}, one concludes that 
\begin{equation*}
\| T^{\lambda} f\|_{\mathrm{BL}^p_{k,A}(B(0,R))} \leq  R^{O(\delta_l)}E_{l, A/2}(\rho) \|f\|_{L^2(B^{n-1})}^2.
\end{equation*}
 To close the induction in this case, it remains to show that 
\begin{equation*}
R^{O(\delta_l)}E_{l, A/2}(\rho) \leq E_{m, A}(R).
\end{equation*}
By \eqref{rho bounds}, 
\begin{align*}
    \rho^{\delta(\log {\bar{A}}_\varepsilon - \log (A/2))} &\leq R^{O(\delta_l)} R^{\delta(\log {\bar{A}}_\varepsilon - \log A)},\\
    \rho^{-e_{k,n}(p) +1/2} &\leq R^{O(\delta_l)} R^{-e_{k,n}(p) +1/2}.
\end{align*}
Combining these observations, the problem is further reduced to showing that $\rho^{\varepsilon-c_n\delta_l}\leq R^{-c\delta_l}R^{\varepsilon -c_n \delta_m}$, where $c > 1$ is a suitably large dimensional constant. By \eqref{rho bounds}, one may ensure that this inequality holds by choosing the constant $c_n$ in \eqref{parameter values} at the outset to be large relative to $c$.




\subsection*{The transverse sub-case}

Now suppose the transverse term dominates; that is,
\begin{equation}\label{equation: transverse case}
\| T^{\lambda} f\|_{\mathrm{BL}^p_{k,A}(B(0,R))}^p \lesssim \sum_{B\in\mathcal{B}_{\rho}}\| T^{\lambda} f_{B,\mathrm{trans}}\|^p_{\mathrm{BL}^p_{k,A/2}(B)}.
\end{equation}

The idea here is somewhat similar to that used in the cellular case. Recall, in the cellular case the number of cells a given tube can enter is controlled by B\'ezout's theorem. In the transverse case, the number of balls $B \in \mathcal{B}_{\rho}$ inside which a given tube can be transverse to $Y^l$ is again controlled due to B\'ezout's theorem, this time by Lemma~\ref{transverse interaction lemma}. This yields the following key inequality.

\begin{claim}
\begin{equation}\label{transverse 1}
\sum_{B\in\mathcal{B}_{\rho}}\|f_{B,\mathrm{trans}}\|_{L^2(B^{n-1})}^2\lesssim_\varepsilon\|f\|_{L^2(B^{n-1})}^2.
\end{equation}
\end{claim}

\begin{proof} This is a fairly direct consequence of the hypothesis of the transverse case together with Lemma~\ref{transverse interaction lemma}. Indeed, note that
\begin{equation*}
\sum_{B\in\mathcal{B}_{\rho}}\|f_{B,\mathrm{trans}}\|_{L^2(B^{n-1})}^2 \sim \sum_{(\theta,v) \in \T}\#\{B\in\mathcal{B}_{\rho}: (\theta,v)\in \mathbb{T}_{B,\mathrm{trans}}\} \|f_{\theta,v}\|_{L^2(B^{n-1})}^2.
\end{equation*}
and so, to prove \eqref{transverse 1}, it suffices to fix an arbitrary $(\theta,v)\in \mathbb{T}_{B,\mathrm{trans}}$ and show that
\begin{equation}\label{transverse 2}
\#\{B\in\mathcal{B}_{\rho}:\; (\theta,v)\in \mathbb{T}_{B,\mathrm{trans}}\}\lesssim_{\varepsilon} 1.
\end{equation}
Let $\Gamma := [\Gamma_{\theta,v}^{\lambda}]_{\varepsilon} \colon \R \to \R^n$ be the polynomial approximant of the core curve $\Gamma_{\theta, v}^{\lambda}$ defined in $\S$\ref{polynomial approximation subsection}. Thus, $\deg \Gamma \lesssim_{\varepsilon} 1$ and, recalling that $R \lesssim_{\varepsilon} \lambda^{1- \varepsilon}$, property \eqref{Taylor approximation 2} of the approximant implies that 
\begin{equation}\label{transverse 2.1}
|\Gamma(t)-\Gamma_{\theta,v}^{\lambda}(t)|\leq R^{1/2} \qquad \textrm{for all $t\in (-R,R)$.} 
\end{equation}
Let $u \in T_{\theta,v}$ and $x\in \Gamma \cap B(0,R)$ with $|u-x| \lesssim R^{1/2+\delta}$. It follows from the definition of $T_{\theta, v}$ and \eqref{transverse 2.1} that there exists some $t \in (-R, R)$ such that 
\begin{equation*}
|u - \Gamma_{\theta,v}^{\lambda}(t)| \lesssim R^{1/2 + \delta} \quad \textrm{and} \quad |x - \Gamma(t)| \lesssim R^{1/2+\delta}.
\end{equation*}
Consequently, recalling Lemma~\ref{Gauss Lipschitz},
\begin{equation*}
\angle( G^{\lambda}(u;\omega_\theta), \mathrm{T}_{x}\Gamma) \lesssim \angle( T_{\Gamma_{\theta,v}^{\lambda}(t)}\Gamma_{\theta,v}^{\lambda}, T_{\Gamma(t)}\Gamma) + R^{1/2 + \delta}/\lambda
\end{equation*}
and therefore, by property \eqref{Taylor approximation 3} of the approximant, 
\begin{equation*}
\angle( G^{\lambda}(u;\omega_\theta), \mathrm{T}_{x}\Gamma)\lesssim_{\varepsilon} \lambda^{-1/2} + R^{1/2 + \delta}/\lambda <\frac{\bar{c}_{\mathrm{tang}}}{4}\rho^{-1/2+\delta_l}.
\end{equation*}
Using the above inequality, one may easily verify that if $B = B(y, \rho) \in \mathcal{B}_{\rho}$ and $(\theta,v) \in \T_{B,\mathrm{trans}}$, then $Y^l_{>\alpha, r, \Gamma} \cap B(y, 2 \rho) \neq \emptyset$ for $\alpha \sim \rho^{-1/2 + \delta_l}$ and $r \sim \rho^{1/2 + \delta_l}$. Here $Y^l_{>\alpha, r, \Gamma}$ is as defined in $\S$\ref{Algebraic preliminaries section}; that is
\begin{equation*}
Y^l_{>\alpha, r, \Gamma} := \big\{z \in Y^l : \exists\, x \in \Gamma \textrm{ with } |x - z| < r \textrm{ and } \angle(T_zY^l, T_x\Gamma) > \alpha \big\}.
\end{equation*}
By Lemma~\ref{transverse interaction lemma}, the number of balls $B=B(y,\rho)\in \mathcal{B}_{\rho}$ for which $B(y,2\rho)$ intersects $Y^l_{>\alpha, r, \Gamma}$ non-trivially is at most $O((\deg \Gamma)^n\cdot (\overline{\deg}\, Y^l)^n) = O_{\varepsilon}(1)$. Combining these observations, one immediately deduces \eqref{transverse 2}, as required. 
\end{proof}

In view of \eqref{transverse 1}, the strategy in the transverse case is to use the radial induction hypothesis to show that for some constant $\bar{c}_{\varepsilon} > 0$ one has
\begin{equation}\label{transverse 3}
\|T^{\lambda} f_{B,\mathrm{trans}}\|_{\mathrm{BL}^p_{k,A/2}(B)} \leq \bar{c}_{\varepsilon} E_{m,A}(R) \|f_{B,\mathrm{trans}}\|_{L^2(B^{n-1})} \quad \textrm{ for all $B \in \mathcal{B}_{\rho}$.}
\end{equation}
Indeed, provided $\bar{c}_{\varepsilon} > 0$ is sufficiently small, depending only on $n$ and $\varepsilon$, the preceding inequality may be combined with \eqref{equation: transverse case}, \eqref{transverse 1} and the simple estimate
\begin{equation*}
\|f_{B,\mathrm{trans}}\|_{L^2(B^{n-1})} \lesssim \|f\|_{L^2(B^{n-1})},
\end{equation*}
to yield 
\begin{align*}
\|T^{\lambda} f\|_{\mathrm{BL}^p_{k,A/2}(B(0,R))} &\lesssim_{\varepsilon} \bar{c}_{\varepsilon} E_{m,A}(R) \|f\|_{L^2(B^{n-1})}^{1-2/p} \big(\sum_{B \in \mathcal{B}_{\rho}}\|f_{B,\mathrm{trans}}\|_{L^2(B^{n-1})}^2 \big)^{1/p} \\
&\leq E_{m,A}(R) \|f\|_{L^2(B^{n-1})},
\end{align*}
closing the induction in this case. 

The main obstacle in carrying out this programme is that the $f_{B, \mathrm{trans}}$ do not, in general, satisfy the hypothesis of Proposition~\ref{k-linear for m proposition} at scale $\rho$, and therefore one cannot directly apply the radial induction hypothesis to these functions. However, one can appeal to the theory developed in $\S$\ref{Adjusting wave packets section}, which essentially allows $f_{B, \mathrm{trans}}$ to be broken into pieces $f_{B, \mathrm{trans}, b}$ which do satisfy the hypothesis of the proposition at scale $\rho$. Here is a sketch of how the argument works. By choosing a suitable set of translates $\mathfrak{B}$, one may essentially write 
\begin{equation}\label{heuristic decomposition}
\|T^{\lambda}f_{B, \mathrm{trans}}\|_{\mathrm{BL}^p_{k,A/2}(B)} \lesssim \big( \sum_{b \in \mathfrak{B}} \|T^{\lambda}f_{B, \mathrm{trans}, b}\|_{\mathrm{BL}^p_{k,A/2}(B)}^p \big)^{1/p}
\end{equation}
where each piece $f_{B, \mathrm{trans}, b}$ is defined so that it is concentrated on scale $\rho$ wave packets which are tangential to some translate $Z - y + b$ of $Z$. By the theory of transverse equidistributions developed in $\S$\ref{Transverse equidistribution section} and $\S$\ref{Adjusting wave packets section}, the $f_{B, \mathrm{trans}, b}$ satisfy favourable $L^2$ estimates and, in particular, the inequality \eqref{transverse equidistribution gain} below holds. The radial induction hypothesis is applied to each of the $T^{\lambda}f_{B, \mathrm{trans}, b}$. To close the induction, one must estimate the resulting sum 
\begin{equation*}
\big(\sum_{b \in \mathfrak{B}} \|f_{B, \mathrm{trans}, b}\|_{L^2(B^{n-1})}^2 \big)^{1/2}
\end{equation*} 
in terms of $\|f_{B, \mathrm{trans}}\|_{L^2(B^{n-1})}$. Here the gain in $\rho/R$ in \eqref{transverse equidistribution gain}, afforded by transverse equidistribution, is crucial to the argument: it allows one to sum up the contributions from the individual pieces $f_{B, \mathrm{trans}, b}$ without any (significant) loss in $R$. It is this gain which accounts for the improved range of estimates for the $k$-broad inequalities under the positive-definite hypothesis (recall, the proof of the transverse equidistribution lemma relied heavily on the positive-definite condition). 

As part of this argument, to ensure that the $f_{B, \mathrm{trans}, b}$ form a reasonable decomposition of $f_{B, \mathrm{trans}}$ so that \eqref{heuristic decomposition} essentially holds, the set of translates $\mathfrak{B}$ must be chosen so that $\bigcup_{b \in \mathfrak{B}} N_{\rho^{1/2+\delta_m}}(Z-y+b)$ covers $N_{R^{1/2 + \delta_m}}(Z)$ (recall, by hypothesis $f_{B, \mathrm{trans}}$ is concentrated on wave packets in $\T_Z$ and so the mass of $T^{\lambda}f_{B,\mathrm{trans}}$ is concentrated in $N_{R^{1/2 + \delta_m}}(Z)$) and so that the $N_{\rho^{1/2+\delta_m}}(Z-y+b)$ are essentially disjoint. This can be achieved using a probabilistic construction. More precisely, fixing $B = B(y, \rho) \in \mathcal{B}_{\rho}$, one may show the following. 

\begin{lemma}\label{choosing the collection} There exist a finite set $\mathfrak{B}\subset B(0, 2R^{1/2+\delta_m})$ and a collection $\mathcal{B}^{\prime}\subseteq \{B_{K^2} \in \mathcal{B}_{K^2} : B_{K^2} \cap B(y,\rho) \neq \emptyset\}$ such that, up to inclusion of a rapidly decreasing error term,
\begin{equation}\label{transverse 5}
\|T^{\lambda} f_{B,\mathrm{trans}} \|_{\mathrm{BL}^p_{k,A/2}(B)}\lesssim (\log R)^2 \big(\sum_{B_{K^2}\in \mathcal{B}^{\prime}}\mu_{T^\lambda f_{B,\mathrm{trans}}}(B_{K^2})\big)^{1/p}
\end{equation}
and for each $B_{K^2}\in\mathcal{B}^{\prime}$ the following hold:
\begin{enumerate}[i)]
\item There exists some $b\in\mathfrak{B}$ such that 
\begin{equation}\label{random containment}
B_{K^2} \subset N_{\frac{1}{2}\rho^{1/2+\delta_m}}(Z+b);
\end{equation}
\item There exist at most $O(1)$ vectors $b\in\mathfrak{B}$ for which 
\begin{equation*}
B_{K^2} \cap  N_{\rho^{1/2+\delta_m}}(Z+b)\neq \emptyset.
\end{equation*}
\end{enumerate}
\end{lemma}

The proof of the lemma, which is slightly technical, is postponed until the end of the section. Temporarily assuming this result, one may argue as follows to complete the proof of Proposition~\ref{k-linear for m proposition}. 

For each $b \in \mathfrak{B}$ let $\mathcal{B}^{\prime}_b$ denote the collection of all $B_{K^2} \in \mathcal{B}^{\prime}$ for which \eqref{random containment} holds. Thus, by \eqref{transverse 5} and property i) in the lemma,
\begin{equation*}
\|T^{\lambda} f_{B,\mathrm{trans}} \|_{\mathrm{BL}^p_{k,A/2}(B)}\lesssim (\log R)^2 \Big(\sum_{b \in \mathfrak{B}} \sum_{B_{K^2}\in \mathcal{B}^{\prime}_b}\mu_{\tilde{T}^\lambda (f_{B,\mathrm{trans}})\;\widetilde{}\;}(B_{K^2}-y)\Big)^{1/p},
\end{equation*}
up to a rapidly decreasing error term.

Define the collection of wave packets
\begin{equation*}
\tilde{\T}_b^{\prime} := \Big\{ (\tilde{\theta}, \tilde{v}) \in \!\!\! \bigcup_{(\theta,v) \in \T_{B(y,\rho),\mathrm{trans}}} \!\!\! \tilde{\T}_{\theta,v} : \tilde{T}_{\tilde{\theta}, \tilde{v}} \cap \bigcup_{B_{K^2} \in \mathcal{B}^{\prime}_b} (B_{K^2}-y) \neq \emptyset \Big\};
\end{equation*}
note this set is a subset of the collection $\tilde{\T}_b$ defined in $\S$\ref{Adjusting wave packets section} and so, by Proposition~\ref{tangential thin packets}, one has $\tilde{\T}_b^{\prime} \subseteq \tilde{\T}_{Z-y+b}$. Therefore, if $f_{B,\mathrm{trans},b}$ is defined by 
\begin{equation*}
(f_{B,\mathrm{trans},b})\;\widetilde{}\; = \sum_{(\tilde{\theta}, \tilde{v}) \in \tilde{\T}_b^{\prime}} (f_{B,\mathrm{trans}})\;\widetilde{}_{\!\!\tilde{\theta}, \tilde{v}}\;,
\end{equation*} 
then $(f_{B,\mathrm{trans},b})\;\widetilde{}\;$ is concentrated on wave packets that are $\rho^{-1/2 + \delta_m}$-tangent to $Z-y+b$. Furthermore, again up to a rapidly decreasing error term, one has
\begin{equation*}
\|T^{\lambda} f_{B,\mathrm{trans}} \|_{\mathrm{BL}^p_{k,A/2}(B)}\lesssim (\log R)^2 \big(\sum_{b \in \mathfrak{B}} \|\tilde{T}^{\lambda} (f_{B,\mathrm{trans},b})\;\widetilde{}\; \|_{\mathrm{BL}^p_{k,A/2}(B(0,\rho))}^p\big)^{1/p}.
\end{equation*}
The function $(f_{B,\mathrm{trans},b})\;\widetilde{}\;$ satisfies the hypotheses of Proposition~\ref{k-linear for m proposition} at scale $\rho$ and therefore the radial induction hypothesis yields
\begin{equation*}
\big( \sum_{b \in \mathfrak{B}} \|\tilde{T}^{\lambda} (f_{B,\mathrm{trans},b})\;\widetilde{}\; \|_{\mathrm{BL}^p_{k,A/2}(B(0,\rho))}^p\big)^{1/p} \leq E_{m, A/2}(\rho)  \big( \sum_{b \in \mathfrak{B}} \|f_{B,\mathrm{trans},b}\|_{L^2(B^{n-1})}^p\big)^{1/p}.
\end{equation*}
On the other hand, it is claimed that 
\begin{equation*}
\big( \sum_{b \in \mathfrak{B}} \|f_{B,\mathrm{trans},b}\|_{L^2(B^{n-1})}^p\big)^{1/p} \lesssim R^{O(\delta_m)}(\rho/R)^{(n-m)(1/4-1/2p)}\|f_{B,\mathrm{trans}}\|_{L^2(B^{n-1})}.
\end{equation*}
Clearly it suffices to prove the above inequality for $p = 2$ and $p = \infty$; the desired estimate for $p = \bar{p}_0(k,m)$ then follows by interpolation (via H\"older's inequality). 

\subsubsection*{$p=2$} Observe that, by the orthogonality between the wave packets,
\begin{equation*}
\sum_{b \in \mathfrak{B}} \|f_{B,\mathrm{trans},b}\|_{L^2(B^{n-1})}^2 \sim  \sum_{(\tilde{\theta}, \tilde{v}) \in \tilde{\T}} \#\mathfrak{B}_{\tilde{\theta},\tilde{v}} \cdot \|(f_{B,\mathrm{trans}})\;\widetilde{}_{\!\!\tilde{\theta}, \tilde{v}}\;\|_{L^2(B^{n-1})}^2
\end{equation*}
where $\mathfrak{B}_{\tilde{\theta},\tilde{v}} := \{b \in \mathfrak{B} : (\tilde{\theta}, \tilde{v}) \in \tilde{\T}_b^{\prime}\}$. Fixing $(\tilde{\theta},\tilde{v}) \in \tilde{\T}$, it suffices to show that $\#\mathfrak{B}_{\tilde{\theta},\tilde{v}} \lesssim 1$. Supposing $\mathfrak{B}_{\tilde{\theta}, \tilde{v}} \neq \emptyset$, there exists some $B_{K^2} \in \mathcal{B}^{\prime}$ with $\tilde{T}_{\tilde{\theta},\tilde{v}} \cap (B_{K^2}-y) \neq \emptyset$. For any $b \in \mathfrak{B}_{\tilde{\theta}, \tilde{v}}$ it follows that $(\tilde{\theta}, \tilde{v}) \in \tilde{\T}_b$ and so $\tilde{T}_{\tilde{\theta}, \tilde{v}} \subseteq N_{\rho^{1/2 + \delta_m}}(Z - y + b)$ by Proposition~\ref{tangential thin packets}. Consequently, $B_{K^2} \cap N_{\rho^{1/2 + \delta_m}}(Z + b) \neq \emptyset$ for all $b \in \mathfrak{B}_{\tilde{\theta},\tilde{v}}$ and so the desired bound follows from property ii) of Lemma~\ref{choosing the collection}.

\subsubsection*{$p=\infty$} In this case, the estimate is a direct consequence of the transverse equidistribution estimates established in $\S$\ref{Transverse equidistribution section} and $\S$\ref{Adjusting wave packets section}. In particular, the function $f_{B, \mathrm{trans}}$ is concentrated on wave packets belonging to $\T_{Z, B}$ and so, by Lemma~\ref{equidistributedinputs2}, on deduces that
\begin{equation}\label{transverse equidistribution gain}
\|f_{B, \mathrm{trans},b}\|_{L^2(B^{n-1})} \lesssim R^{O(\delta_m)} (\rho/R)^{(n-m)/4} \|f_{B, \mathrm{trans}}\|_{L^2(B^{n-1})},
\end{equation}
as required. 

\subsection*{} The preceding analysis shows that $\|T^{\lambda} f_{B,\mathrm{trans}} \|_{\mathrm{BL}^p_{k,A/2}(B(0,R))}$ is bounded above by
\begin{equation*}
R^{O(\delta_m)}E_{m, A}(\rho)(\rho/R)^{(n-m)(1/4-1/2p)}  \|f_{B,\mathrm{trans}}\|_{L^2(B^{n-1})}
\end{equation*}
and therefore, to prove \eqref{transverse 3} and thereby close the induction argument in this case, it suffices to show that
\begin{equation}\label{transverse 6}
R^{O(\delta_m)}E_{m, A}(\rho)(\rho/R)^{(n-m)(1/4-1/2p)}\leq \bar{c}_{\varepsilon} E_{m, A}(R).
\end{equation}
For the exponent $p = \bar{p}(k,m)$ one has
\begin{equation*}
\rho^{-e_{k,n}(p) + 1/2} (\rho/R)^{(n-m)(1/4-1/2p)} \leq R^{-e_{k,n}(p) + 1/2}
\end{equation*}
whilst for the perturbed exponent $p = \bar{p}_0(k,m)$ the same inequality holds up to a $R^{O(\delta)}$ factor. Thus, the left-hand side of \eqref{transverse 6} is dominated by 
\begin{equation*}
 R^{O(\delta_m)} (\rho/R)^{\varepsilon}E_{m, A}(R).
\end{equation*}
Recalling \eqref{rho bounds} and the choice of parameters $\delta_l$ and $\delta_m$, one obtains the desired inequality. 




\subsection*{The probabilistic argument} The above argument establishes Proposition~\ref{k-linear for m proposition} except for the details of the probabilistic argument used to prove Lemma~\ref{choosing the collection}. 

\begin{proof}[Proof (of Lemma~\ref{choosing the collection})] Before commencing the argument proper, a few technical reductions are necessary. By a standard dyadic pigeonholing argument, one may assume that
\begin{equation}\label{transverse 3.5}
\|T^{\lambda} f_{B,\mathrm{trans}} \|_{\mathrm{BL}^p_{k,A/2}(B)}\lesssim \log R \big(\sum_{B_{K^2}\in \mathcal{B}^{\prime \prime}}\mu_{T^\lambda f_{B,\mathrm{trans}}}(B_{K^2})\big)^{1/p}
\end{equation}
for some sub-collection $\mathcal{B}^{\prime \prime} \subseteq \mathcal{B}_{K^2}$ with the property that
\begin{equation}\label{transverse 4}
\mu_{T^\lambda f_{B,\mathrm{trans}}}(B_{K^2}) \sim \mu_{T^\lambda f_{B,\mathrm{trans}}}(\bar{B}_{K^2}) \qquad \textrm{for all $B_{K^2}, \bar{B}_{K^2} \in \mathcal{B}^{\prime \prime}$.}
\end{equation}
Since $f_{B, \mathrm{trans}}$ is concentrated on wave packets from $\T_{Z, B}$, one may further assume that $B_{K^2} \cap B(y, \rho) \cap N_{R^{1/2 + \delta_m}}(Z) \neq \emptyset$ for all $B_{K^2} \in \mathcal{B}^{\prime \prime}$, at the cost of a rapidly decaying term on the right-hand side of \eqref{transverse 3.5}.
 
A set of translates $\mathfrak{B}$ will be selected at random from $\R^n$ according to a probability measure $\mathbf{P}$. The distribution $\mathbf{P}$ is taken to be a mollified version of the uniform probability distribution $\mathbf{P}_{\mathrm{unif}}$ on $B(0,R^{1/2 + \delta_m})$. In particular, let $\omega \colon \R^n \to [0, \infty)$ be given by\footnote{Here $\displaystyle (u)_+ := \left\{ \begin{array}{ll} 
u & \textrm{if $u \geq 0$} \\
0 & \textrm{if $u < 0$} 
\end{array}\right. $ for all $u \in \R$.}
\begin{equation*}
\omega(x) := \exp\Big( \frac{ - (|x| - R^{1/2 + \delta_m})_+}{\rho^{1/2 + \delta_m}} \Big) \qquad \textrm{for all $x\in \R^n$}
\end{equation*}
and $\mathbf{P}$ be the continuous probability measure on $\R^n$ with Radon--Nikodym derivative $\big(\int_{\R^n} \omega\big)^{-1} \omega$ (with respect to Lebesgue measure). This measure approximates $\mathbf{P}_{\mathrm{unif}}$ in the sense that
\begin{equation}\label{probability 1}
\mathbf{P}(\R^n \setminus B(0,2R^{1/2+\delta_m})) = \mathrm{RapDec}(R).
\end{equation}
The motivation behind the definition of $\mathbf{P}$ is that, in contrast with $\mathbf{P}_{\mathrm{unif}}$, it enjoys the doubling property
\begin{equation*}
\mathbf{P}\big(B(x,2r)\big)\lesssim \mathbf{P}\big(B(x,r)\big)\qquad \textrm{for all $x \in \R^n$ and $ 0 < r \lesssim \rho^{1/2+\delta_m}$.}
\end{equation*}
Consequently, by the Vitali covering lemma, for any $E \subseteq \R^n$ one has
\begin{equation}\label{probability 2}
\mathbf{P}\big(N_{2r}(E)\big)\lesssim \mathbf{P}\big(N_{r}(E)\big) \qquad \textrm{for all $0 < r \lesssim \rho^{1/2+\delta_m}$.}
\end{equation}

Recall, if $B(x, K^2) \in \mathcal{B}^{\prime \prime}$, then $B(x, K^2) \cap N_{R^{1/2 + \delta_m}}(Z) \neq \emptyset$ and so
\begin{equation*}
|B(0, R^{1/2 + \delta_m}) \cap N_{\rho^{1/2 + \delta_m}}(Z - x)| \gtrsim |B(0, \rho^{1/2 + \delta_m})|
\end{equation*}
which implies that
\begin{equation*}
\mathbf{P}\big(N_{\rho^{1/2 + \delta_m}}(Z - x)\big) \gtrsim \frac{|B(0, \rho^{1/2 + \delta_m})|}{|B(0, R^{1/2 + \delta_m})|}.
\end{equation*}
 
For any $s\in\mathbb{N}$ with $2^s\gtrsim |B(0, \rho^{1/2 + \delta_m})|$, define
\begin{equation*}
\mathcal{B}^s:=\bigg\{B(x,K^2) \in \mathcal{B}^{\prime \prime}: \mathbf{P}(N_{\rho^{1/2+\delta_m}}(Z-x))\sim \frac{2^s}{|B(0,R^{1/2+\delta_m})|}\bigg\}.
\end{equation*}
By a further pigeonholing argument, there exists some value of $s$ as above such that \eqref{transverse 5} holds with $\mathcal{B}^s$ in place of $\mathcal{B}^{\prime}$.

Let $\bar{C} \geq 1$ be a dimensional constant, chosen to be sufficiently large for the purposes of the following argument, and define $N := \ceil{\bar{C} 2^{-s}|B(0,R^{1/2+\delta_m})|} \in \N$. Recalling \eqref{rho bounds}, it follows that
\begin{equation}\label{probability 3}
N \lesssim \frac{|B(0,R^{1/2+\delta_m})|}{|B(0, \rho^{1/2 + \delta_m})|} \lesssim R^{2n \delta_l}.
\end{equation}

Suppose $\mathfrak{B} = \{b_1, \dots, b_N\}$ is a sequence of vectors in $\R^n$ formed by choosing each term independently at random according to the probability distribution $\mathbf{P}$. The problem is to show that $\mathfrak{B}$ satisfies each of the desired properties with high probability.




\subsubsection*{The containment property $\mathfrak{B}\subset B(0,2R^{1/2+\delta_m})$ } Recalling \eqref{probability 1} and \eqref{probability 3}, it follows that
\begin{align*}
\mathbf{P}\big(\mathfrak{B} \subseteq B(0,2R^{1/2 + \delta_m}) \big) &= 1 + \sum_{k=1}^N \binom{N}{k} (-1)^k \mathbf{P}\big(\R^n \setminus B(0,2R^{1/2 + \delta_m}) \big)^k \\
& = 1 +\mathrm{RapDec}(R).
\end{align*}
Indeed, for the second equality we use the elementary bound 
\begin{equation*}
    \Big|\sum_{k=1}^N \binom{N}{k} (-1)^k u^k \Big| = |(1-u)^N - 1| \leq N|u|  \qquad \textrm{for all $0 \leq u \leq 1$,}
\end{equation*}
which follows from the mean value theorem.
Thus, if $R \geq 1$ is sufficiently large depending only on $n$ and $\varepsilon$, then
\begin{equation}\label{containment is good}
\mathbf{P}\big(\mathfrak{B} \subseteq B(0,2R^{1/2 + \delta_m}) \big) \geq \frac{99}{100},
\end{equation}
which verifies that the desired containment property holds with high probability. 




\subsubsection*{Property i)} Let $B(x, K^2) \in \mathcal{B}^s$ and observe that 
\begin{align*}
\mathbf{P}\Big(B(x, K^2) \subseteq \bigcup_{j=1}^N N_{\frac{1}{2}\rho^{1/2 + \delta_m}}(Z + b_j) \Big) &\geq \mathbf{P}\Big( x \in \bigcup_{j=1}^N N_{\frac{1}{4}\rho^{1/2 + \delta_m}}(Z + b_j) \Big) \\
&=1 - \Big( 1 - \mathbf{P}\big(N_{\frac{1}{4}\rho^{1/2 + \delta_m}}(Z - x) \big) \Big)^N.
\end{align*}
By the definition of $\mathcal{B}^s$ and the doubling property \eqref{probability 2} of $\mathbf{P}$, it follows that $\mathbf{P}\big(N_{\frac{1}{4}\rho^{1/2 + \delta_m}}(Z - x) \big) = c \bar{C}/N$ for some dimensional constant $c > 0$ and, consequently, 
\begin{equation*}
\mathbf{P}\Big(B(x, K^2) \subseteq \bigcup_{j=1}^N N_{\frac{1}{2}\rho^{1/2 + \delta_m}}(Z + b_j) \Big) \geq 1 - (1 - c\bar{C}/N)^{N} > 1- e^{-c\bar{C}}.
\end{equation*}
Let $X$ denote the random variable that counts the number of $B_{K^2} \in \mathcal{B}^s$ for which $B_{K^2} \subseteq N_{\frac{1}{2}\rho^{1/2 + \delta_m}}(Z + b)$ for some $b \in \mathfrak{B}$. If $\bar{C}$ is suitably chosen, then the above inequality implies that the expected value of $X$ satisfies $\mathbf{E}[X] \geq (1 - 10^{-4}) \#\mathcal{B}^s$. By Markov's inequality,
\begin{equation}\label{X is good}
\mathbf{P}\Big(X > \frac{99}{100}\#\mathcal{B}^s \Big) \geq 1- \frac{100}{\#\mathcal{B}^s} \mathbf{E}[\#\mathcal{B}^s - X] \geq \frac{99}{100}, 
\end{equation}
which verifies that property i) of the lemma holds with high probability. 




\subsubsection*{Property ii)}

For each $x \in \R^n$ let $M_x$ denote the random variable that counts the number of sets $N_{\rho^{1/2 + \delta_m}}(Z +b_j)$ containing $x$; that is,
\begin{equation*}
M_x(b_1,\ldots,b_N) :=\sum_{j=1}^N \chi_{N_{2\rho^{1/2+\delta_m}}(Z+b_j)}(x).
\end{equation*}
If $B(x, K^2) \in \mathcal{B}^s$, then
\begin{equation*}
\mathbf{E}[M_{x}] = \sum_{j=1}^N \mathbf{P}\big(N_{2\rho^{1/2+\delta_m}}(Z - x)\big) \sim N \frac{2^s}{|B(0, R^{1/2+\delta_m})|}\sim 1. 
\end{equation*}
Now let $C \geq 1$ be a dimensional constant and $Y$ denote the random variable that counts the number of $B(x, K^2) \in  \mathcal{B}^s$ for which $M_x \leq C$. By a two-fold application of Markov's inequality, if $C$ is chosen to be sufficiently large, then 
\begin{equation}\label{Y is good}
\mathbf{P}\Big( Y > \frac{99}{100}\#\mathcal{B}^s   \Big) \geq \mathbf{P}\Big( \#\mathcal{B}^s - \frac{1}{C} \sum_{B(x,K^2) \in \mathcal{B}^s} M_x > \frac{99}{100}\#\mathcal{B}^s   \Big) \geq \frac{99}{100}, 
\end{equation}
which verifies that property ii) of the lemma holds with high probability. \\ 

In view of \eqref{containment is good}, \eqref{X is good} and \eqref{Y is good}, there exists a choice of $\mathfrak{B} \subseteq B(0,2R^{1/2 + \delta_m})$ and a subset $\mathcal{B}^{\prime} \subseteq \mathcal{B}^s$ of cardinality comparable to that of $\mathcal{B}^s$ for which the desired properties i) and ii) hold. Finally, by \eqref{transverse 4}, the inequality \eqref{transverse 5} also holds for the sub-collection $\mathcal{B}^{\prime}$.

\end{proof}




\section{Going from $k$-broad to linear estimates }\label{k-broad to linear section}




\subsection{Applying the Bourgain--Guth method} Theorem~\ref{main theorem} can be deduced as a consequence of the $k$-broad estimates via the method of \cite{Bourgain2011}. The key proposition is as follows.

\begin{proposition}\label{proposition Bourgain Guth} Suppose that for all $K \geq 1$ and all $\varepsilon > 0$ any H\"ormander-type operator $T^{\lambda}$ with reduced positive-definite phase obeys the $k$-broad inequality
\begin{equation}\label{proposition Bourgain Guth 1}
\|T^{\lambda}f\|_{\mathrm{BL}_{k,A}^p(B(0,R))} \lesssim_{\varepsilon} K^{C_{\varepsilon}} R^{\varepsilon} \|f\|_{L^p(B^{n-1})}
\end{equation}
for some fixed $k, A, p, q, C_{\varepsilon}$ and all $R \geq 1$. If  
\begin{equation*}
2 \cdot \frac{2n - k +2}{2n - k} \leq p \leq 2 \cdot \frac{k-1}{k-2},
\end{equation*}
then any H\"ormander-type operator $T^{\lambda}$ with positive-definite phase satisfies 
\begin{equation*}
\|T^{\lambda}f\|_{L^p(B(0,R))} \lesssim_{\phi, \varepsilon} R^{\varepsilon} \|f\|_{L^p(B^{n-1})}. 
\end{equation*}
\end{proposition}

Theorem~\ref{main theorem} is now a direct consequence of Proposition~\ref{proposition Bourgain Guth} and Theorem~\ref{k-broad theorem}. 

\begin{proof}[Proof (of Theorem~\ref{main theorem})] Theorem~\ref{k-broad theorem} implies that for each $2 \leq k \leq n$ the estimate \eqref{proposition Bourgain Guth 1} is valid for all $p \geq \bar{p}(n,k)$. Thus, for each $k$ satisfying the constraint 
\begin{equation}\label{k constraint}
 2 \cdot \frac{2n - k +2}{2n - k} \leq \bar{p}(n,k) = 2\cdot \frac{n+k}{n+k-2}
\end{equation}
one may apply Proposition~\ref{proposition Bourgain Guth} with $\bar{p}(n,k) \leq  p \leq 2 (k-1)/(k-2)$ to obtain a (potentially empty) range of estimates for the linear problem. Since $\bar{p}(n,k)$ is a decreasing function of $k$, the optimal estimate is given by applying Proposition~\ref{proposition Bourgain Guth} as above with $k$ chosen to be as large as possible subject to \eqref{k constraint}. Rearranging \eqref{k constraint} yields $k \leq n/2 + 1$. Defining $k_* := n/2 + 1$ for $n$ even and $k_* := (n+1)/2$ for $n$ odd, one may readily verify that $\bar{p}(n,k_*) \leq 2 \cdot \frac{k_* - 1}{k_*-2}$ and so the linear estimate holds for all $p \geq \bar{p}(n,k_*)$. A simple computation shows that this corresponds to the range of estimates stated in Theorem~\ref{main theorem}. 
 \end{proof}

For contrast, it is noted that there is also a version of Proposition~\ref{proposition Bourgain Guth} which holds without the positive-definite assumption. This can be combined with the multilinear estimates of Bennett--Carbery--Tao \cite{Bennett2006} to prove Theorem~\ref{general main theorem} (this is essentially the argument used in \cite{Bourgain2011}).

\begin{proposition}\label{general proposition Bourgain Guth} Suppose that for all $K \geq 1$ and all $\varepsilon > 0$ any H\"ormander-type operator $T^{\lambda}$ with reduced phase\footnote{The notation of a \emph{reduced} phase under a general signature hypothesis has not been introduced but is almost identical to that used in the positive-definite case. Indeed, the only difference is that the first condition in \eqref{close to identity} is suitably modified, with $I_{n-1}$ replaced with a diagonal matrix with diagonal entries $1$ and $-1$.} obeys the $k$-broad inequality
\begin{equation*}
\|T^{\lambda}f\|_{\mathrm{BL}_{k,A}^p(B(0,R))} \lesssim_{\varepsilon}  K^{C_{\varepsilon}}R^{\varepsilon} \|f\|_{L^p(B^{n-1})}
\end{equation*}
for some fixed $k, A, p, q, C_{\varepsilon}$ and all $R \geq 1$. If  
\begin{equation*}
2 \cdot \frac{n - k +2}{n - k+1} \leq p,
\end{equation*}
then any H\"ormander-type operator $T^{\lambda}$ satisfies 
\begin{equation*}
\|T^{\lambda}f\|_{L^p(B(0,R))} \lesssim_{\phi, \varepsilon} R^{\varepsilon} \|f\|_{L^p(B^{n-1})}. 
\end{equation*}
\end{proposition}

Theorem~\ref{general main theorem} is now a direct consequence of Proposition~\ref{general proposition Bourgain Guth} and the Bennett--Carbery--Tao theorem. 

\begin{proof}[Proof (of Theorem~\ref{general main theorem})] The proof is precisely the same as that of Theorem~\ref{main theorem} above, but with the exponent $2k/(k-1)$ from the Bennett--Carbery--Tao theorem (that is, Theorem~\ref{Bennett Carbery Tao theorem} or, more precisely, the $k$-broad version given by Corollary~\ref{Bennett Carbery Tao corollary}) playing the r\^ole of $\bar{p}(n,k)$.
 \end{proof}

\begin{remark} From the above, the narrow range of exponents in Theorem~\ref{general main theorem} compared with Theorem~\ref{main theorem} can be broadly attributed to:
\begin{enumerate}
    \item The weaker $k$-broad estimates coming from the Bennett--Carbery--Tao theorem compared with Theorem~\ref{k-broad theorem}. One cannot work with stronger $k$-broad estimates than those given by Corollary~\ref{Bennett Carbery Tao corollary} due to the failure of transverse equidistribution in the mixed-signature case.
    \item The more stringent constraints on $p$ in Proposition~\ref{general proposition Bourgain Guth} compared with Proposition~\ref{proposition Bourgain Guth}. These additional constraints arise due to the fact that hyperbolic paraboloids contain linear subspaces, as discussed below. 
\end{enumerate}
 \end{remark}

To establish the main result, Theorem~\ref{main theorem}, it remains to prove Proposition~\ref{proposition Bourgain Guth}. Both Proposition~\ref{proposition Bourgain Guth} and Proposition~\ref{general proposition Bourgain Guth} can be established using very similar arguments: in fact, the proofs differ only at one (crucial) point. To highlight the essential differences between the positive-definite and mixed-signature cases, at the end of this subsection it is indicated how one may adapt the proof of Proposition~\ref{proposition Bourgain Guth} to yield Proposition~\ref{general proposition Bourgain Guth}.

The proof of Proposition~\ref{proposition Bourgain Guth} is an induction on scales argument. The induction quantity is defined as follows.

\begin{definition} For $1 \leq p \leq \infty$ and $R \geq 1$ let $Q_{p}(R)$ denote the infimum over all constants $C$ for which the estimate
\begin{equation*}
\|T^{\lambda}f\|_{L^p(B(0,r))} \leq C\|f\|_{L^p(B^{n-1})}
\end{equation*}
holds for $1 \leq r \leq R$ and all H\"ormander-type operators $T^{\lambda}$ with reduced positive-definite phase and all $\lambda \geq R$. 
\end{definition}

With this definition, the problem is now to show that, under the hypotheses of Proposition~\ref{proposition Bourgain Guth}, one has
\begin{equation}\label{induction quantity bound}
 Q_{p}(R) \lesssim_{\varepsilon} R^{\varepsilon}
\end{equation}
for all $\varepsilon > 0$ and $1 \leq R \leq \lambda$. Indeed, this establishes the linear estimates in the case of reduced phases, and then the arguments of $\S$\ref{Reductions section} extend the result to general H\"ormander-type operators with positive-definite phase.

It is useful to introduce some of the ingredients of the proof of \eqref{induction quantity bound}. Decompose $B(0,R)$ into balls $B_{K^2}$ of radius $K^2$ and consider $\|T^{\lambda}f\|_{L^p(B_{K^2})}$ for some fixed $B_{K^2}$ with centre $\bar{x}$. To bound this quantity one expresses $f$ as a sum of two terms: a ``narrow'' and a ``broad'' term. The narrow term is of the form 
\begin{equation}\label{narrow term}
\sum_{\tau \in V_a \textrm{ for some }a} f_{\tau} ,
\end{equation}
consisting of contributions to $f$ from caps for which $G^{\lambda}(\bar{x};\tau)$ makes a small angle with some member of a family of $(k-1)$-planes. The broad term consists of the contributions to $f$ from all the remaining caps. One may choose the planes $V_1, \dots, V_A$ so that the broad term can be bounded by the $k$-broad inequality from the hypothesis. Thus, the problem is roughly reduced to studying the case where $f$ is of the form \eqref{narrow term}. To treat this case, the first step is to apply an $\ell^p$-decoupling inequality to isolate the contributions of the different $f_{\tau}$.

\begin{theorem}\label{decoupling theorem} Suppose that $T^{\lambda}$ is a H\"ormander-type operator with reduced positive-definite phase. If $V \subseteq \R^n$ is an $m$-dimensional linear subspace, then for $2 \leq p \leq 2m/(m-1)$ and $\delta > 0$ one has
\begin{equation*}
\big\| \sum_{\tau \in V} T^{\lambda} g_{\tau} \big\|_{L^p(B_{K^2})} \lesssim_{\delta} K^{(m-1)(1/2 - 1/p) + \delta} \big(\sum_{\tau \in V} \| T^{\lambda} g_{\tau}\|_{L^p(w_{B_{K^2}})}^p \big)^{1/p}.
\end{equation*}
Here the sums are over all caps $\tau$ for which $\angle(G^{\lambda}(\bar{x},\tau), V) \leq K^{-1}$ where $\bar{x}$ is the centre of $B_{K^2}$ and $w_{B_{K^2}}$ is a rapidly decaying weight of the form of that defined in \eqref{weight function}.
\end{theorem}

This theorem is a variable coefficient generalisation of a decoupling inequality due to Bourgain \cite{Bourgain2013}. It can be established by adapting the argument of \cite{Bourgain2013} using many of the techniques employed in the current article: see also \cite{Beltran}.\footnote{It is remarked that since the decoupling estimate is applied at a small spatial scale $K^2 \ll \lambda^{1/2}$ one can avoid the use of the full statement of Theorem~\ref{decoupling theorem} by appealing to an approximation argument. If one argues in this way, then only Theorem~\ref{decoupling theorem} for extension operators associated to elliptic-type hypersurfaces is required.}  

Summing together the contributions from the various spatial balls $B_{K^2}$, it remains to estimate the decoupled contributions $\|T^{\lambda}f_{\tau}\|_{L^p(B_{R})}$. Since each $f_{\tau}$ has small support, after rescaling one obtains favourable estimates for $\|T^{\lambda}f_{\tau}\|_{L^p(B(0,R))}$ by invoking the induction hypothesis. This is made precise by the following lemma.

\begin{lemma}[Parabolic rescaling]\label{parabolic rescaling lemma} Let $1 \leq R \leq \lambda$ and suppose $f$ is supported on a ball of radius $\rho^{-1}$ where $1 \leq \rho \leq R$. For all $p \geq 2$ and $\delta >0$ one has
\begin{equation*}
\|T^{\lambda}f\|_{L^p(B(0,R))} \lesssim_{\delta} Q_{p}(R) R^{\delta} \rho^{2n/p - (n-1)}\|f\|_{L^p(B^{n-1})}.
\end{equation*}
\end{lemma}

The proof of the parabolic rescaling lemma is based on the changes of variables previously encountered in \S\ref{parabolic rescaling subsection}. For extension operators the argument is simple, consisting of an affine change of variables. In the variable coefficient case some significant additional complications arise; the details are therefore postponed until the following subsection. 

Having introduced the main tools, the proof of Proposition~\ref{proposition Bourgain Guth} follows easily. 

\begin{proof}[Proof (of Proposition~\ref{proposition Bourgain Guth})] It suffices to demonstrate the linear estimate for $p$ satisfying the additional condition
\begin{equation}\label{Lebesgue exponent inequality}
2 \cdot \frac{2n - k +2}{2n - k} < p;
\end{equation}
the result for the remaining value of $p$ then follows immediately by H\"older's inequality. 

Let $\varepsilon > 0$ be given. By hypothesis,
\begin{equation*}
\sum_{\substack{B_{K^2} \in \mathcal{B}_{K^2} \\ B_{K^2} \cap B(0,R) \neq \emptyset}} \min_{V_1, \dots, V_A} \max_{\tau \notin V_a} \int_{B_{K^2}} |T^{\lambda}f_{\tau}|^p \leq C(K, \varepsilon)R^{p\varepsilon/2}\|f\|_{L^p(B^{n-1})}^p
\end{equation*}
where $V_1, \dots, V_A$ are $(k-1)$-planes and the notation $\tau \notin V_a$ signifies that $\angle(G^{\lambda}(\bar{x}, \tau), V_a) > K^{-1}$ for $\bar{x}$ the centre of the corresponding $K^2$-ball $B_{K^2}$. 

For each $B_{K^2}$ fix a choice of $V_1, \dots, V_A$ which achieves the minimum above. Then one may write
\begin{equation*}
\int_{B_{K^2}} |T^{\lambda}f|^p \lesssim K^{O(1)} \max_{\tau \notin V_a} \int_{B_{K^2}} |T^{\lambda}f_{\tau}|^p + \sum_{a=1}^A \int_{B_{K^2}} |\sum_{\tau \in V_a} T^{\lambda}f_{\tau}|^p.
\end{equation*} 
The first term can be estimated using the hypothesised $k$-broad estimate; in particular, 
\begin{equation*}
\int\displaylimits_{B(0,R)} \!\!\!|T^{\lambda}f|^p \lesssim K^{O(1)}C(K, \varepsilon)R^{p\varepsilon/2}\|f\|_{L^p(B^{n-1})}^p + \!\!\!\!\!\!\sum_{\substack{B_{K^2} \in \mathcal{B}_{K^2} \\[2pt] B_{K^2} \cap B(0,R) \neq \emptyset}} \!\!\! \sum_{a=1}^A \,\, \int\displaylimits_{B_{K^2}}\!\! |\sum_{\tau \in V_a} T^{\lambda}f_{\tau}|^p.
\end{equation*} 
It remains to bound the narrow term, where the contributions come from caps whose directions make a small angle with one of planes $V_a$. By Theorem~\ref{decoupling theorem}, for any $\delta' > 0$ one has
\begin{equation*}
\int_{B_{K^2}} | \sum_{\tau \in V_a} T^{\lambda} f_{\tau} \big|^p \lesssim_{\delta'} K^{(k-2)(p/2 - 1) + \delta'} \sum_{\tau \in V_a} \int_{\R^n} | T^{\lambda} f_{\tau}|^p w_{B_{K^2}}
\end{equation*}
for each $1 \leq a \leq A$. Thus, summing over the $a$ and all the relevant balls $B_{K^2}$, one concludes that
\begin{equation*}
\sum_{\substack{B_{K^2} \in \mathcal{B}_{K^2} \\[2pt] B_{K^2} \cap B(0,R) \neq \emptyset}} \!\!\!\!\sum_{a = 1}^A\int_{B_{K^2}}\!\!\! |\sum_{\tau \in V_a} T^{\lambda}f_{\tau}|^p \lesssim_{\delta'} \! K^{(k-2)(p/2 - 1) + \delta'}\!\!\!\!\!\! \sum_{\tau : K^{-1}\mathrm{-cap}}\int_{B(0,2R)}|T^{\lambda}f_{\tau}|^p.
\end{equation*}
Since each $f_{\tau}$ is supported on a $K^{-1}$-cap, the summands appearing in the right-hand expression are amenable to parabolic rescaling. In particular, letting $\delta > 0$ be a small number chosen to satisfy the requirements of the forthcoming argument, Lemma~\ref{parabolic rescaling lemma} implies that
\begin{equation*}
\int_{B(0,2R)}|T^{\lambda} f_{\tau}|^p \lesssim_{\delta} Q_{p}(R)^p R^{\delta} K^{2n - (n-1)p}\|f_{\tau}\|_{L^p(B^{n-1})}^p.
\end{equation*}
Defining
\begin{equation*}
 e(k,p) := (k-2)(1-p/2) - 2n + (n-1)p
\end{equation*}
and combining these estimates,
\begin{equation*}
\int_{B(0,R)} |T^{\lambda}f|^p \leq \big( K^{O(1)}C(K,\varepsilon)R^{p\varepsilon/2} +  C_{\delta,\delta'}Q_{p}(R)^p R^{\delta} K^{-e(k,p) + \delta'}\big) \|f\|_{L^p(B^{n-1})}^p
\end{equation*}  
and so, by definition,
\begin{equation*}
Q_{p}(R)^p \leq K^{O(1)}C(K,\varepsilon)R^{p\varepsilon/2} +  C_{\delta,\delta'}Q_{p}(R)^pR^{\delta} K^{-e(k,p) +\delta'}.
\end{equation*}
Since $p$ satisfies \eqref{Lebesgue exponent inequality}, it follows that $e(k,p) > 0$ and one may choose $\delta' = e(k,p)/2$ so that the $K$ exponent in the right-hand term is negative. Thus, if
 $K = K_0R^{2\delta/e(k,p)}$ for a sufficiently large constant $K_0$, depending only on $\varepsilon$, $\delta$, $p$ and $n$, it follows that
\begin{equation*}
Q_{p}(R)^p \leq K^{O(1)}C(K_0R^{2\delta/e(k,p)}, \varepsilon)R^{p\varepsilon/2} + Q_{p}(R)^p/2.
\end{equation*}
Recall that, by hypothesis, the constant $C(K, \varepsilon)$ arising from the $k$-broad estimate grows at most polynomially in $K$. Consequently, one may choose $\delta$ to be small enough (depending only on admissible parameters) so that $Q_{p}(R) \lesssim_{\varepsilon} R^{\varepsilon}$, as required. 
\end{proof}

As mentioned above, this argument can be adapted to study the case of general H\"ormander-type operators (with potentially mixed signature) to prove Proposition~\ref{general proposition Bourgain Guth}. The induction quantity $Q_{p}(R)$ is defined as before, but now the supremum is taken over the larger class of all H\"ormander-type operators $T^{\lambda}$ which are in a suitably reduced form. The proof of the parabolic rescaling lemma then extends to this setting \emph{mutatis mutandis}. The key differences arise in the decoupling inequality. In particular, Theorem~\ref{decoupling theorem} does not hold at the required level of generality. To see why this is so, consider the example of the extension operator $E$ associated to (a compact piece of) the hyperbolic paraboloid given by the graph of $h(\omega) := \omega_1\omega_2$. If $V := \{x \in \R^3 : x_1 = 0\}$ and $G$ is the relevant Gauss map, then
\begin{equation*}
 S_{\omega} := \{ \omega \in B^{2} : G(\omega) \in V\}  = \{\omega \in B^2 : \omega_2 = 0\}. 
\end{equation*}
Thus, the $E g_{\tau}$ for $\tau \in V$ are (distributionally) Fourier supported in a neighbourhood of the $\xi_1$-axis (which is, in particular, a curve of everywhere zero curvature); here the notation $\tau \in V$ is used to denote that $\angle(G(\tau), V) \leq K^{-1}$, consistent with the non-standard notion of containment used in \S\ref{k-broad subsection} and Theorem~\ref{decoupling theorem}. As is well-known, in the absence of curvature, no non-trivial decoupling estimates are possible.\footnote{It is remarked that non-trivial $\ell^p$-decoupling estimates are known hold for the \textit{full} hyperbolic paraboloid: see \cite{BD2017}. The problem here arises because one is forced to consider decoupling along the lower dimensional submanifold $S_{\omega} \times \{0\}$.}

The following simple result acts as a substitute for Theorem~\ref{decoupling theorem}.

\begin{lemma}[Bourgain--Guth \cite{Bourgain2011}]\label{general decoupling theorem} Suppose that $T^{\lambda}$ is a H\"ormander-type operator with reduced phase.  If $V \subseteq \R^n$ is an $m$-dimensional linear subspace, then for all $p \geq 2$ and $\delta > 0$ one has
\begin{equation*}
\big\| \sum_{\tau \in V}T^{\lambda} g_{\tau} \big\|_{L^p(B_{K^2})} \lesssim_{\delta} K^{(m-1)(1 - 2/p) + \delta} \big(\sum_{\tau \in V} \|T^{\lambda} g_{\tau}\|_{L^p(w_{B_{K^2}})}^p \big)^{1/p}.
\end{equation*}
Here the sums are over all caps $\tau$ for which $\angle(G^{\lambda}(\bar{x},\tau), V) \leq K^{-1}$ where $\bar{x}$ is the centre of $B_{K^2}$.
\end{lemma}

This lemma provides much weaker estimates than those guaranteed by the $\ell^p$-decoupling theorem in the positive-definite case: here the $K$ exponent is larger by a factor of 2 than that appearing in Theorem~\ref{decoupling theorem}. The proof is implicitly contained in the proof of Theorem 4 in \cite[\S5]{Bourgain2011}. 

To prove Proposition~\ref{general proposition Bourgain Guth} one proceeds as in the proof of Proposition~\ref{proposition Bourgain Guth}, first decomposing $B(0,R)$ into balls of radius $K^2$. For each such ball the broad term is bounded using the hypothesised $k$-broad estimate whilst the narrow term is bounded by Lemma~\ref{general decoupling theorem} together with the induction hypothesis (via parabolic rescaling). The larger exponent incurred by Lemma~\ref{general decoupling theorem} propagates through the argument until one arrives at the estimate
\begin{equation*}
Q_{p}(R)^p \leq K^{O(1)} C(K,\varepsilon)R^{p\varepsilon/2} +  C_{\delta,\delta'}Q_{p}(R)^p R^{\delta}K^{(k-2)(p - 2) + 2n - (n-1)p +\delta'}.
\end{equation*}
In order to close the induction, once again one must ensure that the $K$ exponent is negative. By choosing $\delta'$ appropriately, this is possible if $p$ satisfies the stronger hypothesis $p > 2(n - k +2)/(n - k+1)$, which is precisely the condition featured in the statement of Proposition~\ref{general proposition Bourgain Guth}.




\subsection{Proof of Lemma~\ref{parabolic rescaling lemma}} It remains to establish the parabolic rescaling lemma, which is achieved by adapting arguments implicit in \cite[\S5]{Bourgain2011}. As mentioned in the previous section,  some additional complications arise in the case of H\"ormander operators (as opposed to the extension case) and the proof of the parabolic rescaling is slightly involved. 

It will be useful to work with the following discrete reformulation of the main estimate for the operator $T^{\lambda}$. 

\begin{lemma}\label{discrete formulation} If $\mathcal{D}$ is a maximal $R^{-1}$-separated discrete subset of $\Omega$, then 
\begin{equation}\label{discrete formulation 1}
\big\| \sum_{\omega_{\theta} \in \mathcal{D}} e^{2 \pi i \phi^{\lambda}(\,\cdot\,;\omega_{\theta})}F(\omega_{\theta})\big\|_{L^p(B(0,R))} \lesssim Q_{p}(R)R^{(n-1)/p'} \|F\|_{\ell^p(\mathcal{D})}
\end{equation}
for all $F\colon \mathcal{D} \to \C$. 
\end{lemma}

\begin{proof} Fix $\psi \in C^{\infty}_c(\R^{n-1})$ supported on $B(0,2)$ which satisfies $0 \leq \psi \leq 1$ and $\psi(\omega) = 1$ for all $\omega \in B^{n-1}$ and for each $\omega_{\theta} \in \mathcal{D}$ define $\psi_{\theta}(\omega) := \psi(10R(\omega - \omega_{\theta}))$. Thus, for all $x \in B(0,R)$ the exponential sum appearing in the left-hand side of \eqref{discrete formulation 1} can be expressed as a constant multiple of
\begin{equation*}
R^{n-1}  \int_{\R^{n-1}}e^{2\pi i \phi^{\lambda}(x;\omega)}a^{\lambda}(x;\omega)\big[\tilde{\psi}(x/R)\sum_{\omega_{\theta} \in \mathcal{D}} e^{-2 \pi i \lambda\Omega_{\theta}(x/\lambda;\omega)} F(\omega_{\theta}) \psi_{\theta}(\omega)\big]\,\ud \omega
\end{equation*}
where $\Omega_{\theta}(x;\omega) := \phi(x;\omega)-\phi(x;\omega_{\theta})$,  $\tilde{\psi} \in C^{\infty}_c(\R^{n})$ is a function of $n$ variables which enjoys properties similar to those of $\psi$ and $a^{\lambda}$ is a suitable choice of amplitude. Since
\begin{equation*}
 \sup_{\omega \in \mathrm{supp}\,\psi_{\theta}}|\partial_{x}^{\beta}\Omega_{\theta}(x; \omega)| \lesssim_{\beta} R^{-1}|x| \qquad \textrm{for all $\beta \in \N_0^n$ and $x \in X$,}
\end{equation*}
one may safely remove the $\lambda\Omega_{\theta}(x/\lambda;\omega)$ term from the phase. More precisely, by expanding $\tilde{\psi}(x)e^{-2 \pi i \lambda\Omega_{\theta}(Rx/\lambda; \omega)}$ as a Fourier series in the variable $x$, one can show that
\begin{equation*}
\big|\sum_{\omega_{\theta} \in \mathcal{D}} e^{2 \pi i \phi^{\lambda}(x;\omega_{\theta})} F(\omega_{\theta})\big| \lesssim R^{n-1}  \sum_{k \in \Z^n}(1+|k|)^{-(n+1)} |T^{\lambda}f_k(x)|
\end{equation*}
where $T^{\lambda}$ is a H\"ormander-type operator with phase $\phi^{\lambda}$ and
\begin{equation*}
f_k(\omega) := \sum_{\omega_{\theta} \in \mathcal{D}} F(\omega_{\theta})c_{k,\theta}(\omega)\psi_{\theta}(\omega)
\end{equation*}
for some choice of smooth functions $c_{k,\theta}$ satisfying the uniform bound $\|c_{k,\theta}\|_{L^{\infty}(B^{n-1})} \lesssim 1$. Thus, by the definition of $Q_{p}(R)$ it follows that
\begin{equation*}
\big\| \sum_{\omega_{\theta} \in \mathcal{D}} e^{2 \pi i \phi^{\lambda}(\,\cdot\,;\omega_{\theta})}F(\omega_{\theta})\big\|_{L^p(B(0,R))} \lesssim Q_{p}(R)R^{n-1} \sum_{k \in \Z^n} (1+|k|)^{-(n+1)} \|f_k\|_{L^p(B^{n-1})}
\end{equation*}
and, since the $\psi_{\theta}$ are supported on pairwise disjoint sets,
\begin{equation*}
\|f_k\|_{L^p(B^{n-1})} \lesssim R^{-(n-1)/p} \big(\sum_{\omega_{\theta} \in \mathcal{D}} |F(\omega_{\theta})|^p\big)^{1/p},
\end{equation*}
concluding the proof.
\end{proof}

\begin{proof}[Proof (of Lemma~\ref{parabolic rescaling lemma})] Recall, the phase of $T^{\lambda}$ is given by $\phi^{\lambda}(x;\omega) := \lambda\phi(x/\lambda;\omega)$ where
\begin{equation}\label{nice phase}
 \phi(x;\omega) = \langle x',\omega\rangle + x_nh(\omega) + \mathcal{E}(x;\omega).
\end{equation}
Let $B(\bar{\omega}, \rho^{-1})$ be a ball supporting $f$, where $\bar{\omega} \in B^{n-1}$. If $ \tilde{T}^{\lambda/\rho^2}$ denotes the parabolically rescaled operator defined in \eqref{rescaled operator}, with rescaled phase function
\begin{equation}\label{rescaled phase}
\tilde{\phi}(x; \omega) := \langle x', \omega \rangle + x_n\tilde{h}(\omega) + \tilde{\mathcal{E}}(x; \omega),
\end{equation}
then it follows that 
\begin{equation*}
 \|T^{\lambda}f\|_{L^p(B(0,R))} \lesssim \rho^{(n+1)/p} \|\tilde{T}^{\lambda/\rho^2}\tilde{f}\|_{L^p(\tilde{D}_R)}
\end{equation*}
where now $\tilde{D}_R$ is an ellipse with principal axes parallel to the co-ordinate axes and dimensions $O(R/\rho) \times \dots \times O(R/\rho) \times O(R/\rho^2)$ and $\tilde{f}(\omega) := \rho^{-(n-1)}f(\bar{\omega}+\rho^{-1}\omega)$. Since 
\begin{equation*}
 \|\tilde{f}\|_{L^p(B^{n-1})} = \rho^{-(n-1) + (n-1)/p}\|f\|_{L^p(B^{n-1})},
\end{equation*}
given $\delta > 0$, the problem is to show that
\begin{equation*}
\|\tilde{T}^{\lambda/\rho^2}\tilde{f}\|_{L^p(\tilde{D}_R)} \lesssim_{\delta} Q_{p}(R) R^{\delta} \|\tilde{f}\|_{L^p(B^{n-1})}.
\end{equation*}
Observe that the phase $\tilde{\phi}$ defined in \eqref{rescaled phase} is also positive-definite and of reduced form. To lighten the notation, consider once again a general positive-definite reduced phase $\phi$ as in \eqref{nice phase} and let $T^{\lambda}$ is a H\"ormander-type operator associated to $\phi^{\lambda}$. It suffices to show 
\begin{equation*}
\|T^{\lambda}f\|_{L^p(D_{\mathbf{R}})} \lesssim_{\delta} Q_{p}(R) R^{\delta} \|f\|_{L^p(B^{n-1})}
\end{equation*}
for all $1 \ll R \leq R' \leq \lambda$ and $\delta > 0$ where
\begin{equation*}
 D_{\mathbf{R}} := \Big\{ x \in \R^n : \Big(\frac{|x'|}{R'}\Big)^2 + \Big(\frac{|x_n|}{R}\Big)^2 \leq 1 \Big\}
\end{equation*}
is an ellipse. Of course, if $R = R'$, then this inequality is immediate from the definition of $Q_{p}(R)$. 

Cover $B^{n-1}$ by a collection of essentially disjoint $R^{-1}$-caps $\theta$ and decompose $f$ as $f = \sum_{\theta} f_{\theta}$. Define
\begin{equation*}
T^{\lambda}_{\theta}f(x) :=e^{-2\pi i \phi^{\lambda}(x;\omega_{\theta})} T^{\lambda} f(x)
\end{equation*}
so that
\begin{equation*}
T^{\lambda}f(x) = \sum_{\theta : R^{-1}\mathrm{-cap}} e^{2\pi i \phi^{\lambda}(x;\omega_{\theta})} T^{\lambda}_{\theta} f_{\theta}(x).
\end{equation*}
Fix $\delta > 0$ to be sufficiently small for the purposes of the forthcoming argument. Each $f_{\theta}$ is supported on an $R^{-1}$-ball and is therefore, of course, supported on an $R^{-1+\delta}$-ball. Since $(R^{-1 + \delta})^{-1} \leq \lambda^{1 - \delta}$, one may argue as in the proof of Lemma~\ref{locally constant lemma} to deduce that
\begin{equation*}
T^{\lambda}_{\theta}f_{\theta}(x) = T^{\lambda}_{\theta}f_{\theta}\ast \eta_{R^{1-\delta}}(x) + \mathrm{RapDec}(\lambda)\|f\|_{L^2(B^{n-1})}
\end{equation*}
for some choice of smooth, rapidly decreasing function $\eta$ such that $|\eta|$ admits a smooth, rapidly decreasing majorant $\zeta \colon \R^n \to [0,\infty)$ which is locally constant at scale 1.  In particular, it follows that
\begin{equation}\label{parabolic rescaling 1}
\zeta_{R^{1-\delta}}(x) \lesssim R^{\delta} \zeta_{R^{1-\delta}}(y) \quad \textrm{if $|x-y| \lesssim R$.}
\end{equation}

Cover $D_{\mathbf{R}}$ by finitely-overlapping $R$-balls and let $B_{R}$ be some member of this cover. Combining the above observations, if $\bar{x}$ denotes the centre of $B_{R}$ and $z \in B(0,R)$, then 
\begin{equation*}
 |T^{\lambda}f(\bar{x} + z)| \lesssim R^{\delta}  \int_{\R^n} \big|\sum_{\theta : R^{-1}\mathrm{-cap}} e^{2\pi i \tilde{\phi}^{\lambda}(z;\omega_{\theta})} e^{2\pi i \phi^{\lambda}(\bar{x};\omega_{\theta})}T^{\lambda}_{\theta}f_{\theta}(y)\big| \zeta_{R^{1-\delta}}(\bar{x}-y) \,\ud y,
\end{equation*}
where $ \tilde{\phi}^{\lambda}(z;\omega_{\theta}) := \phi^{\lambda}(\bar{x}+z;\omega_{\theta}) - \phi^{\lambda}(\bar{x};\omega_{\theta})$. Taking the $L^p$-norm in $z$ it follows from Minkowski's inequality that $ \|T^{\lambda}f\|_{L^p(B_{R})}$ is dominated by
\begin{equation*}
 R^{\delta} \int_{\R^n} \big\|\sum_{\theta : R^{-1}\mathrm{-cap}} e^{2\pi i \tilde{\phi}^{\lambda}(\,\cdot\,;\omega_{\theta})} e^{2\pi i \phi^{\lambda}(\bar{x};\omega_{\theta})}T^{\lambda}_{\theta}f_{\theta}(y)\big\|_{L^p(B(0,R))} \zeta_{R^{1-\delta}}(\bar{x}-y) \,\ud y.
\end{equation*}
By Lemma~\ref{discrete formulation} the $L^p$-norm appearing in the above integrand is bounded by a constant multiple of
\begin{equation*}
Q_{p}(R)R^{(n-1)/p'} \big(\sum_{\theta : R^{-1}\mathrm{-cap}} |T^{\lambda}_{\theta}f_{\theta}(y)|^p \big)^{1/p}. 
\end{equation*}
Applying H\"older's inequality and the locally-constant property \eqref{parabolic rescaling 1}, one deduces that
\begin{equation*}
\|T^{\lambda}f\|_{L^p(B_{R})} \lesssim Q_{p}(R)R^{(n-1)/p' + O(\delta)}   \big(\int_{\R^n}\sum_{\theta : R^{-1}\mathrm{-cap}} |T^{\lambda}f_{\theta}(\bar{x} + z - y)|^p \zeta_{R^{1-\delta}}(y) \,\ud y \big)^{1/p}
\end{equation*}
for all $z \in B(0,R)$. By raising both sides of this estimate to the $p$th power, averaging in $z$ and summing over all balls $B_{R}$ in the covering, it follows that $\|T^{\lambda}f\|_{L^p(D_{\mathbf{R}})}$ is dominated by
\begin{equation*}
 Q_{p}(R)R^{(n-1)/p'-n/p +O(\delta)}   \big(\int_{\R^n}\sum_{\theta : R^{-1}\mathrm{-cap}} \|T^{\lambda}f_{\theta}\|_{L^p(D_{\mathbf{R}} - y)}^p \zeta_{R^{1-\delta}}(y) \,\ud y \big)^{1/p}.
\end{equation*}
Observe that, by H\"ormander's theorem (Lemma~\ref{Hormander L2 again}) and H\"older's inequality, one has
\begin{equation*}
\|T^{\lambda}f_{\theta}\|_{L^2(D_{\mathbf{R}} - y)} \lesssim R^{-(n-1)(1/2 - 1/p) + 1/2}\|f_{\theta}\|_{L^p(B^{n-1})}.
\end{equation*}
On the other hand, the trivial estimate
\begin{equation*}
\|T^{\lambda}f_{\theta}\|_{L^{\infty}(D_{\mathbf{R}} - y)} \lesssim R^{-(n-1)/p'}\|f_{\theta}\|_{L^p(B^{n-1})}
\end{equation*}
holds, simply due to H\"older's inequality. Combining the above, 
\begin{equation*}
\|T^{\lambda}f_{\theta}\|_{L^p(D_{\mathbf{R}} - y)} \lesssim R^{-(n-1)/p' + n/p}\|f_{\theta}\|_{L^p(B^{n-1})}.
\end{equation*}
The desired inequality is now immediate. 
\end{proof}




\section{An $\varepsilon$-removal lemma}\label{Epsilon removal section}

The $\lambda^{\varepsilon}$-loss in the linear estimates of Theorems~\ref{general main theorem} and~\ref{main theorem} can be removed away from the endpoint by an appeal to an $\varepsilon$-removal lemma of the type introduced in \cite{Tao1999} (see also \cite{Bourgain2011, Tao1998a}). The precise form of the required lemma does not appear in the literature, but it can be deduced by a minor modification of an argument from \cite{Tao1999}. For completeness, the details are given presently. 

Suppose $T^{\lambda}$ is a H\"ormander-type operator with associated phase function $\phi^{\lambda}$ (note that here no additional positive-definite assumption is assumed). Let $\bar{p} \geq 2$ and suppose for all $\varepsilon > 0$ the estimate
\begin{equation}\label{epsilon removed hypothesis}
 \|T^{\lambda}f\|_{L^p(B_R)} \lesssim_{\varepsilon, \phi, a} R^{\varepsilon}\|f\|_{L^p(B^{n-1})}
\end{equation}
holds for all $p \geq \bar{p}$, all $R$-balls $B_R$ for $1 \leq R \leq  \lambda$ and any choice of amplitude function. Under this hypothesis, one wishes to show that the global estimate
\begin{equation}\label{epsilon removed estimate}
 \|T^{\lambda}f\|_{L^p(\R^n)} \lesssim_{\phi, a} \|f\|_{L^p(B^{n-1})}
\end{equation}
is valid for all $p > \bar{p}$. 

\begin{definition}[Tao \cite{Tao1999}] Let $R \geq 1$. A collection $\{B(x_j, R)\}_{j=1}^N$ of $R$-balls in $\R^n$ is \emph{sparse} if the centres $\{x_1, \dots, x_N\}$ are $(RN)^{\bar{C}}$-separated. Here $\bar{C} \geq 1$ is a fixed constant, chosen large enough for the purposes of the proof.  
\end{definition}

Following \cite{Tao1999}, the first step towards establishing \eqref{epsilon removed estimate} is to reduce the problem to proving estimates for $T^{\lambda}$ over sparse families of balls. 

\begin{lemma}\label{epsilon reduction lemma}
To prove \eqref{epsilon removed estimate} for all $p > \bar{p}$ it suffices to show that for all $\varepsilon >0$ the estimate
 \begin{equation}\label{sparse estimate}
 \|T^{\lambda}f\|_{L^{\bar{p}}(S)} \lesssim_{\varepsilon, \phi, a} R^{\varepsilon}\|f\|_{L^{\bar{p}}(B^{n-1})}
\end{equation}
 holds whenever $R \geq 1$ and $S \subseteq \R^n$ is a union of $R$-balls belonging to a sparse collection, for any choice of amplitude function. 
\end{lemma}

The key step in the proof of Lemma~\ref{epsilon reduction lemma} is the following covering lemma.

\begin{lemma}[Tao \cite{Tao1998a, Tao1999}] Suppose $E \subseteq \R^n$ is a finite union of 1-cubes and $N \geq 1$. Define the radii $R_j$ inductively by
\begin{equation*}
 R_0 := 1, \quad R_{j} := R_{j-1}^{\bar{C}}|E|^{\bar{C}} \textrm{ for $1 \leq j \leq N-1$.}
\end{equation*}
Then for each $0 \leq j \leq N-1$ there exists a family of sparse collections $(\mathcal{B}_{j,\alpha})_{\alpha \in A_j}$ of balls of radius $R_j$ such that the index sets $A_k$ have cardinality $O(|E|^{1/N})$ and
\begin{equation*}
 E \subseteq \bigcup_{j=0}^{N-1} \bigcup_{\alpha \in A_j} S_{j, \alpha}
\end{equation*}
where $S_{j, \alpha}$ is the union of all the balls belonging to the family $\mathcal{B}_{j,\alpha}$.
\end{lemma}

\begin{proof}[Proof (of Lemma~\ref{epsilon reduction lemma})] Let $E \subseteq \R^n$ be a finite union of 1-cubes. For $N \geq 1$, the covering lemma together with the hypothesis \eqref{sparse estimate} imply that
\begin{equation*}
 \|T^{\lambda}f\|_{L^{\bar{p}}(E)} \lesssim_{\varepsilon, \phi, a} N|E|^{1/N + \varepsilon \bar{C}^N}\|f\|_{L^{\bar{p}}(B^{n-1})}.
\end{equation*}
Choosing $N \sim \log(1/\varepsilon)$, it follows that
\begin{equation*}
 \|T^{\lambda}f\|_{L^{\bar{p}}(E)} \lesssim_{\varepsilon, \phi, a} |E|^{\bar{C}/\log(1/\varepsilon)}\|f\|_{L^{\bar{p}}(B^{n-1})}.
\end{equation*}
It will be convenient to work with the dual operator 
\begin{equation*}
 T^*g(\omega) := \int_{\R^n} e^{-2\pi i \phi^{\lambda}(x;\omega)} a^{\lambda}(x; \omega) g(x)\,\ud x
\end{equation*}
so that the above estimate can be reformulated as 
\begin{equation}\label{dual epsilon removed estimate}
 \|T^*g\|_{L^{\bar{p}'}(B^{n-1})} \lesssim_{\varepsilon, \phi, a} |E|^{\bar{C}/\log(1/\varepsilon)}\|g\|_{L^{\bar{p}'}(E)}
\end{equation}
for $g$ supported on the set $E$.

Fix $p > \bar{p}$ and $\tau \in [-1/2, 1/2]^n$. Suppose that $g \in L^{p'}(\R^n)$ satisfies $\|g\|_{L^{p'}(\R^n)} = 1$ and is constant on the mesh of 1-cubes centred on points of the lattice $\tau + \Z^n$. Form a level set decomposition of $g$ by writing $g = \sum_{k \in \Z} g_k$ where $g_k := g\chi_{E_k}$ for 
\begin{equation*}
E_k := \big\{ x \in \R^n : 2^{-k} \leq |g(x)| < 2^{-k+1} \big\}.
\end{equation*}
Chebyshev's inequality implies that $|E_k| \leq 2^{kp'}$ for all $k \in \Z$. Furthermore, each set $E_k$ is a union of 1-cubes and therefore if $E_k \neq \emptyset$, then $|E_k| \geq 1$. Combining these observations, one deduces that $E_k = \emptyset$ for all $k < 0$. Since $g_k$ is supported on $E_k$, one may apply \eqref{dual epsilon removed estimate} to conclude that 
\begin{equation}\label{reduced epsilon removed estimate}
 \|T^*g_k\|_{L^{\bar{p}'}(B^{n-1})} \lesssim_{\varepsilon, \phi, a} |E_k|^{\bar{C}/\log(1/\varepsilon)} \|g_k\|_{L^{\bar{p}'}(\R^n)}.
\end{equation}
Using a simple base-times-height estimate, the right-hand side of \eqref{reduced epsilon removed estimate} can be bounded by (a constant multiple of)
\begin{equation*}
 2^{-k}|E_k|^{\bar{C}/\log(1/\varepsilon) + 1/\bar{p}'} \lesssim 2^{-k(1 - \bar{C}p'/\log(1/\varepsilon) - p'/\bar{p}')}.
\end{equation*}
Since $p' < \bar{p}'$, by choosing $\varepsilon$ sufficiently small one can ensure that the right-hand exponent is negative and therefore 
\begin{equation*}
 \|T^*g\|_{L^{p'}(B^{n-1})} \lesssim \|T^*g\|_{L^{\bar{p}'}(B^{n-1})} \leq \sum_{k \geq 0}  \|T^*g_k\|_{L^{\bar{p}'}(B^{n-1})} \lesssim_{\phi, a} 1 = \|g\|_{L^{p'}(\R^n)}.
\end{equation*}
This establishes the dual of the desired estimate \eqref{epsilon removed estimate} under the additional hypothesis that the function $g$ is constant on 1-cubes.

It remains to remove the condition that $g$ is constant on 1-cubes. The key observation is that this special case of \eqref{epsilon removed estimate} implies the discrete inequality 
\begin{equation}\label{discrete epsilon removal}
 \big\| \sum_{\sigma \in \Z^n} e^{-2\pi i \phi^{\lambda}(\sigma + \tau;\,\cdot\,)} a^{\lambda}(\sigma + \tau; \,\cdot\,) G(\sigma) \big\|_{L^{p'}(B^{n-1})} \lesssim_{\phi, a} \|G\|_{\ell^{p'}(\Z^n)}
\end{equation}
for all $G \in \ell^{p'}(\Z^n)$ and $\tau \in [-1/2, 1/2]^n$. Indeed, once \eqref{discrete epsilon removal} is established, taking $g \in L^{p'}(\R^n)$ belonging to a suitable \emph{a priori} class and applying Minkowski's inequality one deduces that
\begin{equation*}
 \|T^*g\|_{L^{p'}(B^{n-1})} \leq \int_{[-1/2, 1/2]^n} \big\| \sum_{\sigma \in \Z^n} e^{-2\pi i \phi^{\lambda}(\sigma + \tau;\,\cdot\,)} a^{\lambda}(\sigma + \tau; \,\cdot\,) g(\sigma + \tau) \big\|_{L^{p'}(B^{n-1})} \,\ud \tau.
\end{equation*}
Combining this with \eqref{discrete epsilon removal} and H\"older's inequality yields \eqref{dual epsilon removed estimate}. 

Thus, the problem is now reduced to proving \eqref{discrete epsilon removal}. Fix $G \in \ell^{p'}(\Z^n)$ and define 
\begin{equation*}
\tilde{g}(x) := \sum_{\sigma \in \Z^n} G(\sigma)\chi(x - \sigma - \tau)
\end{equation*}
where $\chi$ is the characteristic function of $[-1/2,1/2]^n$. Since $\tilde{g}$ is constant on 1-cubes, one is free to apply \eqref{dual epsilon removed estimate} to this function. In particular, let $\tilde{T}^*$ be the dual of a H\"ormander-type operator with phase $\phi^{\lambda}$ and amplitude $\tilde{a}^{\lambda}$ where
\begin{equation*}
\tilde{a}(x;\omega) := \Big(\prod_{j=1}^n \frac{\sin \pi (\partial_{x_j}\phi)(x;\omega)}{\pi  (\partial_{x_j}\phi)(x;\omega)} \Big)^{-1}a_0(x;\omega)
\end{equation*}
 for $a_0$ a smooth amplitude which is supported on $X \times \Omega$ and satisfies $a_0(x;\omega) = 1$ for $(x; \omega) \in \mathrm{supp}\,a$. By the usual reductions (see $\S$\ref{Reductions section}) one may assume from the outset that $|(\partial_{x_j}\phi)(x;\omega)| \leq 1/2$ for $(x;\omega) \in X \times \Omega$ and $1 \leq j \leq n$ and hence $\tilde{a}$ is a well-defined, smooth function. Thus, the estimate 
\begin{equation*}
\| \tilde{T}^*\tilde{g}\|_{L^{p'}(B^{n-1})} \lesssim_{\phi, a} \|\tilde{g}\|_{L^{p'}(\R^n)} 
\end{equation*}
holds, which can be rewritten as
\begin{equation}\label{almost discrete epsilon removal}
\| \sum_{\sigma \in \Z^n} e^{-2\pi i \phi^{\lambda}(\sigma + \tau;\,\cdot\,)} (A_{\lambda})^{\lambda}(\sigma + \tau; \,\cdot\,) G(\sigma)\|_{L^{p'}(B^{n-1})} \lesssim_{\phi, a} \|G\|_{\ell^{p'}(\Z^n)}
\end{equation}
where
\begin{equation*}
 A_{\lambda}(x;\omega)  := \int_{[-1/2,1/2]^n} e^{-2\pi i\lambda(\phi(x + y/\lambda; \omega) - \phi(x; \omega))}\tilde{a}(x + y/\lambda; \omega)\,\ud y. 
\end{equation*}
Note that \eqref{almost discrete epsilon removal} is almost the desired expression \eqref{discrete epsilon removal} except for the disparity between the amplitude functions. To deal with this slight technicality, observe that, since
\begin{equation*}
\lim_{\lambda \to \infty} A_{\lambda}(x;\omega) = a_0(x;\omega) \quad \textrm{uniformly,}
\end{equation*}
one may assume that $\lambda$ is sufficiently large so that $|A_{\lambda}(x;\omega)| \gtrsim 1$ for all $(x;\omega) \in \mathrm{supp}\,a$. Furthermore, by applying the mean value theorem to the phase,
\begin{equation*}
    \|\partial^{\alpha}_xA_{\lambda}\|_{L^{\infty}(\R^n \times \R^{n-1})} \lesssim_{\alpha, \phi, a} 1 \qquad \textrm{for all $\alpha \in \N_0^n$,}
\end{equation*}
the important observation here being that the derivatives are independent of $\lambda$. Thus, the expression appearing in the norm on the left-hand side of \eqref{discrete epsilon removal} is given by
\begin{equation*}
\sum_{\sigma \in \Z^n} e^{-2\pi i \phi^{\lambda}(\sigma + \tau;\omega)} (A_{\lambda})^{\lambda}(\sigma + \tau; \omega) G(\sigma)(\rho_{\lambda})^{\lambda}(\sigma + \tau;\omega)
\end{equation*}
where the ratio $\rho_{\lambda}(x;\omega) := a(x;\omega)A_{\lambda}(x;\omega)^{-1}$ satisfies
\begin{equation*}
    \|\partial^{\alpha}_x\rho_{\lambda}\|_{L^{\infty}(\R^n \times \R^{n-1})} \lesssim_{\alpha, \phi, a} 1 \qquad \textrm{for all $\alpha \in \N_0^n$.}
\end{equation*}
Taking a Fourier series expansion of $\rho_{\lambda}$ in the $x$ variable and using repeated integration-by-parts to estimate the Fourier coefficients, it follows that
\begin{equation*}
\rho_{\lambda}(x;\omega) = \sum_{k \in \Z^n} (1+ |k|)^{-(n+1)} c_{\lambda, k}(\omega) e^{2\pi i \langle x, k \rangle} 
\end{equation*}
where the $c_{\lambda, k}$ are bounded functions, uniformly in $\lambda$ and $k$ (they do, however, depend on $n$, $\phi$ and $a$). One may therefore bound the left-hand side of \eqref{discrete epsilon removal} by a $(1+|k|)^{-(n+1)}$-weighted sum of the left-hand side of \eqref{almost discrete epsilon removal} applied to modulated versions of $G$. Estimating each summand using \eqref{almost discrete epsilon removal} and summing in $k$ concludes the proof.  
\end{proof}

Given the above reduction, it remains to establish the estimates for $T^{\lambda}$ over sparse collections of $R$-balls. 

\begin{lemma} Under the hypothesis \eqref{epsilon removed hypothesis}, if $p \geq \bar{p}$, then the estimate
 \begin{equation*}
 \|T^{\lambda}f\|_{L^p(S)} \lesssim_{\varepsilon, \phi, a} R^{\varepsilon}\|f\|_{L^p(B^{n-1})}
\end{equation*}
 holds for all $\varepsilon >0$ whenever $S \subseteq \R^n$ is a union of $R$-balls belonging to a sparse collection. 
\end{lemma}

\begin{proof} The proof uses a crude form of wave packet analysis and has much in common with the arguments described in $\S$\ref{Wave packet section}. Let $\{B(x_j, R)\}_{j=1}^N$ be the sparse collection of balls whose union is the set $S$. Clearly it suffices to assume that $R \ll \lambda$ and that all the $B(x_k,R)$ intersect the $x$-support of $a^{\lambda}$. Furthermore, letting $c_{\mathrm{diam}} > 0$ be a small constant chosen to satisfy the requirements of the forthcoming argument, by applying a partition of unity one may assume that $\mathrm{diam}\,X < c_{\mathrm{diam}}$ and so
\begin{equation}\label{not too bad separation}
\frac{|x_{j_1} - x_{j_2}|}{\lambda} \lesssim c_{\mathrm{diam}} \qquad \textrm{for all $1 \leq j_1, j_2 \leq N$.}
\end{equation}

Fix $\eta \in C^{\infty}(\R^{n-1})$ satisfying $0 \leq \eta \leq 1$, $\mathrm{supp}\,\eta \in B^{n-1}$ and $\eta(z) = 1$ for all $z \in B(0,1/2)$. For $R_1 := CNR$, where $C \geq 1$ is a large constant, define $\eta_{R_1}(z) := \eta(z/R_1)$. Further, let $\psi \in C^{\infty}_c(\R^{n-1})$ satisfy $0 \leq \psi \leq 1$, $\mathrm{supp}\,\psi \subset \Omega$ and $\psi(\omega) = 1$ for $\omega$ belonging to the $\omega$-support of $a^{\lambda}$. Fix $1 \leq j \leq N$ and write
\begin{equation*}
e^{2 \pi i \phi^{\lambda}(x_j;\,\cdot\,)}\psi f = P_jf + \big(e^{2 \pi i \phi^{\lambda}(x_j;\,\cdot\,)}\psi f - P_jf) =: P_jf + f_{j,\infty}
\end{equation*}
where $P_jf := \hat{\eta}_{R_1} \ast [e^{2 \pi i \phi^{\lambda}(x_j;\,\cdot\,)}\psi f]$. If one defines
\begin{equation*}
\mathrm{Err}(x) := \int_{\R^{n-1}} e^{2\pi i (\phi^{\lambda}(x;\omega) - \phi^{\lambda}(x_j;\omega))}a^{\lambda}(x;\omega)f_{j,\infty}(\omega)\,\ud \omega,
\end{equation*}
then it follows that
\begin{equation*}
T^{\lambda}f(x) = T^{\lambda}[e^{-2 \pi i \phi^{\lambda}(x_j;\,\cdot\,)}P_jf](x) + \mathrm{Err}(x). 
\end{equation*}
For $x \in B(x_j, R)$ the term $\mathrm{Err}(x)$ is negligible. Indeed, by Plancherel's theorem 
\begin{equation*}
\mathrm{Err}(x) = \int_{\R^{n-1}} \overline{\check{G}_x(z)}\cdot (1-\eta_{R_1}(z))[e^{2\pi i \phi^{\lambda}(x_j;\,\cdot\,)}\psi f]\;\widecheck{}\;(z)\,\ud z 
\end{equation*}
where
\begin{equation*}
\check{G}_x(z) = \int_{\R^{n-1}}e^{2\pi i(\langle z, \omega \rangle - \phi^{\lambda}(x;\omega) + \phi^{\lambda}(x_j;\omega))} a^{\lambda}(x;\omega)\,\ud \omega.
\end{equation*} 
Taking the $\omega$-derivatives of the phase of $\check{G}_x(z)$, one obtains
\begin{align*}
z - \lambda\big(\partial_{\omega}\phi(x/\lambda;\omega) - \partial_{\omega}\phi(x_j/\lambda;\omega)\big) &= z + O(R), \\
-\lambda\big(\partial_{\omega}^{\alpha}\phi(x/\lambda;\omega) - \partial_{\omega}^{\alpha}\phi(x_j/\lambda;\omega)\big) &= O(R) \qquad \textrm{for $|\alpha| \geq 2$.}
\end{align*}
Thus, if $z$ belongs to the support of $1 - \eta_{R_1}$, then integration-by-parts (see Lemma~\ref{integration-by-parts lemma}) shows that $G_x(z)$ is rapidly decaying in $R_1$ and therefore
\begin{equation*}
|\mathrm{Err}(x)| \leq \mathrm{RapDec}(R_1)\|f\|_{L^p(B^{n-1})}.
\end{equation*}

It remains to bound the contributions arising from the frequency localised pieces. By applying the estimate for $T^{\lambda}$ with $R^{\varepsilon}$-loss over each ball $B(x_j, R)$ one obtains 
\begin{align*}
\|T^{\lambda}f\|_{L^p(S)} &\leq \big(\sum_{j=1}^N \|T^{\lambda}[e^{-2\pi i\phi^{\lambda}(x_j;\,\cdot\,)}P_j f]\|_{L^p(B(x_j, R))}^p \big)^{1/p} + \mathrm{RapDec}(R_1)\|f\|_{L^p(B^{n-1})} \\
 &\lesssim_{\varepsilon, \phi, a} R^{\varepsilon} \big(\sum_{j=1}^N \|P_j f\|_{L^p(B^{n-1})}^p \big)^{1/p} + \|f\|_{L^p(B^{n-1})}.
\end{align*}
Thus, it now suffices to show that
\begin{equation*}
\big(\sum_{j=1}^N \| P_j f\|_{L^p(\R^{n-1})}^p \big)^{1/p} \lesssim \|f\|_{L^p(B^{n-1})}.
\end{equation*}
This estimate follows via interpolation between the endpoint cases $p = 2$ and $p= \infty$, which are established presently. The $p= \infty$ case is a trivial consequence of Young's inequality and so it suffices to consider $p=2$. By duality, the desired inequality is equivalent to 
\begin{equation}\label{vector valued estimate}
\big\|\sum_{j=1}^N e^{-2\pi i \phi^{\lambda}(x_j; \,\cdot\,)} \psi \cdot[\hat{\eta}_{R_1} \ast g_j]\big\|_{L^2(\R^{n-1})} \lesssim \big(\sum_{j=1}^N \|g_j\|_{L^2(B^{n-1})}^2 \big)^{1/2}.
\end{equation}
 By squaring the left-hand side of \eqref{vector valued estimate} one obtains
\begin{equation*}
\sum_{j_1, j_2 = 1}^{N} \int_{\R^{n-1}} \overline{G_{j_1,j_2}(\omega)} \hat{\eta}_{R_1} \ast g_{j_1}(\omega)\overline{\hat{\eta}_{R_1} \ast g_{j_2}(\omega)}\,\ud \omega
\end{equation*}
where 
\begin{equation*}
 G_{j_1,j_2}(\omega) := e^{2\pi i (\phi^{\lambda}(x_{j_1}; \omega) - \phi^{\lambda}(x_{j_2}; \omega))} \psi(\omega)^2.
\end{equation*}
By Plancherel's theorem, each summand of the above expression can be written as
\begin{equation}\label{off diagonal term}
 \int_{\R^{n-1}} \overline{\check{G}_{j_1,j_2}(z)} (\eta_{R_1}\check{g}_{j_1}) \ast (\eta_{R_1}\check{g}_{j_2})^{\sim}(z) \, \ud z;
\end{equation}
here $(\eta_{R_1}\check{g}_{j_2})^{\sim}(z) := \overline{(\eta_{R_1}\check{g}_{j_2})(-z)}$. Note that the integrand in \eqref{off diagonal term} is supported on a ball of radius $O(R_1)$ about the origin. 

Fix $1 \leq j_1, j_2 \leq N$ with $j_1 \neq j_2$, let $z \in \R^{n-1}$ with $|z| \lesssim R_1 < |x_{j_2} - x_{j_1}|$ and consider
\begin{equation*}
 \check{G}_{j_1,j_2}(z) = \int_{\R^{n-1}} e^{2\pi i (\langle z, \omega \rangle + \phi^{\lambda}(x_{j_1}; \omega) - \phi^{\lambda}(x_{j_2}; \omega))} \psi(\omega)^2\,\ud \omega.
\end{equation*}
This oscillatory integral can be bounded by a simple stationary phase analysis. For $\alpha \in \N^{n-1}$ with $|\alpha| \leq 2$ consider the function
\begin{equation*}
\partial_{\omega}^{\alpha} \big[\phi^{\lambda}(x_{j_1}; \omega) - \phi^{\lambda}(x_{j_2}; \omega)\big] = \partial_{\omega}^{\alpha}\langle \partial_x^{\lambda}\phi(x_{j_1}; \omega), x_{j_2} - x_{j_1} \rangle + O(c_{\mathrm{diam}}|x_{j_2} - x_{j_1}|), 
\end{equation*}
where the remainder term has been estimated using \eqref{not too bad separation}. 

Let $c_{\mathrm{crit}} > 0$ be another small constant, chosen to satisfy the requirements of the forthcoming argument, and $\omega_0 \in \Omega$.  Suppose that
\begin{equation}\label{critical point}
 \big| \pm \frac{x_{j_2} - x_{j_1}}{|x_{j_2} - x_{j_1}|} - G^{\lambda}(x_{j_1};\omega_0)\big| \geq c_{\mathrm{crit}},
\end{equation}
where the estimate is interpreted as holding for both choices of sign. Condition H1) on the phase implies that for each $\omega_0 \in \Omega$ the vector $G^{\lambda}(x;\omega_0)$ spans the kernel of $\partial_{\omega x}^2\phi^{\lambda}(x;\omega_0)$. Consequently, in view of \eqref{critical point} one has
\begin{equation*}
 |\partial_{\omega}\big[\langle \partial_x\phi^{\lambda}(x_{j_1}; \omega), x_{j_2} - x_{j_1} \rangle\big]|_{\omega = \omega_0}| \gtrsim |x_{j_2} - x_{j_1}|
\end{equation*}
and therefore
\begin{equation*}
 \big|\partial_{\omega} \big[\phi^{\lambda}(x_{j_1}; \omega) - \phi^{\lambda}(x_{j_2}; \omega)\big]|_{\omega = \omega_0} \big| \gtrsim |x_{j_2} - x_{j_1}|, 
\end{equation*}
provided $c_{\mathrm{diam}}$ is sufficiently small. On the other hand, if \eqref{critical point} fails, then
\begin{equation*}
\partial_{\omega}^{\alpha}\Big\langle \partial_x\phi^{\lambda}(x_{j_1}; \omega), \frac{x_{j_2} - x_{j_1}}{|x_{j_2} - x_{j_1}|}\Big\rangle|_{\omega = \omega_0} = \partial_{\omega}^{\alpha}\langle \partial_x\phi^{\lambda}(x_{j_1}; \omega), G^{\lambda}(x;\omega_0) \rangle|_{\omega = \omega_0} + O(c_{\mathrm{crit}}).
\end{equation*}
If $c_{\mathrm{crit}}$ and $c_{\mathrm{diam}}$ are both chosen to be sufficiently small, then condition H2) implies that 
\begin{equation*}
 |\det \partial_{\omega\omega}^{2} \big[\phi^{\lambda}(x_{j_1}; \omega) - \phi^{\lambda}(x_{j_2}; \omega)\big]|_{\omega = \omega_0}| \gtrsim |x_{j_2} - x_{j_1}|^{n-1}.
\end{equation*}
Thus, any critical point of the phase must be (quantitatively) non-degenerate and one may apply higher dimensional versions of van der Corput's lemma (see, for instance, Chapter VIII, Proposition 6 of \cite{Stein1993}) to estimate the oscillatory integral. In particular,
\begin{equation*}
|\check{G}_{j_1,j_2}(z)| \lesssim |x_{j_2} - x_{j_1}|^{-(n-1)/2} \lesssim R_1^{-\bar{C}/2}
\end{equation*}
so that the absolute value of \eqref{off diagonal term} is bounded by
\begin{align*}
R_1^{-\bar{C}/2} \|(\eta_{R_1}\check{g}_{j_1}) \ast (\eta_{R_1}\check{g}_{j_2})^{\sim}\|_{L^1(\R^{n-1})} &\lesssim R_1^{-\bar{C}/2} \prod_{i=1}^2\|\eta_{R_1}\check{g}_{j_i}\|_{L^1(\R^{n-1})} \\
&\lesssim R_1^{-\bar{C}/2+n-1} \prod_{i=1}^2\|g_{j_i}\|_{L^2(\R^{n-1})}.
\end{align*}
Since there are only $O(N^2)$ choice of indices $j_1, j_2$, one may invoke the trivial estimate 
\begin{equation*}
\prod_{i=1}^2\|g_{j_i}\|_{L^2(\R^{n-1})} \lesssim \sum_{j=1}^N\| g_j\|_{L^2(B^{n-1})}^2
\end{equation*}
and then sum all the contributions from all pairs $j_1, j_2$ to bound the off-diagonal terms arising from the left-hand side of \eqref{vector valued estimate}. On the other hand, the diagonal terms provide a favourable contribution of 
\begin{equation*}
\big(\sum_{j=1}^N\big\|\hat{\eta}_{R_1} \ast g_j\big\|_{L^2(B^{n-1})}^2\big)^{1/2} \lesssim \big(\sum_{j=1}^N\| g_j\|_{L^2(B^{n-1})}^2\big)^{1/2}.
\end{equation*}
Combining these observations concludes the proof of \eqref{vector valued estimate} and thereby establishes the lemma.
\end{proof}




\appendix

\section{The integration-by-parts argument}

In this appendix further details of the integration-by-parts argument frequently used in the paper are presented.

\begin{lemma}\label{integration-by-parts lemma} Let $\phi \in C^{\infty}(\R^n)$ be real valued and $a \in C^{\infty}(\R^n)$ supported in $B^n$. Suppose that for some $\lambda, M \geq 1$ and $N \in \N$ and all $z \in \mathrm{supp}\,a$ these functions satisfy the following conditions:
\begin{enumerate}[i)]
    \item $|\partial_{z} \phi(z)| \geq \lambda$,
    \item $|\partial_z^{\alpha} \phi(z)| \leq M|\partial_z \phi(z)|$ for all $\alpha \in \N_0^n$ with $2 \leq |\alpha| \leq N$,
    \item $|\partial_z^{\alpha} a(z)| \leq M^{|\alpha|}$ for all $\alpha \in \N_0^n$ with $|\alpha| \leq N$.
\end{enumerate}
Then
\begin{equation*}
    \Big|\int_{\R^n} e^{i \phi(z)}a(z)\,\ud z \Big| \lesssim_N M^{N}\lambda^{-N}.
\end{equation*}
\end{lemma}

The lemma is a standard application of integration-by-parts and the Leibniz rule. Nevertheless, the details of the proof are provided for completeness. 

\begin{proof}[Proof (of Lemma~\ref{integration-by-parts lemma})] Define $Q \colon \R^n \to \R$ by $Q(z) := |\partial_z \phi(z)|^2$ and consider the mutually adjoint\footnote{In the sense that $\int_{\Omega} (D u) v = \int_{\Omega}  u (D^*v)$ whenever at least one of the functions $u, v \in C^{\infty}(\Omega)$ has compact support.} differential operators 
\begin{equation*}
    D u := \frac{\langle \partial_z u, \partial_z \phi \rangle }{iQ}, \qquad D^*u := i \sum_{k=1}^{n}\partial_{z_k} \big[(\partial_{z_k}\phi) Q^{-1} u \big].
\end{equation*}
Note that $D$ fixes the function $e^{i\phi}$ and, consequently,
\begin{equation*}
    \int_{\R^n} e^{i \phi(z)}a(z)\,\ud z =  \int_{\R^n} \big[ D^Ne^{i \phi(z)}\big]a(z)\,\ud z = \int_{\R^n} e^{i \phi(z)}(D^*)^Na(z)\,\ud z.
\end{equation*}
Thus, it suffices to show that 
\begin{equation*}
    |(D^*)^Na(z)| \lesssim_N M^N \lambda^{-N}. 
\end{equation*}

It is useful to work with the more general statement 
\begin{equation*}
    |\partial_z^{\alpha} (D^*)^\mu a(z)| \lesssim_{N,\alpha} M^{\mu + |\alpha|}\lambda^{-\mu} \qquad \textrm{for all $\mu \in \N_0$, $\alpha \in \N_0^n$ satisfying $\mu + |\alpha| \leq N$},
\end{equation*}
which is amenable to induction on $\mu$. The base case $\mu = 0$ follows directly from hypothesis iii). The inductive step is established by appropriate application of the Leibniz rule and the hypothesised bounds for $\phi$. 

Fix $0 \leq \mu \leq N-1$ and $\alpha \in \N_0^n$ such that $\mu + 1 + |\alpha| \leq N$. Denoting by $e_k$ the standard co-ordinate vectors for $k=1,\ldots,n$, it follows by the definition of $D^*$ and the Leibniz rule that
\begin{equation*}
  \partial_z^{\alpha} (D^*)^{\mu+1} a = i \sum_{k=1}^{n}\sum_{\beta \leq \alpha + e_k} \binom{\alpha + e_k}{\beta} \partial_z^{\alpha - \beta +e_k}\big[(\partial_{z_k}\phi) Q^{-1}\big] \partial_z^{\beta}(D^*)^{\mu} a
\end{equation*}
where, for every fixed $k$, the second sum is over all multi-indices $\beta \in \N_0^n$ satisfying $\beta_j \leq \alpha_j + \delta_{jk}$ for $1 \leq j \leq n$. For each such multi-index, $\mu + |\beta| \leq \mu + 1 + |\alpha| \leq N$ and therefore the induction hypothesis yields 
\begin{equation}\label{integration-by-parts 1}
    |\partial_z^{\beta}(D^*)^{\mu} a(z)| \lesssim_N M^{\mu+|\beta|}\lambda^{-\mu}.
\end{equation}
On the other hand, condition ii) together with the Leibniz rule implies that
\begin{equation*}
    |\partial_z^{\gamma}Q^{-1}(z)| \lesssim_{N} M^{|\gamma|}|Q(z)|^{-1} \qquad \textrm{for all $\gamma \in \N_0^n$ with $|\gamma| \leq N$.}
\end{equation*}
Thus, again using ii) and the Leibniz rule,
\begin{equation*}
    \big|\partial_z^{\gamma} \big[(\partial_{z_k}\phi(z))Q^{-1}(z)\big]\big| \lesssim_{N} M^{|\gamma|}|\partial_z\phi(z)|^{-1} \leq M^{|\gamma|}\lambda^{-1},
\end{equation*}
where the last step is by i). Applying the above estimate with $\gamma = \alpha - \beta + e_k$ and combining this with \eqref{integration-by-parts 1}, one deduces that
\begin{equation*}
      |\partial_z^{\alpha} (D^*)^{k+1} a(z)| \lesssim_N 
      \sum_{k=1}^{n}\sum_{\beta \leq \alpha + e_k}  M^{|\alpha| - |\beta| + 1}\lambda^{-1} M^{\mu+|\beta|}\lambda^{-\mu} \lesssim_N M^{\mu + 1 + |\alpha|}\lambda^{-(\mu+1)},
\end{equation*}
which closes the induction. 
\end{proof}




\bibliography{Reference}
\bibliographystyle{amsplain}

\end{document}